\pdfoutput=1
\documentclass[11pt,twoside]{article}

\usepackage{graphicx, amsmath, amsthm, amssymb, amscd, mathtools}

\usepackage{dsfont}

\usepackage[usetoc,nostar]{titleref}

\newcommand{\oddheader}{Deception, Delay, and Detection of Strategies}

\newcommand{\ifig}[2]{\includegraphics[#2]{#1-eps-converted-to.pdf}}
\newcommand{\fig}[1]{Figure~\ref{#1}}

\newcommand{\vso}{\vspace*{1pt}}
\newcommand{\vst}{\vspace*{2pt}}
\newcommand{\vsr}{\vspace*{3pt}}

\newcommand{\vzsp}{\kern .02em}
\newcommand{\vtsp}{\kern .08em}
\newcommand{\vmsp}{\kern .12em}
\newcommand{\vlsp}{\kern .15em}

\newcommand{\hspc}[0]{\hbox{\hspace*{0.2pt}}}
\newcommand{\hspd}[0]{\hbox{\hspace*{0.4pt}}}
\newcommand{\hsph}[0]{\hbox{\hspace*{0.5pt}}}
\newcommand{\hsps}[0]{\hbox{\hspace*{0.6pt}}}
\newcommand{\hspq}[0]{\hbox{\hspace*{0.75pt}}}
\newcommand{\hspt}[0]{\hbox{\hspace*{1pt}}}

\newtheorem{theorem}{Theorem}
\newtheorem{lemma}[theorem]{Lemma}
\newtheorem{corollary}[theorem]{Corollary}
\newtheorem{construction}[theorem]{Construction}
\newtheorem{definition}[theorem]{Definition}

\newcommand\mydefem[1]{{\underline{\em #1}}}

\newcommand{\setdef}[2]{{\left\lbrace #1\;\left\vert\;#2 \right.\right\rbrace}}

\newcommand{\biginter}{\bigcap}
\newcommand{\bigunion}{\bigcup}

\newcommand{\sdiff}[2]{#1 \setminus #2}

\newcommand{\abs}[1]{\lvert#1\rvert}
\newcommand{\bigabs}[1]{\big\lvert#1\big\rvert}
\newcommand{\matchabs}[1]{\left\lvert#1\right\rvert}

\newcommand{\yless}{\mskip-1mu}
\newcommand{\xadj}{\mskip-1.8mu}
\newcommand{\XxY}{{X}\xadj\times{Y}}

\newcommand{\one}{\bullet}

\newcommand{\OTw}{{{\tt H}\rightarrow{\tt L}}}
\newcommand{\OTh}{{{\tt H}\rightarrow{\tt R}}}
\newcommand{\TwF}{{{\tt L}\rightarrow{\tt P}}}
\newcommand{\ThF}{{{\tt R}\rightarrow{\tt P}}}
\newcommand{\TwTh}{{{\tt L}\rightarrow{\tt R}}}
\newcommand{\ThTw}{{{\tt R}\rightarrow{\tt L}}}
\newcommand{\AB}{{{\tt A}\rightarrow{\tt B}}}

\newcommand{\tH}{{\tt H}}
\newcommand{\tL}{{\tt L}}
\newcommand{\tR}{{\tt R}}
\newcommand{\tP}{{\tt P}}

\newcommand{\tu}{{\tt u}}
\newcommand{\td}{{\tt d}}
\newcommand{\tf}{{\tt f}}
\newcommand{\tm}{{\tt m}}

\newcommand{\nvtsp}{\mskip -1mu}
\newcommand{\ntsp}{\mskip -3mu}

\newcommand{\bonespc}{\vtsp{}b_1}

\newcommand{\inter}{\cap}
\newcommand{\union}{\cup}

\newcommand{\calA}{{\mathcal{A}}}
\newcommand{\calB}{{\mathcal{B}}}
\newcommand{\calC}{{\mathcal{C}}}

\newcommand{\calE}{{\mathcal{E}}}
\newcommand{\calF}{{\mathcal{F}}}
\newcommand{\calK}{{\mathcal{K}}}
\newcommand{\frakA}{{\mathfrak{A}}}
\newcommand{\frakB}{{\mathfrak{B}}}
\newcommand{\frakC}{{\mathfrak{C}}}
\newcommand{\frakD}{{\mathfrak{D}}}
\newcommand{\frakE}{{\mathfrak{E}}}

\newcommand{\CH}{{\mathfrak{C}_H}}
\newcommand{\CHp}{{\mathfrak{C}_{H^*}}}
\newcommand{\CyN}{{\mathcal{C}_N}}
\newcommand{\CyH}{{\mathcal{C}_H}}
\newcommand{\Cp}{{\mathfrak{C}}}

\newcommand{\expansive}{\calE}
\newcommand{\outside}{\calE^{-}}
\newcommand{\exa}{e}
\newcommand{\minexp}{\varrho}

\newcommand{\maction}[2]{$#1\ntsp\rightarrow\ntsp #2$}

\newcommand{\src}{\mathop{\rm src}}
\newcommand{\trg}{\mathop{\rm trg}}

\newcommand{\wrep}{\diamond}
\newcommand{\wrepone}{\Diamond}
\newcommand{\wreptwo}{\Box}

\newcommand{\DG}{\Delta_G}
\newcommand{\DTH}{\Delta_H}
\newcommand{\maxGam}{{\mathfrak{M}}}
\newcommand{\maxmarked}{{\mathcal{M}}}

\newcommand{\DGW}{\Delta_{G/W}}
\newcommand{\DHp}{\Delta_{H^*}}

\newcommand{\DHi}{\Delta_{H_{\scriptstyle i}}}
\newcommand{\DHim}{\Delta_{H_{\scriptstyle {i-1}}}}

\newcommand{\Go}{\overline{G}}
\newcommand{\DGo}{\Delta_{\Go}}
\newcommand{\sigo}{\overline{\sigma}}
\newcommand{\Ao}{\overline{A}}

\newcommand{\tauc}{{\tau_{\circ}}}
\newcommand{\taup}{{\tau_{+}}}
\newcommand{\taum}{{\tau_{-}}}
\newcommand{\taucp}{{\tau^{{\kern .1em}\prime}_{\circ}}}
\newcommand{\taupp}{{\tau^{{\kern .15em}\prime}_{+}}}
\newcommand{\taump}{{\tau^{{\kern .15em}\prime}_{-}}}

\newcommand{\tauz}{{\tau_{\star}}}
\newcommand{\tauzp}{{\tau^{{\kern .1em}\prime}_{\star}}}

\newcommand{\sphere}{\mathbb{S}}

\newcommand{\Sone}{\sphere^{1}}
\newcommand{\Stwo}{\sphere^{2}}
\newcommand{\Ssix}{\sphere^{6}}
\newcommand{\Snt}{\sphere^{\vtsp{n-2}}}

\newcommand{\dl}{\mathop{\rm dl}}

\newcommand{\dowx}{\Psi_R}
\newcommand{\dowy}{\Phi_R}

\newcommand{\clsy}{\phi_R \circ \psi_R}
\newcommand{\clsx}{\psi_R \circ \phi_R}

\newcommand{\dowAx}{\Psi_A}
\newcommand{\dowAy}{\Phi_A}
\newcommand{\Aoney}{A^{(1)}}
\newcommand{\dowAoney}{\Phi_{\mskip-2mu\Aoney}}

\newcommand{\clsAy}{\phi_A \circ \psi_A}

\newcommand{\gammacls}{\overline{\gamma}}

\title{Deception, Delay, and Detection\\[4pt]of Strategies}

\author{%
\parbox[c]{2.2in}{%
\begin{center}
Michael Erdmann\thanks{This report is based upon work supported in
part by the National Science Foundation under award number
IIS-1409003.  Any opinions, findings and conclusions or
recommendations expressed in this report are those of the author and
do not necessarily reflect the views of the Government or the National
Science Foundation.}\\ Carnegie Mellon University\\
June 27, 2019
\end{center}}}

\date{{\small \copyright\ 2019 Michael Erdmann}}

\begin{document}

\pagenumbering{roman}

\maketitle{}
\thispagestyle{empty}

\vspace*{-0.1in}
\begin{abstract}
Homology generators in a relation offer individuals the ability to
delay identification, by guiding the order via which the individuals
reveal their attributes \cite{paths:privacy}.  This perspective
applies as well to the identification of goal-attaining strategies in
systems with errorful control, since the strategy complex of a fully
controllable nondeterministic or stochastic graph is homotopic to a
sphere.  Specifically, such a graph contains for
each state $v$ a maximal strategy $\sigma_v$ that converges to state
$v$ from all other states in the graph and whose identity may be
shrouded in the following sense: One may reveal certain actions of
$\sigma_v$ in a particular order so that the full strategy becomes
known only after at least $n-1$ of these actions have been revealed,
with none of the actions revealed definitively inferable from those
previously revealed.  Here $n$ is the number of states in the graph.
Moreover, the strategy contains at least $(n-1)!$ such {\em
informative action release sequences}, each of length at least $n-1$.

The earlier work described above sketched a proof that {\em every\,}
maximal strategy in a {\em pure nondeterministic\,} or {\em pure
stochastic\,} graph contains {\em at least one\,} informative action
release sequence of length at least $n-1$.  The primary purpose of the
current report is to fill in the details of that sketch.  To build
intuition, the report first discusses several simpler examples.  These
examples suggest an underlying structure for hiding capabilities or
bluffing capabilities, as well as for detecting such deceit.
\end{abstract}

\vspace*{0.2in}
\markright{Contents}

\tableofcontents

\cleardoublepage

\pagenumbering{arabic}

\section{Introductory Examples}
\label{intro}

\subsection{Paths and Constituent Transitions}
\markright{Paths and Constituent Transitions}

\label{lake}

\fig{LakeFig} shows four islands connected by bridges, as might be
found in one of the great oceanic cities of the world.  One of the
bridges allows traffic in two directions, the others are one-way
bridges.  Of interest are the possible paths a bus of tourists or the
motorcade of a prominent dignitary might take from the {\sc Hotel
Island} to the {\sc Palace Island}, via one or two intermediary
islands (the {\sc Left Island} and/or the {\sc Right Island}).

\begin{figure}[h]
\begin{center}
\ifig{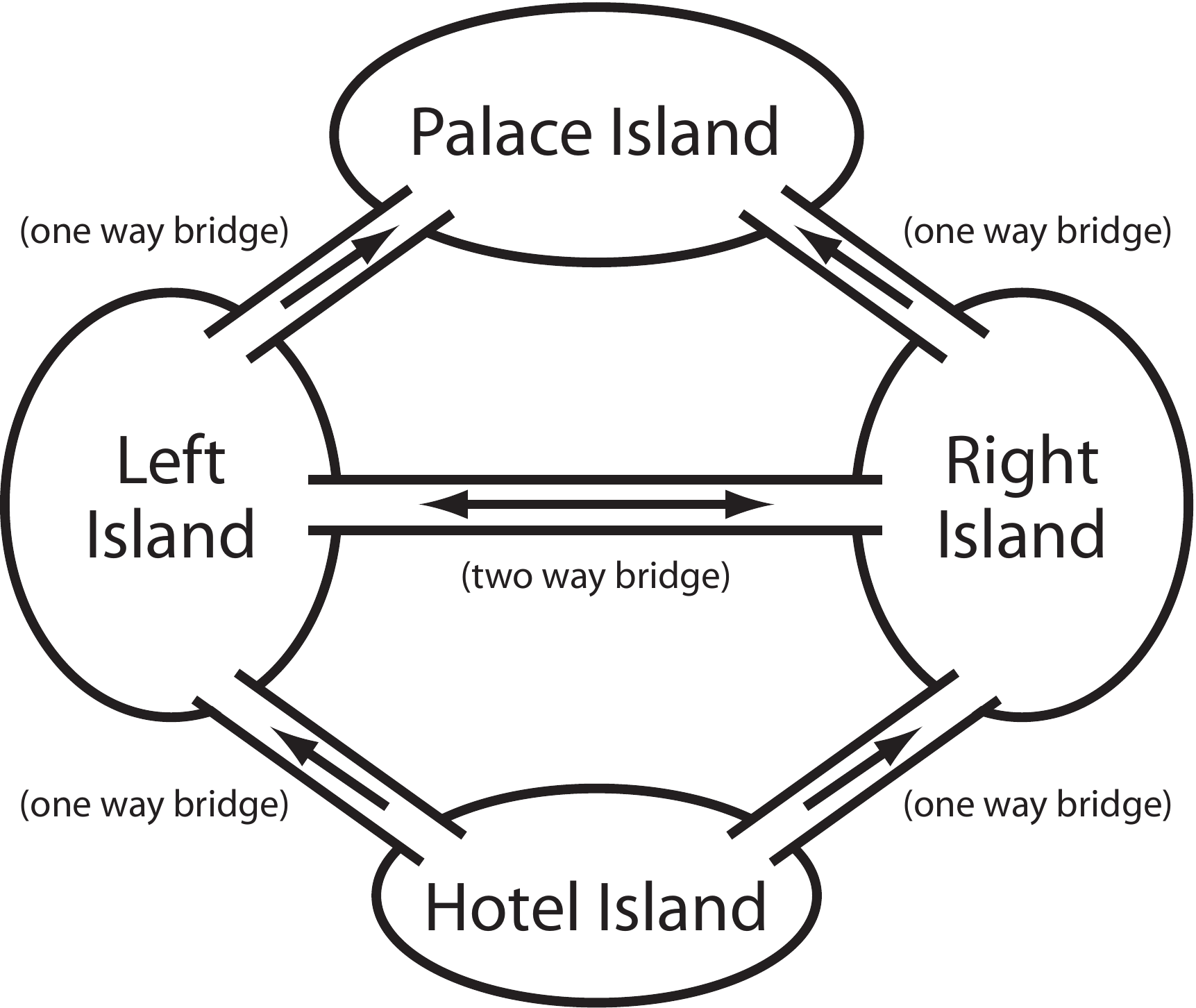}{scale=0.55}
\end{center}
\vspace*{-0.1in}
\caption[]{Four islands along with some directional bridges connecting
  the islands.}
\label{LakeFig}
\end{figure}

Since the bridge between {\sc Left Island} and {\sc Right Island} is
bidirectional, there are infinitely many such paths, parameterized by
the number of times the bus or motorcade cycles over the two-way
bridge.  For the purposes of this report, we will disallow such
infinite cycling.  One can imagine different restrictions.
In this first example, we impose the restriction that a vehicle may
traverse the two-way bridge at most once in each of its possible
directions (perhaps for legal or monetary reasons).  In a more general
setting, we would disallow traversing any bridge in the same direction
more than once.  Later, in Section~\ref{river}, we will discuss a
different example with a different restriction that prevents infinite
cycling.

We define a {\vtsp\em permissible path\vlsp} to be any path that a
vehicle might take from {\sc Hotel Island} to {\sc Palace Island},
subject to the ``no directional transition twice'' restriction.  The
next page enumerates all permissible paths; there are six.
For clarity, we abbreviate each island name to its first letter and
give paths the names $\pi_1, \ldots, \pi_6$.  Throughout this
subsection, we consider only these six paths, each of which starts at
{\sc Hotel Island} and ends at {\sc Palace Island}.

$$\begin{matrix}
\hbox{$\pi_1$:\ } & \tH & \rightarrow & \tL & \rightarrow & \tP & & & & \\[1pt]
\hbox{$\pi_2$:\ } & \tH & \rightarrow & \tR & \rightarrow & \tP & & & & \\[1pt]
\hbox{$\pi_3$:\ } & \tH & \rightarrow & \tL & \rightarrow & \tR & \rightarrow & \tP & & \\[1pt]
\hbox{$\pi_4$:\ } & \tH & \rightarrow & \tR & \rightarrow & \tL & \rightarrow & \tP & & \\[1pt]
\hbox{$\pi_5$:\ } & \tH & \rightarrow & \tL & \rightarrow & \tR & \rightarrow & \tL & \rightarrow & \tP\\[1pt]
\hbox{$\pi_6$:\ } & \tH & \rightarrow & \tR & \rightarrow & \tL & \rightarrow & \tR & \rightarrow & \tP\\[1pt]
\end{matrix}\label{permissiblepaths}$$

\vspace*{0.1in}

\fig{LakeGraph} describes the islands and bridges of \fig{LakeFig} as
a directed graph, and the six permissible paths as a relation.  The
relation has a row for each permissible path and a column for each
directed edge in the graph, that is, for each directional transition
across a bridge.  Since a permissible path may traverse any bridge
direction at most once, each permissible path defines a {\vtsp\em
set\vlsp} of directed edges, modeling all directional bridge
transitions in the path.  The relation therefore contains a nonblank
entry $\one$ for a given path $\pi_i$ and a given directed edge
$\AB\vtsp$ if and only if path $\pi_i$ includes transition $\AB$.

\begin{figure}[h]
\vspace*{0.1in}
\begin{center}
\begin{minipage}{1.5in}
\ifig{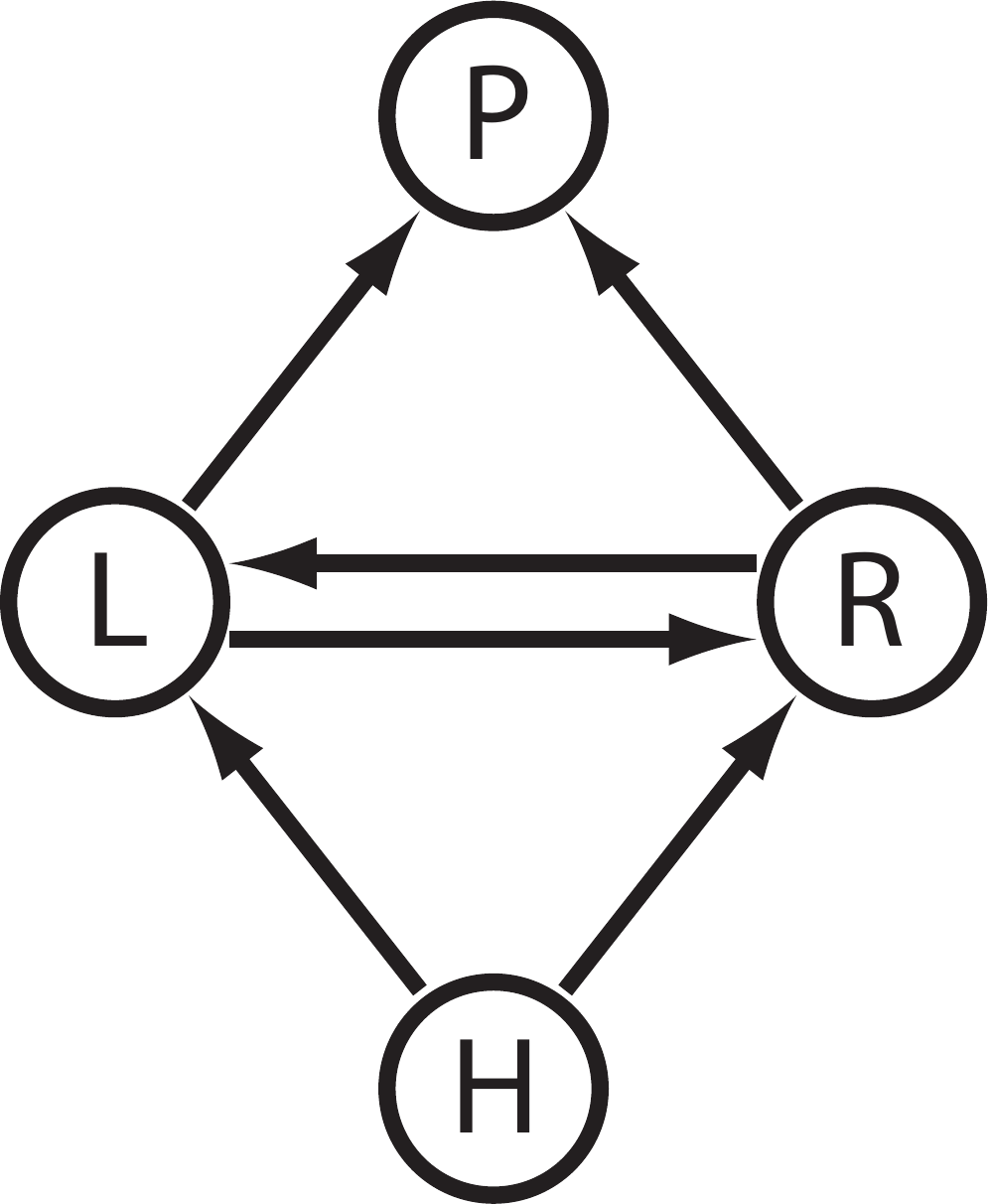}{scale=0.35}
\end{minipage}
\hspace*{0.65in}
\begin{minipage}{3.2in}{$\begin{array}{c|cccccc}
       & \OTw & \OTh & \TwF & \ThF & \TwTh & \ThTw \\[2pt]\hline
 \pi_1 & \one &      & \one &      &       &       \\[2pt]
 \pi_2 &      & \one &      & \one &       &       \\[2pt]
 \pi_3 & \one &      &      & \one & \one  &       \\[2pt]
 \pi_4 &      & \one & \one &      &       & \one  \\[2pt]
 \pi_5 & \one &      & \one &      & \one  & \one  \\[2pt]
 \pi_6 &      & \one &      & \one & \one  & \one  \\[2pt]
\end{array}$}
\end{minipage}
\end{center}
\vspace*{-0.15in}
\caption[]{{\bf Left Panel:} Directed graph representing the sketch of
  \fig{LakeFig}.  (Island names appear as first letter
  abbreviations.)\quad
  {\bf Right Panel:} A relation describing all paths leading from the
  {\sc Hotel Island} to the {\sc Palace Island}, while traversing any
  bridge direction at most once.}
\label{LakeGraph}
\end{figure}

\subsubsection*{Identifying Paths from Transitions at Execution Time}

Suppose an observer is watching a bus drive from the {\sc Hotel
Island} to the {\sc Palace Island}.  At what point during the trip can
the observer identify uniquely the specific path followed by the bus,
assuming the bus is traversing one of the permissible paths $\pi_1,
\pi_2, \pi_3, \pi_4, \pi_5$, or $\pi_6\vtsp$?

\begin{itemize}

\item Certainly, once the bus arrives at its destination, {\sc Palace
  Island}, the observer can identify the path uniquely, since at that
  point the observer knows that he/she has seen the entire path.

\item The observer cannot identify any path uniquely after observing
  only the first bridge transition.  For instance, after observing
  transition $\OTw$, the possible paths consistent with this
  observation are $\pi_1$, $\pi_3$, and $\pi_5$.  Similarly, after
  observing transition $\OTh$, the possible paths consistent with the
  observation are $\pi_2$, $\pi_4$, and $\pi_6$.

\item Consider paths $\pi_1$ and path $\pi_2$, each of which consists
  of two transitions.  By the previous point, the observer must see
  the entire path in order to identify either of these paths uniquely.

\item The first two transitions of paths $\pi_3$ and $\pi_5$ are the
  same, namely $\OTw$ and $\TwTh$.  Consequently, upon observing these
  transitions, the observer cannot identify a path uniquely; the path
  could be either $\pi_3$ or $\pi_5$.  Path $\pi_3$ consists of three
  transitions.  Thus, if the actual path is $\pi_3$, the observer must
  see the entire path before identifying the path as $\pi_3$.  A similar
  argument holds for path $\pi_4$.

\item Paths $\pi_5$ and $\pi_6$ contain four transitions.  For each of
  these two paths, the observer only needs to see the first three
  transitions in order to identify the path; no other permissible path
  shares those same three transitions with the path being observed.

\end{itemize}

In summary: Paths $\pi_1$, $\pi_2$, $\pi_3$, and $\pi_4$ can only be
identified uniquely after seeing all their transitions, assuming one
observes transitions in consecutive order.  Paths $\pi_5$ and $\pi_6$
can be identified uniquely after seeing the first three of their four
transitions, again assuming one observes transitions in consecutive
order.

\subsubsection*{Identifying Paths from Transitions in Arbitrary Order}

Previously we assumed that the observer was observing consecutive
motions of a bus.  Suppose now that the observer merely learns of
particular transitions made by the bus, {\em without\hspt} any explicit
ordering in time.  For instance, perhaps the observer is listening to
stories told by tourists on the bus after their trip, from which the
observer attempts to reconstruct the path taken.  Or perhaps the
observations are coming from many trips taken over the course of
several days by a bus following a particular fixed bus route each day.
Or perhaps the observer overhears the bus driver commenting on
particular bridges he will encounter on his next trip, from which the
observer is trying to predict the path yet to be taken.

\vspace*{0.1in}

We now ask: \ What {\em set\,} of transitions allows an observer to
identify a path uniquely?

\begin{itemize}

\item Recall that path $\pi_1$ consists of the set of transitions
  $\{\OTw,\, \TwF\}$.  Previously, when observing transitions in
  consecutive order, seeing both these transitions identified path
  $\pi_1$ uniquely.  That is no longer true when transitions may be
  observed nonconsecutively.  The reason is that path $\pi_5$ contains
  these same transitions, plus others.  In fact, it is {\em no longer
  possible} to identify path $\pi_1$ uniquely.  Similarly, it is no
  longer possible to identify path $\pi_2$ uniquely.

\item If the observer learns that a path contains the transitions
  $\OTw$ and $\ThF$, then the observer can infer that the path must
  also contain the transition $\TwTh$ and must in fact be path
  $\pi_3$.  Whereas previously an observer needed to see the entire
  path $\pi_3$ in order to identify it uniquely, now a pair of
  nonconsecutive transitions identifies the path.  In effect,
  continuity of paths allows the observer to infer an unobserved
  transition.  A similar argument holds for path $\pi_4$.  Of course,
  if a story teller wishes to draw out identification of the path,
  he/she might simply talk about the sights seen during the bus ride
  in consecutive order, thus preventing such a leap of inference.

\item If the observer learns that a path contains the transitions
  $\OTw$ and $\ThTw$, then the observer can actually infer two
  unobserved transitions, namely $\TwTh$ and $\TwF$, thereby
  concluding that the path is $\pi_5$.  A similar inference is
  possible for path $\pi_6$.  The observer is in effect taking
  advantage both of path continuity and knowledge of the path's
  destination.  Again, a story teller could draw out identification of
  path $\pi_5$ slightly by reporting transitions in consecutive order.

\end{itemize}

\fig{LakeComplexes} encodes these conclusions geometrically, using
simplicial complexes \cite{paths:munkres, paths:rotman,
paths:privacy}.  As in the relation of \fig{LakeGraph}, we now view
each path as a {\em set\vmsp} of directed edges.  These sets
constitute the generating simplices of the left simplicial complex
shown in \fig{LakeComplexes}.  The vertices in this complex are the
directed edges of the graph of \fig{LakeGraph}.  Path $\pi_1$
generates a one-dimensional simplex (edge) in the complex.  This
simplex is a subset of the three-dimensional simplex (tetrahedron)
generated by path $\pi_5$, modeling the earlier conclusion that one
cannot identify path $\pi_1$ uniquely when observing transitions in
arbitrary order.  Observe that the set $\{\OTw,\, \ThF\}$ is a free
face\footnote{Simplicial complexes in this report are {\em abstract},
i.e., collections of sets and all their subsets.  A simplex is a {\em
free face{\hspace*{0.8pt}}} of an abstract simplicial complex if it is
a proper subset of exactly one maximal simplex in the complex.} in the
complex and is not itself a path.  This geometry models the inference
and identification of path $\pi_3$ discussed previously.  Similarly,
the set $\{\OTw,\, \ThTw\}$ forms a free face in the complex, modeling
the inferences and identification of path $\pi_5$ discussed above.
(The set $\{\OTw,\, \ThTw\}$ is an undrawn ``diagonal'' of the
tetrahedron labeled $\pi_5$.)

\begin{figure}[t]
\begin{center}
\ifig{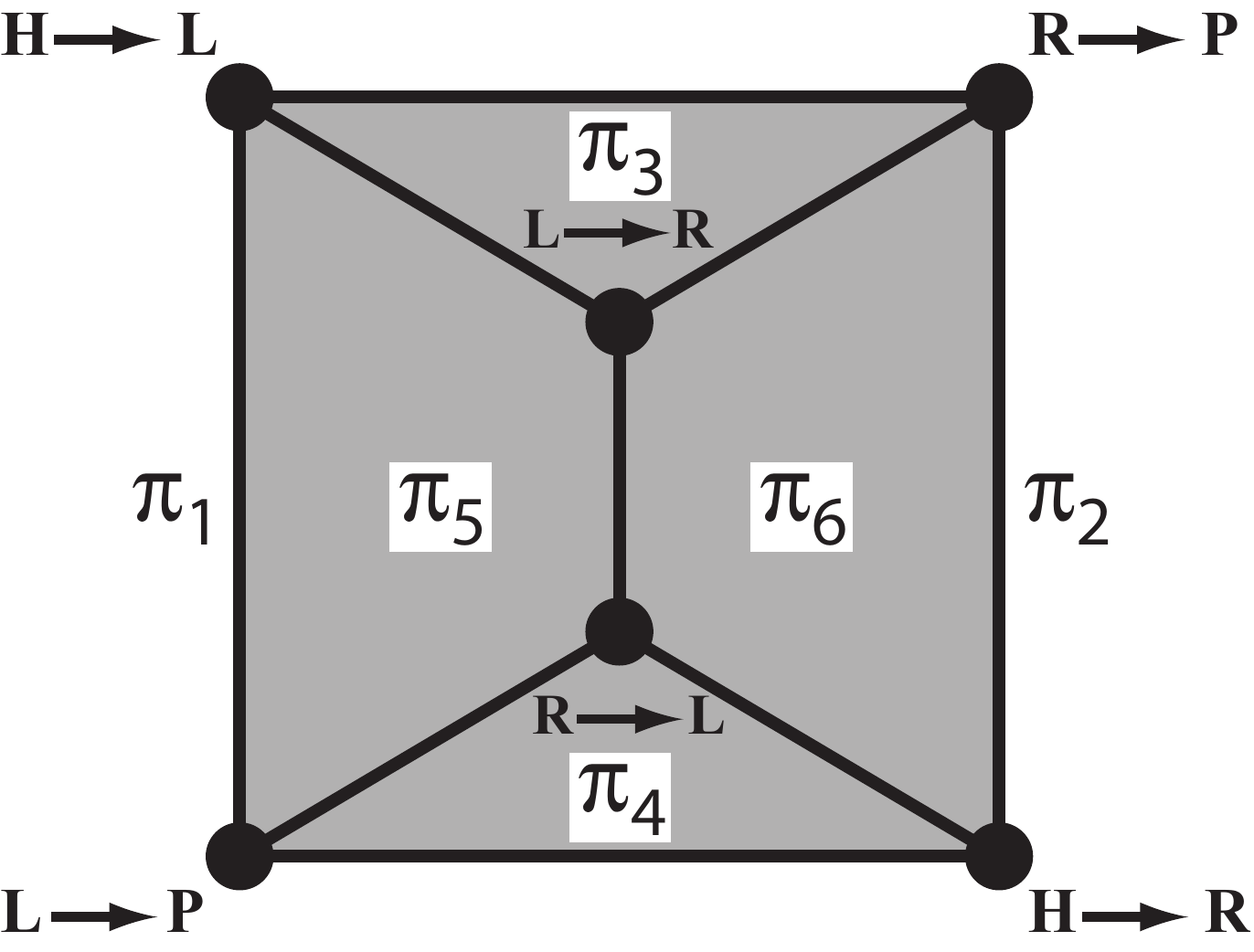}{scale=0.48}
\hspace*{0.3in}
\ifig{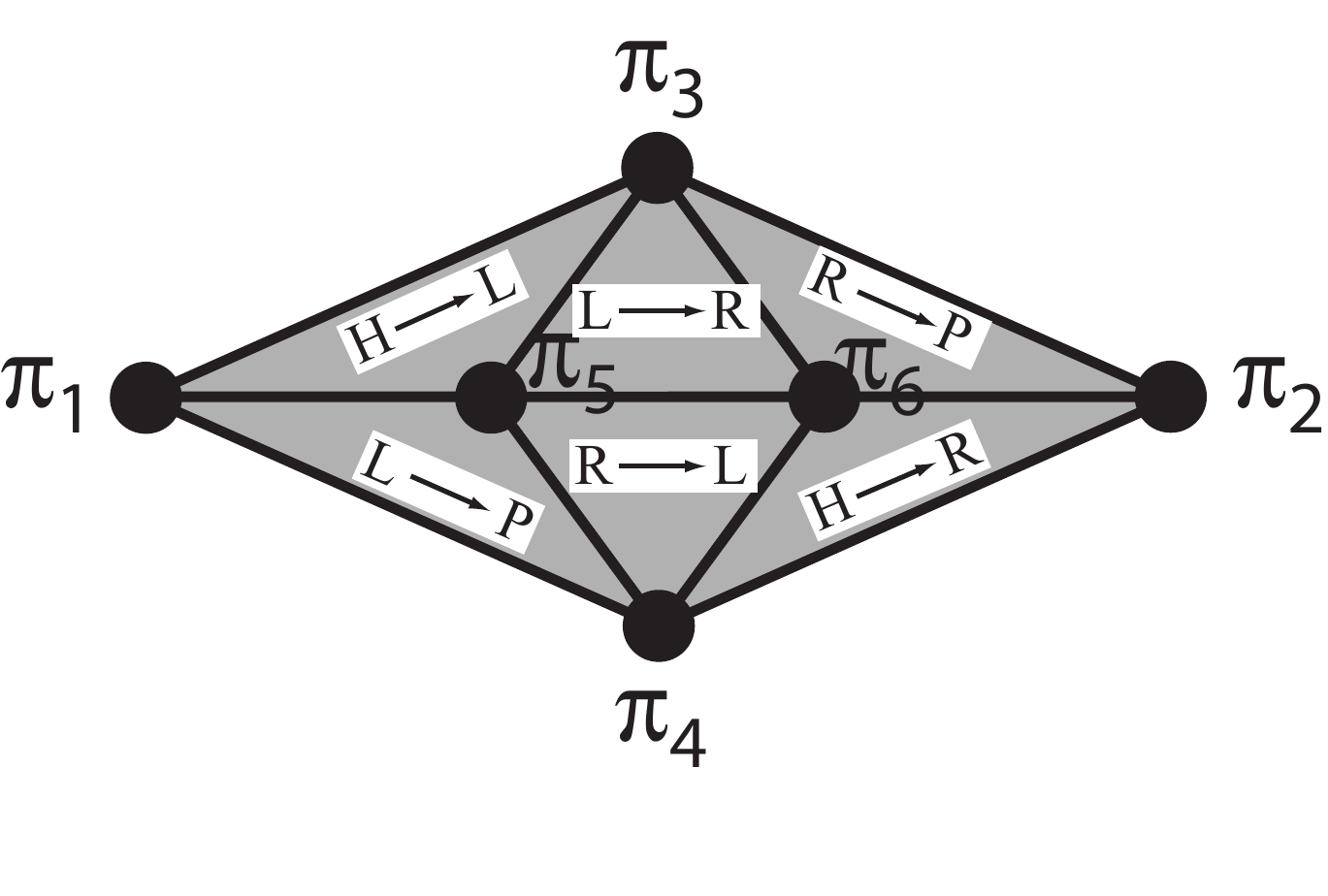}{scale=0.53}
\end{center}
\vspace*{-0.2in}
\caption[]{Two simplicial complexes derived from the relation of
  \fig{LakeGraph}.  The two complexes are Dowker dual
  \cite{paths:privacy} to each other with respect to that relation. \quad
  {\bf Left Panel:} The underlying vertex set of this complex is the
  collection of directed edges in the graph of \fig{LakeGraph}.
  The paths $\pi_1$, $\pi_2$, $\pi_3$, $\pi_4$, $\pi_5$, $\pi_6\mskip2.5mu$
  generate simplices as indicated by the path labels. (The two
  quadrilaterals are actually solid tetrahedra, flattened for ease of
  viewing in the figure.)\quad
  {\bf Right Panel:} The underlying vertex set of this complex is the
  collection of permissible paths in the graph of
  \fig{LakeGraph}.  Each maximal simplex in the complex is a
  triangle, reflecting the fact that each possible directed edge in
  the graph appears in three permissible paths. Each triangle is
  labeled with that directed edge.}
\label{LakeComplexes}
\end{figure}

The right simplicial complex of \fig{LakeComplexes} contains the same
information as the left complex, but in dual form.  The duality is
with respect to the relation of \fig{LakeGraph}.  (Details of such
``Dowker duality'' are discussed further in \cite{paths:privacy}.)

The vertices in the right complex are the permissible paths of the
graph of \fig{LakeGraph}.  The generating simplices of the complex are
given by the columns of the relation of \fig{LakeGraph}.  In other
words, each generating simplex consists of all the paths that share a
given directed edge.  Thus the right complex tells us how to interpret
observations of transitions as intersections of generating simplices.
For instance, if we know that a path contains the transitions $\OTw$
and $\TwF$, then we can intersect the triangle labeled with $\OTw$ and
the triangle labeled with $\TwF$ to see that the possible paths are
$\pi_1$ and $\pi_5$.  (This geometric intersection is exactly the
intersection of the two columns indexed by $\OTw$ and $\TwF$ in the
relation of \fig{LakeGraph}.)

\vspace*{0.05in}

Considering such intersections, our earlier observations
plus some others are immediate:

\vspace*{-0.05in}

\begin{itemize}
\addtolength{\itemsep}{-4pt}

\item $\pi_1$ and $\pi_2$ are not uniquely identifiable.
\item Observing $\OTw$ and $\ThF$ identifies path $\pi_3$.
\item Observing $\OTh$ and $\TwF$ identifies path $\pi_4$.
\item Observing $\OTw$ and $\ThTw$ identifies path $\pi_5$ (as does observing $\TwF$ and $\TwTh$).
\item Observing $\OTh$ and $\TwTh$ identifies path $\pi_6$ (as does observing $\ThF$ and $\ThTw$).
\end{itemize}

\vspace*{-0.05in}

(These are the smallest sets of identifying observations for each
path; there exist larger sets of observations as well.)

\clearpage

\subsection{Strategies and Underlying Capabilities}
\markright{Strategies and Underlying Capabilities}

\label{river}

\fig{Stream} shows a river with two islands, a fishing area upstream
of the islands, and a marina downstream from the islands.  The two
islands create three passages within the river that boats may
traverse, either upstream or downstream, as they move between the
fishing area and the marina.  The three passages produce currents of
different strengths.  One of these currents is so strong that only
boats with powerful motors are able to traverse the current going
upstream.  Fish in the fishing area like to gather near the start of
that strong current.  Consequently, boats with powerful motors have an
advantage reaching nice fish over boats with weaker motors.  On the
other hand, revealing that one has a powerful motor leads to envy and
other competitions.  As a result, skippers tend to underplay the power
of their motors.

\vst

We will examine the possible strategies for reaching the fishing area
(along with strategies for reaching the marina).  We will further
examine the extent to which someone can reveal portions of a strategy
without revealing the entire strategy.  Conversely, we will examine
the extent to which an observer can infer that a boat has a powerful
motor even when the observer never sees the boat traversing upstream
over the strong current.

\begin{figure}[h]
\begin{center}
\ifig{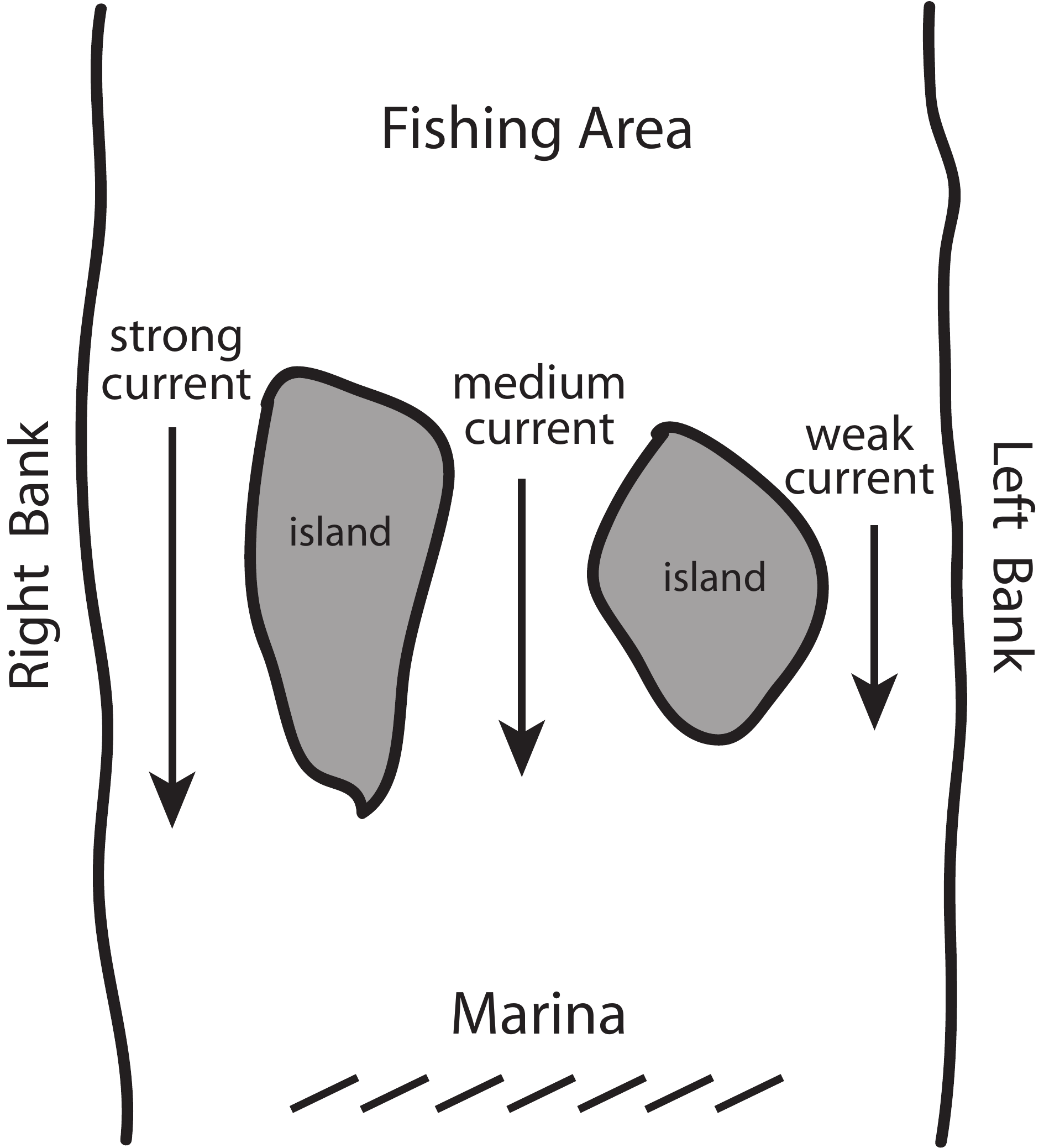}{scale=0.5}
\end{center}
\vspace*{-0.1in}
\caption[]{A river along with islands that create three passages and
  consequent currents of different strengths, with a fishing area
  upstream and a marina downstream.}
\label{Stream}
\end{figure}

\clearpage

\begin{figure}[t]
\begin{center}
\ifig{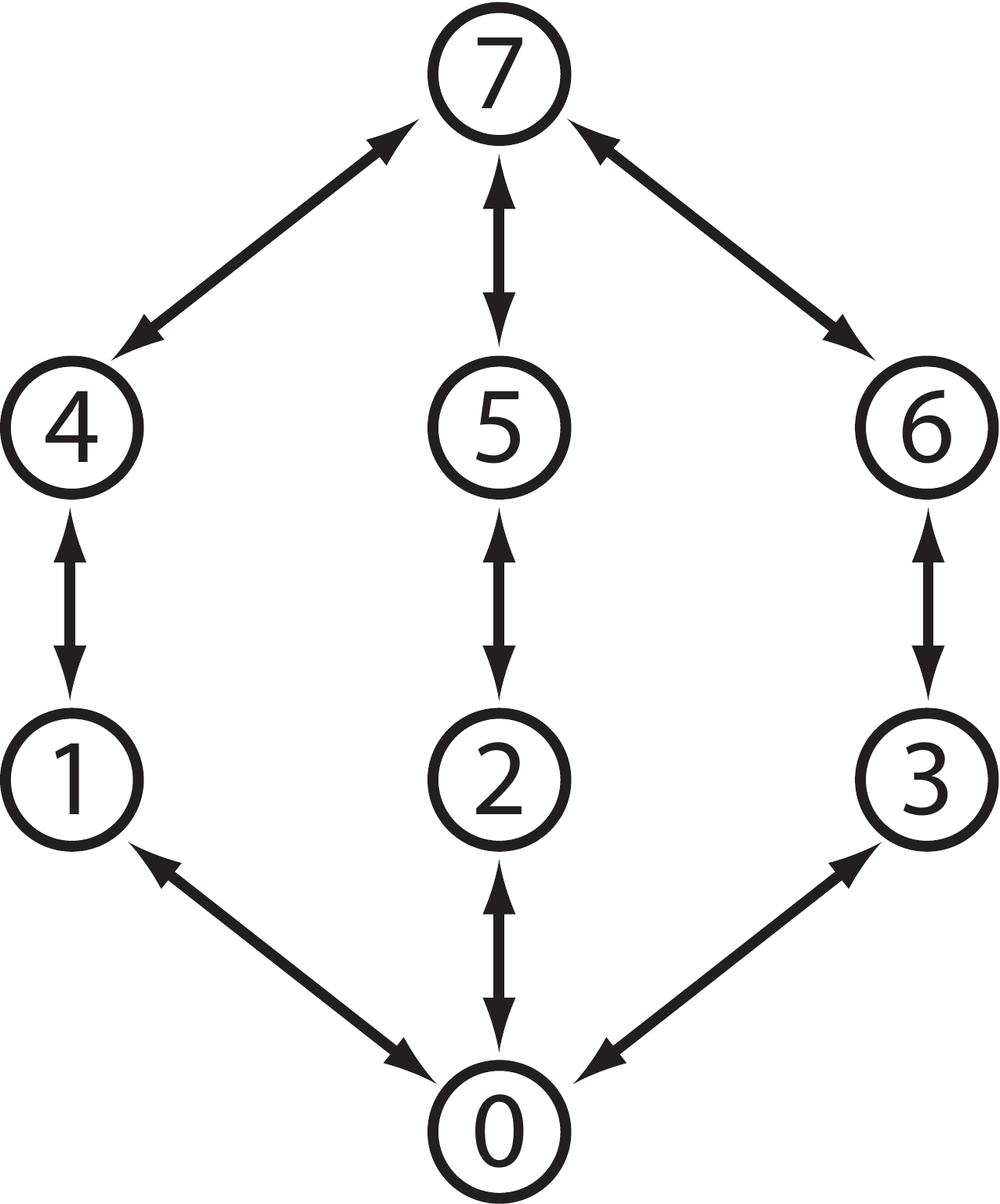}{scale=0.4}
\hspace*{0.675in}
\ifig{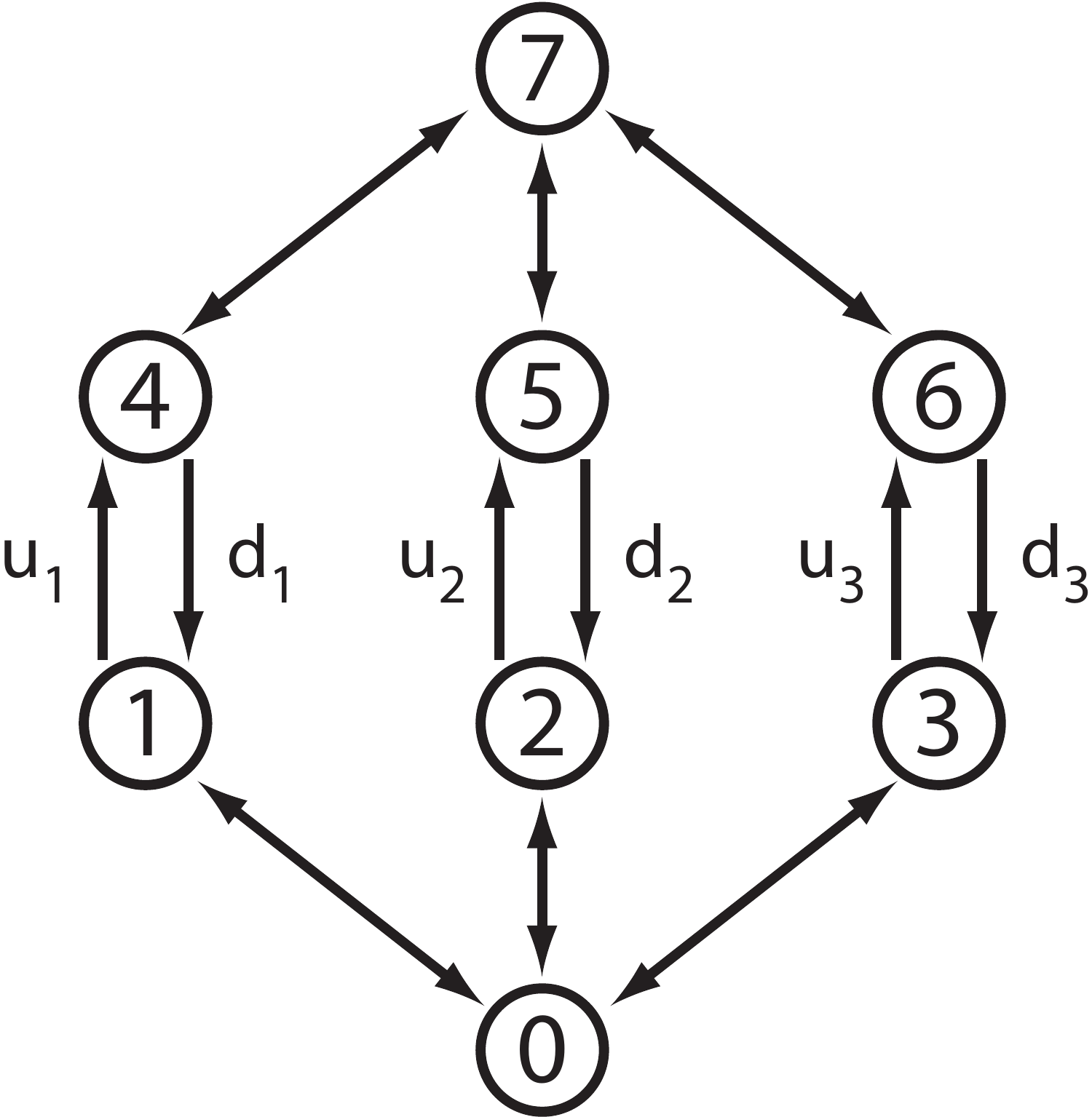}{scale=0.4}
\end{center}
\vspace*{-0.1in}
\caption[]{{\bf Left Panel:} A directed graph that describes the
  possible transitions a boat might make while traversing the river of
  \fig{Stream} between the marina and the fishing area.  (State
  \#0 represents the marina, while state \#7 represents the fishing
  area.)\\[1pt]
  {\bf Right Panel:} The same graph, now with the transitions that
  move upstream or downstream through the passages given explicit
  names.}
\label{StreamGraphs}
\end{figure}

\fig{StreamGraphs} describes the river setting of \fig{Stream} using a
directed graph, much as we did in the earlier example of
Section~\ref{lake}.  We will be interested primarily in the
transitions upstream and downstream through the passages beside the
islands, so we give those directed edges explicit names: $\tu_1$,
$\td_1$, $\tu_2$, $\td_2$, $\tu_3$, $\td_3$, as shown in the right
panel of the figure.

\vst

In the example of Section~\ref{lake}, we focused on paths.  In the
current example, we adopt a slightly different perspective.  We are
interested in a generalization of what is frequently called a {\em
control law}, namely a mapping from states to commanded motions.  The
generalization is that of a {\em strategy}, to be reviewed in
Section~\ref{strategies}.  A strategy is a mapping from states to {\em
sets\hsps} of possible motions, in this case sets of directed edges.
The semantics are as follows: When a boat is at a particular location,
a strategy specifies a set of directed edges leading from that
location to some neighboring locations.  The boat must move along some
one of those directions, with the particular direction determined
possibly by circumstance rather than chosen by the skipper.  (The
strategy specifies a set since sometimes the precise direction is not
so important as is a general direction. For instance, a boat with a
powerful motor that is currently at the marina might be instructed to
move toward any of the three passages in the river.  A corresponding
strategy would therefore include the set of transitions
$\{0\rightarrow 1,\, 0\rightarrow 2,\, 0\rightarrow 3\}$.)  If the set
specified for a particular location is empty, then the boat must stop
if it is at that location.

\vst

In the example of Section~\ref{lake}, we prevented infinite cycling by
disallowing any path that traversed any directed edge more than once.
With strategies, it is more natural to disallow any strategy whose
motion sets might cause the system to revisit a state.

\vst

There is a well-developed theory for strategies in graphs with
directed edges \cite{paths:bjorner-welker, paths:hultman,
paths:jonsson} as well as in graphs with nondeterministic and/or
stochastic transitions \cite{paths:strategies, paths:plans}.  One can
model the collection of all strategies as a simplicial complex,
similar to the constructions of Section~\ref{lake}.  In a directed
graph, a strategy is a set of directed edges that produces no cycle(s)
in the graph.  A graph's strategies constitute the simplices of a
simplicial complex whose underlying vertex set consists of the graph's
directed edges.  For a strongly connected directed graph, this
simplicial complex has the homotopy type of a sphere, namely $\Snt$,
with $n$ the number of states in the graph \cite{paths:hultman}.  For
the graph of \fig{StreamGraphs}, the simplicial complex is therefore
homotopic to $\Ssix$.

Since a sphere has homology, our prior work on privacy
\cite{paths:privacy} offers some lower bounds on how long a skipper
may delay identification of a strategy or a boat's final destination,
relative to all possible strategies in the graph of
\fig{StreamGraphs}.  However, rather than explore the entire space of
strategies, we will focus in this example on some simpler scenarios.

\subsubsection*{Strategies for Attaining the Fishing Area}

For the moment, let us consider only all maximal strategies that
ultimately attain the fishing area from anywhere in the graph.  (By a
{\em maximal strategy\hspace*{0.86pt}} we mean here a cycle-free set
of directed edges in the graph of \fig{StreamGraphs} that is maximal
among all such sets.)

\vspace*{-0.05in}

\begin{itemize}

\item Here is one such strategy, consisting of all possible upstream
  motions in the graph of \fig{StreamGraphs}:

\vspace*{-0.13in}

$$\sigma_{123} \;=\; 
  \{0\rightarrow 1,\; 0\rightarrow 2,\; 0\rightarrow 3,\; 
    1\rightarrow 4,\; 2\rightarrow 5,\; 3\rightarrow 6,\; 
    4\rightarrow 7,\; 5\rightarrow 7,\; 6\rightarrow 7\}.$$

(This strategy contains the three upstream transitions $\tu_1$,
    $\tu_2$, and $\tu_3$, with $\tu_i = i \rightarrow i+3$.)

\item The strategy $\sigma_{123}$ only makes sense for a boat with a
powerful enough motor to traverse the strong current of \fig{Stream}.  A
boat without such a powerful motor might instead use the following
strategy:

\vspace*{-0.12in}

$$\sigma_{23} \;=\; 
  \{1\rightarrow 0,\; 0\rightarrow 2,\; 0\rightarrow 3,\; 
    4\rightarrow 1,\; 2\rightarrow 5,\; 3\rightarrow 6,\; 
    4\rightarrow 7,\; 5\rightarrow 7,\; 6\rightarrow 7\}.$$

(This strategy contains downstream transition $\td_1$ and 
    upstream transitions $\tu_2$ and $\tu_3$.)

Strategy $\sigma_{23}$ is very similar to strategy $\sigma_{123}$, but
in place of the upstream transitions $0\rightarrow 1$ and
$1\rightarrow 4$, the strategy contains the downstream transitions
$1\rightarrow 0$ and $4\rightarrow 1$.  As a result, if necessary, the
boat will first return to the marina via the leftmost passage of
\fig{Stream}, then move up to the fishing area via either of the other
two passages.

\end{itemize}

\paragraph{Permissible Strategies:}\ Strategy $\sigma_{23}$ specifies
two transitions at state \#4, namely $4\rightarrow 7$ and
$4\rightarrow 1$.  A boat moving under strategy $\sigma_{23}$ may
therefore reach the fishing area from state \#4 either by moving
directly to the fishing area or by moving first downstream to the
marina then upstream via one of the other passages.  Intuitively, this
bifurcation arises because there are two arcs between any two points
on a circle.  Some maximal strategy must contain both.

\vspace*{0.025in}

\label{permissiblestrat}
While generally useful, motion multiplicity may merely add bookkeeping
clutter, so we restrict it: To start, we define a {\em permissible
strategy\hsph} to be a maximal strategy that (i) attains the fishing
area from anywhere in the graph and (ii) specifies a unique motion at
each state in the set $\{1, 2, 3\}$.  We also stipulate that whenever
a strategy specifies a passage transition and some other motion at a
boat's current location, then the boat will move through the passage.

\vspace*{0.05in}

We may now model permissible strategies via the relation of
\fig{StreamFishingStrategyRelation}.  The relation contains a row for
each permissible strategy and a column for each passage transition.
An entry in the relation is nonblank if and only if the given strategy
contains the given transition.  Every maximal strategy in the graph of
\fig{StreamGraphs} must contain exactly one transition from each of
the three sets $\{\tu_i, \td_i\}$, $i = 1, 2, 3$.  Furthermore, among
the \hspc{\em permissible}$\mskip2mu$ strategies, each strategy is
uniquely characterized by the three passage transitions it contains.
(Of course, it is impossible for a permissible strategy to contain all
three downstream transitions $\td_1$, $\td_2$, $\td_3$, since then the
strategy would not be guaranteed to attain the fishing area from the
marina.)
\fig{StreamFishingStrategyRelation} further depicts a simplicial
complex generated by the strategies of the relation.  The underlying
vertex set of this complex is $\{\tu_1, \td_1, \tu_2, \td_2, \tu_3,
\td_3\}$, comprising the six passage transitions in the river.

\begin{figure}[t]
\begin{center}
\begin{minipage}{2.1in}{$\begin{array}{c|cccccc}
              & \tu_1& \td_1& \tu_2& \td_2& \tu_3& \td_3  \\[2pt]\hline
 \sigma_{123} & \one &      & \one &      & \one &      \\[2pt]
 \sigma_{12}  & \one &      & \one &      &      & \one \\[2pt]
 \sigma_{23}  &      & \one & \one &      & \one &      \\[2pt]
 \sigma_{13}  & \one &      &      & \one & \one &      \\[2pt]
 \sigma_1     & \one &      &      & \one &      & \one \\[2pt]
 \sigma_2     &      & \one & \one &      &      & \one \\[2pt]
 \sigma_3     &      & \one &      & \one & \one &      \\[2pt]
\end{array}$}
\end{minipage}
\hspace*{0.7in}
\begin{minipage}{2.4in}
\ifig{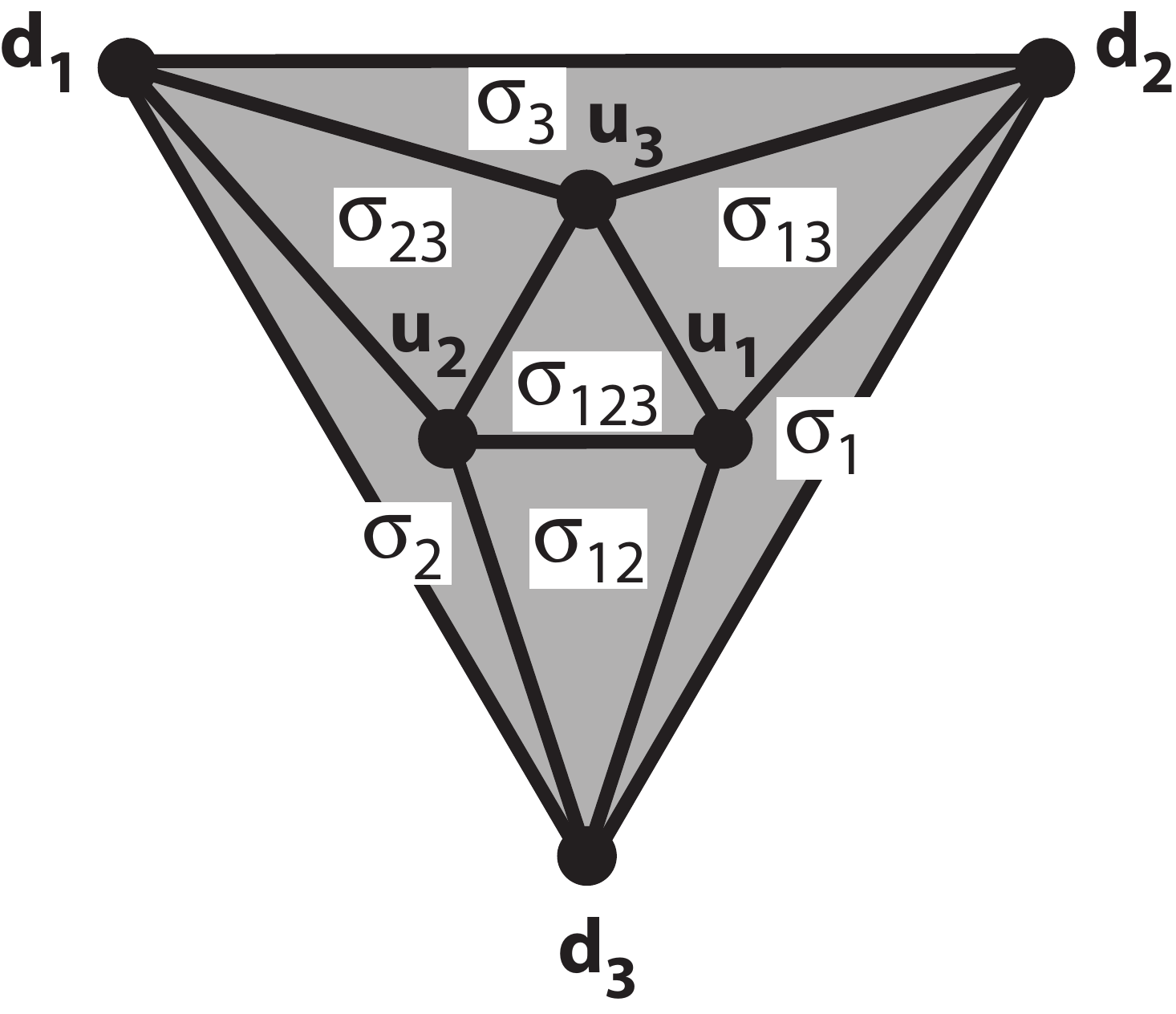}{scale=0.425}
\end{minipage}
\end{center}
\vspace*{-0.2in}
\caption[]{{\bf Left Panel:} Relation describing permissible
  strategies for attaining state \#7 in \fig{StreamGraphs}.  (State
  \#7 is the fishing area in \fig{Stream}.  See
  page~\pageref{permissiblestrat} for the meaning of {\em
  permissible}.)  For each permissible strategy, the relation shows
  only the subset of upstream and downstream transitions $\{\tu_1,
  \td_1, \tu_2, \td_2, \tu_3, \td_3\}$ contained in the strategy.
  These transitions describe motions through the three passages of the
  river.  They fully determine the strategy, given that it is
  permissible.\quad
  {\bf Right Panel:} A simplicial complex derived from the relation,
  with underlying vertex set being the upstream and downstream
  transitions $\{\tu_1, \td_1, \tu_2, \td_2, \tu_3, \td_3\}$. \,Each
  maximal simplex is labeled with its strategy name, as specified by
  its row in the relation.}
\label{StreamFishingStrategyRelation}
\end{figure}

Each of the free faces in the complex of
\fig{StreamFishingStrategyRelation} consists of a pair of downstream
transitions, e.g., $\{\td_1, \td_2\}$, suggesting inference of an
upstream transition, e.g., $\tu_3$.  Indeed, if an observer learns
that a permissible strategy specifies downstream motion through two
passages, then the observer can infer that the strategy must specify
an upstream motion through the remaining passage (since the fishing
area is given as destination).  Consequently, observing two downstream
transitions in a permissible strategy identifies the strategy
uniquely.
There are no other free faces in the complex.  Consequently, observing
any other proper subset of a permissible strategy's passage
transitions does not identify that strategy uniquely.

\vspace*{-0.125in}

\paragraph{Comment:}\ On a given fishing expedition, an observer may
only see a boat move through a single passage, but over the course of
several days the observer may see the boat take different routes.  Or
perhaps a crewmember speaks of the transitions specified by a
strategy.  Assuming the skipper's strategy is constant, the observer
may be able to eventually infer the overall permissible strategy, much
like an observer could infer a bus route in Section~\ref{lake}, after
observing different bridge crossings on different days or by listening
to tourist stories.

\paragraph{Inferring Motor Strength:}\  How might an observer of a
boat's transitions infer that the boat has a strong motor?  Directly
observing the upstream transition $u_1$ is one way, of course.
Additionally, if the observer learns that a strategy specifies
downstream transitions $\td_2$ and $\td_3$ {\em and\,} if the observer
knows that these are part of a strategy to reach the fishing area,
then the observer can infer that the strategy must contain the
upstream transition $\tu_1$, implying that the boat has a powerful
motor.  \ (We assume that each boat only follows strategies it can execute.)

Four of the permissible strategies in
\fig{StreamFishingStrategyRelation}, namely $\sigma_{123}$,
$\sigma_{12}$, $\sigma_{13}$, and $\sigma_1$, presuppose a powerful
motor.  If a skipper is following one of the strategies
$\sigma_{123}$, $\sigma_{12}$, or $\sigma_{13}$, then the skipper can
carefully reveal up to two different passage transitions while still
hiding the motor's power.  In contrast, for strategy $\sigma_1$, the
skipper can reveal at most one passage transition; revealing a second
transition necessarily exposes or implies the motor's power.

\vspace*{-0.1in}

\subsubsection*{Strategies for Attaining the Fishing Area and Strategies
  for Attaining the Marina}

Let us augment our collection of {\em permissible strategies}, in
order to model boat excursions that are {\em
\hspace*{-0.05pt}either\hspace*{1.05pt}} outbound to the fishing area
\hspace*{0.1pt}{\em or\hspace*{0.8pt}} returning to the marina.  We
now permit any maximal strategy for attaining the fishing area that
specifies a unique motion at each state in the set $\{1, 2, 3\}$, {\em
plus}\hspace*{0.5pt} any maximal strategy for attaining the marina
that specifies a unique motion at each state in the set $\{4, 5, 6\}$.
\ (We retain the stipulation regarding passage transitions.)

Focusing on the subcollection of strategies for attaining the marina,
we may again construct a simplicial complex whose underlying vertex
set is $\{\tu_1, \td_1, \tu_2, \td_2, \tu_3, \td_3\}$, much as in
\fig{StreamFishingStrategyRelation}, now with the upstream and
downstream transitions interchanged.  We thus have two simplicial
complexes, one for the fishing-attaining strategies, the other for the
marina-attaining strategies.  We wish to combine these complexes.  In
order to not confuse simplices, we conify each complex with a vertex
identifying the complex.  We then glue the resulting two complexes
together at common boundary locations.  The final simplicial complex
thus obtained is homotopic to $\Stwo$.

In order to visualize this construction more easily, let us simplify
the problem, by removing one of the islands, as in \fig{StreamNarrow}.
Now there are only two passages, one with a strong current requiring a
powerful motor for the upstream direction, the other with a mild
current, traversable by all boats in both directions.  (The new
graph's state and transition names are consistent with those of
\fig{StreamGraphs}. \ State \#0 is the marina and state \#7 is the
fishing area.)

There are now three permissible strategies that attain the fishing
area from anywhere in the graph.  We name them $\sigma_{12}$,
$\sigma_1$, and $\sigma_2$.  Similarly, there are three permissible
strategies that attain the marina from anywhere in the graph.  We name
them $\tau_{12}$, $\tau_1$, and $\tau_2$.  \fig{NarrowRelation}
describes these six strategies via a relation.  Strategy names index
the rows of the relation.  The upstream and downstream transitions,
$\tu_1$, $\td_1$, $\tu_2$, and $\td_2$, plus two additional
attributes, $\tf$ and $\tm$, index the columns of the relation.
Previously, when we were considering only permissible strategies for
attaining the fishing area, the component upstream and downstream
transitions of that strategy fully determined the strategy (given that
the strategy was permissible).  Now, knowing a strategy's upstream and
downstream transitions {\em may $\mskip-0.5mu$not}$\mskip1.5mu$ fully
determine the strategy.  However, each permissible strategy in the set
$\{\sigma_{12}, \sigma_1, \sigma_2, \tau_{12}, \tau_1, \tau_2\}$ is
fully determined by its upstream and downstream transitions {\em
and\vmsp} by its destination.  The attributes $\tf$ and $\tm$ model
this destination.  Attribute $\tf$ means a strategy's destination is
the fishing area and attribute $\tm$ means the strategy's destination
is the marina.  For each permissible strategy, the relation lists the
strategy's upstream and downstream transitions along with its
destination.

\clearpage

\begin{figure}[t]
\begin{center}
\ifig{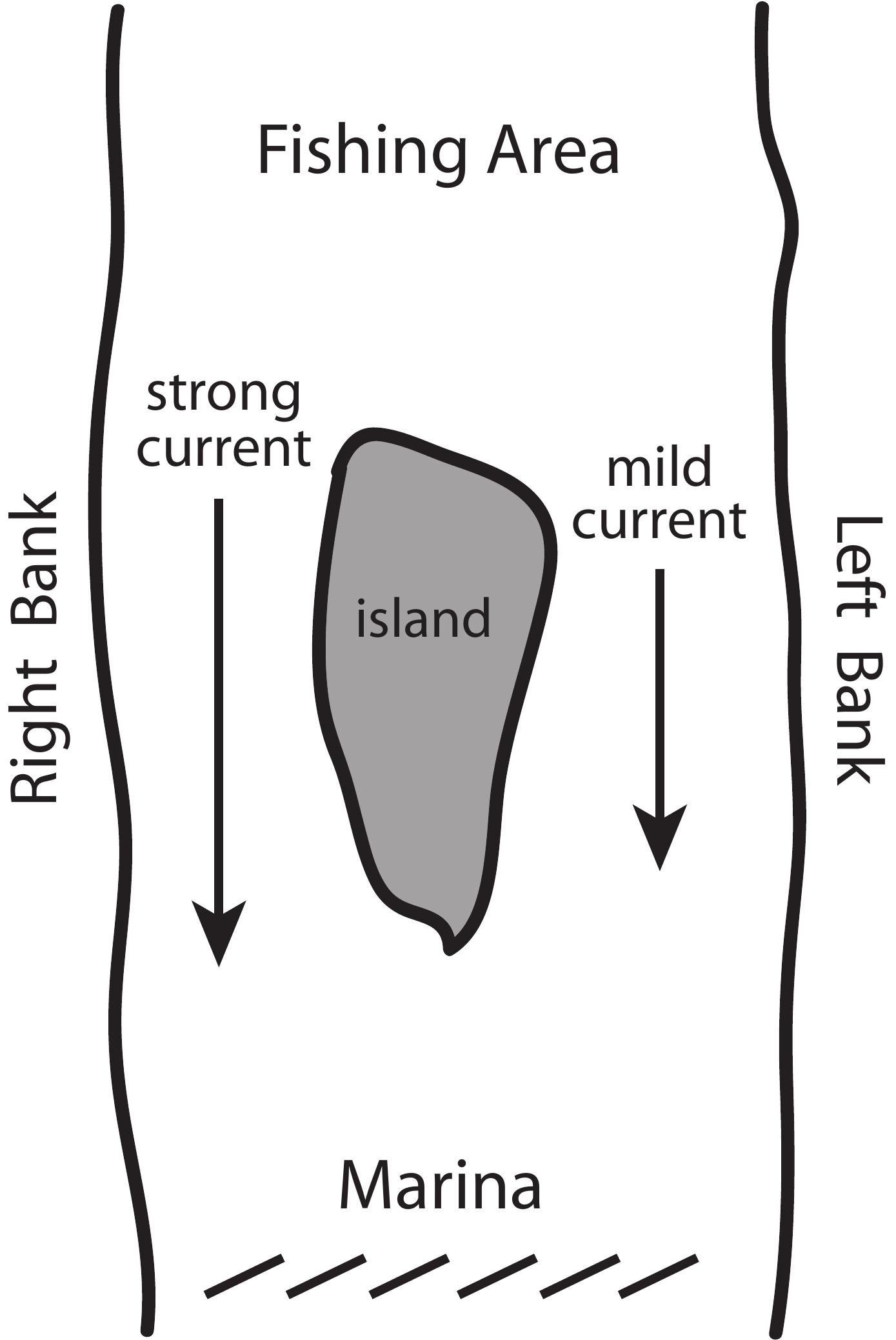}{scale=0.33}
\hspace*{0.8in}
\ifig{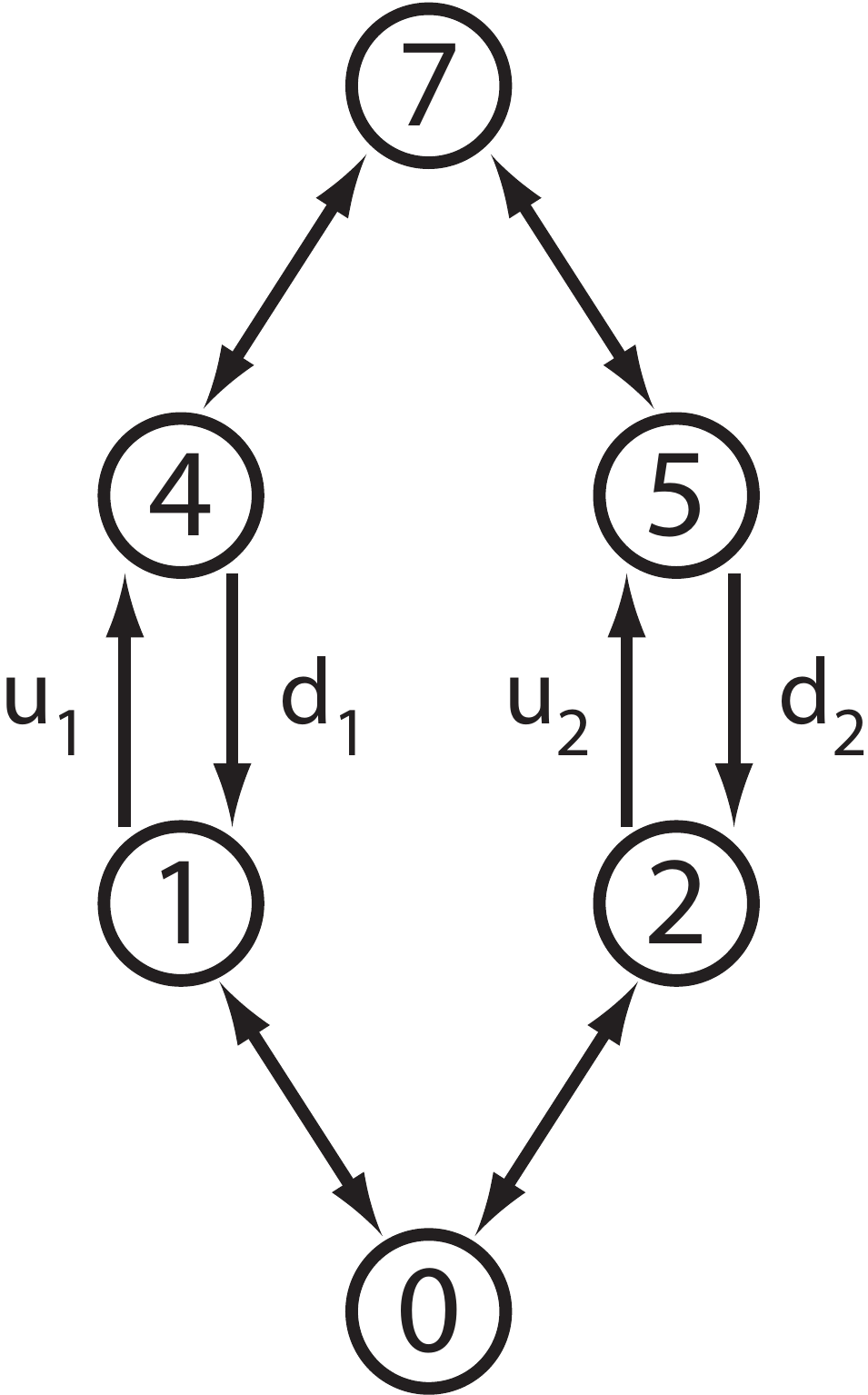}{scale=0.45}
\end{center}
\vspace*{-0.2in}
\caption[]{{\bf Left Panel:} Simplified version of the river from
  \fig{Stream}, in which there now are only two passages and
  consequent currents.\quad
  {\bf Right Panel:} Corresponding simplified graph, again
  highlighting the upstream and downstream transitions through the
  passages.}
\label{StreamNarrow}
\end{figure}

\begin{figure}[h]
\begin{center}
\hspace*{0.7in}\begin{minipage}{3in}{$\begin{array}{c|cccccc}
             & \tu_1& \td_1& \tu_2& \td_2& \tf  & \tm  \\[2pt]\hline
 \sigma_{12} & \one &      & \one &      & \one &      \\[2pt]
 \sigma_1    & \one &      &      & \one & \one &      \\[2pt]
 \sigma_2    &      & \one & \one &      & \one &      \\[2pt]
 \tau_1      &      & \one & \one &      &      & \one \\[2pt]
 \tau_2      & \one &      &      & \one &      & \one \\[2pt]
 \tau_{12}   &      & \one &      & \one &      & \one \\[2pt]
    \end{array}$}
\end{minipage}
\end{center}
\vspace*{-0.2in}
\caption[]{Relation describing permissible strategies $\sigma_{12}$,
 $\sigma_1$, $\sigma_2$ for attaining state \#7 (fishing area) and
 permissible strategies $\tau_{12}$, $\tau_1$, $\tau_2$ for attaining
 state \#0 (marina) in \fig{StreamNarrow}.  The relation uses the
 upstream and downstream transitions for each strategy as attributes.
 Without knowing a strategy's destination, these transitions are not
 necessarily enough to uniquely determine the strategy.  Consequently,
 the relation includes two additional attributes, to indicate the
 destination, as either the fishing area (attribute \tf) or the marina
 (attribute \tm).}
\label{NarrowRelation}
\end{figure}

\vspace*{-0.1in}

\paragraph{Comment:}\ One may readily observe a boat traversing a
passage, but what does it mean to observe attribute $\tf$ or
$\tm\vtsp$?  \ One possibility is that a skipper announces a boat's
destination.  \ Another possibility is that the ``observation'' of
$\tf$ or $\tm$ is actually an inference made from other observations.
For instance, if an observer learns that the skipper has made
preparations for both transitions $4\rightarrow 7$ and $5\rightarrow
7$, then the observer may conclude that the boat's destination is the
fishing area.  Or perhaps someone observes a boat departing the
marina, for instance by making the transition $0\rightarrow 2$.  Then
the observer knows that the boat's destination cannot be the marina,
and so has ``observed $\tf$'' (assuming the strategy being observed is
permissible).

\clearpage

\begin{figure}[h]
\begin{center}
\ifig{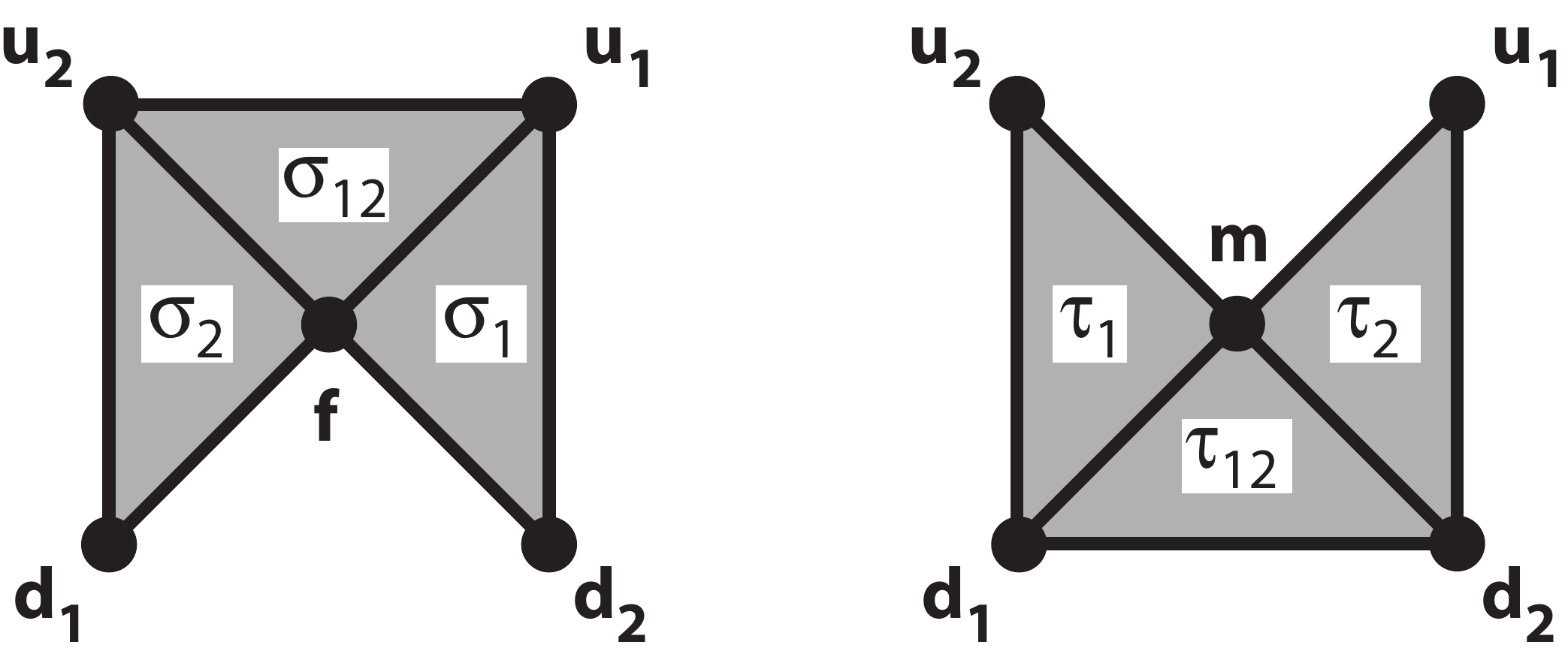}{scale=0.45}
\end{center}
\vspace*{-0.2in}
\caption[]{Two simplicial complexes derived from the relation of
  \fig{NarrowRelation}, with underlying vertex set being the
  attributes $\{\tu_1, \td_1, \tu_2, \td_2, \tf, \tm\}$.  The left
  complex shows the permissible strategies $\sigma_{12}$, $\sigma_1$,
  and $\sigma_2$ for attaining the fishing area.  The right complex
  shows the permissible strategies $\tau_{12}$, $\tau_1$, and $\tau_2$
  for attaining the marina.  See \fig{StreamNarrowComplex} as well.}
\label{StreamNarrowComplexHemispheres}
\end{figure}

We will now construct a simplicial complex to represent the relation
of \fig{NarrowRelation}, much as we constructed the complex in
\fig{StreamFishingStrategyRelation}.  Let us proceed in steps.  First,
we construct a simplicial complex representing the permissible
strategies for attaining the fishing area and a separate simplicial
complex representing the permissible strategies for attaining the
marina.  In both cases, we let the underlying vertex set be $\{\tu_1,
\td_1, \tu_2, \td_2, \tf, \tm\}$.  See
\fig{StreamNarrowComplexHemispheres}.  Observe that attribute $\tf$ is
a cone apex for the complex generated by the fishing-attaining
strategies, while attribute $\tm$ is a cone apex for the complex
generated by the marina-attaining strategies.\footnote{A {\em cone
apex\vlsp} for a finite simplicial complex is a vertex contained in
every maximal simplex of the complex.}  Finally, we glue these two
simplicial complexes together along shared simplices, obtaining the
complex of \fig{StreamNarrowComplex}.  The homotopy type of this
complex is $\Sone$, suggesting some ability to delay identification of
strategies \cite{paths:privacy}.  Observe in particular that every
upstream or downstream transition appears in three permissible
strategies.

\begin{figure}[h]
\vspace*{0.05in}
\begin{center}
\ifig{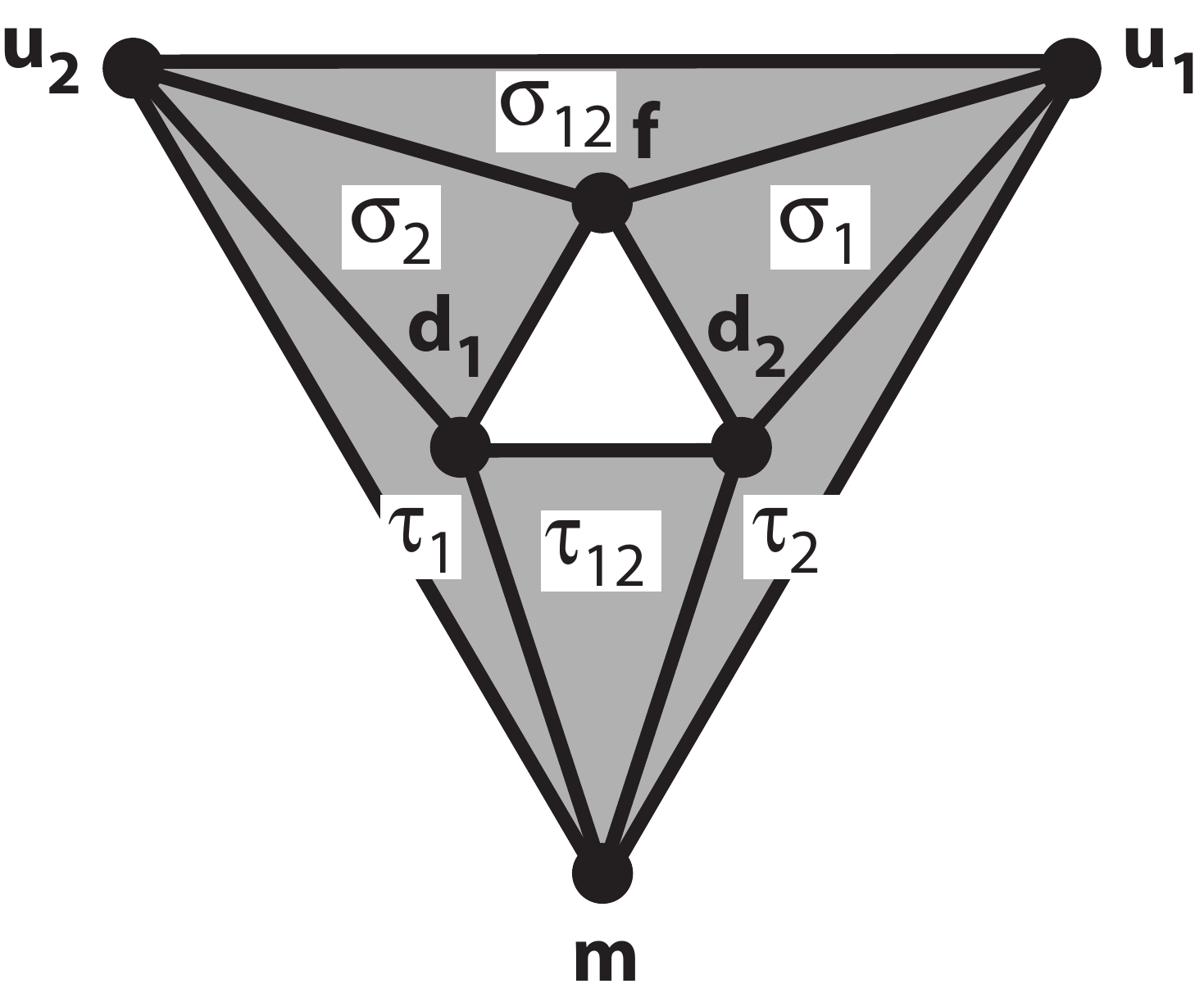}{scale=0.45}
\end{center}
\vspace*{-0.2in}
\caption[]{Simplicial complex derived from the relation of
  \fig{NarrowRelation}, with underlying vertex set being the
  attributes $\{\tu_1, \td_1, \tu_2, \td_2, \tf, \tm\}$.  (This
  complex is obtained from the two complexes shown in
  \fig{StreamNarrowComplexHemispheres} by gluing those
  complexes together along common vertices and edges.)  Each maximal
  simplex is labeled with its strategy name, as specified in the
  relation.}
\label{StreamNarrowComplex}
\end{figure}

\clearpage

Considering the free and nonfree faces of the simplicial complex in
\fig{StreamNarrowComplex}, or directly from the relation of
\fig{NarrowRelation}, we may make the following inferences:

\begin{itemize}

\item Observing upstream transitions $\tu_1$ and $\tu_2$ in a boat's
  strategy implies attribute $\tf$.  In other words, one may infer
  that the boat's destination is the fishing area and that the
  encompassing permissible strategy is $\sigma_{12}$.  This inference
  reflects the physical reality that a boat cannot have the marina as
  destination if its strategy always entails moving upstream, via the
  two passages on both sides of the island.

\item Similarly, observation of $\td_1$ and $\td_2$ implies $\tm$
  and identifies strategy $\tau_{12}$.

\item Observing an upstream transition in one passage and a downstream
  transition in the other passage (without knowing which occurred
  first) leaves the boat's destination ambiguous.  This ambiguity
  reflects the physical reality that the boat could have rounded
  either the upstream end or the downstream end of the island.

\item Knowing the boat's destination and observing a passage
  transition {\em away\,} from that destination implies the other
  passage transition and thus the overall permissible strategy.  For
  instance, knowing that the boat is heading to the fishing area
  ($\tf$) and observing the boat move downstream along the strong
  current ($\td_1$) means the boat's strategy must also specify a
  motion upstream over the mild current ($\tu_2$).  The geometry of
  the river forces this implication and thus identifies the strategy
  $\sigma_2$ \,(assuming strategies are permissible).

\item Knowing the boat's destination and observing a passage
  transition {\em toward\,} that destination leaves open the
  directionality of the transition through the other passage.  For
  instance, knowing that the boat is heading to the marina ($\tm$) and
  observing the boat move downstream along the mild current ($\td_2$)
  does not nail down whether the strategy specifies an upstream or a
  downstream transition through the passage with the strong current.
  There remains an ambiguity as to whether the encompassing
  permissible strategy is $\tau_{12}$ or $\tau_2$.

\end{itemize}

\subsubsection*{Deception and Detection}

Skippers of boats with motors capable of moving upstream over the
strong current tend to hide their strength for when it is really
needed, such as a competition to snag nice fish at the head of the
strong current.  They may hide their strength either by never
exercising it or by moving upstream over the strong current only under
cover of fog or darkness.

Let us suppose that this deception is so pervasive that, for all
intents and purposes, the upstream transition $\tu_1$ is {\vtsp\em
never\,} observable.  Two questions emerge:

\vspace*{-0.03in}

\begin{enumerate}
\item How does the unobservability of $\tu_1$ change the relation of
  \fig{NarrowRelation} and the complex of \fig{StreamNarrowComplex}?

\vspace*{-0.03in}

\item How can one distinguish between a boat that is truly incapable
  of moving upstream over the strong current and a boat whose skipper
  is merely hiding that capability?
\end{enumerate}

\vspace*{-0.03in}

An initial answer to question \#2 is that one cannot distinguish the
two types of boats {\em if\,} the powerful boat {\em always\,} acts
like the weaker boat.  However, if the powerful boat does sometimes
exercise its capabilities, then other observations {\em may\,} imply
the boat's power.

\clearpage

\begin{figure}[h]
\begin{center}
\begin{minipage}{2.25in}{$\begin{array}{c|ccccc}
             & \td_1& \tu_2& \td_2& \tf  & \tm  \\[2pt]\hline
 \sigma_{12} &      & \one &      & \one &      \\[2pt]
 \sigma_1    &      &      & \one & \one &      \\[2pt]
 \sigma_2    & \one & \one &      & \one &      \\[2pt]
 \tau_1      & \one & \one &      &      & \one \\[2pt]
 \tau_2      &      &      & \one &      & \one \\[2pt]
 \tau_{12}   & \one &      & \one &      & \one \\[2pt]
    \end{array}$}
\end{minipage}
\hspace*{0.3in}
\begin{minipage}{2.25in}
\ifig{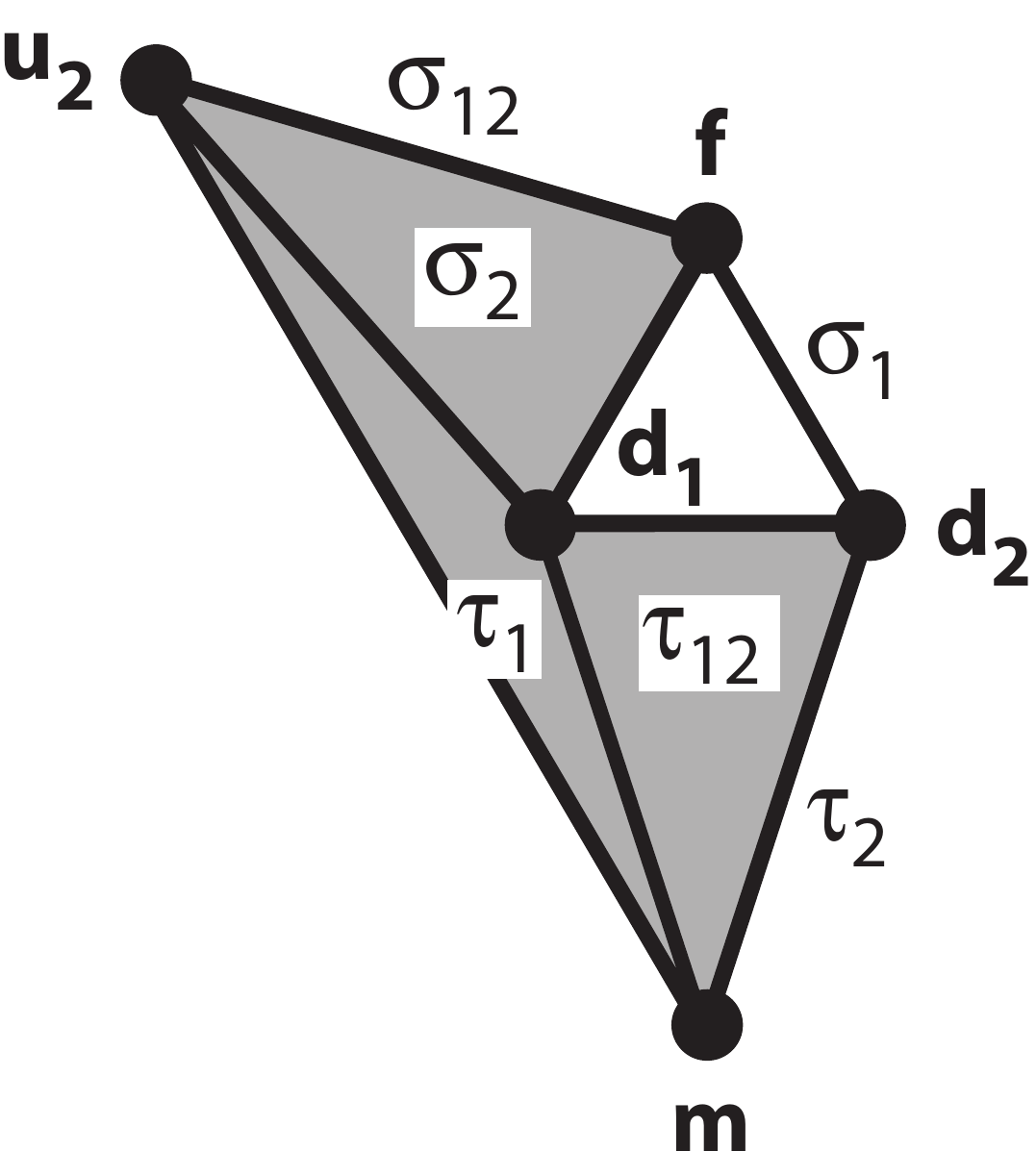}{scale=0.45}
\end{minipage}
\end{center}
\vspace*{-0.1in}
\caption[]{Here is a version of the relation from \fig{NarrowRelation}
  and the simplicial complex from \fig{StreamNarrowComplex} in
  which the attribute $\tu_1$ exists but is not observable.  Each
  strategy in the relation generates a simplex in the complex, as
  indicated by the labels.  Some strategies now appear as edges rather
  than triangles, due to the unobservable transition $\tu_1$ in those
  strategies.  Although this transition is unobservable, one can infer
  its existence indirectly if one observes attributes $\td_2$ and
  $\tf$.  This is because the edge $\{\td_2, \tf\}$ is not a
  subsimplex of some larger strategy.}
\label{StreamNarrowHiddenRelationComplex}
\end{figure}

\vspace*{0.1in}

\fig{StreamNarrowHiddenRelationComplex} answers question \#1.  The
relation one obtains when $\tu_1$ is unobservable is the same as the
original relation of \fig{NarrowRelation}, except that the column
indexed by the upstream transition $\tu_1$ disappears.  The resulting
simplicial complex is now the deletion $\dl(\Gamma, \tu_1)$, with
$\Gamma$ the complex of \fig{StreamNarrowComplex}.  (Formally,
$\dl(\Gamma,\tu_1) = \setdef{\gamma\in\Gamma}{\tu_1\not\in\gamma}$.
In other words, the resulting complex contains all simplices of the
original complex except those that included $\tu_1$.)

\vso

In the new complex, some strategies that appeared as triangles
originally now appear as edges.  These are the strategies
$\sigma_{12}$, $\sigma_1$, and $\tau_2$, namely all strategies that
include the now unobservable upstream transition $\tu_1$.

\vso

The observable portions of two of those strategies, namely
$\sigma_{12}$ and $\tau_2$, are subsets of other strategies.  For
instance, the observable portion of $\sigma_{12}$ is a subset of
strategy $\sigma_2$.  This means: If one observes the upstream
transition $\tu_2$ and if one knows that the boat's destination is the
fishing area (attribute $\tf$), then there remains an ambiguity
regarding the strategy's specified transition over the strong current;
it could be either upstream (as in $\sigma_{12}$) or downstream (as in
$\sigma_2$).  Consequently, if a skipper with a powerful motor always
follows strategy $\sigma_{12}$, but traverses the strong current only
during fog and the mild current during clear weather, then no one will
know of the boat's power.  (This assumes that the traversal times are
unmeasured or constant, e.g., the skipper avoids traversing the mild
current excessively quickly with the strong motor.)

\vso

In contrast, strategy $\sigma_1$ generates a maximal simplex even in
the new complex and is thus uniquely identifiable from its observable
attributes.  This means: If one observes the downstream transition
$\td_2$ but knows that the overall destination of the boat is the
fishing area (attribute $\tf$), then one can conclude that the
strategy must be $\sigma_1$ and that the boat will traverse the strong
current upstream.  In other words, even though $\tu_1$ is
unobservable, one can infer its existence and conclude that the boat
has a powerful motor.

\clearpage

\begin{figure}[h]
\begin{center}
\begin{minipage}{2.25in}{$\begin{array}{c|ccccc}
             & \td_1& \tu_2& \td_2& \tf  & \tm  \\[2pt]\hline
 \sigma_2    & \one & \one &      & \one &      \\[2pt]
 \tau_1      & \one & \one &      &      & \one \\[2pt]
 \tau_{12}   & \one &      & \one &      & \one \\[2pt]
    \end{array}$}
\end{minipage}
\hspace*{0.3in}
\begin{minipage}{2.25in}
\ifig{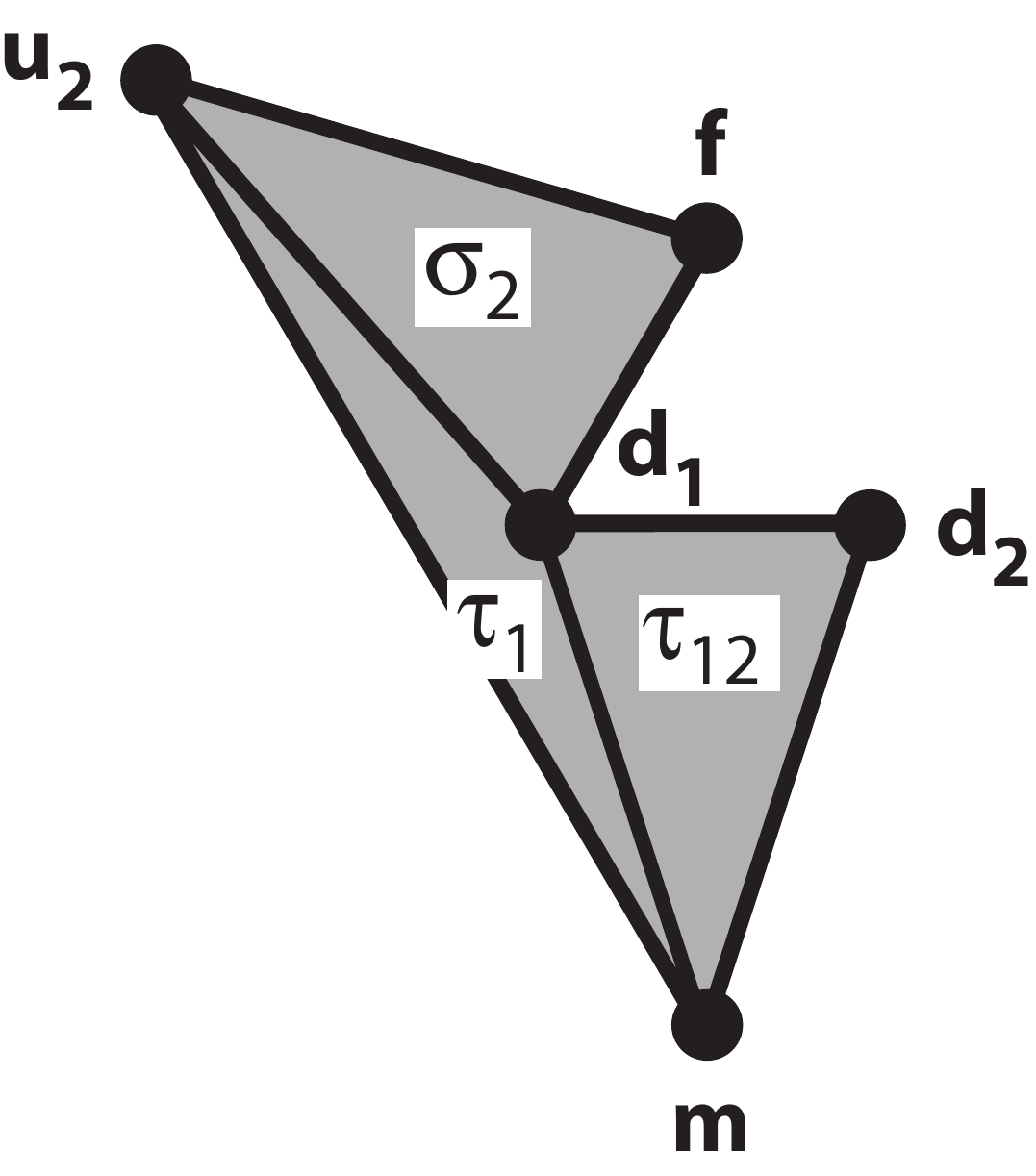}{scale=0.45}
\end{minipage}
\end{center}
\vspace*{-0.1in}
\caption[]{{\bf Left Panel:} A relation describing the permissible
  strategies for a boat that is incapable of moving upstream over the
  strong current of \fig{StreamNarrow}.  This incapacity amounts to
  removing the directed edge $\tu_1$ from the graph of$\mskip2mu$
  \fig{StreamNarrow} $\mskip2.25mu${\em as well as}$\mskip3.75mu$
  removing all strategies from the relation of \fig{NarrowRelation}
  that contain transition $\tu_1$.\quad
  {\bf Right Panel:} The resulting simplicial complex, with maximal
  simplices labeled by strategy names.}
\label{StreamNarrowLimitedRelationComplex}
\end{figure}

It is instructive to construct the strategy relation and attendant
simplicial complex for a boat that truly is incapable of traversing
the strong current upstream.  These appear in
\fig{StreamNarrowLimitedRelationComplex}.  Not only does the
transition $\tu_1$ now disappear, but so do all strategies that relied
on that transition.  (Again, we assume that the boat only follows
strategies which it is capable of executing.)

The space of strategies is much smaller now.  The resulting simplicial
complex looks similar to that in which $\tu_1$ was merely
unobservable, but there is a key difference: The simplex $\{\tf,
\td_2\}$ is gone.  (In fact, all three strategies that once contained
$\tu_1$ are now gone, but only one of those, namely $\sigma_1$, formed
a maximal simplex previously in the complex of
\fig{StreamNarrowHiddenRelationComplex}.)  The new complex tells us
that it is {\em inconsistent\,} for a boat with a weak motor to
announce its destination as the fishing area (attribute $\tf$) but to
traverse the mild current downstream (attribute $\td_2$).

\vst

Report \cite{paths:privacy} described how inconsistent observations
sometimes suggest or identify unmodeled properties.  Here, observing
the inconsistency $\{\tf, \td_2\}$ might suggest that the actual
strategies for the boat being observed are not those of
\fig{StreamNarrowLimitedRelationComplex} but those of
\fig{StreamNarrowHiddenRelationComplex}.

\vspace*{0.1in}

We may therefore interpret the difference between the two relations
and complexes of Figures~\ref{StreamNarrowHiddenRelationComplex} and
\ref{StreamNarrowLimitedRelationComplex} as follows:

\vspace*{-0.025in}

\begin{itemize}
\item The relations and complexes tell an {\em actor\,} what behavior
  to {\em avoid}, in order to {\em be successful\,} at deception.
\item The relations and complexes tell an {\em observer\,} how to
  \hspq{\em look for\,} inconsistencies in behavior, in order to {\em
  detect\,} deception.
\end{itemize}

\clearpage

\subsubsection*{Inferences, Free Faces, Geometry}

The simplicial complex of \fig{StreamNarrowLimitedRelationComplex} is
a cone with apex $\td_1$.  This geometry arises because {\em every\,}
permissible strategy in the relation of that figure contains the
downstream transition $\td_1$.  Said differently, if a boat is
incapable of moving upstream through the passage with the strong
current, then there is no choice but to include the downstream
transition in every permissible strategy for that boat.  (Recall that
a permissible strategy consists of a {\em maximal\,} set of motions
that a boat {\em might\,} make to attain its destination without
cycling.)  One can thus infer the transition $\td_1$ ``for free'',
i.e., without observing or learning of the motion directly, assuming
one knows that the boat has a weak motor.

\vspace*{0.075in}

Let us therefore remove vertex $\td_1$ from the complex of
\fig{StreamNarrowLimitedRelationComplex} to obtain the following:

\begin{figure}[h]
\begin{center}
\vspace*{0.1in}
\ifig{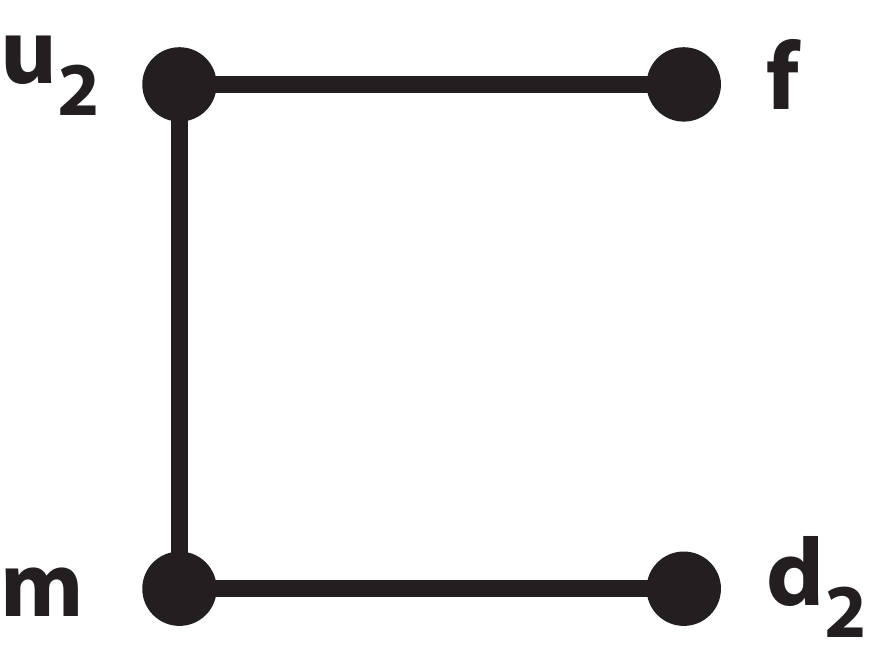}{scale=0.4}
\end{center}
\vspace*{-0.2in}
\caption[]{In the complex of \fig{StreamNarrowLimitedRelationComplex},
  the downstream transition $\td_1$ is a cone apex.  Removing $\td_1$
  produces the simplicial complex shown here.  The free vertices of
  this complex highlight the implications $\tf\!\implies\!\tu_2$ and
  $\td_2\!\implies\!\tm$ that occur when transition $\tu_1$ does not
  exist, as can also be seen from the relation of
  \fig{StreamNarrowLimitedRelationComplex}.}
\label{StreamNarrowLimitedComplexUncone}
\end{figure}

\vspace*{0.1in}

This simplified complex highlights some inferences that are possible
when observing boats known to have weak motors (as usual, we also
assume that the strategies under consideration are permissible
strategies and that each boat only follows strategies it can execute):

\vspace*{-0.025in}

\begin{itemize}

\item $\tf\!\implies\!\tu_2$: \ If one knows that a boat is heading to
  the fishing area, then one can infer that the boat must move
  upstream when traversing the mild current.

\item $\td_2\!\implies\!\tm$: \ If one observes a boat moving
  downstream over the mild current, then the boat must be heading to
  the marina.

\end{itemize}

\vspace*{0.05in}

These conclusions reflect the geometry of the fishing area and marina
relative to the passages.  Said differently, the geometry of the
simplicial complex models the geometry of the river in such a way that
one can draw conclusions about boat motions from the free faces in the
complex.  We saw this property as well when modeling paths in the
example of Section~\ref{lake}.

\clearpage

\subsection{Hidden State}
\markright{Hidden State}

The examples and discussion of this introduction have assumed that the
underlying state is known to the observer.  For example, the analysis
of Section~\ref{river} assumed a state space defined by the river
geometry, along with strategies whose motions were not allowed to
revisit states.

\vso

In reality, a skipper might have additional {\em hidden state}.  The
restriction on revisiting states would then apply to the composite of
observable and hidden states but not necessarily to the observable
states alone.  For instance, a skipper might be willing to revisit
some parts of the river once or twice, relative to some hidden
internal counter (perhaps fuel consumption).  A skipper could then
selectively hide some motions and reveal other motions, in order to
bluff the capability of a strong motor.  (An observer unaware of the
skipper's hidden state might infer $\tu_1$ from observations $\{\tf,
\td_2\}$ even though the boat only has a weak motor and has made
surreptitious cyclic motions not involving $\tu_1$ in order to give
the appearance of a strong motor.)

\vso

One possible approach for dealing with such hidden state is to
construct many possible models of that state, then hope to observe
inconsistencies in behavior to rule out or imply some of these models.
We leave such higher-order deception and detection for future work.

\clearpage
\section{Review of Prior Work and Notation}
\markright{Review}

This section briefly reviews key concepts and notation regarding
strategy complexes and relations.  Detailed discussions of strategies
and strategy complexes appear in \cite{paths:strategies, paths:plans}.
The connection of strategies to relations appears in
\cite{paths:privacy}.  Background material on topology may be found in
\cite{paths:munkres, paths:rotman}, on posets in \cite{paths:wachs},
and on privacy in \cite{paths:sweeneykanon, paths:DinurNissim,
paths:netflix, paths:dworkcacm}.

\vspace*{0.1in}

\noindent {\bf Assumption:}\quad All graphs, relations, and simplicial
complexes in this report are finite.

\subsection{Graphs and Strategies}
\markright{Graphs and Strategies}
\label{strategies}

This subsection reviews material on strategies, taken fairly directly
from \cite{paths:plans}, with some descriptions verbatim.

\vspace*{-0.15in}

\paragraph{Nondeterministic and Stochastic Graphs:}\ We are interested
in finite graphs, viewed as state spaces with errorful transitions.
We model any such graph $G$ as a pair $(V, \frakA)$, consisting of a
finite set of {\em states\hspd} $\mskip1muV\mskip-2mu$ and a finite
collection of {\em actions\,} $\frakA$.  Each action $a\in\frakA$
consists of a {\em source\,} state $v$ and a nonempty set
$T\mskip-2mu$ of {\em targets}, with $v \in V$ and $\emptyset \neq T
\subseteq V$.

If $\vmsp{}T$ consists of a single state $t$, we say that the action
is {\em deterministic}.  We may write a deterministic action $a$ as
\vtsp\maction{v}{t}, just like a directed edge in a directed graph.
If $\vmsp{}T$ contains more than one state, there are two
possibilities: The action is either {\em nondeterministic\,} or {\em
stochastic}.  We discuss each of these possibilities next.  We may
also regard a deterministic action as a special instance of a
nondeterministic action and/or as a special instance of a stochastic
action.

We write a nondeterministic action $a\,$ as \maction{v}{T}.  The
semantics of such an action are as follows: Action $a$ may be executed
whenever the system is at state $v$.  When action $a$ is executed, the
system moves from state $v$ to one of the target states in $T$.  \ If
$\,\abs{T} > 1$, then the precise target attained is not predictable in
advance, but is known after execution completes.  Different execution
instances of action $a$ could attain different target states within
$T$.  An adversary might be choosing the target attained.

We write a stochastic action $a\,$ as \maction{v}{\vtsp{p}\,T}, with
$p\vtsp$ a strictly positive probability distribution $p : T
\rightarrow (0, 1]$, such that $\sum_{t\in{T}}p(t) = 1$.
The semantics of a stochastic action are very similar to those of a
nondeterministic action, except that the target state attained at
execution time is now determined stochastically, according to the
probability distribution $p$, rather than nondeterministically.
Different execution instances of action $a$ are assumed to be
independent of each other.

We say that a graph $(V, \frakA)$ is a {\em pure nondeterministic}
graph if all the actions in $\frakA$ are either deterministic or
nondeterministic.  We say that a graph $(V, \frakA)$ is a {\em pure
stochastic}$\mskip1.5mu$ graph if all the actions in $\frakA$ are
either deterministic or stochastic.  This report is primarily
interested in such pure graphs.  However, see \cite{paths:plans} for a
discussion of graphs with a mix of deterministic, nondeterministic,
and stochastic actions.  See also Section~\ref{mixed}.

Suppose action $a$ has source $v$ and targets $T$.  We refer to each
possible transition $\vtsp{}v \rightarrow t\vtsp$, with $t\in T$, as
an {\em action edge (of action $\vtsp{}a$)}.

\vspace*{0.1in}

{\bf Comment:} \ We permit multiple actions to be distinct yet have
the same source $v$ and the same target set $T$ (and the same
probability distribution $p$ if the actions are stochastic).  Such
duplication flexibility is useful, for instance when forming quotient
graphs (see page~\pageref{quotient}).

\paragraph{Strategies and Strategy Complexes:}\ We next define {\em
strategies} and {\em strategy complexes} via a series of intermediate
concepts.  Intuitively, a strategy is a generalization of a control
law, now viewed as a mapping from states to sets of actions.

\label{actionsemantics}Let $G=(V, \frakA)$ be a graph as on
page~\pageref{strategies}, and suppose $\calA\subseteq\frakA$.  \ We
view $\calA$ as a generalized control law as follows: Suppose the
system is currently at state $v$.  The set $\calA$ may contain zero,
one, or several actions with source $v$.  The system {\em stops\vtsp}
moving precisely when $\calA$ contains no action with source $v$.
Otherwise, the system {\em must\vtsp} execute some action $a\in\calA$
with source $v$.  If there are several such actions, any one of the
actions might execute, determined nondeterministically.  (Worst-case,
an adversary might make the choice.  In Section~\ref{river}, perhaps
sometimes a boat's skipper could.)  Upon execution of action $a$, the
system finds itself at one of the targets $t$ of action $a$.  The
process then repeats, with $t$ the system's new current state.

We are interested in only those control laws that eventually stop at
some state or states.  Intuitively, for pure nondeterministic graphs,
this means that executing any of the actions contained in $\calA$ will
never cause the system to cycle (i.e., revisit a previously
encountered state).  \ For pure stochastic graphs, it means that no
subset of the actions contained in $\calA$ forms a recurrent Markov
chain.  We model these requirements with the following definitions,
again taken fairly directly from \cite{paths:plans}:

\vspace*{0.1in}

Let $G=(V, \frakA)$ be a graph as on page~\pageref{strategies}.

\vspace*{-0.025in}

\begin{itemize}

\item With $a\in\frakA$, let $\src(a)$ denote the source of action
  $a$.  When $\calA\subseteq\frakA$, define $\calA$'s {\em source set\,}
  $\src(\calA)$ as the set of all the individual
  actions' sources: \ $\src(\calA)=\setdef{\mskip1mu\src(a)\!}{a\in\calA}$.

\item With $a\in\frakA$, let $\trg(a)$ denote the set of targets of
  action $a$.

\item \label{movesoff}Let $W \subseteq V$ and $a \in \frakA$.  We say
  that {\em action $a$ moves off\, $\hspc{W}\!$ (in $G$)} \vmsp if
  \vmsp$\src(a) \in W$ and one (or both) of the following is true: (i)
  action $a$ is nondeterministic with all of its targets in $V
  \setminus W$, or (ii) action $a$ is stochastic with at least one of
  its targets in $V \setminus W$.  (The two requirements are identical
  when $a$ is deterministic.)

\item Let $\calA \subseteq \frakA$.  We say $\calA$ {\em contains a
  circuit\,} if, for some nonempty subset $\calB$ of $\calA$, no
  action of $\calB$ moves off $\src(\calB)$.  We say $\calA$ {\em
  converges\,} or {\em is convergent\,} if $\calA$ does not contain a
  circuit.  \hspace*{0.05in} Comment: \ If \vmsp$\calA$ is convergent
  and $V\!\neq\emptyset$, then necessarily $\src(\calA) \neq V$.

\item \label{strategy}If \vmsp$V\! \neq \emptyset$, then the {\em
  strategy complex\, $\DG$ of\, $G$} is the simplicial complex whose
  underlying vertex set is $\frakA$ and whose simplices are all the
  convergent subsets $\calA$ of $\frakA$.  Every simplex of $\DG$ is
  called a {\em strategy}.  The empty simplex $\emptyset$ is one such
  strategy, modeling no motion.  It appears in $\DG$ whenever
  $V\mskip-2.5mu\neq\emptyset$.  \ If \vmsp$V\mskip-2.5mu=\emptyset$,
  then one would let $\DG$ be the void complex, containing no
  simplices.  \ (This report will always require $V\! \neq
  \emptyset$.)

\item If $\sigma$ is a strategy in $\DG$, then
  $\sdiff{V}{\src(\sigma)}$ is called the {\em goal\,} or {\em goal
  set\,} of $\sigma$.  The goal set consists of all states in the
  graph at which $\sigma$ does not specify a motion.  We may say that
  \vmsp$\sigma$ {\em converges to\,} $\sdiff{V}{\src(\sigma)}$.  \ If
  the goal set is a singleton $\{v\}$, we may refer directly to state
  $v$ as $\mskip1mu\sigma$'s {\em goal}.

\item Suppose $V\!\neq\emptyset$.  We say that $G$ is {\em fully
  controllable\hspd} if every nonempty subset of $V\mskip-2mu$ is the
  goal set of some strategy in $\DG$.  Observe that $G$ is fully
  controllable if and only if every singleton state is the goal set of
  at least one {\em maximal\,} strategy (simplex) in $\DG$.

\item Topologically, $G$ is fully controllable if and only if $\DG$ is
  homotopic to the sphere $\Snt$, with $n=\abs{V}$.  \ See
  \cite{paths:strategies, paths:plans}.

\end{itemize}

We will soon model strategy complexes via relations.  The maximal
simplices of a strategy complex will index the rows of the relation
and the graph's actions will index the columns.  Of interest will be
how to reveal the constituent actions of a maximal strategy in such a
way as to delay identification of the strategy for as long as
possible.

\vspace*{0.2in}

We end this subsection with two definitions that will be useful later
in the report:

\vspace*{-0.05in}

\paragraph{Subgraphs:} \
Suppose $G=(V,\frakA)$ is a graph as on page~\pageref{strategies}.
\ By a {\vmsp\em subgraph\vtsp} $\mskip1muH\mskip-1.5mu$ {\em of}
$\mskip4muG\mskip2mu$ we mean a graph $H=(W, \frakB)\vtsp$ in its
own right such that $W \subseteq V$ and $\frakB \subseteq \frakA$.

\paragraph{Quotient graphs:} \
\label{quotient}Suppose $G=(V,\frakA)$ is a graph as on
page~\pageref{strategies} and suppose $\emptyset \neq W \subseteq V$.
\\[2pt]
We define the {\em quotient graph\,} $G/W =(V^\prime,\frakA^\prime)$
as follows:

\vspace*{-0.05in}

\begin{itemize}
\item The state space is $\,V^\prime \;=\; (\sdiff{V}{W})\;\union\;\{\wrep\}$.

      Here $\wrep$ is a new state. \ It represents the set of states
      $W\!$ all identified to one state.\footnote{In this report,
      \hspace*{-0.75pt}{\em identify}\hspace*{0.325pt} typically means
      {\em determine identity \hspace*{-0.5pt}of}, but sometimes, as
      here, it means \hspc{\em treat as same\hspq}.}

\item The actions $\frakA^\prime$ are in one-to-one correspondence
  with the actions $\frakA$, but source and target states of actions
  in $\frakA^\prime$ are relabeled to match the new state space.
  Specifically, any source or target in $\sdiff{V}{W}$ remains
  unchanged, while any source or target in $W\mskip-0.5mu$ becomes $\wrep$.
  
  The relabeling of targets may identify some or all of the targets of
  an action $a\in\frakA$.  For a stochastic action of the form
  \maction{v}{\vtsp{p}\,T}, with $T \inter W \mskip-2mu \neq
  \emptyset$, one therefore sums the transition probabilities of the
  targets in $T \inter W$ in order to determine the transition
  probability to state $\wrep$ of the relabeled action
  $a^\prime\in\frakA^\prime$.

\vso

  {\bf Comment:}\ ``one-to-one correspondence'' means that distinct
  actions of $\frakA$ remain distinct in $\frakA^\prime$ even if their
  sources become the same and their target sets become the same (and
  even if their probability distributions become the same, in the
  stochastic case).

\end{itemize}

\vst

\noindent The following facts are easy to establish:

\begin{enumerate}

\label{quotientfacts}
\item A convergent set of actions in $G$ may contain a circuit once
one views the actions in $G/W$.  \ (In particular, individual actions
may become nonconvergent.) \ However, the set of actions remains
convergent if its source set does not overlap $W$.

\item Any convergent set of actions in $G/W\mskip-1.5mu$ will remain
convergent if one views the actions back in their original form in
$G$.

\item If $G$ is fully controllable, then so is $G/W$.

\end{enumerate}

\paragraph{Generalization:} Suppose $W_1, \ldots, W_k$ are
nonempty pairwise disjoint subsets of $V$, with $k \geq 1$.
\ Let $\wrep_1, \ldots, \wrep_k$ be new and distinct states.  The
definition of quotient graph given above generalizes to this setting:
For each $i=1, \ldots, k$, we identify all states of $G$ that lie in
$W_i$ to the single state $\wrep_i$.  \ We denote the resulting
quotient graph by $G/\{W_1, \ldots, W_k\}$.

\clearpage

\subsection{Relations and Dowker Complexes}
\markright{Relations and Dowker Complexes}
\label{relations}

This subsection reviews material on relations, taken fairly directly
from \cite{paths:privacy}, with some descriptions verbatim.

\vst

Let $X$ and $Y$ be nonempty finite discrete spaces.  A {\em relation}
$R$ on $\XxY$ is a set of ordered pairs constituting a subset of the
cross product $\XxY$.  We frequently view $R$ as a matrix of blank and
nonblank entries, with $X$ indexing rows and $Y\yless$ indexing
columns.  We often refer to elements of $\hspt{}X\mskip-0.5mu$ as
\hspc{\em individuals\vlsp} and to elements of $\hspt{}Y\mskip-2.3mu$
as {\em attributes}.

For each $x\in{X}$, we let $Y_x$ be the set of attributes that
individual $x$ has.  Formally, $Y_x = \setdef{\mskip2muy\in
Y}{(x,y)\in R}$.  We may view $Y_x$ as the row of $R\mskip0.75mu$
indexed by $x$ (or more precisely, as all the attributes with nonblank
entries in the row indexed by $x$).
Similarly, for each $y\in{Y}$, we let $X_y$ be the set of individuals
who have attribute $y$, that is, $X_y = \setdef{\mskip1mux\in
X}{(x,y)\in R}$.  We may view $X_y$ as the column of $R\mskip0.5mu$
indexed by $y$ (or more precisely, as all the individuals with
nonblank entries in the column indexed by $y$).

\vspace*{0.05in}

Given a relation $R$, we define two simplicial complexes, $\dowy$ and
$\dowx$, as follows:

\vst

$\dowy$ is called the {\em Dowker attribute complex}.  It has
  underlying vertex set $Y\!$ and is generated by the rows of $R$.
\ $\dowx$ is called the {\em Dowker association complex}.  It has
  underlying vertex set $X$ and is generated by the columns of $R$.
\ Thus:
$$\dowy=\bigunion_{x\in X}\hbox{$<$$Y_x$$>$} \qquad\hbox{and}\qquad
  \dowx=\bigunion_{y\in Y}\hbox{$<$$X_y$$>$}.$$
(The symbol \hbox{$<$$\sigma$$>$} means the {\em simplicial complex
generated by} $\sigma$, that is, the collection of all subsets of
$\sigma$, including the empty simplex $\mskip2mu\emptyset$ and
$\sigma$ itself.)

\vspace*{0.1in}

We define two {\em interpretation maps}, $\vtsp\phi_R$ and $\psi_R$, as
follows:

\vspace*{-0.15in}

\begin{eqnarray*}
\phi_R(\sigma) &=& \biginter_{x\in\sigma}Y_x, \qquad \hbox{for any $\sigma\subseteq{X}$,}\\[2pt]
\psi_R(\gamma) &=& \biginter_{y\in\gamma}X_y, \qquad \hbox{for any $\gamma\subseteq{Y}$.}\\
\end{eqnarray*}

\vspace*{-0.1in}

Thus $\phi_R(\sigma)$ consists of all attributes shared by at least
all the individuals in $\sigma$, while $\psi_R(\gamma)$ consists of
all individuals who each have at least all the attributes in $\gamma$.

\vso

Observe that $\phi_R(\emptyset) = Y$ and $\psi_R(\emptyset) = X$.

\vso

One may regard $\phi_R$ both as an interpretation map as well as a
test for membership in the Dowker complex $\dowx$.  Specifically, for
all $\sigma\subseteq{X}$, $\hspt\sigma\in\dowx\hspt$ if and only if
$\hspt\phi_R(\sigma)\neq\emptyset$.  Moreover, if
$\hspt\emptyset\neq\sigma\in\dowx$, then
$\emptyset\neq\phi_R(\sigma)\in\dowy$.

\vso

Similarly, for all $\gamma\subseteq{Y}$, $\hspt\gamma\in\dowy\hspt$ if
and only if $\hspt\psi_R(\gamma)\neq\emptyset$.  And
$\emptyset\neq\psi_R(\gamma)\in\dowx$ whenever
$\emptyset\neq\gamma\in\dowy$.

\vspace*{0.1in}

We say that individual $x \in X$ is \hspace*{0.075pt}{\em identifiable
via $\mskip0.5muR\mskip1.5mu$} whenever $\psi_R(Y_x) = \{x\}$.  In
other words, an individual $x$ is identifiable when no other
individual's attributes include all of $x$'s attributes.

\clearpage

The following facts are useful to remember \cite{paths:privacy}:

\vspace*{0.05in}

\begin{minipage}{5.7in}
\begin{enumerate}
\label{closurefacts}
\item Each of $\phi_R$ and $\psi_R$ is inclusion-reversing.
\item For all $\gamma\subseteq{Y}$, \
  $\gamma\subseteq(\clsy)(\gamma)$.  \ Similarly for $\clsx$.
\item If $\gamma$ is a maximal simplex of $\dowy$, then
  $(\clsy)(\gamma)=\gamma$.  \ Similarly for $\clsx$.
\item Each of the compositions $\clsy$ and $\clsx$ is idempotent.
\item $\phi_R \circ \psi_R \circ \phi_R = \phi_R$.  \ Similarly with
  the roles of $\phi_R$ and $\psi_R$ interchanged.
\end{enumerate}
\end{minipage}

\vspace*{0.15in}

The maps $\phi_R$ and $\psi_R$ define homotopy equivalences between
the two Dowker complexes \cite{paths:privacy}.  In particular, the
compositions $\clsy$ and $\clsx$ are homotopy equivalent to the
identity maps on their respective simplicial complexes.  These
equivalences allow one to construct a poset whose elements may be
viewed as pairs of sets $(\sigma, \gamma)$ satisfying $\emptyset \neq
\sigma = \psi_R(\gamma) \in \dowx$ and $\emptyset \neq \gamma =
\phi_R(\sigma) \in \dowy$.  This poset has an encompassing lattice
structure and is amenable to topological analysis: When the poset has
high-dimensional homology, one can be assured that it contains long
chains.  We will not need the details of that poset construction in
this report.  Instead, we jump directly to one additional definition
that we will need:

\paragraph{Informative Attribute Release Sequences:}\label{iarsdef}
\ An {\em informative attribute release sequence (for relation $R$)},
abbreviated as {\em iars}, is a nonempty set of attributes in
$\hspace{0.7pt}{Y}$ released in a particular sequential order

\vspace*{-0.15in}

$$y_1, y_2, \ldots, y_k, \quad\hbox{with $k \geq 1$},$$
satisfying

\vspace*{-0.1in}

$$y_i \not\in (\clsy)(\{y_1, \ldots, y_{i-1}\}), \quad
  \hbox{for all $1 \leq i \leq k$}.$$

\vspace*{0.05in}

In order to understand this last condition, recall from
\cite{paths:privacy} that $\sdiff{(\clsy)(\gamma)}{\gamma}$, with
$\gamma\in\dowy$, is the set of all attributes inferable from
$\gamma$.  For instance, one may have directly observed attributes
$\gamma$ for some unknown individual known to be modeled by relation
$R$.  Then the set of attributes $\sdiff{(\clsy)(\gamma)}{\gamma}$ is
inferable without direct observation.  Thus the condition above
requires that no attribute $y_i$ in the sequence be inferable from the
attributes $y_1, \ldots, y_{i-1}$ released before $y_i$.  In
particular, $y_1$ must not be inferable ``for free'', i.e., without
observation.  (The cone apex $\td_1$ of the complex in
\fig{StreamNarrowLimitedRelationComplex} on
page~\pageref{StreamNarrowLimitedRelationComplex} was inferable for
free and thus would never appear in an informative attribute release
sequence for the relation in that figure.)

\vst

{\bf Comment:}\  An informative attribute release sequence $y_1, y_2,
\ldots, y_k$ might not form a simplex in $\dowy$, but any proper prefix
of the sequence will.  It may at first seem counterintuitive to have
a nonsimplex be informative, but the inconsistency one obtains with
the last attribute released may provide information in some relation
containing $R$, as discussed in \cite{paths:privacy}.

\vst

Of interest in some privacy settings is how long one can delay
identifying an individual while revealing information: Given an
individual $x$, how large can one make $k$ in defining an informative
attribute release sequence $y_1, y_2, \ldots, y_k$ for which
$\psi_R(\{y_1, \ldots, y_k\}) = \{x\}$?  Topology offers lower bounds
\cite{paths:privacy}.  In this report, we will consider that question
with strategies in place of individuals and actions in place of
attributes.  We will argue from first principles.
\clearpage

\paragraph{Relations from Complexes:}\ \label{actionrelation}Suppose
$\Gamma$ is a nonvoid simplicial complex with underlying vertex set
$Y\!\neq\emptyset$.  We can define a relation $R$ on
$\vtsp\maxGam\times{Y}$, with $\maxGam$ the maximal simplices of
$\Gamma$:
$$R = \setdef{(\gamma, y)}{y \in \gamma \in \maxGam}.$$
One readily sees that $\dowy = \Gamma$.
\ \ (See footnote
\footnote{Void vs.~Empty: The {\em void complex}\, is the simplicial
  complex $\Gamma=\emptyset$, containing no simplices.  In contrast,
  the {\em empty complex}\, is the simplicial complex
  $\Gamma=\{\emptyset\}$, containing a single simplex, namely the {\em
  empty simplex}\, $\emptyset$.  One may construct $R\mskip1mu$ from
  the empty complex, assuming $\mskip1muY\mskip-4mu\neq\emptyset$.  In
  that case, $\maxGam=\{\emptyset\}$, $R=\emptyset$, and $\dowy =
  \{\emptyset\} = \Gamma$.}
for a side comment.)

\vspace*{0.1in}

\noindent Observe that every maximal simplex of $\Gamma$, i.e., every
$\gamma \in \maxGam$, is identifiable via relation $R$.

\vspace*{0.1in}

\paragraph{Action Relations:}\ Suppose $G=(V, \frakA)$ is a graph as
on page~\pageref{strategies}, now with both $V\mskip-3.65mu \neq
\emptyset$ and $\mskip1mu\frakA \neq \emptyset$.  \ We may substitute
$\DG$ for $\Gamma$ in the previous construction to obtain a relation
on $\maxGam\times\frakA$, with $\maxGam$ the maximal simplices of
$\DG$.  We refer to such a relation as an {\em action relation}, or
more specifically, as {\em graph $G$'s $\mskip-1mu$action relation}.
In an action relation, maximal strategies\footnote{The term {\em
maximal strategy (in $\mskip2muG\mskip-1mu$ or
$\mskip1mu\DG$)\hspd} is synonymous with {\em maximal simplex (in
$\DG$)}.}  play the role of individuals while actions play the
role of attributes.  The strategy relations of Section~\ref{river}
were of this form, though there we were only considering subcomplexes
of the full strategy complex.

\vspace*{0.1in}

In this context, informative {\em attribute\hspq} release sequences
become informative {\em action\hspq} release sequences.
\ We may thus ask the question:

\vspace*{0.05in}

\begin{center}
\begin{minipage}{4in}
\raggedright
{\bf How many actions can one reveal informatively before one has
identified a maximal strategy?}
\end{minipage}
\end{center}

\paragraph{Caution:}\ The order via which actions are revealed in an
informative action release sequence need not correspond to the order
in which actions might be executed at runtime.

\vspace*{0.1in}

\paragraph{Terminology:}\ Let $G$ be a graph as on
page~\pageref{strategies}.  \ We will make statements of the form
``maximal strategy $\sigma$ in $G\mskip1.5mu$ contains informative
action release sequence $a_1, a_2, \ldots, a_k$''.  This statement
means that the following three conditions hold:

\vspace*{-0.06in}

\begin{center}
\begin{minipage}{5.65in}
\begin{itemize}
\addtolength{\itemsep}{-4pt}
\item[(a)] $\sigma$ is a maximal simplex in $\DG$.
\item[(b)] $\{a_1, a_2, \ldots, a_k\} \subseteq \sigma$.
\item[(c)] $a_1, a_2, \ldots, a_k\mskip2mu$ is an informative
  attribute release sequence for $G$'s action relation.
\end{itemize}
\end{minipage}
\end{center}

In particular, the order of the actions $\mskip1mua_1, a_2, \ldots,
a_k\mskip1mu$ is significant.  If we view the same set of actions in a
different order we obtain a different sequence.  Consequently, a
statement of the form ``$\sigma$ contains $k!$ different informative
action release sequences'' has meaning even when $\abs{\sigma}=k$.  \
(Caution: A permutation of an informative attribute release sequence
need not itself be an informative attribute release sequence.  See
page~\pageref{permsnotiars}.)

\clearpage
\subsection{Minimal Nonfaces}
\markright{Minimal Nonfaces}
\label{nonfaces}

Suppose $\Gamma$ is a simplicial complex with underlying vertex set
$Y$.  A {\em minimal nonface} of (or in) $\Gamma\hspd$ is a subset of
$\vmsp{}Y\yless$ that is not itself a simplex in $\Gamma$ but all of
whose proper subsets are simplices in $\Gamma$.

\vst

If we arrange the vertices of a minimal nonface in any order, we
obtain an informative attribute release sequence.  That fact is the
content of our first lemma:

\begin{lemma}[Minimal Nonfaces as Informative Attribute Release Sequences]\label{minnonfaceiars}
\ Let $R$ be a relation on $\XxY$, with both $\hspt{}X\!$ and
$\hspt{}Y\!$ nonempty.
Suppose $\kappa$ is a minimal nonface of $\vlsp\dowy$.  Then any
ordering of the attributes in $\kappa$ is an informative attribute
release sequence for $R$.
\end{lemma}

\begin{proof}
Suppose the lemma fails for some minimal nonface $\kappa$ of $\dowy$.
Necessarily, $\kappa\neq\emptyset$.  Let $k = \abs{\kappa}$.  Then,
for some ordering $y_1, y_2, \ldots, y_k$ of the attributes in
$\kappa$, $y_k$ must be implied by $y_1, \ldots, y_{k-1}$ (if $k=1$,
this means $y_1$ is inferable ``for free'').  Since $\kappa$ is a
minimal nonface of $\dowy$, $\{y_1, \ldots, y_{k-1}\}$ is a simplex in
$\dowy$.  Thus $\{y_1, \ldots, y_{k-1}\}\subseteq(\clsy)(\{y_1,
\ldots, y_{k-1}\})\in\dowy$.  By supposition, $y_k\in(\clsy)(\{y_1,
\ldots, y_{k-1}\})$.  Consequently, $\kappa \subseteq (\clsy)(\{y_1,
\ldots, y_{k-1}\}) \in \dowy$, contradicting the assumption that
$\kappa$ is a nonface of $\dowy$.
\end{proof}

\vspace*{0.05in}

Suppose all individuals in a relation $R$ are identifiable.  Then all
rows of $R$ are distinct and each row forms a maximal simplex in the
attribute complex $\dowy$.  Suppose an observer has observed
attributes $\eta$ for some unknown individual $x$ known to be modeled
by relation $R$.  Even if $\eta$ is a proper subset of $Y_x$, it is
possible that $\eta$ identifies $x$, meaning $\psi_R(\eta)=\{x\}$.  In
that case, $Y_x$ is the only maximal simplex containing $\eta$.
Conversely, if the observed attributes $\eta$ do not identify
individual $x$, then $\eta$ must be contained in some maximal simplex
besides $Y_x$.  Thus there exists an attribute $y$ that is not one of
$x$'s attributes but that is consistent with all the observed
attributes $\eta$ of $x$, meaning $\eta\union\{y\} \in \dowy$.  The
following lemma characterizes this situation more generally for a
simplicial complex, in terms of minimal nonfaces.

\begin{lemma}[Minimal Nonfaces between a Maximal Simplex and a Separate Vertex]\label{minnonfacesmaxsimplex}
Suppose $\Gamma$ is a simplicial complex with underlying vertex set
$\vtsp{}Y$.
\,Let $\gamma$ be a maximal simplex of $\,\hspc\Gamma$ and let
\hspq$y\in Y\!$ such that $y\not\in\gamma$. \ Define

\vspace*{-0.125in}

$$\calK \;=\; \setdef{\kappa \subseteq \gamma \union
  \{y\}}{\hbox{$\kappa$ is a minimal nonface of $\vlsp\Gamma$}}.$$

\vspace*{0.02in}

Suppose $\eta\subseteq\gamma$.
Let $\eta^\prime = \eta\union\{y\}$.

\vspace*{2.5pt}

Then $\vtsp\eta^\prime \in \Gamma$ \ if and only if \ 
$\,\sdiff{\kappa}{\eta^\prime}\neq\emptyset\vtsp$ for every
$\kappa\in\calK$.

\end{lemma}

Comments: \ (i) $\calK\mskip-2mu\neq\emptyset$, since $\gamma$ is a
maximal simplex in $\Gamma$ and $y\not\in\gamma$. \
(ii) If $\{y\}\not\in\Gamma$, then $\calK$ consists solely of $\{y\}$
and $\sdiff{\{y\}}{\eta^\prime}=\emptyset$, no matter what $\eta$ is.
Indeed, no $\eta^\prime$ can be in $\Gamma$.
(iii) If $\{y\}\in\Gamma$, then $\gamma$ cannot be the empty simplex.
Every $\kappa\in\calK$ now contains at least two vertices, namely $y$
and some element of $\gamma$.  Therefore the lemma's assertion for
$\eta=\emptyset$ is clear.
(iv) More generally, the lemma says: Even though vertex $y$ cannot
enlarge simplex $\gamma$, it may be able to enlarge a face $\eta$ of
$\gamma$.  Such enlargement is possible precisely when the enlarged
set contains no minimal nonfaces of the type described by $\calK$.

\begin{proof}
Suppose $\eta^\prime\in\Gamma$.  If for some $\kappa\in\calK$,
$\sdiff{\kappa}{\eta^\prime}=\emptyset$, then $\eta^\prime$ would
contain $\kappa$ as a minimal nonface, a contradiction.  Now suppose
$\eta^\prime\not\in\Gamma$.  Then $\eta^\prime$ must contain some
minimal nonface, necessarily a set $\kappa$ in $\calK$ since
$\eta^\prime\subseteq\gamma\union\{y\}$.  Thus
$\,\sdiff{\kappa}{\eta^\prime}=\emptyset$.
\end{proof}

\paragraph{Minimal Nonfaces in a Strategy Complex:}\ 
Specializing to minimal nonfaces of a strategy complex yields
additional results, as discussed below.

\begin{lemma}[Minimal Nonfaces in Strategy Complexes]\label{minnonfacestrat}
\ Let $G=(V,\frakA)$ be a graph as on page~\pageref{strategies}, with
$\mskip1muV\! \neq \emptyset$.
Suppose $\kappa$ is a minimal nonface of $\vlsp\DG$.  Then the actions
in $\kappa$ all have distinct sources and no action in $\kappa$ moves
off $\,\src(\kappa)$ in $\mskip1.5muG$.
\end{lemma}

\begin{proof}
Let $k = \abs{\kappa}$.  \ Since $V\! \neq \emptyset$, $\,k>0$.

Write $\kappa = \{a_1, \ldots, a_k\}$.  For each $i=1,\ldots,k$,
define $\kappa_i = \sdiff{\kappa}{\{a_i\}}$, that is, remove one
action from $\kappa$.  Then $\kappa_i\in\DG$, for $i=1,\ldots,k$.
Thus, for every $\emptyset\neq\tau\subseteq\kappa_i$, some action in
$\tau$ must move off $\src(\tau)$.  On the other hand,
$\kappa\not\in\DG$, so for some $\emptyset\neq\xi\subseteq\kappa$, no
action in $\xi$ moves off $\src(\xi)$.  Consequently $\xi=\kappa$,
establishing the second assertion of the lemma.

The first assertion is trivial if $k=1$, so assume $k>1$ and suppose
$\src(a_1)=\src(a_2)$.  Then
$\src(\kappa)=\src(\kappa_1)=\src(\kappa_2)$.  Some action in
$\kappa_1$ moves off $\src(\kappa_1)=\src(\kappa)$.  That contradicts
the previous paragraph, thereby establishing the first assertion of
the lemma.
\end{proof}

\paragraph{Interpreting Minimal Nonfaces in Strategy Complexes:}\ 
Let us examine the meaning of minimal nonfaces for the two types of
pure graphs discussed in this report.  Assume the notation of
Lemma~\ref{minnonfacestrat} and its proof.\label{minnonfacesemantics}

\vspace*{-0.05in}

\begin{itemize}

\item Suppose $G$ is a pure nondeterministic graph.  Inductively,
  Lemma~\ref{minnonfacestrat} produces a cycle of actions $a_1$,
  \ldots, $a_k$, such that $\src(a_{i+1}) \in \trg(a_i)$, for $i=1,
  \ldots, k$ (here indices wrap around, so that $a_{k+1}$ again means
  $a_1$).  Moreover, for each action $a_i$, {\em exactly
  one}\hspace*{1pt} of the action's targets lies in $\src(\kappa)$;
  any additional targets lie outside $\src(\kappa)$.  (Otherwise, one
  could create a shorter cycle and thus a proper subset of $\kappa$
  would be a nonface of $\DG$.)

\item Suppose $G$ is a pure stochastic graph.  In the definition of
  ``moves off'' from page~\pageref{movesoff}, the quantification over
  targets is different for stochastic actions than for
  nondeterministic actions.  Consequently, Lemma~\ref{minnonfacestrat}
  now implies that {\em all\,\hspc} targets of every action in
  $\kappa$ must lie within $\src(\kappa)$.  One may therefore create a
  subgraph $H$ of $G$ defined by $H=(\src(\kappa), \kappa)$.  One sees
  that $\kappa$ is also a minimal nonface in $\DTH$, that $\DTH$ is
  the boundary complex\footnote{The {\em boundary complex on the set
  $Z$} (with $Z$ finite) is the simplicial complex whose underlying
  vertex set is $Z$ and whose simplices are all the proper subsets of
  $Z$.} on the set $\kappa$, and that $H$ is a fully controllable pure
  stochastic graph.  In fact, $H$ defines an irreducible Markov chain
  \cite{paths:feller, paths:ktII, paths:plans}.

\end{itemize}

\clearpage
\subsection{Sample Graphs, Relations, and Informative Action Release Sequences}

This subsection provides examples of graphs, action relations, and
strategy complexes, along with discussion of the extent to which
strategy or goal identification may be delayed. \ (Actions here are
deterministic or nondeterministic.  Stochastic actions appear in
Sections \ref{stochastic} and \ref{mixed}.)

\subsubsection{A Directed Cycle Graph}
\markright{A Directed Cycle Graph}
\label{directedcycle}

\begin{figure}[h]
\begin{center}
\begin{minipage}{1.5in}
\ifig{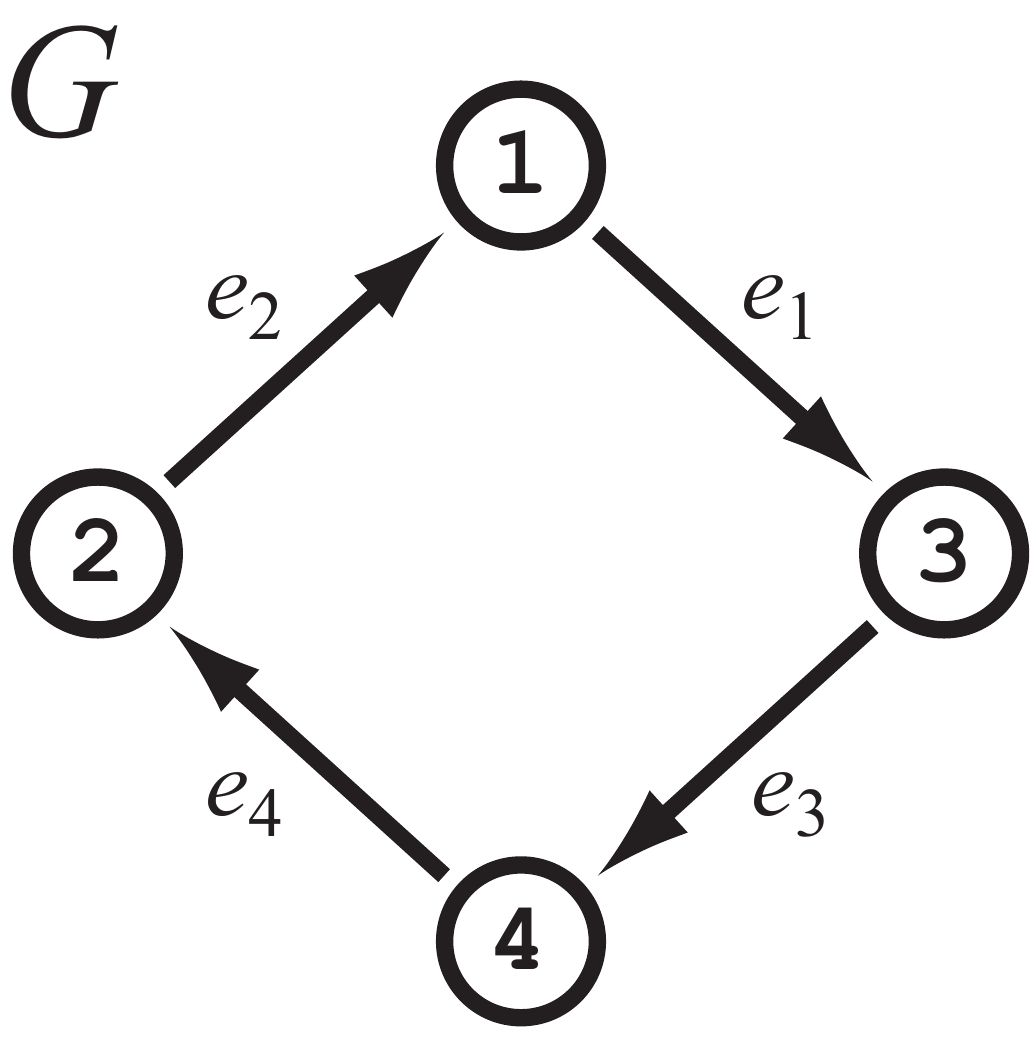}{height=1.6in}
\end{minipage}
\hspace*{0.4in}
\begin{minipage}{1.4in}{$\begin{array}{c|cccc}
    A       & e_1  & e_2  & e_3  & e_4  \\[2pt]\hline
\sigma_1    &      & \one & \one & \one \\[2pt]
\sigma_2    & \one &      & \one & \one \\[2pt]
\sigma_3    & \one & \one &      & \one \\[2pt]
\sigma_4    & \one & \one & \one &      \\[2pt]
\end{array}$}
\end{minipage}
\begin{minipage}{0.5in}{$\begin{array}{c}
\hbox{Goal} \\[2pt]\hline
1 \\[2pt]
2 \\[2pt]
3 \\[2pt]
4 \\[2pt]
\end{array}$}
\end{minipage}
\hspace*{0.25in}
\begin{minipage}{1.25in}
\ifig{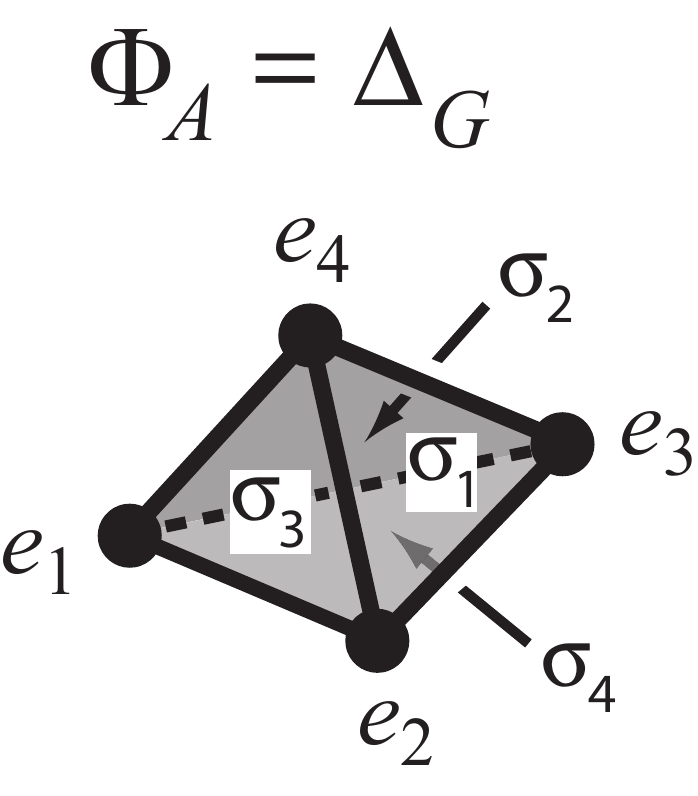}{scale=0.5}
\end{minipage}
\end{center}
\vspace*{-0.15in}
\caption[]{{\bf Left Panel:}\ A deterministic graph with four
  states and four actions that form a directed cycle.
\ \ 
  {\bf Middle Panel:}\ The graph's action relation $\mskip-2muA$,
  along with each maximal strategy's goal.
\ \ 
{\bf Right Panel:}\ The Dowker attribute complex $\dowAy$ (which is
  necessarily the same as the strategy complex $\DG$).  It is a hollow
  tetrahedron.  Vertices are actions and triangles are maximal
  strategies, as indicated by the labels.}
\label{CycleGraph4}
\end{figure}

As a first example, consider the directed graph $G$ in the left panel
of \fig{CycleGraph4}.  The graph contains four states and four
directed edges.  (The directed edges represent deterministic actions.)
These edges form a directed cycle.  Any proper subset of the four
directed edges does not form a cycle.  Consequently, any set of three
directed edges forms a strategy, in fact a maximal strategy, that
converges to one of the states in the graph, from any other state in
the graph.  For instance, the strategy $\sigma_4$, consisting of the
set $\{e_1, e_2, e_3\}$ of directed edges, converges to state \#4, for
any initial starting state of the system.

\vspace*{0.05in}

{\bf Comment:}\ Any subset of a maximal strategy is also a strategy,
since it too will be acyclic.  For instance, the set of directed edges
$\{e_1, e_4\}$ is a strategy that stops at either state \#2 or state
\#3.  (The precise stopping point depends on the starting point during
a particular execution of the strategy.)  The set $\{e_1, e_4\}$ is a
strategy but it it is {\em not\,} a {\em maximal\,} strategy.

\vspace*{0.05in}

The middle panel of \fig{CycleGraph4} shows graph $G$'s action
relation, describing each maximal strategy by its constituent actions.
For each state $v$, there is a maximal strategy $\sigma_v$ converging
to that state from anywhere else in the graph.  Therefore, $G$ is
fully controllable.  The strategy complex $\DG$ is in fact generated
by four such maximal strategies, each consisting of three directed
edges.  Consequently, the strategy complex is a hollow tetrahedron, as
shown in the right panel of the figure.  In particular, the strategy
complex contains a single minimal nonface, namely the set $\{e_1, e_2,
e_3, e_4\}$, consisting of all four directed edges (actions) in the
graph.

\vspace*{0.0425in}

There are no free faces in the strategy complex, so it is impossible to
infer any actions of a strategy from any actions revealed ---
``attribute privacy is preserved'' \cite{paths:privacy}.  Thus
it is impossible to identify a maximal strategy uniquely if one knows
only a proper subset of its actions.
Each maximal strategy consists of three actions and has no free faces.
Consequently, each maximal strategy contains $3!$ different
informative action release sequences that identify the strategy, and
each such sequence has length 3.  For instance, the six sequences for
strategy $\sigma_4$ are:

\vspace*{-0.2in}

\begin{eqnarray*}
e_1, e_2, e_3 & \quad e_2, e_3, e_1 & \quad e_3, e_1, e_2 \\
e_3, e_2, e_1 & \quad e_2, e_1, e_3 & \quad e_1, e_3, e_2.
\end{eqnarray*}

\vspace*{0.025in}

\paragraph{Ability to Delay Strategy Identification:}\ \label{stratdefer}
Let $G=(V,\frakA)$ be a fully controllable graph with $n=\abs{V} > 1$.
The following property holds \cite{paths:privacy}: For every state
$v\in{V}$, there is some maximal strategy $\sigma_v\in\DG$ such that
$\sigma_v$ has goal $v$ and contains at least $(n-1)!$ different
informative action release sequences of length at least $n-1$ each.
\ For each such sequence $a_1, \ldots, a_\ell$, this means the
following: An observer cannot infer (via $G$'s action relation) that
$\sigma_v$ contains action $a_i$ merely from knowing that $\sigma_v$
contains the set $\{a_1, \ldots, a_{i-1}\}$ of actions appearing
earlier in the sequence.  In particular, an observer cannot identify
$\sigma_v$ uniquely before seeing all actions in the sequence $a_1,
\ldots, a_\ell$.  \ Moreover, $\ell \geq n-1$.

\vspace*{0.05in}

{\bf Comment:}\quad In the example of \fig{CycleGraph4}, the six
informative action release sequences of length 3 within each maximal
strategy were permutations of the strategy's three constituent
actions.  In general, the $\mskip0.75mu(n-1)!\mskip1mu$ different
sequences need not be permutations of each other.

\vspace*{0.05in}

\subsubsection{A Graph with a Subspace Cycle}
\markright{A Graph with a Subspace Cycle}

\begin{figure}[h]
\begin{center}
\ifig{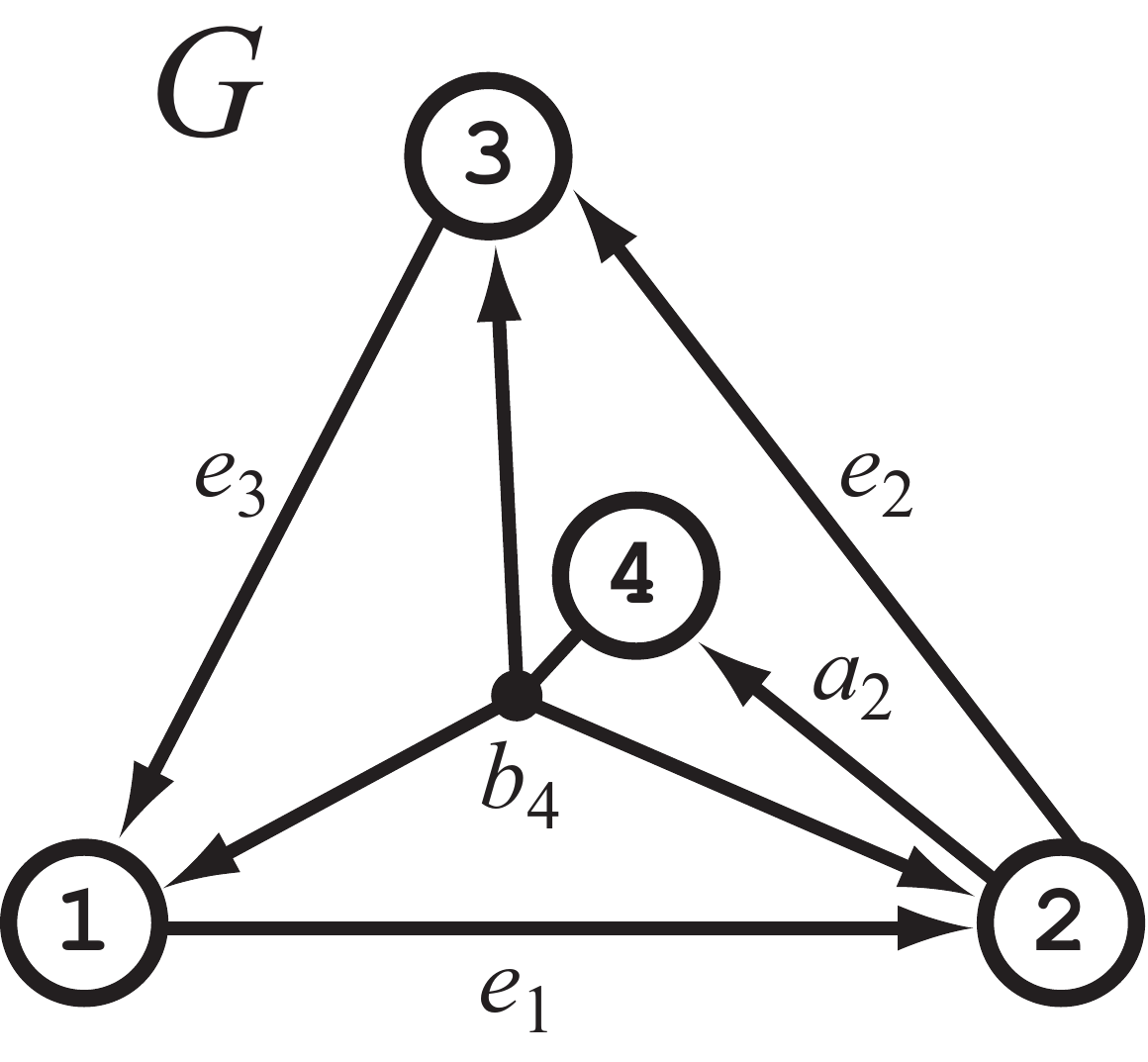}{height=1.8in}
\end{center}
\vspace*{-0.15in}
\caption[]{A pure nondeterministic graph with four states, $1,2,3,4$,
  \ four deterministic actions, $e_1$, $e_2$, $e_3$, $a_2$, \ and one
  nondeterministic action, $b_4$.}
\label{Ex202b_Graph}
\end{figure}

As a second example, let us consider a graph with a directed cycle
merely on a proper subset of the state space, as shown in
\fig{Ex202b_Graph}.  The graph again consists of four states.  The set
of deterministic actions $\{e_1, e_2, e_3\}$ forms a directed cycle on
the set of states $\{1, 2, 3\}$.  In addition, there is a
deterministic action $a_2$ that moves off this cycle space,
specifically from state \#2 to state \#4.  Finally, there is a
nondeterministic action $b_4$ that moves from state \#4 back to the
set of target states $\{1, 2, 3\}$.  (To say that the action is
nondeterministic means that the precise target state attained cannot
be predicted in advance, not even stochastically.)

\vst

The actions $e_1$, $e_2$, and $e_3$ form a directed cycle.  Any two of
these actions form a convergent strategy.  The remaining two actions,
$a_2$ and $b_4$, taken together, could cause the system to cycle
between states \#2 and \#4.  Any one of these actions is convergent by
itself.  Therefore, the two sets of actions $\{e_1, e_2, e_3\}$ and
$\{a_2, b_4\}$ each form a minimal nonface in the strategy complex
$\DG$.  In fact, these are the only minimal nonfaces in the strategy
complex.  They are independent of each other.  Consequently, $G$'s
strategy complex is the simplicial join of the boundary of the
triangle $\{e_1, e_2, e_3\}$ and the boundary of the edge $\{a_2,
b_4\}$.  In other words, $\DG$ is a {\em suspension}
\cite{paths:munkres, paths:wachs} of a triangle boundary.
\fig{Ex202b} depicts this complex along with $G$'s action relation.
Observe that the complex is homotopic to $\Stwo$, consistent with $G$
being fully controllable.

\begin{figure}[t]
\vspace*{0.2in}
\begin{center}
\begin{minipage}{1.75in}{$\begin{array}{c|ccccc}
    A       & e_1  & e_2  & e_3  & a_2  & b_4  \\[2pt]\hline
\sigma_1    &      & \one & \one &      & \one \\[2pt]
\sigma_2    & \one &      & \one &      & \one \\[2pt]
\sigma_3    & \one & \one &      &      & \one \\[2pt]
\sigma_4    & \one &      & \one & \one &      \\[2pt]
\sigma_{14} &      & \one & \one & \one &      \\[2pt]
\sigma_{34} & \one & \one &      & \one &      \\[2pt]
\end{array}$}
\end{minipage}
\begin{minipage}{0.5in}{$\begin{array}{c}
\hbox{Goal} \\[2pt]\hline
1 \\[2pt]
2 \\[2pt]
3 \\[2pt]
4 \\[2pt]
\{1,4\} \\[2pt]
\{3,4\} \\[2pt]
\end{array}$}
\end{minipage}
\hspace*{0.35in}
\begin{minipage}{3in}
\ifig{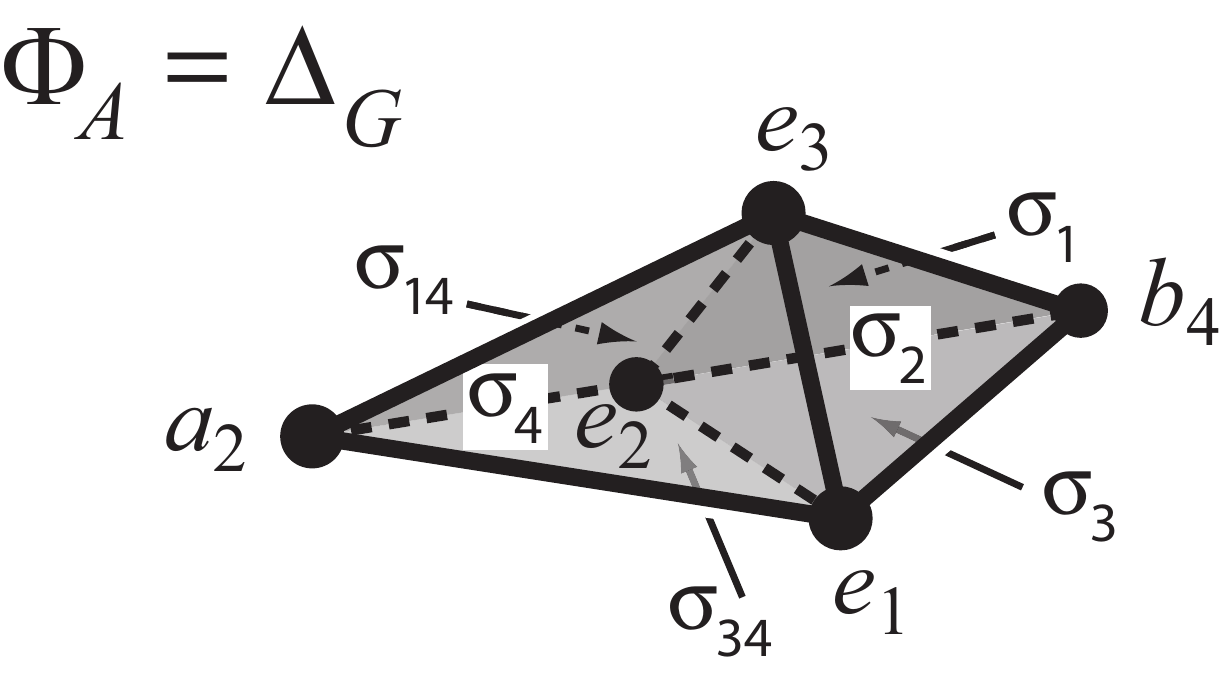}{scale=0.625}
\end{minipage}
\end{center}
\vspace*{-0.15in}
\caption[]{{\bf Left Panel:}\ The action relation $A$ for the graph $G$
  of \fig{Ex202b_Graph}, along with each maximal strategy's goal (or
  goal set).
\ \ 
  {\bf Right Panel:}\ The Dowker attribute complex $\dowAy$ (which is
  necessarily the same as the strategy complex $\DG$).  The complex
  consists of two party hats glued together, forming an $\Stwo$ hole.
  Vertices are actions and triangles are maximal strategies, as
  indicated by the labels.}
\label{Ex202b}
\end{figure}

\vso

As in the example of Section~\ref{directedcycle}, $\DG$ contains no
free faces.  So, again, it is impossible to identify a maximal
strategy uniquely from a proper subset of its constituent actions.

\vso

One salient difference between this example and the previous one is
that some maximal strategies now have goal sets with more than one
state in them.  For instance, strategy $\sigma_{34}$, consisting of
actions $\{e_1, e_2, a_2\}$ has goal set $\{3, 4\}$.  This multi-state
goal arises because the actions $e_1$ and $e_2$, taken together,
converge to state \#3 {\em assuming\hspc} the system state lies within
the subset of states $\{1, 2, 3\}$.  However, if the starting state
happens to be state \#4, then the system will simply remain at that
state.  Consequently, the strategy $\{e_1, e_2\}$ has goal set $\{3,
4\}$.  That strategy is not itself maximal.  One can augment it either
with action $b_4$, in which case the resulting maximal strategy would
be $\sigma_3$, converging to state \#3.  Or one can augment $\{e_1,
e_2\}$ with action $a_2$ to produce $\sigma_{34}$.  Adding action
$a_2$ introduces some nondeterminism at state \#2, but nothing that
changes the overall goal set; it remains $\{3, 4\}$.

\paragraph{Informative Action Release Sequences:} \
As in the example of Section~\ref{directedcycle}, the longest
informative action release sequences within each maximal strategy in
$\DG$ are simply permutations of the strategy's constituent actions.
Whenever a maximal strategy has no free faces, one can release its
actions in any order without definitively identifying the strategy
before all actions have been released.

In order to understand the more general picture, suppose we simply
interchange the roles of strategies and actions in this example.  The
``individuals'' are now $e_1$, $e_2$, $e_3$, $a_2$, $b_4$ and the
``attributes'' are $\sigma_1$, $\sigma_2$, $\sigma_3$, $\sigma_4$,
$\sigma_{14}$, $\sigma_{34}$.  We are thus interested in the Dowker
association complex $\dowAx$ of the original action relation.  That
complex appears in \fig{Ex202b_dowAx}.
\,(Comment: \label{permsnotiars}We are not asserting that this complex
is the strategy complex of a fully controllable graph, merely using
the complex to illustrate a point. One can however construct fully
controllable graphs with strategy complexes that make the same
underlying point: permutations of informative attribute release
sequences need not themselves be informative attribute release
sequences.)

\begin{figure}[h]
\begin{center}
\ifig{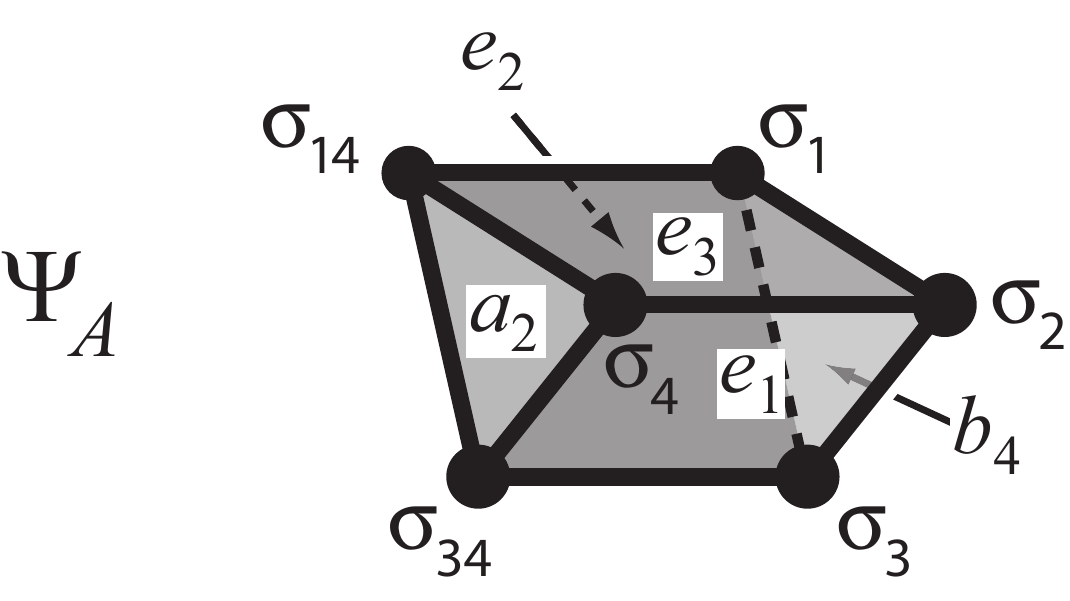}{scale=0.65}
\end{center}
\vspace*{-0.2in}
\caption[]{The Dowker association complex $\dowAx$ for the action
  relation of \fig{Ex202b}, drawn as a hollow cylinder with a
  triangular cross-section and two triangular endcaps.  The
  quadrilaterals drawn in the figure are actually solid tetrahedra,
  flattened in the figure for ease of viewing.}
\label{Ex202b_dowAx}
\end{figure}

Each maximal simplex of $\dowAx$ continues to offer (at least) $3!$
different informative release sequences of length 3 each, with
elements in each sequence drawn from the simplex's vertices.  Two of
the maximal simplices (namely, the ``endcaps'' in the figure) are
solid triangles with no free faces, so their sequences are again
simply permutations of each other.  Three of the maximal simplices are
solid tetrahedra.  The undrawn ``diagonals'' of these tetrahedra are
free faces, so releasing their endpoints would completely identify the
tetrahedron.  For instance, releasing vertices $\sigma_1$ and
$\sigma_4$ identifies the tetrahedron labeled $e_3$.  Consequently,
one cannot simply choose arbitrary sequences of length 3 and expect
them to be informative.  Nonetheless, each tetrahedron does contain 16
informative release sequences of length 3.  Here are the sequences for
the tetrahedron labeled $e_3$:

$$\begin{array}{llll}
\sigma_1, \sigma_2, \sigma_4    & \quad \sigma_2, \sigma_1, \sigma_4    & \quad  \sigma_1, \sigma_2, \sigma_{14} & \quad \sigma_2, \sigma_1, \sigma_{14} \\
\sigma_4, \sigma_{14}, \sigma_1 & \quad \sigma_{14}, \sigma_4, \sigma_1 & \quad  \sigma_4, \sigma_{14}, \sigma_2 & \quad \sigma_{14}, \sigma_4, \sigma_2 \\
\sigma_1, \sigma_{14}, \sigma_4 & \quad \sigma_{14}, \sigma_1, \sigma_4 & \quad  \sigma_1, \sigma_{14}, \sigma_2 & \quad \sigma_{14}, \sigma_1, \sigma_2 \\
\sigma_2, \sigma_4, \sigma_{14} & \quad \sigma_4, \sigma_2, \sigma_{14} & \quad  \sigma_2, \sigma_4, \sigma_1    & \quad \sigma_4, \sigma_2, \sigma_1.    
\end{array}$$

\vspace*{0.3in}

\paragraph{Incorporating Additional Constraints:}\ 
Suppose, in some context, the system only executes strategies that
converge to singleton goals.  From an inference perspective, the
action relation $A$ and strategy complex $\DG$ of \fig{Ex202b} would
be misleading.  To understand the possible inferences, one should
consider a relation $\Aoney$ that models all the maximal strategies
with singleton goal sets, and only those, as shown in
\fig{Ex202btruncated}.

\begin{figure}[h]
\begin{center}
\begin{minipage}{1.75in}{$\begin{array}{c|ccccc}
\Aoney      & e_1  & e_2  & e_3  & a_2  & b_4  \\[2pt]\hline
\sigma_1    &      & \one & \one &      & \one \\[2pt]
\sigma_2    & \one &      & \one &      & \one \\[2pt]
\sigma_3    & \one & \one &      &      & \one \\[2pt]
\sigma_4    & \one &      & \one & \one &      \\[2pt]
\end{array}$}
\end{minipage}
\begin{minipage}{0.75in}{$\begin{array}{c}
\hbox{Goal} \\[2pt]\hline
1 \\[2pt]
2 \\[2pt]
3 \\[2pt]
4 \\[2pt]
\end{array}$}
\end{minipage}
\hspace*{0.35in}
\begin{minipage}{2.95in}
\ifig{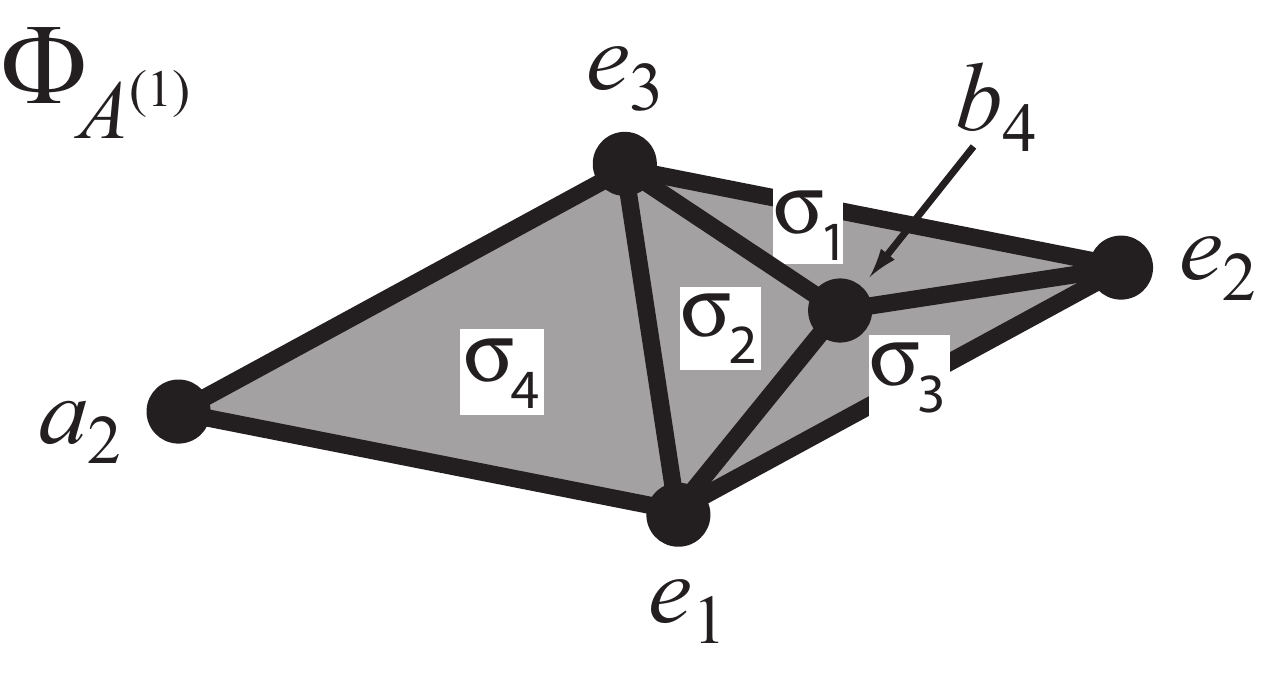}{scale=0.6}
\end{minipage}
\end{center}
\vspace*{-0.25in}
\caption[]{{\bf Left Panel:}\ Modified relation from \fig{Ex202b},
  containing only those maximal strategies that each have a singleton
  state as goal set.
\ 
  {\bf Right Panel:}\ The corresponding Dowker attribute complex
  $\dowAoney$, along with labels for actions and maximal strategies.}
\label{Ex202btruncated}
\end{figure}

With this added information, it is no longer true that one can find
$6$ different informative action release sequences of length $3$
within every maximal strategy.  For instance, action $a_2$ identifies
strategy $\sigma_4$.  Consequently, as soon as one releases that
action, the other two actions in $\sigma_4$, if not previously
released, would be implied.  As a result, there are only two
informative action release sequences of length 3 for identifying
strategy $\sigma_4$, namely $e_1, e_3, a_2\mskip1mu$ and $e_3, e_1,
a_2$.

Similarly, action $e_2$ implies action $b_4$, again limiting the
ordering of any sequences containing both those actions.  Strategy
$\sigma_1$ now contains only three, rather than six, informative
action release sequences of length 3, namely:

\vspace*{-0.1in}

$$e_3, b_4, e_2 \qquad b_4, e_3, e_2 \qquad b_4, e_2, e_3.$$

\vspace*{0.1in}

\paragraph{Ability to Delay Goal Identification:}\
Suppose $G=(V,\frakA)$ is a fully controllable graph with $n=\abs{V}>1$.
The following property holds \cite{paths:privacy}: 
For every state $v\in{V}$, there is some maximal strategy
$\tau_v\in\DG$ such that $\tau_v$ has goal $v$ and contains an
informative action release sequence whose sequential release leaves
the goal ambiguous at least until all actions in the sequence have
been revealed.  \ Moreover, the sequence reduces the goal ambiguity by
at most one state with each action revealed, so the sequence has
length at least $n-1$.

One sees this property in the complex of \fig{Ex202btruncated} since:
(i) every maximal simplex contains a vertex shared by three strategies
with different goals and (ii) the vertex lies within one of the
simplex's edges that is shared by two strategies with different goals.

\subsubsection{An Augmented Cycle Graph}
\markright{An Augmented Cycle Graph}

Let us augment the graph of \fig{CycleGraph4} with two nondeterministic
actions, as shown in \fig{Ex200_Graph}.  The actions are $a_1 =
1\ntsp\rightarrow\ntsp\{2, 3\}$ and $a_2 = 2\ntsp\rightarrow\ntsp\{3,
4\}$.

\begin{figure}[h]
\begin{center}
\begin{minipage}{2.25in}
\ifig{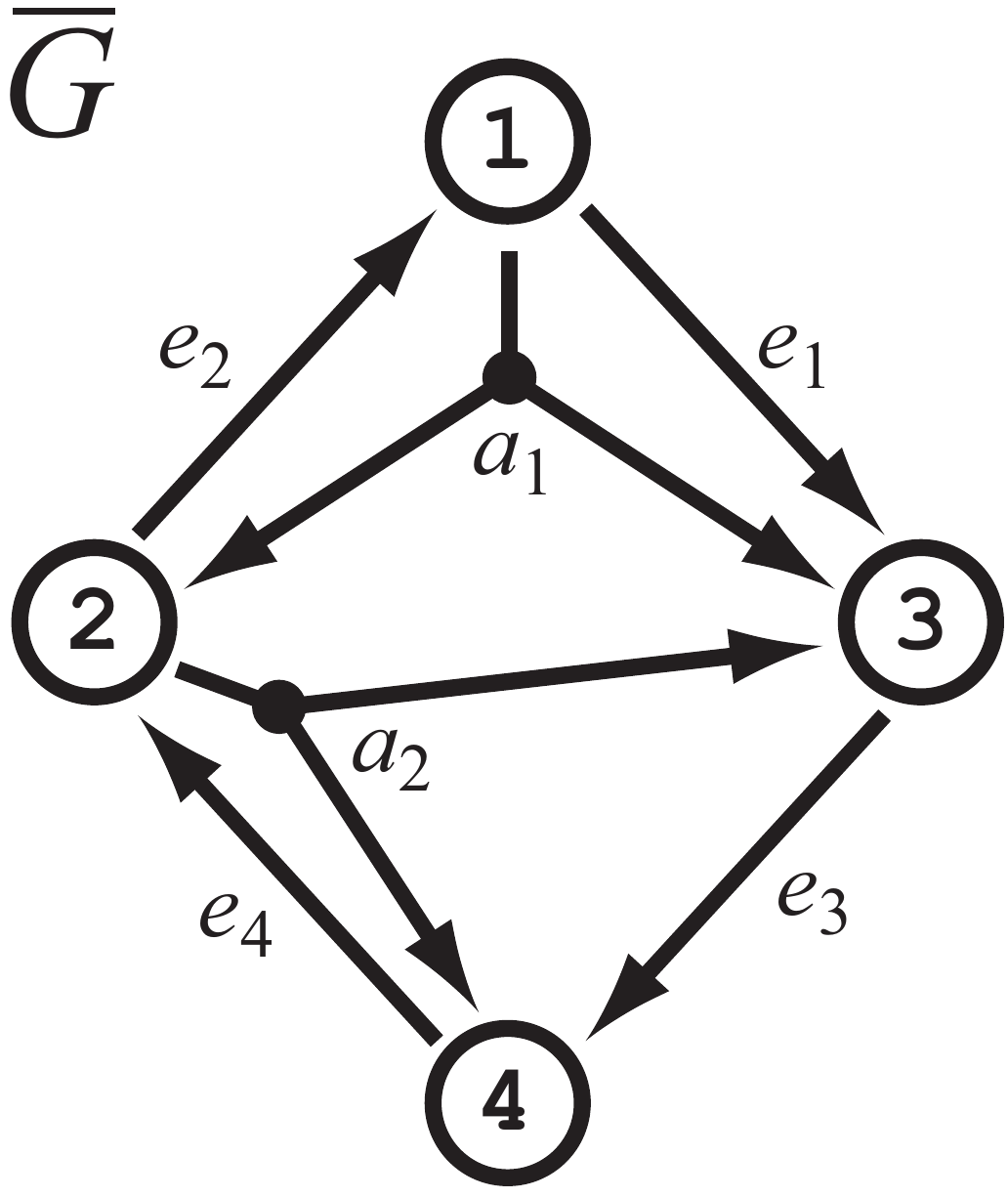}{height=2.2in}
\end{minipage}
\hspace*{0.35in}
\begin{minipage}{2in}{$\begin{array}{c|cccccc}
\Ao        & e_1  & e_2  & e_3  & e_4  & a_1  & a_2  \\[2pt]\hline
\sigo_1    &      & \one & \one & \one &      &      \\[2pt]
\sigo_2    & \one &      & \one & \one & \one &      \\[2pt]
\sigo_3    & \one & \one &      & \one &      &      \\[2pt]
\sigo_4    & \one & \one & \one &      &      & \one \\[2pt]
\sigo_5    & \one &      & \one &      & \one & \one \\[2pt]
\end{array}$}
\end{minipage}
\begin{minipage}{0.5in}{$\begin{array}{c}
\hbox{Goal} \\[2pt]\hline
1 \\[2pt]
2 \\[2pt]
3 \\[2pt]
4 \\[2pt]
4 \\[2pt]
\end{array}$}
\end{minipage}
\end{center}
\vspace*{-0.1in}
\caption[]{{\bf Left Panel:}\ A pure nondeterministic graph with four
  states, $\!1, 2, 3, 4$, four deterministic actions, $e_1$, $e_2$,
  $e_3$, $e_4$, and two nondeterministic actions, $a_1$, $a_2$.  The
  graph here is the graph from \fig{CycleGraph4} augmented with
  two nondeterministic actions.
\\[2pt]
  {\bf Right Panel:}\ The graph's action relation $\Ao$, along with
  each maximal strategy's goal.}
\label{Ex200_Graph}
\end{figure}

The new graph $\Go$ has a strategy complex $\DGo$ described by the
action relation $\Ao$ shown in \fig{Ex200_Graph}.  The complex is a
partially puffed up version of the hollow tetrahedron from
\fig{CycleGraph4}, now consisting of three solid tetrahedra and two
solid triangles glued together to enclose an $\Stwo$ hole.
\fig{Ex200_1skel} shows the 1-skeleton of this complex.

The maximal strategies from the original strategy complex $\DG$ are
still present in $\DGo$.  Two of these strategies now lie within
larger maximal simplices.  For instance, strategy $\sigma_2$ for
attaining goal state \#2 in the original graph $G$ consisted of the
actions $\{e_1, e_3, e_4\}$, whereas now the corresponding maximal
simplex $\sigo_2$ consists of the actions $\{e_1, e_3, e_4, a_1\}$.
The original strategy $\sigma_2$ always executed action $e_1$ when the
system was at state \#1, thus transitioning to state \#3.  In the new
graph $\Go$, the new $\sigo_2$ might execute either action $e_1$ or
action $a_1$ at state \#1, selected nondeterministically (possibly by
an adversary).  Action $a_1$ will transition either to state \#2 (the
goal) or to state \#3.  (If an adversary controls the outcome of
action $a_1$, then the adversary might choose to make action $a_1$
mimic action $e_1$, in which case the old $\sigma_2$ and the new
$\sigo_2$ would behave equivalently.)

Actions $e_2$ and $e_4$ both appear in both the original strategies
$\sigma_1$ and $\sigma_3$ of $\DG$.  In the new graph $\Go$, the set
$\{e_2, a_1\}$ contains a circuit, as does the set $\{e_4, a_2\}$.
Consequently, one cannot augment the strategies $\sigma_1$ and
$\sigma_3$ with either of the actions $a_1$ or $a_2$.  These
strategies remain unchanged as one passes from $G$ to $\Go$, that is,
$\sigo_1 = \sigma_1$ and $\sigo_3 = \sigma_3$.

There are five maximal strategies in the new graph, whereas there were
four previously.  New strategy $\sigo_5$ has the same goal state,
namely state \#4, as does strategy $\sigo_4$, but arrives there with
different actions, trading off action $e_2$ for action $a_1$.  The
minimal nonface $\{e_2,\vtsp a_1\}$ of $\vtsp\DGo\vtsp$ hints at this
possible tradeoff.

\vso

The ability to delay strategy identification mentioned on
page~\pageref{stratdefer}, as well as our analysis of the original
graph $G$, ensures that each of $\sigo_1, \sigo_2, \sigo_3, \sigo_4$
contains (at least) 6 informative action release sequences of length
(at least) 3 each.  What can we say about strategy $\sigo_5$?

Releasing both of the two nondeterministic actions $a_1$ and $a_2$
identifies $\sigo_5$.  Releasing either one of these actions implies
both deterministic actions in $\sigo_5$.  Consequently, one obtains
the longest possible informative action release sequences within
$\sigo_5$ by first revealing the two deterministic actions (in either
order) and then the two nondeterministic actions (in either order).
Here are the four possible longest informative action release
sequences within $\sigo_5$:

\vspace*{-0.1in}

$$e_1, e_3, a_1, a_2 \qquad e_1, e_3, a_2, a_1 
      \qquad e_3, e_1, a_1, a_2 \qquad e_3, e_1, a_2, a_1.$$

(The previous reasoning can be generalized and formalized using
lattice representations of links, as discussed in
\cite{paths:privacy}, but we will not develop or use that machinery in
this report.)

\clearpage

\begin{figure}[h]
\vspace*{0.1in}
\begin{center}
\ifig{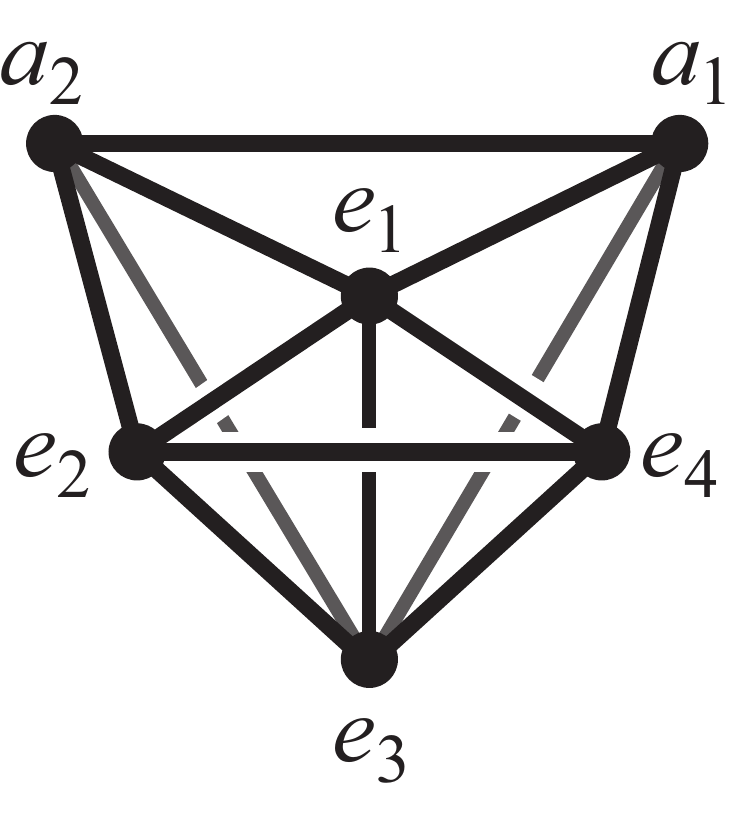}{height=2in}
\end{center}
\vspace*{-0.25in}
\caption[]{The 1-skeleton of the strategy complex $\DGo$, with $\Go$
   the graph of \fig{Ex200_Graph}.  The complex consists of three
   solid tetrahedra and two solid triangles, enclosing a hollow
   tetrahedron.  The hollow tetrahedron has vertices $\{e_1, e_2, e_3,
   e_4\}$, just as in \fig{CycleGraph4}.}
\label{Ex200_1skel}
\end{figure}

\paragraph{Informative Action Release Sequences for Maximal Strategies in Pure Graphs:}
Suppose $G$ is a fully controllable graph containing $n$ states, with
$n > 1$.  Suppose further that $G$ is either pure nondeterministic or
pure stochastic.  The remainder of this report will provide proofs
that {\em every\,} maximal strategy in $G$ contains at least one
informative action release sequence of length at least $n-1$.  The
proofs will be different for the two types of graphs.  This property
need not hold for graphs containing a mix of nondeterministic and
stochastic actions.  See \cite{paths:privacy} and Section~\ref{mixed}
for counterexamples.

\clearpage
\section{Basic Tools}
\markright{Basic Tools}
\label{tools}

Graphs and strategies were defined in Section~\ref{strategies}.
\ We wish to prove the following theorem:

\begin{theorem}[Informative Action Release Sequences for Maximal Strategies]\label{longiars}
$\phantom{0}$\\[1pt]
Let $G=(V,\frakA)$ be a fully controllable graph with $n = \abs{V} >
1$.  Suppose that $G$ is either pure nondeterministic or pure
stochastic.
\ Then every maximal strategy in $\DG$ contains an informative action
release sequence of length at least \hspace*{1pt}$n-1$.
\end{theorem}

Section~\ref{nondet} will prove Theorem~\ref{longiars} for the pure
nondeterministic case, while Section~\ref{stochastic} will prove
Theorem~\ref{longiars} for the pure stochastic case.

This section builds some tools that are useful for both settings.
Throughout this section the graph $G$ may in fact contain a mix of
deterministic, nondeterministic, and stochastic actions.

\paragraph{Terminology and Notation:}

\begin{itemize}
\item We frequently abbreviate {\em informative action release
  sequence} as {\em iars}, for both the singular and plural cases.

\item If $G$ is a graph, the phrase {\em $a_1, \ldots, a_k$ is an iars
  for $G$} means that $a_1, \ldots, a_k$ is an informative action
  release sequence for $G$'s action relation as defined on
  page~\pageref{actionrelation}.

\item Suppose $G=(V,\frakA)$ is a graph and
  $G^\prime=(V^\prime,\frakA^\prime)$ is some quotient graph of $G$.
  Recall from page~\pageref{quotient} that the sets of actions
  $\frakA$ and $\frakA^\prime$ are in one-to-one correspondence.
  Corresponding actions $a\in\frakA$ and $a^\prime\in\frakA^\prime$
  differ only in that the source and/or target states of action
  $a^\prime$ may have changed from those of action $a$ in order to
  reflect the quotient graph's state identifications.  One may
  therefore view any action of $G^\prime$ in the form $a^\prime$, with
  $a$ being the unique action of $G$ corresponding to $a^\prime$.  \
  We will use this prime notation from now on.

\end{itemize}

\vspace*{0.05in}

Our first tool allows us to combine two strategies when one of the
strategies comes from a subgraph and the other comes from a quotient
graph formed by collapsing that subgraph to a single state.

\begin{lemma}[Combining Quotient and Subgraph Strategies]\label{unquotienting}
\quad Let $G=(V,\frakA)$ be a graph and let $H=(W,\frakB)$ be a subgraph of $G$.
\ (As always, both $\hspq{}V\mskip-4mu$ and $\hsph{}W\mskip-4mu$ are assumed to be nonempty.)      

\vst

If $\hspd\sigma^\prime \in \DGW$ and $\hspc\gamma\in\DTH$,
      then $\hspc\sigma\union\gamma\in\DG$.

\vst

(Prime notation indicates corresponding actions in a graph and a
quotient, as discussed.)
\end{lemma}

Comment:\ Since all sources and targets of actions in $\gamma$ lie
within $W$, no such action is convergent in the quotient graph
$G/W$. Therefore $\sigma\inter\gamma=\emptyset$.

\begin{proof}
Suppose $\emptyset\neq\tau\subseteq\sigma\union\gamma$.  We need to
show that some action of $\tau$ moves off $\src(\tau)$ in $G$.

\vso

If $\mskip1mu\src(\sdiff{\tau\!}{\!\gamma})\inter{W}=\emptyset$, then
this assertion follows from Lemma~7.3(b)(i) in \cite{paths:plans}.

\vso

Otherwise, let $\kappa = \sdiff{\tau\!}{\!\gamma}$.  Then
$\emptyset\neq\kappa^\prime\subseteq\sigma^\prime\in\DGW$, so some
action $b^\prime$ in $\kappa^\prime$ moves off $\src(\kappa^\prime)$
in $G/W$.  If $b^\prime$ is nondeterministic, then all its targets lie
in the set
$\sdiff{\big((\sdiff{V}{W})\union\{\wrep\}\big)}{\hspc\src(\kappa^\prime)}$.
If $b^\prime$ is stochastic, then at least one of its targets lies in
that set.  (As usual, $\mskip0.5mu\wrep$ represents
$\mskip.75mu{W}\mskip-1.5mu$ identified to a single state in
$\mskip1muG/W$.)  \ Since $\mskip0.1mu\kappa\mskip0.4mu$ contains an
action with source in $W\mskip-1.5mu$, $\mskip1mu\kappa^\prime$
contains an action with source $\wrep$.  \ Thus
$\sdiff{\big((\sdiff{V}{W})\union \{\wrep\}\big)}{\hspc\src(\kappa^\prime)}
\;=\;
\sdiff{V}{\big(W\union\mskip0.5mu\src(\kappa^\prime)\big)}
\;=\;
\sdiff{V}{\big(W\union\mskip1mu\src(\kappa)\big)}
\;\subseteq\;
\sdiff{V}{\src(\tau)}
$.
That means action $b$ lies in $\tau$ and moves off $\src(\tau)$ in $G$.
\end{proof}

The next lemma ensures that unquotienting a simplicial informative
action release sequence (iars) again produces an informative action
release sequence, when the quotienting is over a proper subspace that
is fully controllable.

\begin{lemma}[Lifting Quotient {\footnotesize IARS}]\label{liftquotientiars}
\quad Let $G=(V,\frakA)$ be a graph
      and let $H=(W,\frakB)$ be a fully controllable subgraph of $G$,
      with \hsph$\emptyset\neq W \subsetneq V$.

\vst

Suppose $a^\prime_1, \ldots, a^\prime_k$ is an iars for $G/W$, with
$\{a^\prime_1, \ldots, a^\prime_k\} \in \DGW$ and $\vmsp{}k \geq 1$.

\vst

Then $a_1, \ldots, a_k$ is an iars for $G$.
\end{lemma}

\begin{proof}
Let $A$ be $G$'s action relation.  We need to show that

\vspace*{-0.1in}

$$a_i \not\in (\clsAy)(\{a_1, \ldots, a_{i-1}\}), \quad\hbox{for $i=1, \ldots, k$.}$$

Suppose this assertion is false for some $i$.  Then $a_i$ is contained in every
maximal simplex of $\DG$ that contains $\{a_1, \ldots, a_{i-1}\}$.

\vspace*{0.1in}

Let $\frakA^\prime$ be the actions of the quotient graph $G/W$.  Since
$a^\prime_1, \ldots, a^\prime_k$ is an iars for
$\mskip.5muG/W\mskip-.5mu$ and $\{a^\prime_1, \ldots, a^\prime_k\} \in
\DGW$, there exist actions $\tau^\prime$, with
$\emptyset\neq\tau^\prime\subseteq\sdiff{\frakA^\prime}{\{a^\prime_1,
\ldots, a^\prime_i\}}$, such that

\vspace*{-0.1in}

$$\{a^\prime_1, \ldots, a^\prime_{i-1}\} \union \tau^\prime \in \DGW
  \qquad\hbox{but}\qquad
  \{a^\prime_1, \ldots, a^\prime_i\} \union \tau^\prime \not\in \DGW.$$

Let $\eta^\prime$ be a minimal nonface of $\DGW$ contained in
$\{a^\prime_1, \ldots, a^\prime_i\} \union \tau^\prime$.
(Since $\{a^\prime_i\}\in\DGW$, $\eta^\prime$ contains at least two
actions, one of them being $a^\prime_i$.)
\ By Lemma~\ref{minnonfacestrat} on page~\pageref{minnonfacestrat}, no
action of $\eta^\prime$ moves off $\src(\eta^\prime)$ in $G/W$.

\vspace*{0.05in}

We consider two cases below, deriving a contradiction for each.

By Lemma~\ref{minnonfacestrat}, there are no further cases.

\vspace*{0.05in}

State $\wrep$ represents $W\mskip-1.5mu$ identified to a single state
in the quotient graph $G/W$, as per page~\pageref{quotient}.

\begin{itemize}

\item[I:] \underline{No action in $\eta^\prime$ has source $\wrep$:}\ 

Then $\src(\eta) = \src(\eta^\prime)$, so no action of $\eta$ moves
  off $\src(\eta)$ in $G$.

On the other hand, $\{a_1, \ldots, a_{i-1}\} \union \tau \in \DG$, by
  Fact 2 on page~\pageref{quotientfacts}.\\  So $\{a_1, \ldots, a_i\}
  \union \tau \in \DG$, by the falsity assumption above.  Since
  $\eta\subseteq\{a_1, \ldots, a_i\} \union \tau$, that means
  $\emptyset\neq\eta\in\DG$ and some action of $\eta$ must move off
  $\src(\eta)$ in $G$, a contradiction.

\item[II:] \underline{Exactly one action in $\eta^\prime$ has source
  $\wrep$:}\ 

Suppose the source of the corresponding action in $G$ is $w$.  Then
  $w\in W$.  Let $\kappa = \eta \union \gamma$, with $\gamma\in\DTH$ a
  strategy that attains $w$ from anywhere in $W\mskip-2mu$ using
  actions of $H$.  Then $\src(\kappa) =
  \left(\sdiff{\src(\eta^\prime)}{\{\wrep\}}\right) \union W$.  Since
  no action of $\eta^\prime$ moves off $\src(\eta^\prime)$ in
  $\mskip0.5muG/W\mskip-1.5mu$ and since $\gamma$ has all its sources
  and targets in $W$, no action of $\kappa$ moves off $\src(\kappa)$
  in $G$.

Since $\{a^\prime_1, \ldots, a^\prime_{i-1}\} \union \tau^\prime \in
\DGW$ and $\gamma\in\DTH$, Lemma~\ref{unquotienting} on
page~\pageref{unquotienting} implies that $\{a_1, \ldots, a_{i-1}\}
\union \tau \union \gamma \in \DG$.  By the falsity assumption,
$\{a_1, \ldots, a_i\} \union \tau \union \gamma \in \DG$.  Now
$\kappa\subseteq\{a_1, \ldots, a_i\} \union \tau \union \gamma$, so
$\emptyset\neq\kappa\in\DG$, and some action of $\kappa$ must move off
$\src(\kappa)$ in $G$, again a contradiction.
\end{itemize}

\vspace*{-0.35in}

\end{proof}

\clearpage
\paragraph{Combining Informative Action Release Sequences:}\ 
The next lemma shows how one may combine an iars in a graph with an
iars from a subgraph.  The subsequent corollary leverages this result
with those discussed earlier, showing how one may combine an iars from
a quotient graph with an iars from a fully controllable subgraph.
That combinability forms a stepping stone in several proofs during the
rest of the report.

\vspace*{0.1in}

\begin{lemma}[Combining Graph and Subgraph Informative Action Release Sequences]\label{combinesubgraphiars}
\quad Let $G=(V,\frakA)$ be a graph and let $H=(W,\frakB)$ be a
subgraph of $G$ (with both $\mskip1muV\mskip-4.5mu$ and
$\mskip1.2muW\mskip-4mu$ nonempty).

\vspace*{0.05in}

\noindent Suppose $a_1, \ldots, a_k$, with $k \geq 1$, is an iars for
$G$, such that:

\vspace*{0.12in}

\hspace*{1in}\begin{minipage}{4in}
\begin{itemize}
   \addtolength{\itemsep}{-3pt}

   \item[(i)] $a_i \in \sdiff{\frakA}{\frakB}$, for $i=1, \ldots, k$, and

   \item[(ii)] $\{a_1, \ldots, a_k\} \union \tau \in \DG$, for every $\tau\in\DTH$.

\end{itemize}
\end{minipage}

\vspace*{0.15in}

\noindent Suppose $b_1, \ldots, b_\ell$ is an iars for $H$, with
$\{b_1, \ldots, b_\ell\} \in \DTH$ and $\vlsp\ell \geq 1$.

\vspace*{0.05in}

\noindent Then $\mskip1mua_1, \ldots, a_k, \bonespc, \ldots, b_\ell$
is an iars for $G$,
with $\{a_1, \ldots, a_k, \bonespc, \ldots, b_\ell\}\in\DG$.
\end{lemma}

\noindent Comment:\ The lemma also holds when $k=0$, meaning every
iars for $H$ is also an iars for $G$.

\begin{proof}
Suppose the iars part of the assertion is false. \ Let $A$ be $G$'s
action relation.

\vso

Then, for some $i\in\{0, 1, \ldots, \ell - 1\}$,
\,$b_{i+1}\in(\clsAy)(\{a_1, \ldots, a_k, \bonespc, \ldots, b_i\})$.

(When $i=0$, this notation means $b_1\in(\clsAy)(\{a_1, \ldots,
a_k\})$.)

Consequently, every maximal simplex of $\DG$ containing $\{a_1,
\ldots, a_k, \bonespc, \ldots, b_i\}$ also contains $b_{i+1}$.

\vst

Since $b_1, \ldots, b_\ell$ is an iars for $H$'s
action relation, there exists a maximal simplex $\tau\in\DTH$ such
that $\{b_1, \ldots, b_i\}\subseteq\tau$ but $\{b_1, \ldots,
b_{i+1}\}\not\subseteq\tau$.

\vst

By assumption, $\{a_1, \ldots, a_k\} \union \tau \in \DG$.
Consequently, $\{a_1, \ldots, a_k, b_{i+1}\} \union \tau \in \DG$.
Thus $\tau\union\{b_{i+1}\} = \frakB\inter\big(\{a_1, \ldots, a_k,
b_{i+1}\} \union \tau\big) \in \DTH$, contradicting the maximality
of $\mskip1mu\tau$ in $\DTH$.
\end{proof}

\vspace*{0.1in}

\begin{corollary}[Lifting and Combining Informative Action Release Sequences]\label{combinequotientgraphiars}
\ Let $G=(V,\frakA)$ be a graph and let $H=(W,\frakB)$ be a fully
controllable subgraph of \vtsp$G$ with \hsph$\emptyset\neq W \subsetneq V$.

\vst

Suppose $a^\prime_1, \ldots, a^\prime_k$ is an iars for $G/W$, with
$\{a^\prime_1, \ldots, a^\prime_k\} \in \DGW$ and $\vlsp{}k \geq 1$.

\vst

Suppose further that $\vtsp{}b_1, \ldots, b_\ell$ is an iars for $H$, with
$\{b_1, \ldots, b_\ell\} \in \DTH$ and $\vlsp\ell \geq 1$.

\vst

Then $\mskip1mua_1, \ldots, a_k, \bonespc, \ldots, b_\ell$ is an iars
for $G$,
with $\{a_1, \ldots, a_k, \bonespc, \ldots, b_\ell\}\in\DG$.
\end{corollary}

\begin{proof}
By Lemma~\ref{liftquotientiars}, $a_1, \ldots, a_k$ is an iars for
$G$.  By Lemma~\ref{unquotienting}, $\{a_1, \ldots, a_k\} \union
\tau\in\DG$, for every $\tau\in\DTH$.  Since actions of $H$ become
self-loops in $G/W$, $a_i\not\in\frakB$, for $i=1, \ldots, k$.  The
desired result therefore follows from Lemma~\ref{combinesubgraphiars}.
\end{proof}

\noindent{Comment:} \ The corollary also holds if one of $k$ or $\ell$ is $0$.

\clearpage
\section{The Nondeterministic Setting}
\markright{The Nondeterministic Setting}
\label{nondet}

The aim of this section is to prove Theorem~\ref{longiars} from
page~\pageref{longiars} for the case in which the graph $G$ is pure
nondeterministic.  Throughout, this section assumes that all graphs
are pure nondeterministic, meaning each action is either deterministic
or nondeterministic (but not stochastic).
\ First, we need some additional definitions and results.

\subsection{Hierarchical Cyclic Graphs}
\markright{Hierarchical Cyclic Graphs}

We start with a recursive definition:

\begin{definition}[Hierarchical Cyclic Graph]
\ A pure nondeterministic graph $\mskip3muG=(V,\frakA)$ is a
\mydefem{hierarchical cyclic graph}\vmsp{} if\vtsp{} one of conditions
(i) or (ii) holds:
\begin{enumerate}

\item[(i)] $\abs{V}=1$ and $\mskip4.5mu\frakA=\emptyset$.

\item[(ii)] There exist $\,V_1, \ldots, V_k$, $\,\frakA_1, \ldots,
            \frakA_k$, $\,a_1, \ldots, a_k$, with $k > 1$, such that:

  \vspace*{-0.05in}

   \begin{enumerate}
   \addtolength{\itemsep}{2pt}

    \item $V_1, \ldots, V_k$ are nonempty pairwise disjoint subsets of
    $\mskip4muV\!$ and $\,V\!=\bigcup_{i=1}^kV_i$.

    \item $\frakA_i$ consists of all actions in $\,\frakA$ whose sources
    and targets lie in $V_i$, for $i=1, \ldots, k$.

    \item $(V_i, \,\frakA_i)$ is a hierarchical cyclic graph, for $i=1, \ldots, k$.

    \item $\frakA \,=\, \{a_1, \ldots, a_k\} \,\union\mskip4.5mu \bigcup_{i=1}^k\frakA_i$.

    \item For $\vmsp{}i=1, \ldots, k$, $\,\vtsp\src(a_i) \in V_i\,$ and $\mskip4mu\trg(a_i) \subseteq V_{i+1}$
          \\[1pt] (here indices wrap around, so $V_{k+1}$ again means $V_1$).
   \end{enumerate}

\end{enumerate}

The decomposition above need not be unique.  We implicitly assume a
specific decomposition when stating that a graph is hierarchical
cyclic.  We refer to it as the \vtsp\mydefem{tree decomposition} of
$\,G$.

\vspace*{0.05in}

A graph of type (i) is a \vtsp\mydefem{leaf} and a graph of type (ii)
is a \vtsp\mydefem{node}.

\vspace*{0.05in}

When $\,G$ is a node, we refer to the subgraphs $\,(V_1, \frakA_1),
\ldots, (V_k, \frakA_k)$ in {\kern .09em}$G$'s tree decomposition as
the {\kern .06em}\mydefem{children of $G$}.  Each subgraph
$\mskip1mu(V_i, \frakA_i)$ is itself either a leaf or a node, with
{\kern .07em}\mydefem{parent} $(V, \frakA)$.  When $(V_i, \frakA_i)$
is a node, we may then speak of its children, and so forth.
Transitively, we may therefore speak of all the \vzsp\mydefem{nodes
and leaves within $G$} \hspc (that includes $(V, \frakA)$).  Finally,
we may speak of the \vmsp\mydefem{root} of the tree decomposition of
$\mskip1.75muG$, meaning the node or leaf $\,(V, \frakA)$, i.e., $G$
itself.

\vspace*{0.05in}

For a graph of type (ii), the actions $a_1, \ldots, a_k$ are the
\mydefem{(top-level) cycle actions of $G$}.  Similarly, if $N\!$ is
any node within $G$, the \mydefem{cycle actions of $N$} are the
top-level cycle actions of $N\!$ when $N$ is viewed as a hierarchical
cyclic graph in its own right.

\end{definition}

\paragraph{Comments and Observations:}

\begin{itemize}
\item Given a hierarchical cyclic graph $G$ of type (ii) as above, we can
  form the quotient graph $G/\{V_1, \ldots, V_k\}$ (see again
  page~\pageref{quotient}).   This quotient graph has state space
  $\{\wrep_1, \ldots, \wrep_k\}$, where $\wrep_i$ represents all of
  $V_i$ identified to a single state, for $i=1, \ldots, k$.

  All actions in each $\frakA_i$ become nonconvergent in $G/\{V_1,
  \ldots, V_k\}$ (actions in $\frakA_i$ become self-loops on state
  $\wrep_i$), so we may ignore them.  In contrast, each action $a_i$
  turns into a {\em deterministic transition} \hsph$a^\prime_i$ from
  state $\wrep_i$ to state $\wrep_{i+1}$.

  We may therefore view the quotient graph $G/\{V_1, \ldots, V_k\}$ as
  the cycle graph\label{cyclequotient}

  \vspace*{-0.1in}

  $$\wrep_1 \;\xrightarrow{\phantom{1}a^\prime_1\phantom{1}}\; \wrep_2
       \;\xrightarrow{\phantom{1}a^\prime_2\phantom{1}}\;
          \cdots
       \;\xrightarrow{\phantom{!}a^\prime_{k-1\phantom{.}}}\; \wrep_k
         \;\xrightarrow{\phantom{1}a^\prime_k\phantom{1}}\; \wrep_1$$
  (the first and last states in the diagram above are the same state,
   namely $\wrep_1$).

\item \label{ignoreloops}More generally, suppose $G$ is a hierarchical
  cyclic graph and $(W, \frakB)$ is some node that appears within the
  tree decomposition of $G$.  We can form the quotient graph
  $\mskip1.25muG/W$.  The quotienting identifies all of
  $\mskip0.5muW\mskip-2mu$ to a single state $\wrep$.  The actions
  $\frakB$ become self-loops on state $\wrep$.  Technically,
  $\mskip0.5muG/W\mskip-1.5mu$ includes these self-loops, but there is
  no harm ignoring them, thereby allowing us to view $G/W\mskip-2mu$
  as a hierarchical cyclic graph.  If $(W, \frakB)$ is $G$ itself,
  then we may view $\mskip0.5muG/W\mskip-1.75mu$ as the leaf
  $(\{\wrep\}, \emptyset)$.  Otherwise, the tree decomposition of
  $\mskip1muG/W\mskip-1.5mu$ is largely unchanged from that of $G$,
  except that one node, along with the subtree rooted at that node,
  has now become a leaf, and any actions of $G$ with source or target
  states in $W\mskip-1mu$ have had those states relabeled as
  $\wrep$.  The only actions that become nonconvergent (by creating
  self-loops) are those in $\frakB$, which we now ignore and discard.

\item Conversely, suppose $(\{s\}, \emptyset)$ is a leaf that
  appears in the tree decomposition of a hierarchical cyclic graph
  $G=(V, \frakA)$.  Suppose $H=(W, \frakB)$ is another hierarchical
  cyclic graph, with states and actions distinct from those of $G$.

  We can replace the leaf $(\{s\}, \emptyset)$ with node $H$, to form
  a new hierarchical cyclic graph $\overline{G} = (\overline{V},
  \,\overline{\frakA}\union\frakB)$.

  Here $\,\overline{V} \,=\, (\sdiff{V}{\{s\}})\,\union\,{W}$.
  \ In forming $\overline{\frakA}$ from $\frakA$, we have some choices:

  Suppose $a=v\rightarrow{T}$ is an action in $\frakA$.  We create a
  corresponding action $\overline{a}\in\overline{\frakA}$ as follows:

  \vspace*{-0.05in}

  \begin{itemize}

   \item If $v=s$, \,we let $\overline{v}$ be {\em any\,} state
   in $\mskip0.5muW\mskip-2mu$ and define $\overline{a}=\overline{v}\rightarrow{T}$.

   \item If $s\in T$, \,we let $S$ be {\em any\,} nonempty subset of
   $\mskip1muW\mskip-1.5mu$ and then define
   $\overline{a}=v\rightarrow\overline{T}$, with $\,\overline{T} \,=\,
   (\sdiff{T}{\{s\}})\,\union\,{S}$.

   \item In all other cases, $\,\overline{a} = a$.

  \end{itemize}

\item A \label{insertcycle}special case of the previous construction
  is to replace a single state $s$ in a hierarchical cyclic graph with
  a deterministic cycle on some new set of states, while adjusting all
  other actions of the encompassing graph accordingly.  Actions of the
  encompassing graph that used to start at $s$ now start at an
  arbitrary state of the cycle.  Actions that used to have a
  transition to $s$ now might transition to one or more states
  comprising the cycle.

\item Every hierarchical cyclic graph is fully controllable and each
  of its actions is convergent.

\item Conversely, the lemma below shows that every fully controllable
  pure nondeterministic graph contains a hierarchical cyclic subgraph
  with the same state space.  \ (There may be more than one such
  subgraph.)

\end{itemize}

\vspace*{0.05in}

\begin{lemma}[Hierarchical Cyclic Subgraphs]\label{hierarch}
Let $G=(V,\frakA)$, with $V\mskip-4mu\neq\emptyset$, be a fully
controllable pure nondeterministic graph.  Then $G$ contains a
hierarchical cyclic subgraph $H=(V, \frakB)$.
\end{lemma}

\begin{proof}
By strong induction on $\abs{V}$.  The base case $\abs{V}=1$ is clear,
so suppose $\abs{V} > 1$.  \ 
For every state $v$ in $V\mskip-2mu$ one can find a nonlooping {\em
deterministic\vtsp} action with target $v$ (since $G$ is fully
controllable and pure nondeterministic).  Backchaining such actions
produces a deterministic cycle $\calC$ on some subspace $W\!$ of $V$
(possibly all of $V$), containing at least two states.

Consider $G/W = (V^\prime, \frakA^\prime)$.  Here $\,V^\prime \,=\,
(\sdiff{V}{W})\,\union\,\{\wrep\}$, with $\wrep$ representing $W$.
\ $G/W$ is fully controllable (by Fact 3 on
page~\pageref{quotientfacts}) and pure nondeterministic, with $0 <
\abs{V^\prime} < \abs{V}$, so the induction hypothesis applies.  We
therefore obtain a hierarchical cyclic subgraph $H^\prime=(V^\prime,
\frakB^\prime)$ of $G/W$.

We may now replace leaf $(\{\wrep\}, \emptyset)$ in $H^\prime$ with
cycle $\calC$ on state space $W$.  When adjusting the encompassing
actions $\frakB^\prime$, we choose sources and targets so as to undo
any relabeling of states that occurred in forming $G/W$.  These
adjustments produce a hierarchical cyclic subgraph $H=(V, \frakB)$ of
$G$.
\end{proof}

\subsection{Core Cycle Actions, Leaf Covers, Disruptive Sets of Actions}
\markright{Core Cycle Actions, Leaf Covers, Disruptive Sets of Actions}

Suppose $H=(W,\frakB)$ is a hierarchical cyclic graph with $\abs{W} >
1$.  Each state $\mskip1.5mut$ of $\hspd{}W\!$ appears as a leaf
$(\{t\}, \emptyset)$ in the tree decomposition of $\mskip1.5muH$ and
has some parent node $N=(U, \frakE)$.  Some action $c_t\in\frakE$,
necessarily a cycle action of $N$, must be deterministic with {\em
target} $\mskip1.5mut$.  We refer to $c_t$ as {\em $t$'s core cycle
action.}  This action is determined uniquely by $t$ and the tree
decomposition of $H$.  ($H$ may contain multiple deterministic actions
with target $t$, but one and only one of those actions will be a cycle
action in the parent node of $(\{t\}, \emptyset)$.)

With that construct in mind, we now make a series of definitions and observations.

\begin{definition}[Core Cycle Actions]
Let $H=(W,\frakB)$ be a hierarchical cyclic graph.  The set
$\CH$ of \,\mydefem{core cycle actions of $H$} is

\begin{equation*}
\CH \;=\; \left\{\,
  \begin{aligned}
    \setdef{\,c_t}{t\in W}, & & \hbox{if $\,\abs{W} > 1$} & & \hbox{(with $c_t\mskip-1mu$  as defined above);}\\[1pt]
    \emptyset, & & \hbox{otherwise}. & & \\
  \end{aligned}
\right.
\end{equation*}
\end{definition}

\vspace*{0.05in}

\begin{definition}[Leaf Covers]
Let $N$ be a node in a  hierarchical cyclic graph $H$.  We say that $N$
\mydefem{covers only leaves in $H$} whenever each of $N$'s children is a
leaf in $H$'s tree decomposition.
\end{definition}

\vspace*{0.05in}

\begin{definition}[Disruptive Sets of Actions]
Let $H=(W,\frakB)$ be a hierarchical cyclic graph and suppose $\frakD
\subseteq \frakB$.  We say that $\frakD$ is \mydefem{disruptive (in $H$)}
whenever the following condition is satisfied:

\vspace*{-0.05in}

\hspace*{0.8in}\begin{minipage}{4.7in}
\raggedright
For every node that covers only leaves in $H$,\\
at least two of the node's cycle actions are missing from $\frakD$.
\end{minipage}
\end{definition}

\paragraph{Observations:}
\begin{itemize}
\addtolength{\itemsep}{-2pt}
\item When a hierarchical cyclic graph $H=(W,\frakB)$ contains at least
  two states, $\abs{\CH} = \abs{W}$.
\item A node covers only leaves in $H$ if and only if all the node's
  cycle actions lie in $\CH$.
\item The empty set of actions is always disruptive, even when $H$ is a leaf.
\end{itemize}

\subsection{Cycle-Breaking Strategies}
\markright{Cycle-Breaking Strategies}

Sets of actions that do not contain any node's full set of cycle
actions are convergent and may be arranged informatively, as the
following definition and lemmas make precise.

\begin{definition}[Cycle-Breaking]
\ Suppose $H=(W,\frakB)$ is a hierarchical cyclic graph.  A set of
actions $\tau\subseteq\frakB$ is \hsps\mydefem{cycle-breaking (in
$H$)} \hsph if, for each node $N$ in the tree decomposition of $H$,
$\hspc\tau\mskip-1mu$ does not contain all of $N$'s cycle actions.
\end{definition}

\vspace*{0.1in}

\begin{lemma}[Cycle-Breaking is Convergent]\label{nocyclestrats}
Suppose $\tau$ is a cycle-breaking set of actions in a hierarchical
cyclic graph $H$.  \ Then $\tau\in\DTH$.
\end{lemma}

\vspace*{-0.1in}

\begin{proof}
By structural induction on the tree decomposition of $H$.  The lemma
holds if $\mskip0.5muH\mskip-0.5mu$ is a leaf, since only
$\tau=\emptyset\,$ is possible.  \ Otherwise, suppose the children of
$\mskip1muH\mskip-1mu$ are $\,(W_1, \frakB_1), \ldots, (W_k,
\frakB_k)$.  Inductively, the lemma holds for the set of actions $\tau
\inter \frakB_i$ in the hierarchical cyclic graph $(W_i, \frakB_i)$,
for $i=1, \ldots, k$.  \ Let $\sigma$ consist of the top-level cycle
actions of $\mskip1.2muH\mskip-0.2mu$ that are in $\tau$.  \ Since
${H/\{W_1, \ldots, W_k\}}$ is a directed cycle graph (see top of
page~\pageref{cyclequotient}) and since $\tau$ is cycle-breaking,
$\sigma^\prime\in\Delta_{H/\{W_1, \ldots, W_k\}}$.  \ Thus, by
repeated application of Lemma~\ref{unquotienting} on
page~\pageref{unquotienting}, $\tau\in\DTH$.
\end{proof}

\vspace*{-0.1in}

\paragraph{Caution:} \hspace*{-0.05in}Not all strategies in a
hierarchical cyclic graph need be cycle-breaking (see
page~\pageref{Ex241_H1}).

\vspace*{0.1in}

\begin{lemma}[Cycle-Breaking is Informative]\label{nocycleiars}
Suppose $\tau$ is a nonempty cycle-breaking set of actions in a
hierarchical cyclic graph $H$.  \ Then some ordering of all the
actions in $\mskip1mu\tau$ is an informative action release sequence
for $H$.
\end{lemma}

\begin{proof}

The proof will associate to each leaf and node of $H$ an informative
action release sequence, with the sequence associated to the root of
$H$ comprising all of $\tau$.

For the purposes of this proof, it will be convenient to consider the
empty sequence of actions as an informative action release sequence.
Since $\tau$ is nonempty, the final sequence produced below will be
nonempty, satisfying the standard requirement of
page~\pageref{iarsdef} that informative attribute release sequences be
nonempty.

\vst

Base Case: \ Associate to each leaf of $H$ the empty sequence.

\vst

Inductive Step: \ Consider a node $N$ of $H$ and assume each child $C$
of $N$ has an associated informative action release sequence
consisting of all the actions of $\tau$ that appear in the graph $C$.
View $N$ as a hierarchical cyclic subgraph in its own right, and form
the quotient graph $N^\prime$ obtained by identifying each child to a
singleton state.  The cycle actions of $N$ create a deterministic
directed cycle in $N^\prime$.  This cycle forms a minimal nonface in
$\Delta_{N^\prime}$.  By Lemma~\ref{minnonfaceiars} on
page~\pageref{minnonfaceiars}, any sequential ordering of the directed
edges comprising this cycle forms an iars for $N^\prime$, any proper
subset of which is convergent.  Let $\{a_1, \ldots, a_\ell\}$ be the
set of $N$'s cycle actions in $\tau$, this being $\emptyset$ with
$\ell=0$ when none of $N$'s cycle actions lie in $\tau$.  Since $\tau$
is cycle-breaking, the reasoning just given implies $\{a^\prime_1,
\ldots, a^\prime_\ell\}\in\Delta_{N^\prime}$ and $a^\prime_1, \ldots,
a^\prime_\ell$ is an iars for $N^\prime$.

\vst

Since the children of $N$ are fully controllable subgraphs of $N$,
repeated application of Corollary~\ref{combinequotientgraphiars} on
page~\pageref{combinequotientgraphiars} shows that $a_1, \ldots,
a_\ell, \hsph{}c_1, \ldots, c_m$ is an iars for $N$, with $c_1,
\ldots, c_m$ being some concatenation of all the informative action
release sequences associated to $N$'s children.  Associate $a_1,
\ldots, a_\ell, \hsph{}c_1, \ldots, c_m$ to $N$.  Observe that this
iars consists of all actions of $\tau$ that appear in the graph $N$.
\ Associated to $H$ itself therefore is an iars consisting of all of
$\tau$.
\end{proof}

\clearpage
\subsection{Markings}
\markright{Markings}
\label{markings}

Let $H=(W,\frakB)$ be a hierarchical cyclic graph.  We will view each
node of $H$ as being either {\em marked\,} or {\em unmarked}.  \ Each
node is unmarked initially. \ Later, we will define an algorithm that
{\em marks\vmsp} nodes according to some criteria.  \ Once marked, a
node remains marked.

\vso

In order to consider marking a node $N$, we first require that each
child of $N$ be either a leaf or an already marked node.  The
collection of marked nodes at any instant therefore defines a set
$\maxmarked$ of {\em maximal marked nodes}, consisting of those nodes
that are marked but have no marked parent.
\ After some nodes have been marked, $\maxmarked = \{(W_1, \frakB_1),
\ldots, (W_\ell, \frakB_\ell)\}$, for some $\ell \geq 1$, with the
sets $W_1, \ldots, W_\ell$ nonempty and pairwise disjoint.  We may
therefore form the quotient graph $H/\{W_1, \ldots, W_\ell\}$, which
we abbreviate as $H/\maxmarked$.  We view $H/\maxmarked$ as a
hierarchical cyclic graph, much as on page~\pageref{ignoreloops}, by
discarding any actions that have become self-loops.

\vso

The process will be iterative, adding an additional node to the
collection of marked nodes with each step.  We abbreviate the notation
by writing $\maxmarked^{(j)}$ to mean the maximal marked nodes at the
$j^{\scriptstyle{th}}$ step, with $j \geq 1$, and by writing $H^{(j)}$
to mean $H/\maxmarked^{(j)}$.  We also define
$\maxmarked^{(0)}=\emptyset$ and $H^{(0)}=H$.

\vspace*{-0.1in}

\paragraph{Observation:}\ Any node covering only leaves in $H^{(j)}$
corresponds to a node in $H$ that is not yet marked but that could be
marked at the $(j+1)^{\scriptstyle{st}}$ step, and vice-versa.

\subsection{Forward Projections}
\markright{Forward Projections}

Strategies in a pure nondeterministic graph define partial orders.
One may view those partial orders as forward projections of possible
system states.

\begin{definition}[A Strategy's Partial Order]
Let $G=(V,\frakA)$ be a pure nondeterministic graph.  If
$\mskip1mu\sigma\in\DG$, then \vtsp\mydefem{$\sigma$ induces a partial
order $\,\geq_\sigma\vtsp$ on $V\!$} as follows:

\vso

For each $w,v \in V\!$, $\,w \geq_\sigma v\,$ if and only if either
$w=v$ or there exist actions $a_1, \ldots, a_k \in \sigma$, with $k
\geq 1$, such that:

\vspace*{-0.1in}

\hspace*{1.5in}\begin{minipage}{3.5in}
\begin{itemize}
\addtolength{\itemsep}{-3pt}
\item[(i)] $\src(a_1) = w$,
\item[(ii)] $\src(a_{i+1}) \in \trg(a_i)$, \ for $i=1,\ldots,k-1$, \ and
\item[(iii)] $v \in \trg(a_k)$.
\end{itemize}
\end{minipage}
\end{definition}
Thus, $w \geq_\sigma \mskip-0.75muv$ if and only if $w$ is $v$ or the
system might move from $w$ to $v$ while executing strategy $\sigma$
\,(in the diagram below, $a_i\in\sigma$, $v_i=\src(a_i)$, and
$v_{i+1}\in\trg(a_i)$, for $i=1,\ldots,k$):

\vspace*{-0.1in}

$$w=v_1\;\xrightarrow{\phantom{1}a_1\phantom{1}}\;
    v_2\;\xrightarrow{\phantom{1}a_2\phantom{1}}\;
    \cdots\;
    v_k\;\xrightarrow{\phantom{1}a_k\phantom{1}}\; v_{k+1}=v.$$
The partial order $\geq_\sigma$ is well-defined since $\sigma$ cannot
create any cycles.

\begin{definition}[Forward Projection]
Suppose $G=(V,\frakA)$ is a pure nondeterministic graph.
Let $\sigma\in\DG$ and $\,\emptyset\neq{W}\subseteq{V}$.
\quad The \,\mydefem{forward projection of $W\!$ under $\sigma$} is the set
$$\calF_\sigma(W) \;=\; \setdef{v \in V}{w \geq_\sigma v, \ \hbox{for
    some $w\in{W}$}}.$$
\end{definition}

In other words, $\calF_\sigma(W)$ consists of all states that the
system might pass through or stop at, assuming the system starts at
some state in $W\!$ and moves according to strategy $\sigma$.  (In
some papers, {\em forward projection} refers only to the states the
system might stop at.  Here, {\em forward projection\hspt} includes
all states through which the system might move, including starting
states.)

\begin{lemma}[Disjoint Forward Projections --- Core Cycle Actions]\label{disjointforwardcore}
$\phantom{0}$

\vst

Let $H=(W,\frakB)$ be a hierarchical cyclic graph and suppose
$\tau\in\DTH$.

\vst

Define $\vtsp\taup =\, \CH \inter \tau$ \ and \ $\taum =\,
\sdiff{\CH}{\tau}$.

\vspace*{0.05in}

For each $c\in\taum$, let $J_c = \calF_\taup(\{t\})$, with $t$ being
the unique target of action $c$. \ Then:

\vspace*{0.05in}

(a) The sets in the family
    $\mskip1mu\{J_c\}_{c\mskip1mu\in\mskip1mu\taum}\!$ are pairwise
    disjoint.

\vsr

(b) Suppose further that $\mskip2mu\tau$ is disruptive.  Let
$c\in\taum$.  Write $c = w \rightarrow t$.  Then \,$w\not\in J_c$.
\end{lemma}

In words: We split the core cycle actions $\CH$ of $H$ into two sets,
those that lie in the strategy $\tau$ and those that do not.  The
first set is itself a strategy, so we can consider forward projections
under that strategy.  For each core cycle action that is not in
$\tau$, we consider the forward projection of that action's target
state.  The lemma asserts that the resulting forward projections are
pairwise disjoint.  Moreover, if $\mskip2mu\tau$ is disruptive, then
each such forward projection does not loop back far enough to include
the source state of its generating core cycle action.

These properties will help us later to construct minimal nonfaces from
which we can then extract an informative action release sequence that
is sufficiently long to establish Theorem~\ref{longiars}.

\begin{proof}
(a) Let $\geq$ be the partial order induced by $\taup$ on $W$.
  Suppose $v\in J_c \inter J_d$, with $c,d \in \taum$.  Write $c = w
  \rightarrow t$ and $d = u \rightarrow s$.  Then $t \geq v$ and $s
  \geq v$.  Since $\taup \subseteq \CH$, backchaining from $v$
  produces a unique backwards path of action edges in $\taup$, with
  each edge actually being a deterministic action.  (The path could be
  degenerate, consisting of no edges, just the state $v$.)  That
  backwards path eventually encounters both $t$ and $s$, establishing
  that $t$ and $s$ are comparable.  For example, \
 $s \rightarrow \cdots \rightarrow t \rightarrow \cdots \rightarrow v$
 \ would establish $s \geq t$.  Since core cycle actions $w
  \rightarrow t$ and $u \rightarrow s$ are missing from $\taup$, this
  is only possible if $s=t$, meaning $c=d$.

\vspace*{0.1in}

(b) Suppose $w\in J_c$, with $c\in\taum$ and $c = w \rightarrow t$.
Arguing as in (a), we now obtain a cycle:

\vspace*{-0.1in}

 $$w \rightarrow t \rightarrow w_1 \rightarrow \cdots
 \rightarrow w_k = w, \quad \hbox{with $k \geq 1$}.$$

All but one of the actions comprising this cycle lie in $\taup$.

\vst

The exception is $w \rightarrow t$, which lies in $\taum$.

\vst

All the actions comprising the cycle lie in $\CH$. Consider any action
$u \rightarrow s$ of $\CH$.  The depth\footnote{Here, the depth of a
node or leaf in a tree is defined recursively as follows: \\
\hspace*{0.25in} The depth of the tree's root is 0.  \ The depth of a
child is one more than the depth of its parent.} of the leaf
$(\{u\}, \emptyset)$ in the tree decomposition of $H$ must be greater
than or equal to the depth of the leaf $(\{s\}, \emptyset)$.
Consequently, all the states in the cycle appear in $H$ as leaves at
the same depth and with the same parent node.  The cycle must
therefore consist of that parent node's cycle actions and the parent
node cannot contain any other children.  So, the parent node covers
only leaves.  Since $\tau$ is disruptive, at least two of the node's
cycle actions lie in $\taum$, not just one, establishing a
contradiction.
\end{proof}

\clearpage
\subsection{Quotienting until Disruption}
\markright{Quotienting until Disruption}
\label{quottodisrupt}

The proof path now is to iteratively mark and quotient by nodes that
prevent a strategy from being disruptive.  Concurrently, one assembles
several sets of actions that satisfy a property similar to the
disjointness of forward projections described in
Lemma~\ref{disjointforwardcore}.

\begin{construction}[Acyclic Dissection]\label{acyclicdissection}
Let $H=(W,\frakB)$ be a hierarchical cyclic graph.  Suppose
$\tau\subseteq\frakB$.  An \vtsp\mydefem{acyclic dissection
$(\tauc,\vtsp \taup,\vtsp \taum,\vtsp \xi)\vtsp$ of $\vtsp\tau\vtsp$
in $H$} is defined iteratively as follows:
\end{construction}

\begin{enumerate}

\addtolength{\itemsep}{1pt}

\item Initialize $\tau^{(0)}=\tau$ \,and $\,\kappa^{(0)} = \emptyset$.
      Assume all nodes in the tree decomposition of $H$ are unmarked
      and initialize $H^{(0)}=H$, as per Section~\ref{markings}.

      Set {\sc Done} to {\tt true} if $\mskip1mu\tau$ is disruptive in
      $H$ and to {\tt false} otherwise.

\item While not {\sc Done}, run the following loop, starting from $j=0$:
      
  \vspace*{-0.1in}

  \begin{enumerate}

  \addtolength{\itemsep}{2pt}

  \item At this stage, $\tau^{(j)}$ consists of actions in $H^{(j)}$
        and is not disruptive in $H^{(j)}$.  \ Let $N$ be some
        unmarked node in $H$ such that the corresponding quotient node
        $N^\prime$ in $H^{(j)}$ covers only leaves and at most one of
        the cycle actions in $N^\prime$ is absent from $\tau^{(j)}$.

  \item Suppose $N^\prime$ has $k$ cycle actions $\{c^\prime_1,
        \ldots, c^\prime_k\}$.  Discard one of these, so that the rest
        all lie in $\tau^{(j)}$.  Without loss of generality, assume
        one may discard $c^\prime_k$.  Now let

        \vspace*{-0.1in}

        $$\kappa^{(j+1)} \;=\; \kappa^{(j)} \union \{c_1, \ldots, c_{k-1}\}.$$

        Inductively: $\,\kappa^{(j+1)}$ consists of (unquotiented)
        actions in $H$.  \ In fact, $\kappa^{(j+1)}\subseteq\tau$.

  \item Mark node $N$, then let $H^{(j+1)}$ be the quotient graph
        formed from the resulting maximal marked nodes, as per
        page~\pageref{markings}, again viewed as a hierarchical cyclic
        graph.

  \item Suppose $H^{(j+1)} = (W^\prime, \,\frakD^\prime)$.  \,Let
        $\tau^{(j+1)} = \setdef{a'\in\frakD^\prime}{a\in\tau}$.  So
        $\tau^{(j+1)}$ is nearly the same as $\tau^\prime$, except
        that $\tau^{(j+1)}$ ignores any action of
        $\mskip1.75mu\tau\mskip-0.1mu$ whose source and targets all
        lie within any one maximal marked node of $H$.  (Prime
        notation indicates the correspondence between an action in
        $\mskip1muH$ and its relabeled form in a quotient graph.)

  \item If $\tau^{(j+1)}$ is disruptive in $H^{(j+1)}$, set {\sc Done}
        to {\tt true}.  The loop ends.  Otherwise, the loop continues,
        with $j+1$ in place of $j$.

  \end{enumerate}

\item If $\tau$ was already disruptive in $H$, let $H^*=H$,
      $\tau^*=\tau$, and $\tauc=\emptyset$.  Otherwise, let
      $H^*=H^{(j+1)}$, $\,\tau^*=\tau^{(j+1)}$, and
      $\,\tauc=\kappa^{(j+1)}$, with $j+1$ as above when the loop
      ends.  In either case, $\tau^*$ is disruptive in $H^*$.
      \ Finally, let $\Cp = \setdef{c\in\frakB}{c'\in\CHp}$.  In
      other words, $\Cp$ is the set of actions in $H$ that become
      core cycle actions in the quotient graph $H^*$.

\item Define $\xi$ as follows ($H$ contains a marked node if and only if
      the loop of step 2 was run):

      Start with $\xi=\emptyset$.  Then, for each {\em unmarked\,}
      node $N$ in $H$, let $\CyN$ be $N$'s cycle actions.  If
      $\CyN\inter\tau$ is a proper subset of $\CyN$, add all of
      $\CyN\inter\tau$ to $\xi$.  Otherwise, select an action $c$ in
      $\sdiff{\CyN}{\Cp}$.  Add the actions $\sdiff{\CyN}{\{c\}}$ to
      $\xi$.  \ (Why does $c$ exist?  If not, let $N^\prime$ be the
      node in $H^*$ corresponding to $N$.  It is well-defined since
      $N$ is unmarked.  Then $N^\prime$ would cover only leaves in
      $H^*$ and thus $\tau^*$ would not be disruptive in $H^*$, a
      contradiction.)  \

\item Step 3 defined $\tauc$.  Now define
            $\,\taup=\,\Cp\inter\xi$, \,and
            $\,\taum=\,\sdiff{\Cp}{\xi}$.

\end{enumerate}

\begin{lemma}\label{acyclicprop}
Construction~\ref{acyclicdissection} produces an acyclic dissection
$(\tauc, \taup, \taum, \vtsp\xi)$ of $\,\tau$ such that:

\vspace*{0.15in}

\hspace*{0.75in}\begin{minipage}{5in}
\begin{itemize}
\addtolength{\itemsep}{-3pt}
\item[(i)] $\taup\subseteq\xi\vtsp$ and $\vtsp\taum\subseteq\frakB$,
\item[(ii)] $\tauc\union\xi\subseteq\tau\vtsp$ and $\vtsp\tauc\inter\xi=\emptyset$,
\item[(iii)] $\tauc \union \xi\hspc$ is cycle-breaking in $H$, and
\item[(iv)] $\taum\inter\tau=\emptyset$.
\end{itemize}
\end{minipage}
\end{lemma}

\vso

\begin{proof}
The loop in step 2 of Construction~\ref{acyclicdissection} runs at most
a finite number of times, since the graph $H$ is finite.  As a result,
an acyclic dissection $(\tauc, \taup, \taum, \xi)$ of $\tau$ is
well-defined by step 5.

\vst

Assertions (i), (ii), and (iii) are clear from the construction.

\vst

To establish assertion (iv), suppose $a\in\taum\inter\tau$.  Let prime
notation denote quotienting from $H$ to $\mskip1muH^*$, with
$\mskip1muH^*$ as defined in step 3 of the construction, and assume
the rest of the notation from the construction.

Then $a\in\Cp$, $a\in\tau$, and $a\not\in\xi$.

So $a^\prime\in\CHp$, implying $a\in\CyN$, with $\CyN$ the cycle
actions of some unmarked node $N$ in $H$.

\vst

If $\CyN\inter\tau$ is proper subset of $\CyN$, then $a\in\xi$, by
step 4 of the construction, producing a contradiction.

So $\CyN\inter\tau=\CyN$.  Let $c$ be the action removed in step 4 of
the construction.  So $c\not\in\Cp$.  Since $a\not\in\xi$, $a\in\CyN$,
and $\sdiff{\CyN\!}{\!\{c\}}\subseteq\xi$, it must be that $a=c$, but
that contradicts $a\in\Cp$.
\end{proof}

And here is a generalization of Lemma~\ref{disjointforwardcore}:

\begin{lemma}[Disjoint Forward Projections]\label{disjointforward}
\ Suppose $H=(W,\frakB)$ is a hierarchical cyclic graph and
$\tau\subseteq\frakB$.  \ Construct $\vtsp(\tauc, \taup, \taum, \xi)$
from $\mskip3mu\tau\mskip-0.5mu$ as per Construction~\ref{acyclicdissection}.  \
Let $\eta = \tauc \union \taup$.

\vspace*{0.05in}

Given $\vmsp{}c\in\taum$, \vmsp{}write $\vmsp{}c = w \rightarrow T$ and define

\vspace*{-0.1in}

$$J_c = \calF_\eta(T).$$

\hspace*{0.74in}(The definition is sensible since $\eta$ is
cycle-breaking in $H\!$ and so $\eta\in\DTH$.)

\vspace*{0.1in}

Then:

\hspace*{0.6in}\begin{minipage}{5in}
\begin{itemize}
\addtolength{\itemsep}{-3pt}
\item[(a)] The sets in the family
  $\mskip1mu\{J_c\}_{c\mskip1mu\in\mskip1mu\taum}\!$ are pairwise
  disjoint.
\item[(b)] For each $c\in\taum$, \,$\src(c)\not\in J_c$.
\end{itemize}
\end{minipage}
\end{lemma}

\begin{proof}
Let prime notation denote quotienting from $H$ to $H^*$, where
sensible, with $H^*$ as in step 3 of the construction.  Write
$H^*=(W^\prime, \frakD^\prime)$, viewed with a tree decomposition
derived from that of $H$.

\vsr

Since $\xi$ arises only from cycle actions of unmarked nodes, each
action in $\xi^\prime$ is a well-defined convergent action in
$\frakD^\prime$.  By construction, $\xi^\prime$ is cycle-breaking in
$H^*$, so $\xi^\prime\in\DHp$.  Since $\xi^\prime\subseteq\tau^*$,
$\xi^\prime$ is disruptive in $H^*$.  Consequently,
Lemma~\ref{disjointforwardcore} applies to the graph $H^*$ and the
disruptive strategy $\xi^\prime$.

\vsr

In $H$, we have $\taup={\Cp}\inter{\xi}$ and $\taum=\sdiff{\Cp}{\xi}$.
Therefore, each action $c\in\taup$ corresponds to an action $c^\prime
\in \taupp = {\CHp}\inter{\xi^\prime}$ in $H^*$, and each action
$c\in\taum$ corresponds to an action $c^\prime \in \taump =
\sdiff{\CHp}{\xi^\prime}$.  \ (Recall the comment about ``one-to-one
correspondence'' on page~\pageref{quotient}.)

\vsr

Let $\geq$ be the partial order induced by $\eta$ on $W$ and let
$\geq^*$ be the partial order induced by $\taupp$ on $W^\prime$.
Suppose $v \geq w$, with $v,w \in W$.  Then $v^\prime \geq^*
w^\prime$, with $v^\prime, w^\prime \in W^\prime$ being the state
relabelings of $v$ and $w$, respectively.  (Why?  If there is a path
of action edges from $v$ to $w$ with the actions drawn from $\eta$,
then there is a path of action edges from $v^\prime$ to $w^\prime$
with the actions drawn from $\eta^\prime$.  Some of the action edges
between states in $W\mskip-2.3mu$ may become self-loops when sources
and targets are relabeled as states in $W^\prime$.  Indeed,
$v^\prime=w^\prime$ is possible even if $v\neq w$.  Any such
self-loops could only come from actions in $\taucp$.  Conversely, all
actions in $\taucp$ are self-loops.  One discards those actions in
forming $H^*$, leaving only $\taupp$ from $\eta^\prime$.  Thus there
is a path of action edges from $v^\prime$ to $w^\prime$ with the
actions drawn from $\taupp$.)

\vsr

It follows that $v\in J_c$ implies $v^\prime\in J_{c^\prime}$, with
$J_{c^\prime}$ defined for $H^*$ and $\xi^\prime$ as in
Lemma~\ref{disjointforwardcore}, now using $c^\prime$ in place of $c$,
$\vtsp\taupp$ in place of $\taup$, and $\taump$ in place of $\taum$.
\ (To see this, write $c = w \rightarrow T$.  The set of targets $T$
becomes a single state $t^\prime\in W^\prime$, since
$c^\prime\in\CHp$.  Write $c^\prime = w^\prime \rightarrow t^\prime$.
If $v\in J_c$, then $t \geq v$ for some $t\in T$, so $t^\prime \geq^*
v^\prime$, and thus $v^\prime\in J_{c^\prime}$.)

\vsr

Consequently, Lemma~\ref{disjointforwardcore} establishes the claims of
the current lemma.
\end{proof}

\vspace*{0.1in}
\subsection{Alternate Development: Quotienting until Disruption}
\markright{Alternate Development: Quotienting until Disruption}

This subsection restates Construction~\ref{acyclicdissection}
recursively without mentioning markings, then provides induction
proofs of the corresponding lemmas.  The key steps are the same as
before.  The rest of Section~\ref{nondet} will prove
Theorem~\ref{longiars} for pure nondeterministic graphs using the
earlier iterative construction, side-stepping any issue of strategy
maximality in quotient graphs.  \ Section~\ref{stochastic} will engage
that issue when proving Theorem~\ref{longiars} for pure stochastic
graphs.

\vsr

\begin{construction}[Alternate Construction: Acyclic Dissection]\label{altacyclicdissection}
\ Let $H=(W,\frakB)$ be a hierarchical cyclic graph. \ Suppose
$\tau\subseteq\frakB$.  An \vtsp\mydefem{acyclic dissection
$(\tauc,\vtsp \taup,\vtsp \taum,\vtsp \xi)\vtsp$ of $\vtsp\tau\vtsp$
in $H$} is defined recursively as follows:
\end{construction}

\vspace*{-0.1in}

\begin{itemize}

\item[I.]  Suppose $\tau$ is disruptive in $H$:

\vspace*{-0.1in}

\begin{enumerate}

  \item Define $\xi$ as follows, starting from $\xi=\emptyset$:

        For each node $N$ in $H$, let $\CyN$ be $N$'s cycle actions.
        If $\CyN\inter\tau$ is a proper subset of $\CyN$, add all of
        $\CyN\inter\tau$ to $\xi$.  Otherwise, there is at least one
        action $c$ in $\sdiff{\CyN}{\CH}$.  (If not, then $N$ would
        cover only leaves in $H$ and thus $\tau$ would not be
        disruptive.)  \ Pick one such action $c$ and add the remaining
        actions $\sdiff{\CyN}{\{c\}}$ to $\xi$.

  \item Let $\,\tauc=\emptyset$,
            $\,\taup=\,\CH\inter\xi$, \,and
            $\,\taum=\,\sdiff{\CH}{\xi}$.

\end{enumerate}

\item[II.] Suppose $\tau$ is not disruptive in $H$:

\vspace*{-0.1in}

\begin{enumerate}

  \item Let $N=(U, \frakE)$ be a node in $H$ that covers only leaves
        and at most one of whose cycle actions is absent from $\tau$.
        The actions $\frakE$ are necessarily $N$'s cycle actions.
        Discard one of those actions, so the rest all lie in $\tau$.
        Denote that resulting set by $\calC$.

  \item Let $H^*=(W^\prime,\frakD^\prime)$ be the hierarchical cyclic
        graph formed from the quotient graph $H/U$ by discarding
        self-loops.  Let $\,(\tauzp,\, \taupp,\, \taump,\,
        \xi^\prime)\,$ be a recursively constructed acyclic dissection
        of $\tau^\prime\inter\frakD^\prime$ in $H^*$.  (As usual,
        prime notation describes the correspondence between actions of
        $\mskip2.5muH$ and actions of $\mskip2.5muH/U$.)

  \item Now define the sets of actions $\,\taup$, $\,\taum$, and
        $\,\xi$ by unquotienting, that is, by direct correspondence
        from the sets of actions $\,\taupp$, $\,\taump$, and
        $\,\xi^\prime$, respectively.
        \quad Finally, let $\tauc = \tauz \union \calC$, with $\tauz$
        formed from $\tauzp$ by unquotienting.

\end{enumerate}

\end{itemize}

\paragraph{Proof of Lemma~\ref{acyclicprop}, assuming alternate acyclic
  dissection given by Construction~\ref{altacyclicdissection}:}

\begin{proof}
The construction terminates because $H$ is finite and each recursive
invocation of the construction replaces a node with a leaf.

The proof of the specific assertions is by induction, with Case I of
the construction defining the base case and Case II defining the
inductive step:

\vspace*{-0.05in}

\begin{itemize}

\item[I:] In the base case, assertions (i), (ii), and (iii) are
          immediate from the construction.  Assertion (iv) follows as
          it did in the earlier proof of Lemma~\ref{acyclicprop}, but
          now working directly with $H$ rather than needing to form a
          quotient.

\item[II:] Inductively, we assume assertions (i)--(iv) hold for an
           acyclic dissection $\,(\tauzp,\, \taupp,\, \taump,\,
           \xi^\prime)\,$ of
           $\mskip1mu\tau^\prime\inter\frakD^\prime\mskip1mu$ in
           $H^*$, using the notation from the construction.  Then:

  \vspace*{-0.05in}

   \begin{itemize}
   \addtolength{\itemsep}{2pt}

   \item[(i)] $\taupp\subseteq\xi^\prime$, so $\taup\subseteq\xi$
              \ and \ 
              $\taump\subseteq\frakD^\prime$, so
              $\taum\subseteq\frakD\subseteq\frakB$.

   \item[(ii)] $\tauzp\union\xi^\prime\subseteq\tau^\prime\inter\frakD^\prime$,
              so
              $\tauc\union\xi=\tauz\union\calC\union\xi\subseteq\tau$,
              since $\calC\subseteq\tau$.

              $\tauzp\inter\xi^\prime=\emptyset$, so
              $\tauz\inter\xi=\emptyset$.  \ Since all actions in
              $\frakE$ become self-loops and are discarded when
              forming $H^*$ from $H$, $\,\calC\inter\xi=\emptyset$,
              and so $\tauc\inter\xi=\emptyset$.

   \item[(iii)] $\tauzp \union \xi^\prime\hspc$ is cycle-breaking in $H^*$,
              so $\tauc\union\xi=\tauz\union\calC\union\xi$ is cycle-breaking
              in $H$, since $\calC$ consists of a proper subset of one
              node's cycle actions and since $\tauz\union\xi$ does not
              include any of that node's actions.

   \item[(iv)] $\taump\inter\tau^\prime\inter\frakD^\prime=\emptyset$,
                so $\taum\inter\tau\inter\frakD=\emptyset$.  Since
                $\taum\subseteq\frakD$, it follows that
                $\taum\inter\tau=\emptyset$.
   \end{itemize}

\end{itemize}

\vspace*{-0.35in}

\end{proof}

\paragraph{Proof of Lemma~\ref{disjointforward}, assuming alternate acyclic
  dissection given by Construction~\ref{altacyclicdissection}:}

\begin{proof}
Again by induction:

\vspace*{-0.1in}

\begin{itemize}

\item[I:] The base case follows from Lemma~\ref{disjointforwardcore},
          with $\xi$ in place of $\tau$, since $\xi$ is a disruptive
          strategy by Lemmas~\ref{acyclicprop} and \ref{nocyclestrats}.

\item[II:] Inductively, the argument is much the same as in the
           earlier proof of this lemma.  One assumes the assertions
           hold for the hierarchical cyclic quotient graph $H^*$.  In
           moving back to $H$, one state of $H^*$ turns back into a
           cycle of states, with all but one of the cycle actions
           added to $\tauz$ to form $\tauc$.  The other sets of
           actions in the dissection do not change as one moves back
           from $H^*$ to $H$, except for relabelings of sources and
           targets. Consequently, execution paths of
           $\tauc\union\taup$ (within $H$) imply execution paths of
           $\tauzp\union\taupp$ (within $H^*$), thereby establishing
           the lemma's assertions for $H$.
\end{itemize}

\vspace*{-0.34in}

\end{proof}

\clearpage
\subsection{Acyclic Dissection Sizes}
\markright{Acyclic Dissection Sizes}

This subsection measures the size of the set
$\tauc\union\taup\union\taum$ in an acyclic dissection.\\
When reading the lemma below, recall that
Construction~\ref{acyclicdissection} marks nodes in $H$.

\begin{lemma}[Subgraph Sizes]\label{subcyclesize}
\ Suppose $H=(W,\frakB)$ is a hierarchical cyclic graph and
$\tau\subseteq\frakB$.  \ Let \hspq$H^*=(W^\prime,\frakD^\prime)$ and
$\mskip2mu\tauc$ be derived from $\mskip2mu\tau$ as per step 3 in
Construction~\ref{acyclicdissection} on
page~\pageref{acyclicdissection}.

\vso

For each $u\in W^\prime$, define $\,(W_u, \frakB_u)$ as follows: \hspd
If $\mskip2muu\in W$, let $\,(W_u, \frakB_u)$ be the leaf $\,(\{u\},
\emptyset)$ of $H$.  If $\vtsp{}u\not\in W$, let $\vtsp(W_u,
\frakB_u)$ be the maximal marked node of $\vtsp{}H$ for which $u$
represents $W_u$.

\vsr

Then \hspd$\abs{\tauc \inter \frakB_u} = \abs{W_u} - 1$,\, for each
$u\in W^\prime$.
\end{lemma}

\begin{proof}
The proof is by induction on the iteration count $j\hspc$ in the loop
of Construction~\ref{acyclicdissection}, now using $H^{(j)}$ in place
of $H^*$, $\kappa^{(j)}$ in place of $\tauc$, and with the collection
of marked nodes dependent on $j$.  The base case, $j=0$, corresponds
to all $u$ being in $W$, for which the lemma's assertion is clear.
Inductively, suppose the lemma's assertion is true for $H^{(j)}$.

In forming $H^{(j+1)}$ from $H^{(j)}$, one marks an unmarked node
$N=(V, \frakA)$ of $H$ whose corresponding node $N^\prime$ in
$H^{(j)}$ covers only leaves.  So $N$ is now a maximal marked node.
\ Let $W^{(j)}$ be the states of $H^{(j)}$, $W^{(j+1)}$ the states of
$H^{(j+1)}$, and $V^\prime$ the states of $N^\prime$.  Then
$W^{(j+1)}=\big(\sdiff{W^{(j)}}{V^\prime}\big)\;\union\;\{\wrep\}$,
with $\wrep$ representing the states $V^\prime$ identified to a
singleton.

\vst

Let $\{c_1, \ldots, c_k\}$ and $\{(V_1, \frakA_1), \ldots, (V_k,
\frakA_k)\}$ be the cycle actions and children of $N$ in $H$, respectively.
\ Without loss of generality, $\,\kappa^{(j+1)} = \kappa^{(j)} \union
\{c_1, \ldots, c_{k-1}\}$.  One has $k=\abs{V^\prime}>1$.

\vspace*{0.1in}

Case I: \ Suppose $u\in W^{(j+1)} \inter W^{(j)}$.  \ Then the
definition of $(W_u, \frakB_u)$ is the same via $H^{(j+1)}$ as via
$H^{(j)}$.  So $\abs{\kappa^{(j)} \inter \frakB_u} = \abs{W_u} - 1$.
Since $(W_u, \frakB_u)$ is either a leaf or a marked node of $H$ at
the $j^{\scriptstyle{th}}$ iteration of the loop in
Construction~\ref{acyclicdissection}, $\frakB_u$ contains none of node
$N$'s cycle actions.  Thus $\kappa^{(j+1)} \inter \frakB_u =
{\kappa^{(j)} \inter \frakB_u}$ and so $\abs{\kappa^{(j+1)} \inter
\frakB_u} = \abs{W_u} - 1$, inductively.

\vspace*{0.1in}

Case II: \ Suppose $u=\wrep$.  \ Then the definition of $(W_\wrep,
\frakB_\wrep)$ via $H^{(j+1)}$ is $N$, so $W_\wrep = V$ and
$\frakB_\wrep = \frakA$.  The states $V^\prime$ of $N^\prime$ in
$H^{(j)}$ are in one-to-one correspondence with the children $\{(V_i,
\frakA_i)\}$ of $N$.  Inductively, $\abs{\kappa^{(j)} \inter \frakA_i}
= \abs{V_i} - 1$.  Again, $\kappa^{(j+1)} \inter \frakA_i =
\kappa^{(j)} \inter \frakA_i$.  \ Moreover,

\vspace*{-0.05in}

$$W_\wrep  \;=\; V_1 \union \cdots \union V_k
  \qquad\hbox{and}\qquad
  B_\wrep  \;=\;  \{c_1, \ldots, c_k\} \union \frakA_1 \union \cdots \union \frakA_k.$$

\vspace*{0.05in}

\noindent By reasoning about markings, one further knows that
$\kappa^{(j)}\inter\{c_1, \ldots, c_k\}=\emptyset$.  \ Therefore

\vspace*{-0.1in}

$$\kappa^{(j+1)} \inter B_\wrep \;\,=\,\;
   \{c_1, \ldots, c_{k-1}\}
    \;\,\union\,\; 
    \bigunion_{i=1}^k\big(\kappa^{(j)} \inter \frakA_i\big)$$

\noindent and

\vspace*{-0.3in}

\begin{eqnarray*}
\abs{\kappa^{(j+1)} \inter B_\wrep}
  & = &
  (k-1) \;+\; \sum_{i=1}^k \abs{\kappa^{(j)} \inter \frakA_i} \\[2pt]
  & = &
  (k-1) \;+\; \sum_{i=1}^k \big(\abs{V_i} - 1\big) \\[2pt]
  & = &
  - 1 \;+\; \sum_{i=1}^k \abs{V_i} \\[5pt]
  & = &
  \abs{W_\wrep} - 1.
\end{eqnarray*}

\vspace*{-0.25in}

\end{proof}

\begin{corollary}[Dissection Sizes]\label{dissize}
\ Suppose $H=(W,\frakB)$ is a hierarchical cyclic graph and
$\tau\subseteq\frakB$.  \ Let $\tauc$, $\taup$, $\taum$, and $H^*$ be
derived from $\tau$ as per Construction~\ref{acyclicdissection} on
page~\pageref{acyclicdissection}.

\vso

Let $n=\abs{W}$ and $m=\abs{\tauc\union\taup\union\taum}$.

\vso

If \hspd$H^*\mskip-1.5mu$ is a leaf, then \hspd$m=n-1$.  \ 
If \hspd$H^*\mskip-1.5mu$ is a node, then \hspd$m=n$.
\end{corollary}

\begin{proof}
Suppose $H^*$ is a leaf.  Then $H^*=(\{u\}, \emptyset)$, for some $u$,
and $\taup=\taum=\emptyset$.  Using the notation of
Lemma~\ref{subcyclesize}, $(W_u, \frakB_u)$ must be all of $H$.
\ The lemma then implies that $m=\abs{\tauc} = \abs{\tauc\inter\frakB}
= \abs{\tauc\inter\frakB_u} = \abs{W_u}-1 = \abs{W}-1=n-1$, as
claimed.

\vspace*{0.1in}

Suppose $H^*$ is a node.  Let $W^\prime$ be the states of $H^*$.  For each
$u\in{W^\prime}$, let $(W_u, \frakB_u)$ be defined as in
Lemma~\ref{subcyclesize}.  Since $\tauc$ is formed from cycle actions
in marked nodes of $H$, $\,\tauc =
\bigunion_{u\in{W^\prime}}(\tauc\inter\frakB_u)$.  We also know that
$W=\bigunion_{u\in{W^\prime}}W_u$. \ Thus, by the lemma,

\vspace*{-0.2in}

\begin{eqnarray*}
\abs{\tauc} & = & \sum_{u\in{W^\prime}}\abs{\tauc\inter\frakB_u} \\[2pt]
            & = & \sum_{u\in{W^\prime}}(\abs{W_u} - 1) \\[2pt]
            & = & \abs{W} - \abs{W^\prime} \\[2pt]
            & = & n - \abs{W^\prime}.
\end{eqnarray*}

Let $\frakC$ and $\CHp$ be as in Construction~\ref{acyclicdissection}.
Then $\taup\union\taum = \frakC$, so
$\abs{\taup\union\taum}=\abs{\frakC}=\abs{\CHp} = \abs{W^\prime}$.
Consequently, $m=\abs{\tauc}+\abs{\taup\union\taum}=n$, as claimed.
\end{proof}

\vspace*{0.1in}
\subsection{Informative Action Release Sequences for Maximal Strategies}
\markright{Informative Action Release Sequences for Maximal Strategies}

This subsection assembles the previous results to prove
Theorem~\ref{longiars} for pure nondeterministic graphs.

\vst

\begin{lemma}[Minimal Nonfaces Overlapping Forward Projections]\label{overlap}
\ Let $G=(V,\frakA)$ be a fully controllable pure nondeterministic graph
with $V\mskip-4mu\neq\emptyset$ and suppose $H=(V, \frakB)$ is a
hierarchical cyclic subgraph of $G$. \ (Recall that $H\!$ exists, by
Lemma~\ref{hierarch} on page~\pageref{hierarch}.)

\vst

Let $\sigma$ be a maximal strategy in $\DG$, \hsph define
$\tau=\sigma\inter\frakB$, \hsph let $(\tauc, \taup, \taum, \xi)$ be
an acyclic dissection of $\mskip3mu\tau\mskip-0.25mu$ obtained from
Construction~\ref{acyclicdissection}, \hsph
and set $\,\eta = \tauc \union \taup$.

\vst

Suppose $c\in\taum$.  \ Write $c = w \rightarrow T$ and define

\vspace*{-0.2in}

\begin{eqnarray*}
\calK_c  &=& \setdef{\kappa \subseteq \sigma \union
    \{c\}}{\hbox{$\kappa$ is a minimal nonface of $\vmsp\DG$}},\\[5pt]
J_c      &=& \calF_\eta(T),\\[5pt]
\sigma_c &=& \setdef{a\in\sigma}{\src(a)\in{}J_c\;\hbox{and}\;\trg(a)\not\subseteq{}J_c}.
\end{eqnarray*}

Then $\,\calK_c\neq\emptyset\,$ and $\,\kappa\inter\sigma_c\neq\emptyset\,$
for every $\kappa\in\calK_c$.
\end{lemma}

\paragraph{Comment:}\  The lemma tells us that every action
$c\in\taum$ is part of a minimal nonface of $\DG$ whose remaining
actions lie in $\sigma$, and that at least one of those actions has
its source, but not all its targets, in the forward projection of
$c$'s targets.  Here the forward projection is based on those actions
of $\mskip2mu\sigma$ that lie in the acyclic dissection sets
$\mskip1mu\tauc$ and $\taup$.  Intuitively, it is useful to think of
$\mskip0.5muH\mskip-1.5mu$ as a single node defining a directed cycle,
with the projection of $\mskip1.5mu\sigma\mskip0.5mu$ onto that cycle
being disruptive.  Disruption means that the cycle splits into at
least two pairwise disjoint directed arcs, as follows: The cycle edges
present in $\sigma$ constitute $\taup$, the cycle edges missing from
$\sigma$ constitute $\taum$, and $\tauc$ is empty in this simple
scenario.  There is one directed arc for each action $c\in\taum$,
starting at $c$'s target.  Each arc is formed from contiguous action
edges of $\taup$.  Each arc has a forward projection flow defined on
it by the directionality of those action edges.  A directed arc ends
when it encounters the source of another action in $\taum$.  An arc may
be degenerate, consisting of a single state.  The lemma says that, for
each missing cycle edge $c$, there is some action $a\in\sigma$ whose
source lies in an arc that starts at $c$'s target, such that at least
one of $a$'s targets lies outside this arc and such that $a$ and $c$
appear together in a minimal nonface of $\DG$.  \ As we will see
shortly, the ``$a$ or $c$?'' choice is therefore informative.

\begin{proof}
By Lemma~\ref{acyclicprop}(iv) on page~\pageref{acyclicprop},
$c\not\in\sigma$.  So, since $\sigma$ is maximal in $\DG$,
$\,\calK_c\neq\emptyset$.

Let $\kappa\in\calK_c$ be given.
Define $\gamma=\setdef{a\in\kappa}{\src(a)\in{}J_c}$.
Since $\kappa$ is a minimal nonface of $\DG$, no action of $\kappa$
moves off $\src(\kappa)$, by Lemma~\ref{minnonfacestrat} on
page~\pageref{minnonfacestrat}.  Since $c\in\kappa$ and $\{c\}\in\DG$,
$\,\gamma\neq\emptyset$.  \ By Lemma~\ref{disjointforward}(b) on
page~\pageref{disjointforward}, $\src(c)\not\in{}J_c$, so
$c\in\sdiff{\kappa}{\gamma}$.  \ Thus
$\emptyset\neq\gamma\subsetneq\kappa$ and $\gamma\subseteq\sigma$.

Now suppose the lemma's second assertion is false for this $\kappa$.
Then every action in $\gamma$ has all its targets in $J_c$.  Pick some
$a\in\gamma$.  Since $a\in\kappa$, there exists $b\in\kappa$ such that
$\src(b)\in\trg(a)\subseteq{}J_c$.  We see therefore that $b\in\gamma$
and that no action of $\gamma$ moves off $\src(\gamma)$.
Consequently, $\gamma\not\in\DG$, which contradicts $\kappa$ being a
minimal nonface of $\DG$.
\end{proof}

Imagine revealing actions of some secret maximal strategy
$\mskip0.25mu\sigma\mskip-3mu\in\mskip-3mu\DG\mskip0.5mu$ to an
observer who knows $G\mskip0.85mu$ but initially merely that $\sigma$
is maximal in $\DG$.  Suppose $c\mskip1mu$ is an action in $\taum$, as
previously defined.  So $c\not\in\sigma$ and
$\sigma\union\{c\}\not\in\DG$.  Let $\sigma_c$ be as before.  The next
corollary says that so long as one has not explicitly revealed any
actions of $\sigma_c$, the observer cannot exclude the possibility
that one is revealing actions of some maximal strategy other than
$\sigma$, some strategy that does include action $c$.  Moreover, there
exists some unrevealed and unimplied action in $\sigma_c$ that one may
yet release informatively.  (The explicitly revealed actions may imply
some actions in $\sigma_c$, but so long as none of the explicitly
revealed actions themselves lie in $\sigma_c$, these assertions hold.)

\vspace*{0.1in}

\begin{corollary}[Informative Actions in Forward Projections]\label{iarsfwd}
Let the hypotheses and notation be as in Lemma~\ref{overlap}.
\ In particular, $\sigma$ is maximal in $\DG$ and $c\in\taum$.

\vso

Suppose $\gamma\subseteq\sigma$ such that $\gamma\inter\sigma_c=\emptyset$.

\vso

Let $A$ be $G$'s action relation and define $\gammacls = (\clsAy)(\gamma)$.
\ Then:

\vspace*{0.1in}

\hspace*{0.6in}\begin{minipage}{5in}
\begin{itemize}
\addtolength{\itemsep}{-3pt}
\item[(i)] $\gammacls\union\{c\}\in\DG$.
\item[(ii)] $\calK_c\neq\emptyset\,$ and
  $\,\sdiff{\big(\kappa\inter\sigma_c\big)}{\gammacls}\;\neq\;\emptyset\,$
  for every $\kappa\in\calK_c$.
\end{itemize}
\end{minipage}
\end{corollary}

\begin{proof}
We may prove (i) by establishing that
$\psi_A\big(\gammacls\union\{c\}\big)\neq\emptyset$.  By reasoning
similar to that on page 122 in \cite{paths:privacy},
$\psi_A\big(\gammacls\union\{c\}\big)=\psi_A\big(\gamma\union\{c\}\big)$,
so it is enough to show that $\gamma\union\{c\}\in\DG$.  Suppose this
is false.  Then there exists a minimal nonface $\kappa$ of $\DG$ such
that $\kappa\subseteq\gamma\union\{c\}\subseteq\sigma\union\{c\}$, so
$\kappa\in\calK_c$.  By Lemma~\ref{overlap},
$\kappa\inter\sigma_c\neq\emptyset$.  That establishes a contradiction
to $\gamma\inter\sigma_c=\emptyset$ and $c\not\in\sigma_c$.

\clearpage

Turning to (ii), $\calK_c\neq\emptyset$ by Lemma~\ref{overlap}.
\ Suppose now that $\gamma=\sdiff{\sigma}{\sigma_c}$.  Establishing
the second part of (ii) for this particular $\gamma$ will establish it
for all hypothesized $\gamma$, by monotonicity of closure operators.
By (i), $\gammacls\union\{c\}\in\DG$.  Since $\sigma$ is maximal in
$\DG$, $\,\gammacls\subseteq\sigma$.  \ Since $c\not\in\sigma$ and by
Lemma~\ref{minnonfacesmaxsimplex} on page~\pageref{minnonfacesmaxsimplex},
$\,\sdiff{\kappa}{(\gammacls\union\{c\})}\neq\emptyset\mskip1.5mu$ for
every $\kappa\in\calK_c$.
\ Since $\sdiff{\sigma}{\sigma_c} \,=\, \gamma \,\subseteq\, \gammacls
\,\subseteq\, \sigma\,$ and $\,\kappa \,\subseteq\,
\sigma\union\{c\}$, $\;\sdiff{\kappa}{(\gammacls\union\{c\})} =
\sdiff{\big(\kappa\inter\sigma_c\big)}{\gammacls}$, \,completing the
proof.
\end{proof}

\noindent The following theorem has as corollary
Theorem~\ref{longiars} of page~\pageref{longiars} for pure
nondeterministic graphs:

\begin{theorem}[Informative Action Release Sequences :
    Pure Nondeterministic Graphs]$\phantom{0}$\label{longiarsnondet}

Let $G=(V,\frakA)$ be a fully controllable pure nondeterministic graph
with $n = \abs{V} > 1$ and suppose $H=(V, \frakB)$ is a hierarchical
cyclic subgraph of $\vmsp{}G$.

Suppose $\sigma$ is a maximal strategy in $\DG$.  Set
$\tau=\sigma\inter\frakB$, then define $H^*$ by step 3 of
Construction~\ref{acyclicdissection} on
page~\pageref{acyclicdissection}.

\vspace*{0.1in}
\hspace*{0.25in}\begin{minipage}{5in}
\begin{itemize}
\addtolength{\itemsep}{-3pt}
\item[I.] If $\vmsp{}H^*\mskip-0.5mu$ is a leaf, then $\vlsp\sigma$
contains an informative action release sequence for $\vlsp{}G$ of
length at least $\vtsp{}n-1$.
\item[II.] If $\vmsp{}H^*\mskip-0.5mu$ is a node, then $\vlsp\sigma$
contains an informative action release sequence for $\vlsp{}G$ of
length at least $n$.
\end{itemize}
\end{minipage}
\end{theorem}

\begin{proof}
Throughout the proof we assume notation as given in
Construction~\ref{acyclicdissection} and Lemma~\ref{overlap}.  Observe
that $\sigma\nvtsp\neq\emptyset$, since $G$ is fully controllable with $n > 1$.

\vspace*{0.075in}

I. Suppose $H^*$ is a leaf.  Then $\taup=\taum=\emptyset$.  \ By
Lemma~\ref{acyclicprop} on page~\pageref{acyclicprop},
$\tauc\subseteq\sigma$ and $\tauc$ is cycle-breaking in $H$; by
Corollary~\ref{dissize} on page~\pageref{dissize}, $\abs{\tauc}=n-1$;
and by Lemma~\ref{nocycleiars} on page~\pageref{nocycleiars}, one may
find an ordering of the actions in $\tauc$ such that they form an
informative action release sequence for $H$.  This sequence is also
informative for $G$ by the comment after
Lemma~\ref{combinesubgraphiars} on page~\pageref{combinesubgraphiars}.

\vspace*{0.075in}

II. Suppose $H^*$ is a node.  As in part I, one may find an ordering
of the actions in $\tauc\union\taup$ such that they form an
informative action release sequence for $G$
($\tauc\union\taup\neq\emptyset$, by maximality of $\sigma$ and full
controllability of $G$).  Write this sequence as $a_1, \ldots,
a_\ell$.  It is contained in $\sigma$.

\vso

By Corollary~\ref{dissize}, $\abs{\tauc\union\taup\union\taum}=n$.
Since $\tau^*$ is disruptive in $H^*$, $\taum\neq\emptyset$.  Of
course, one cannot release the actions in $\taum$, since they are not
in $\sigma$.  Instead, as we will see shortly, for each $c\in\taum$
one may release some action of $\sigma_c$ informatively, thereby
completing the proof.

\vso

First, observe that
$\big(\tauc\union\taup\big)\inter\sigma_c=\emptyset$, for every
$c\in\taum$.  To see this, write $\eta=\tauc\union\taup$ and suppose
$a\in\eta$ and $\src(a)\in J_c$ for some $c\in\taum$.  Write $c = w
\rightarrow T$.  Let $\geq$ be the partial order induced on
$\mskip1mu{V}$ by $\eta$.  Then $t \geq \src(a)$ for some $t\in T$.
Since $a\in\eta$, $\src(a) \geq s$ for every $s\in\trg(a)$.  So $t
\geq s$ for every $s\in\trg(a)$, meaning $\trg(a)\subseteq J_c$.
Consequently, $a\not\in\sigma_c$.

\vso

Inductively, suppose we have released, for some sequence of distinct
actions $c_1, ..., c_k$ in $\taum$, with $k\geq 0$, a corresponding
sequence of distinct actions $b_1, ..., b_k$ in $\sigma$, such that
$b_i\in\sigma_{\!c_{\scriptstyle{i}}}$, for $i=1, \ldots, k$, and such
that the overall sequence $a_1, \ldots, a_\ell, \bonespc, \ldots, b_k$
is an iars for $G$.  If $k=\abs{\taum}$, we are done.  Otherwise, we
need to show how to extend this sequence.

\vso

Let $\gamma=\{a_1, \ldots, a_\ell, \bonespc, \ldots, b_k\}$ and
$\gammacls = (\clsAy)(\gamma)$, with $A$ being $G$'s action relation.
Pick some $c\in\sdiff{\taum}{\{c_1, \ldots, c_k\}}$.  We already
observed that $\big(\tauc\union\taup\big)\inter\sigma_c=\emptyset$.
By construction, $\src(b_i)\in J_{\mskip-0.5muc_{\scriptstyle{i}}}$,
for $i=1, \ldots, k$.  By part (a) of Lemma~\ref{disjointforward} on
page~\pageref{disjointforward},
$J_{\mskip-0.5muc_{\scriptstyle{i}}}\inter J_c = \emptyset$, for $i=1,
\ldots, k$.  Consequently, $\gamma \inter \sigma_c = \emptyset$.  By
Corollary~\ref{iarsfwd}, there exist $\kappa\in\calK_c$ and $b\in
\sdiff{\big(\kappa\inter\sigma_c\big)}{\gammacls}$, so $b$ may be
released informatively.  \ Let \hspq$c_{k+1}=c\,$ and $\,b_{k+1}=b$.
\end{proof}
\clearpage

\subsection{Examples for Pure Nondeterministic Graphs}
\markright{Examples for Pure Nondeterministic Graphs}

This subsection shows how the proof of Theorem~\ref{longiarsnondet}
produces informative action release sequences for various pure
nondeterministic graphs and strategies.

\subsubsection{A Hierarchical Pure Nondeterministic Graph}

The first example considers the pure nondeterministic graph of
\fig{Ex202b_Graph} on page~\pageref{Ex202b_Graph}.  The graph
may be viewed directly as a hierarchical cyclic graph, as indicated by
\fig{Ex202b_HTree}.

\begin{figure}[h]
\vspace*{0.1in}
\begin{center}
\ifig{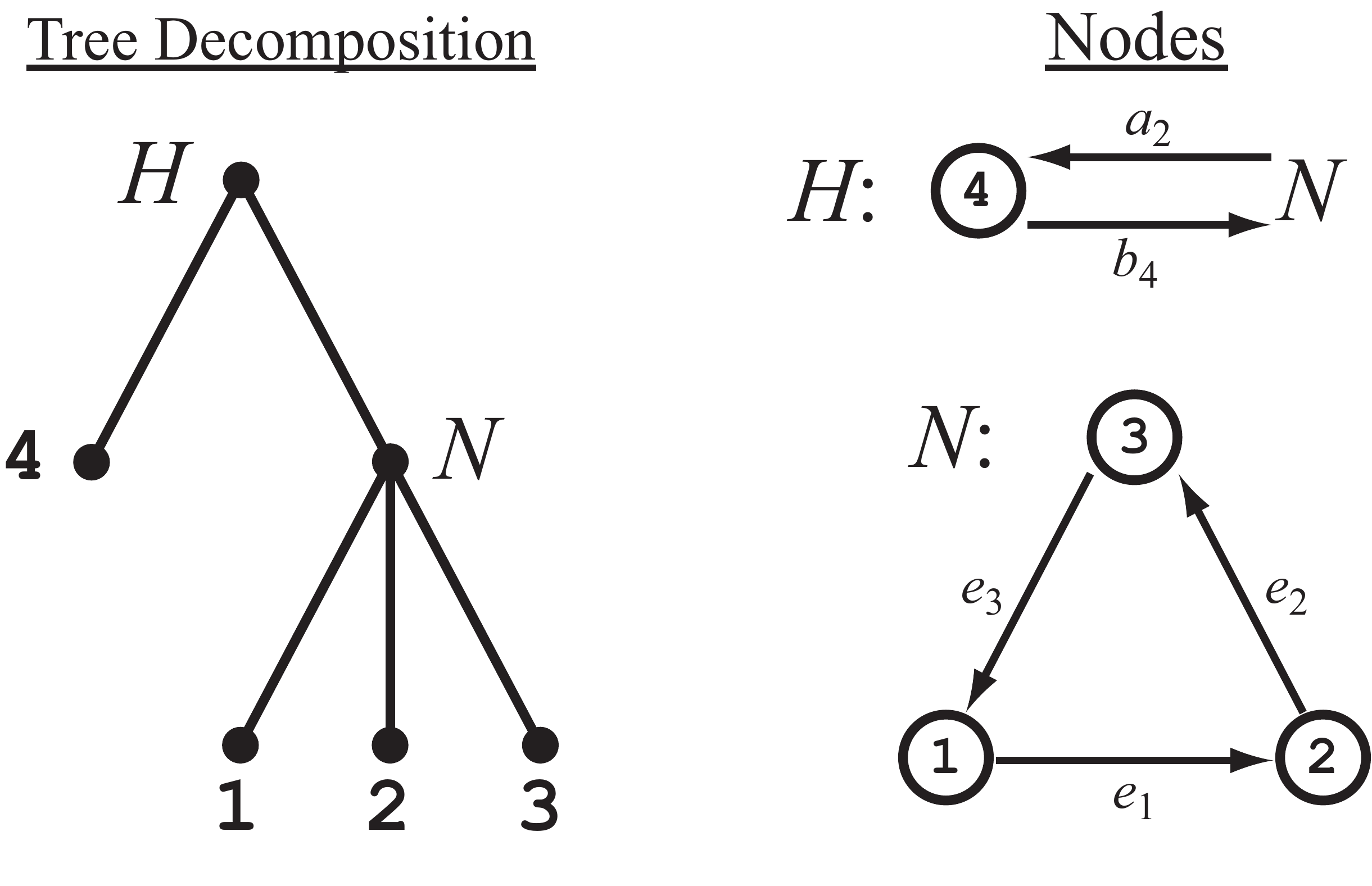}{height=3in}
\end{center}
\vspace*{-0.25in}
\caption[]{A view of graph $G$ from \fig{Ex202b_Graph} directly
  as a hierarchical cyclic graph $H$.  The left panel shows the tree
  decomposition of $H$.  The right panel shows each node's constituent
  parts.  The root $H$ contains two children, a leaf modeling state
  \#4 and a node $N$, along with two cycle actions, $b_4$ and $a_2$.
  Although action $b_4$ is nondeterministic with multiple targets
  inside node $N$, for simplicity the figure merely depicts an arrow
  pointing from state \#4 to node $N$.  \ Node $N$ contains three
  leaves as children, modeling the set of states $\{1, 2, 3\}$, along
  with three cycle actions, $e_1$, $e_2$, and $e_3$.}
\label{Ex202b_HTree}
\end{figure}

Let us consider two maximal strategies and see how our constructions
generate informative action release sequences using the hierarchical
cyclic graph $H$.  Since $H$ and $G$ have the same actions, $\tau$ in
Construction~\ref{acyclicdissection} on
page~\pageref{acyclicdissection} is the maximal strategy under
consideration.

\vspace*{-0.1in}

\paragraph{\underline{$\tau = \{e_2, e_3, b_4\}$}}\quad (This strategy
converges to state \#1.)

\begin{itemize}
\addtolength{\itemsep}{-3pt}

\item Not used by the construction, but just for reference: $\tau$ is
  cycle-breaking in $H$.

\item $\tau$ is not disruptive in $H$, so we run the loop of step 2 in
  Construction~\ref{acyclicdissection}:

\vspace*{-0.1in}

  \begin{enumerate}
  \addtolength{\itemsep}{-3pt}
    \item First we mark node $N$, defining $\kappa^{(1)}=\{e_2, e_3\}$.
    \item Then we mark node $H$, defining $\kappa^{(2)}=\{e_2, e_3, b_4\}$.
  \end{enumerate}

\item At step 3, $H^*$ is a leaf.  So $\tauc=\kappa^{(2)}=\tau$ and
  $\taup=\taum=\emptyset$.

\item The proof of Lemma~\ref{nocycleiars} on
  page~\pageref{nocycleiars} now produces either the sequence $\;b_4,
  e_2, e_3\,$ or the sequence $\,b_4, e_3, e_2\,$ as an informative
  action release sequence.

\end{itemize}

\paragraph{\underline{$\tau = \{e_1, e_2, a_2\}$}}\quad (This strategy
converges to the set of states $\{3, 4\}$.)

\begin{itemize}
\addtolength{\itemsep}{-3pt}

\item $\tau$ is cycle-breaking in $H$.

\item $\tau$ is not disruptive in $H$, so we run the loop of step 2 in
  the construction:

\vspace*{-0.1in}

  \begin{enumerate}
  \addtolength{\itemsep}{-3pt}
    \item First we mark node $N$, defining $\kappa^{(1)}=\{e_1, e_2\}$.
    \item Then we mark node $H$, defining $\kappa^{(2)}=\{e_1, e_2, a_2\}$.
  \end{enumerate}

\item At step 3, $H^*$ is a leaf.  So $\tauc$ is again all of $\tau$
  and $\taup=\taum=\emptyset$.

\item Again, one may release the actions of $\tauc$ informatively, as
  per the proof of Lemma~\ref{nocycleiars}, for instance as the sequence
  $\,a_2, e_1, e_2$.

\end{itemize}

{\bf Comments:}\ (i) $G$'s action relation in \fig{Ex202b} on
page~\pageref{Ex202b} shows that no maximal strategy is disruptive in
$H$, so Construction~\ref{acyclicdissection} will always run the loop
of step 2.
(ii) The construction will always assemble the entire strategy as an
iars.  In fact, as \fig{Ex202b} shows, the strategy complex $\DG$
is a triangulation of $\Stwo$, and in particular has no free faces.
Consequently, any ordering of the actions in a maximal strategy will be
an informative action release sequence for $G$.

\vspace*{0.15in}

\subsubsection{A Pure Nondeterministic Graph with Several Nondeterministic Actions}

Let us add some nondeterministic actions to the previous graph, as shown
in \fig{Ex202_Graph}.

\begin{figure}[h]
\vspace*{0.1in}
\begin{center}
\begin{minipage}{2.9in}
\ifig{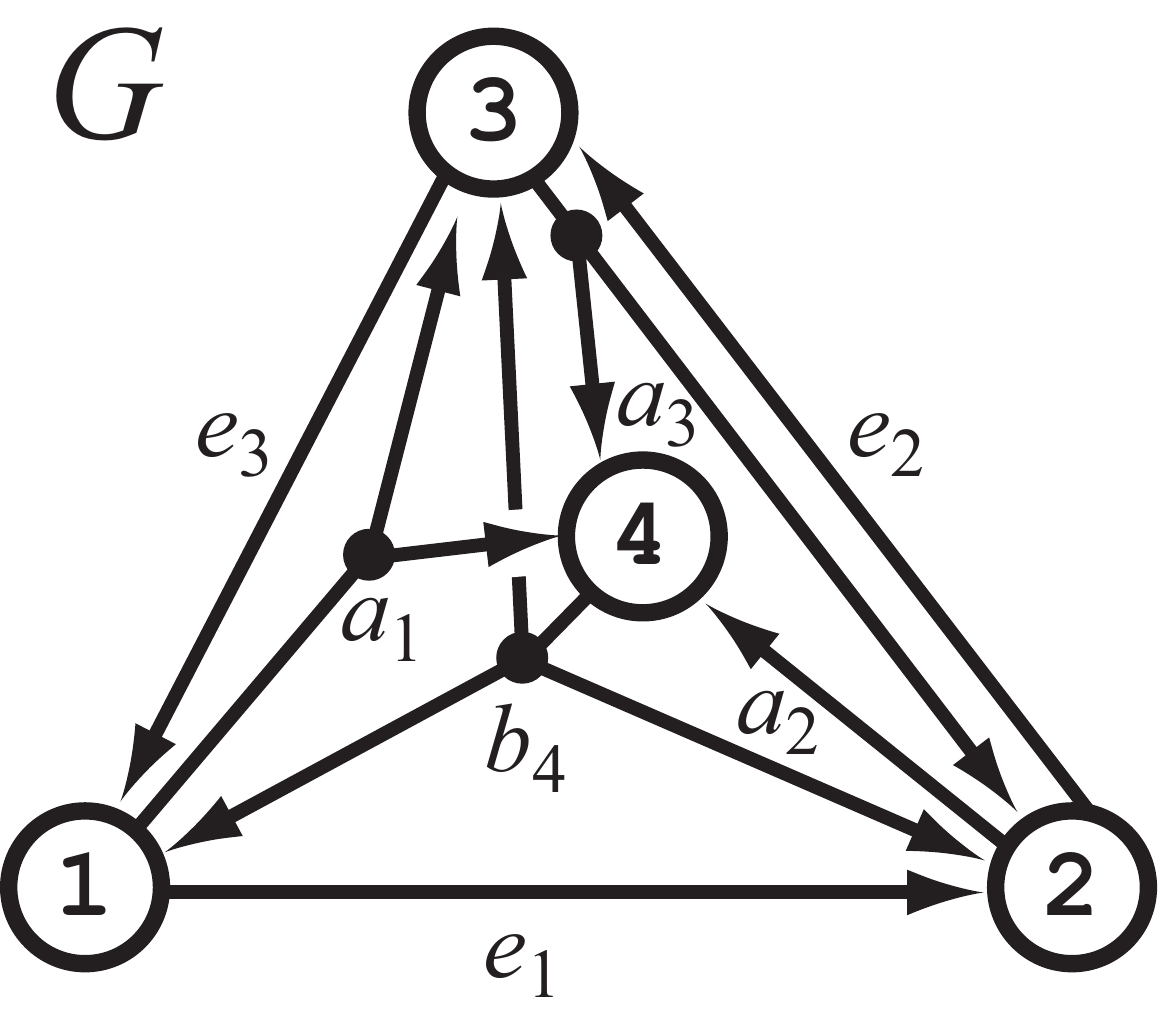}{height=2in}
\end{minipage}
\hspace*{0.2in}
\begin{minipage}{2.3in}{$\begin{array}{c|ccccccc}
       A & e_1  & e_2  & e_3  & a_1  & a_2  & a_3  & b_4 \\[2pt]\hline
\sigma_1 &      & \one & \one &      &      &      & \one \\[2pt]
\sigma_2 & \one &      & \one &      &      &      & \one \\[2pt]
\sigma_3 & \one & \one &      &      &      &      & \one \\[2pt]
\sigma_4 & \one &      & \one &      & \one & \one &      \\[2pt]
\sigma_5 & \one &      &      & \one & \one & \one &      \\[2pt]
\sigma_{14} &      & \one & \one &      & \one &      &      \\[2pt]
\sigma_{34} & \one & \one &      & \one & \one &      &      \\[2pt]
\end{array}$}
\end{minipage}
\begin{minipage}{0.5in}{$\begin{array}{c}
\hbox{Goal} \\[2pt]\hline
1 \\[2pt]
2 \\[2pt]
3 \\[2pt]
4 \\[2pt]
4 \\[2pt]
\{1,4\} \\[2pt]
\{3,4\} \\[2pt]
\end{array}$}
\end{minipage}
\end{center}
\vspace*{-0.1in}
\caption[]{{\bf Left Panel:} A pure nondeterministic graph $G$ with
  four states, $1,2,3,4$, four deterministic actions, $e_1$, $e_2$,
  $e_3$, $a_2$, and three nondeterministic actions, $a_1$, $a_3$,
  $b_4$.\quad\\[2pt]
  {\bf Right Panel:} $G$'s action relation and goal sets.\quad\\[2pt]
  (This figure is a copy of Figures 47 and 48 in \cite{paths:privacy}.)}
\label{Ex202_Graph}
\end{figure}

The earlier hierarchical cyclic graph $H$ of \fig{Ex202b_HTree}
is a subgraph of the new $G$, on the same state space (but with fewer
actions), so we can use the same $H$ as before to construct
informative action release sequences for maximal strategies, now in
the new $G$.  Almost every maximal strategy in the new $\DG$ is either
identical to or a proper superset of a maximal strategy in the old
$\DG$.  Intersecting one of these strategies with the actions of $H$,
as Theorem~\ref{longiarsnondet} requires, therefore produces the same
constructions as before.

\vspace*{0.05in}

There is one exception: The new $\DG$ contains a maximal strategy,
namely $\sigma_5$, that does not restrict to a maximal strategy in the
old $\DG$.  \ Let us look at that strategy more carefully:

\paragraph{\underline{$\sigma = \sigma_5 = \{e_1, a_1, a_2, a_3\}$}}\quad (This strategy
converges to state \#4.)

\begin{itemize}
\addtolength{\itemsep}{-3pt}

\item $\tau$ is the intersection of $\sigma$ with the actions of $H$, so
      $\tau=\{e_1, a_2\}$.

\item $\tau$ is cycle-breaking in $H$.

\item Now $\,\tau$ {\,\em is\,} disruptive in $H$, so
      Construction~\ref{acyclicdissection} $\mskip2.5mu${\em does
      not}$\mskip4mu$ run the loop of step 2, but skips directly to
      step 3.

\item At step 3, $H^*$ is all of $H$, so $\tauc=\emptyset$.

\item $\frakC$ consists of all the core cycle actions of $H$, so
      $\frakC = \{e_1, e_2, e_3, a_2\}$.

\item The construction of $\xi$ in step 4 incorporates all of $\tau$,
      starting from $\xi=\emptyset$, as follows:

\vspace*{-0.1in}

  \begin{enumerate}
  \addtolength{\itemsep}{-3pt}
    \item For node $N$, step 4 adds action $e_1$ to $\xi$.
    \item For node $H$, step 4 adds action $a_2$ to $\xi$.
  \end{enumerate}

\item Thus $\taup=\frakC\inter\xi=\{e_1, a_2\}$ and
      $\taum=\sdiff{\frakC}{\xi}=\{e_2, e_3\}$.

\item The actions of $\taup$ may be released informatively in depth
      order, as the sequence $\,a_2, e_1$.

\item For each action in $\taum$, one finds an action in $\sigma$
      as per the proof of Theorem~\ref{longiarsnondet}:

\vspace*{-0.1in}

  \begin{enumerate}
  \addtolength{\itemsep}{-2pt}
    \item For action $e_2\in\taum$, action $a_3\in\sigma$ lies
    ``downstream'' from $e_2$, forms a minimal nonface with $e_2$, and
    is not implied by $\{a_2, e_1\}$.

    \item For action $e_3\in\taum$, action $a_1\in\sigma$ lies
    ``downstream'' from $e_3$, forms a minimal nonface with $e_3$, and
    is not implied by $\{a_2, e_1, a_3\}$.

  \end{enumerate}

  (The term ``downstream'' refers to the partial order determined by
   $\eta=\tauc\union\taup$.  \ Since $\tauc=\emptyset$, that simply
   means $\taup$ here.
   \ Specifically, the phrase ``action $b\mskip1mu$ lies downstream
   from action $a$'' means that $b$'s source lies in the forward
   projection of $a$'s targets under $\eta$, that is,
   ``$\mskip2mut\geq_{\eta} \src(b)$, for some
   $\mskip1mut\in\trg(a)$''.

   Moreover, this and subsequent examples, following the proof of
   Theorem~\ref{longiarsnondet}, further choose $\mskip1mub\mskip1.75mu$
   so that {\em not\,} all of $\vtsp{}b$'s targets lie within the
   forward projection of $a$'s targets.)

\end{itemize}

Consequently, one may arrange all four actions of $\sigma$ ($=\sigma_5$)
into an informative action release sequence for $G$.  This is consistent
with Theorem~\ref{longiarsnondet}, since $H^*$ is a node in the
construction.  For instance, the sequence $\,a_2, e_1, a_3, a_1\,$ is an
iars.  There are other orderings that will also produce iars of length
4, but not all will do so.  For instance, releasing action $a_1$ as the
first action in a sequence would limit the length of that sequence as an
iars to 2.  See \cite{paths:privacy} for further discussion of this
example.

\clearpage
\subsubsection{A Directed Graph with Several Cycles, Represented Hierarchically}
\label{multicyclegraph}

Consider the directed graph $G$ of \fig{Ex241_GraphSigma}.  All the
actions in this graph are deterministic.  The graph has several
directed cycles in it, giving us the opportunity to explore more than
one hierarchical decomposition for $G$.  The figure also shows a
maximal strategy $\sigma$ in $\DG$.  We will focus on this one
strategy, using two different hierarchical cyclic subgraphs of $G$ to
construct informative action release sequences for $G$ in two
different ways, such that each sequence consists of actions contained
in $\sigma$.  \ For reference, $G$'s full action relation appears in
\fig{Ex202_relation}.

\begin{figure}[h]
\begin{center}
\hspace*{0.1in}
\begin{minipage}{2.7in}
\ifig{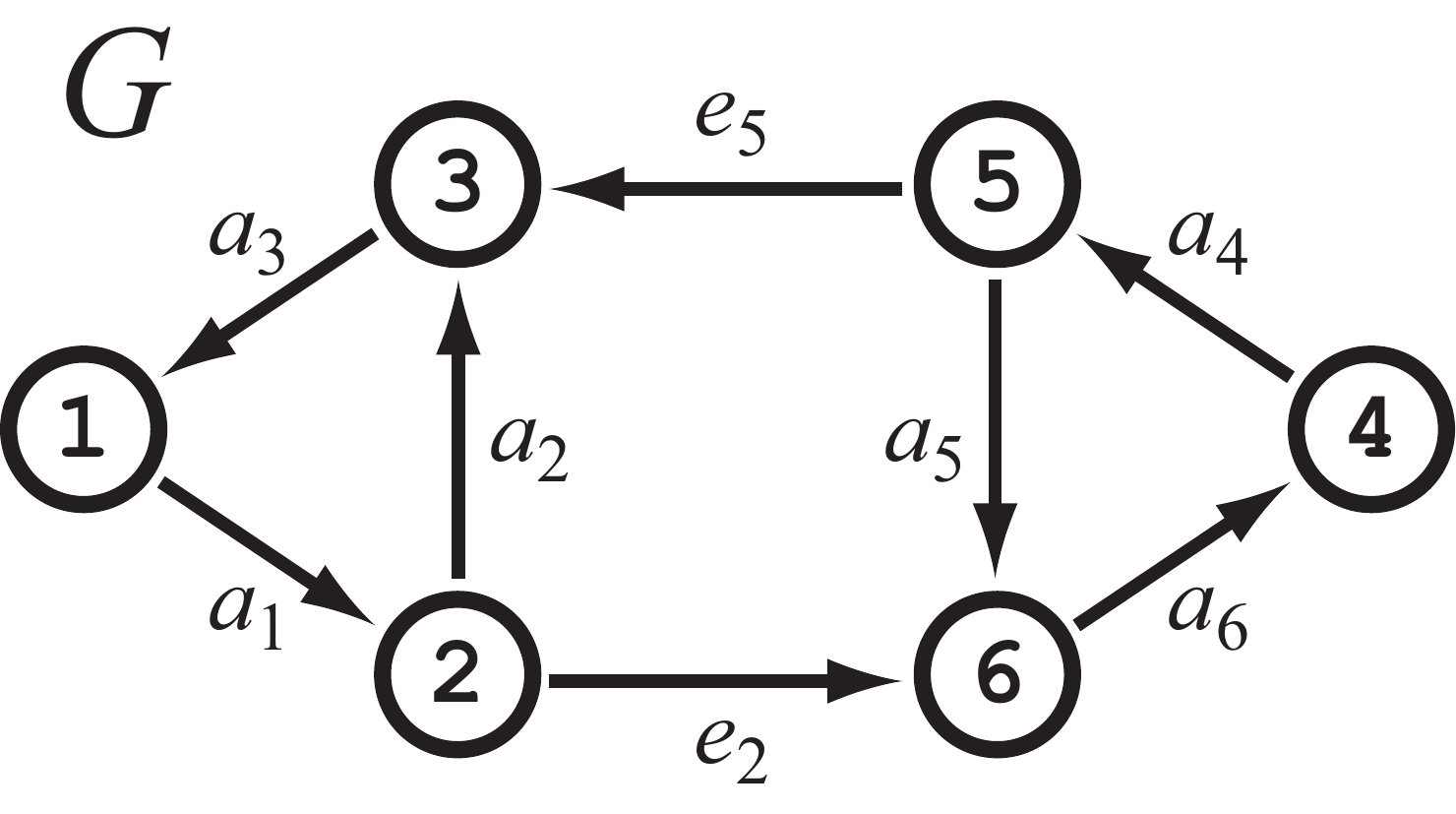}{height=1.5in}
\end{minipage}
\hspace*{0.4in}
\begin{minipage}{2.7in}
\ifig{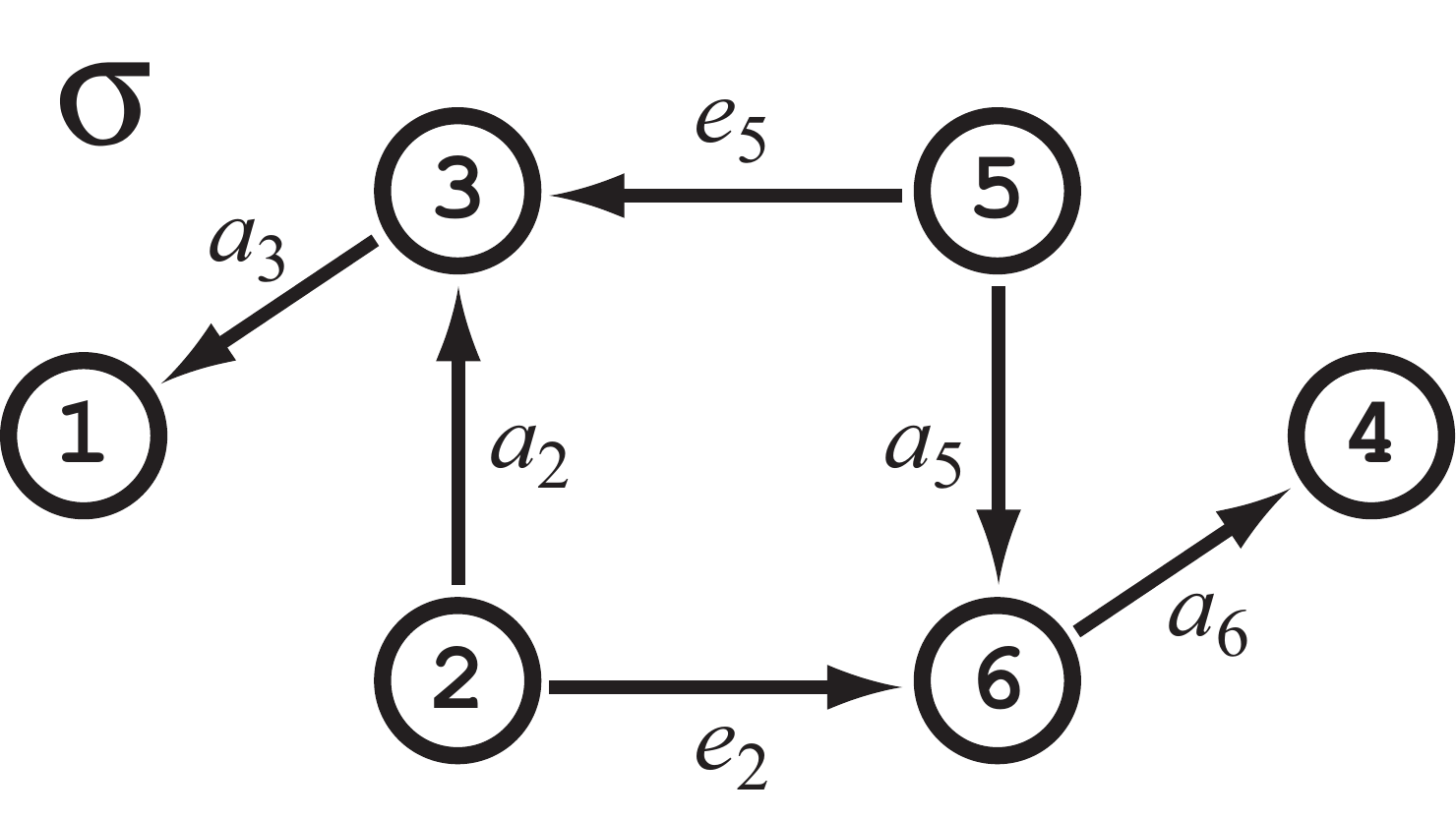}{height=1.5in}
\end{minipage}
\end{center}
\vspace*{-0.1in}
\caption[]{{\bf Left Panel:} A directed graph $G$, consisting of six
  states and eight directed edges.\quad
  {\bf Right Panel:} A maximal strategy $\sigma\in\DG$, depicted by
  its directed edges.}
\label{Ex241_GraphSigma}
\end{figure}

\begin{figure}[h]
\begin{center}
\vspace*{0.1in}
\begin{minipage}{2.5in}{$\begin{array}{c|cccccccc}
     A & a_1  & a_2  & a_3  & a_4  & a_5  & a_6  & e_2  & e_5  \\[2pt]\hline
\sigma &      & \one & \one &      & \one & \one & \one & \one \\[2pt]
       & \one &      & \one &      & \one & \one & \one & \one \\[2pt]
       & \one & \one &      &      & \one & \one & \one & \one \\[2pt]
       &      & \one & \one & \one &      & \one & \one & \one \\[2pt]
       & \one &      & \one & \one &      & \one & \one &      \\[2pt]
       & \one &      & \one & \one &      & \one &      & \one \\[2pt]
       & \one & \one &      & \one &      & \one & \one & \one \\[2pt]
       &      & \one & \one & \one & \one &      & \one & \one \\[2pt]
       & \one &      & \one & \one & \one &      & \one & \one \\[2pt]
       & \one & \one &      & \one & \one &      & \one & \one \\[2pt]
\end{array}$}
\end{minipage}
\hspace*{0.2in}
\begin{minipage}{0.5in}{$\begin{array}{c}
\hbox{Goal} \\[2pt]\hline
\{1,4\} \\[2pt]
4 \\[2pt]
\{3,4\} \\[2pt]
1 \\[2pt]
5 \\[2pt]
2 \\[2pt]
3 \\[2pt]
\{1,6\} \\[2pt]
6 \\[2pt]
\{3,6\} \\[2pt]
\end{array}$}
\end{minipage}
\end{center}
\vspace*{-0.2in}
\caption[]{Action relation and goal sets for the graph of
   \fig{Ex241_GraphSigma}.  The row corresponding to maximal
   strategy $\sigma$ is labeled.  This strategy has a multi-state goal,
   namely $\{1, 4\}$.}
\label{Ex202_relation}
\end{figure}

\paragraph{A Multi-Node Hierarchical Decomposition:} \ The 
decomposition $H$ shown in \fig{Ex241_H1} models $G$ directly
as a hierarchical cyclic graph, meaning $H$ and $G$ contain the same
states and actions.  In this decomposition, the smaller two cycles of
$G$ define two nodes.  Each of these nodes contains only leaves,
comprising the state spaces $\{1, 2, 3\}$ and $\{4, 5, 6\}$,
respectively.  The root of the tree has these two nodes as children,
connected by a two-cycle.

\clearpage

\begin{figure}[h]
\vspace*{0.1in}
\begin{center}
\ifig{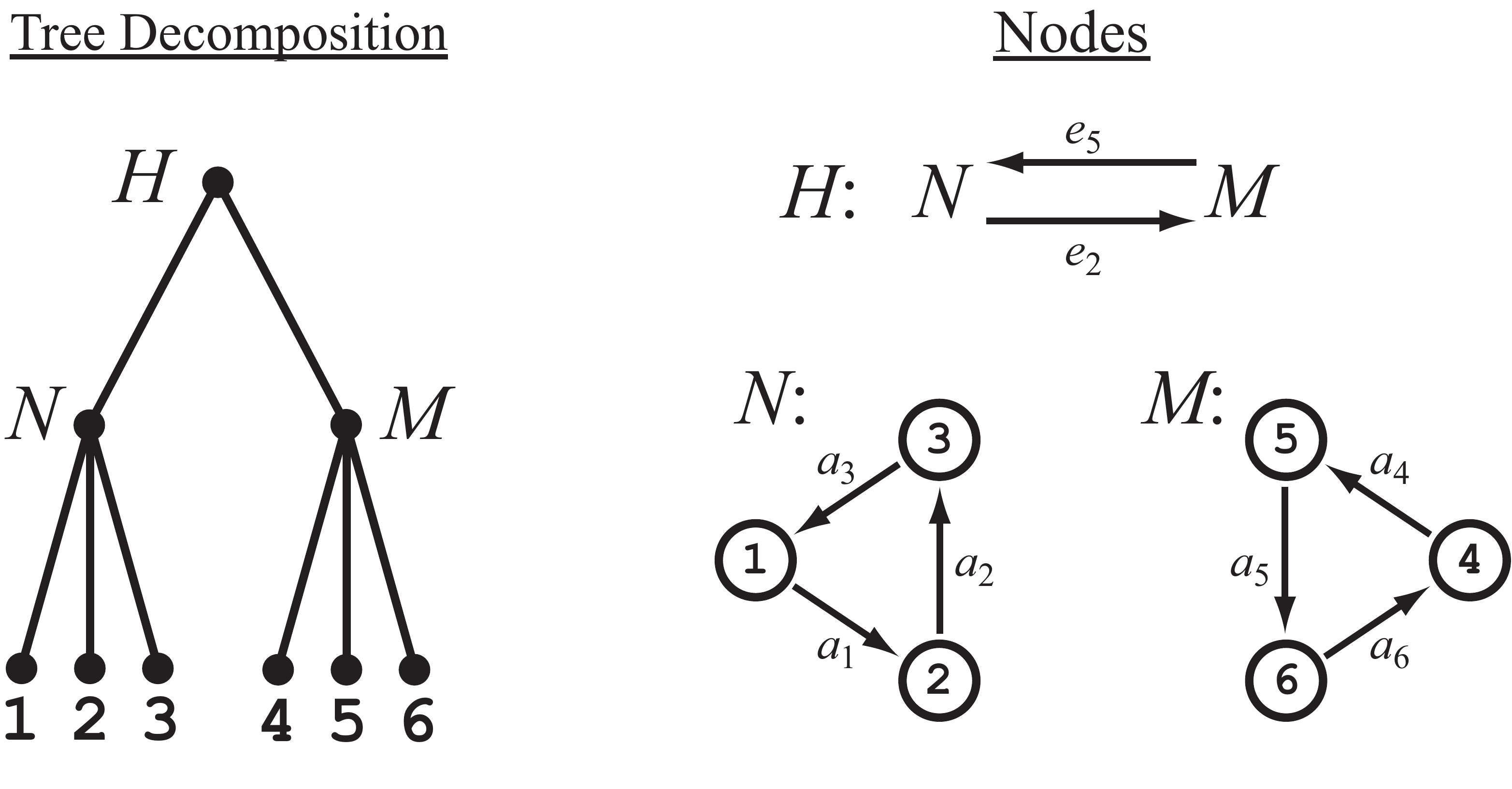}{height=3in}
\end{center}
\vspace*{-0.25in}
\caption[]{A view of graph $G$ from \fig{Ex241_GraphSigma}
  directly as a hierarchical cyclic graph $H$.  The left panel shows
  the tree decomposition of $H$.  The right panel shows each node's
  constituent parts.  The root of $H$ contains two nodes as children,
  along with two cycle actions.  \ The child nodes $N$ and $M$ each
  contain three leaves as children along with three cycle actions.}
\label{Ex241_H1}
\end{figure}

\vspace*{-0.2in}

\paragraph{\underline{$\sigma=\{e_2, e_5, a_2, a_3, a_5, a_6\}$}}

\begin{itemize}
\addtolength{\itemsep}{-3pt}

\item Since $H=G$, also $\,\tau=\sigma$.

\item $\tau$ is not cycle-breaking in $H$, since it contains
      both cycle actions of $H$'s root node.

\item $\tau$ is not disruptive in $H$, so
      Construction~\ref{acyclicdissection} runs the loop of step 2.
      The construction may mark nodes $N$ and $M$ in either order.
      Here we start with $N$.
  
\vspace*{-0.1in}

  \begin{enumerate}
  \addtolength{\itemsep}{-1pt}
    \item Mark node $N$, defining $\kappa^{(1)}=\{a_2, a_3\}$.
    \item Mark node $M$, defining $\kappa^{(2)}=\{a_2, a_3, a_5, a_6\}$.
    \item Mark node $H$.  Since $\tau$ contains both of $H$'s cycle
          actions, the construction could add either action to
          $\kappa^{(2)}$ in defining $\kappa^{(3)}$.  Here we add
          $e_2$, \hspd{}so $\kappa^{(3)}=\{a_2, a_3, a_5, a_6, e_2\}$.
  \end{enumerate}

\item At step 3, $H^*$ is a leaf.  So $\tauc=\{a_2, a_3, a_5, a_6,
      e_2\}$ and $\taup=\taum=\emptyset$.

\item One may release the actions of $\tauc$ informatively in depth
      order, as per the proof of Lemma~\ref{nocycleiars} on
      page~\pageref{nocycleiars}, for instance as the sequence $e_2,
      a_2, a_3, a_5, a_6$.

\item Observe that $\tauc$ is almost all of $\sigma$, excluding only
      action $e_5$.  The construction discarded that one action when
      forming $\kappa^{(3)}$.

\end{itemize}

\vspace*{0.2in}

\paragraph{A Flat Decomposition:} \ \fig{Ex241_H2} shows another
hierarchical cyclic subgraph $H$ of $\mskip2muG$, on the same state
space but with fewer actions.  In this subgraph, the Hamiltonian cycle
of $G$ defines a single node, necessarily the root of $H$, with all
six states as leaves.  Two of $G$'s actions do not appear in $H$.

\begin{figure}[h]
\vspace*{0.075in}
\begin{center}
\ifig{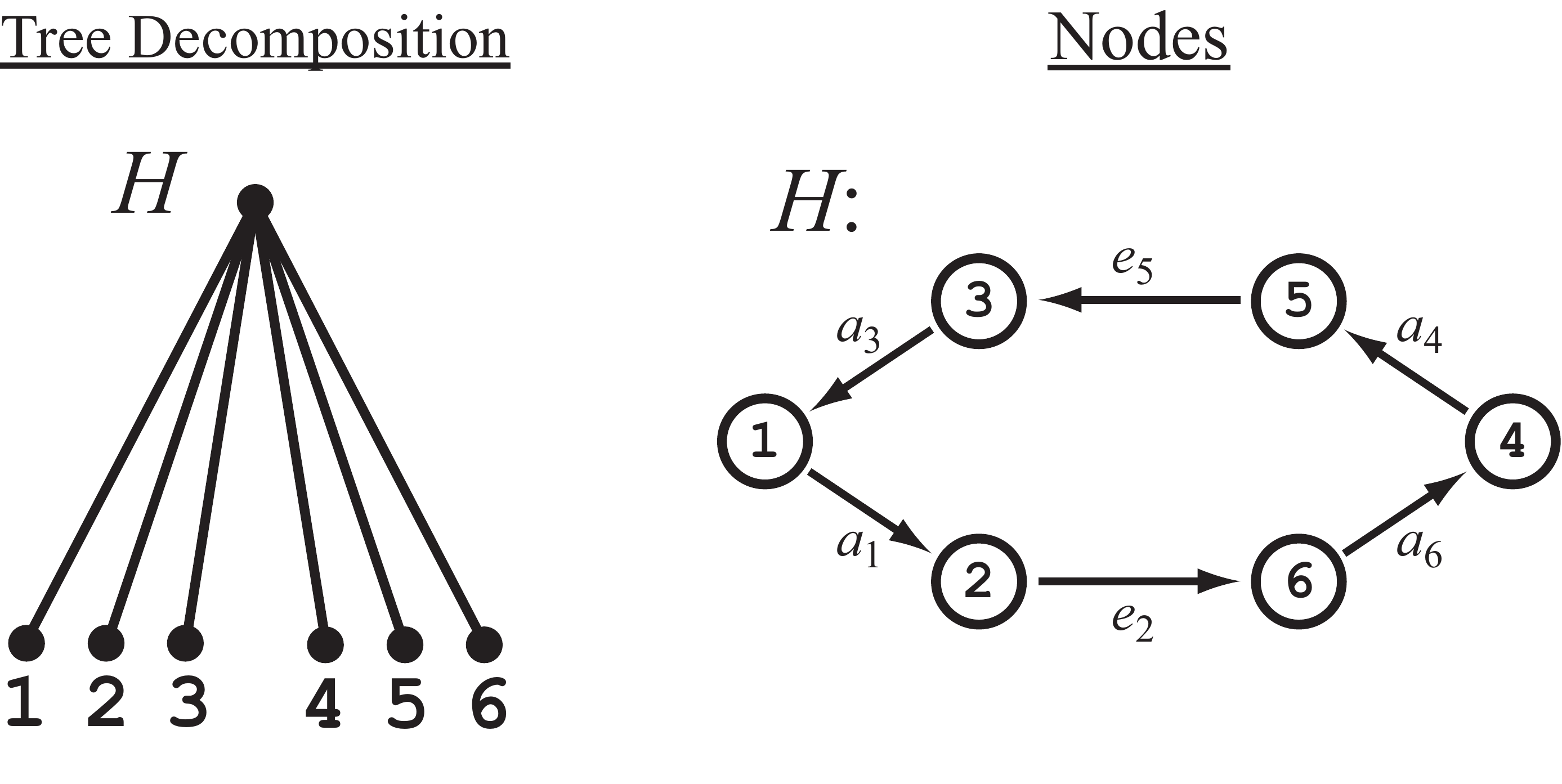}{height=2.5in}
\end{center}
\vspace*{-0.3in}
\caption[]{A hierarchical cyclic subgraph $H\mskip-1mu$ of graph $G$
  from \fig{Ex241_GraphSigma}, with the same states.  The left panel
  shows the tree decomposition of $H$.  The right panel shows
  constituent parts.  The root of $H\mskip-1mu$ contains six leaves,
  connected by six actions forming a directed cycle.}
\label{Ex241_H2}
\end{figure}

\vspace*{-3pt}

\paragraph{\underline{$\sigma=\{e_2, e_5, a_2, a_3, a_5, a_6\}$}}

\begin{itemize}
\addtolength{\itemsep}{-3pt}

\vspace*{3pt}

\item $\tau$ is the intersection of $\sigma$ with $H$'s actions, so
      $\tau=\{e_2, e_5, a_3, a_6\}$.

\item $\tau$ is both cycle-breaking and disruptive in $H$.

\item Since $\tau$ is disruptive, $H^*=H$ and $\tauc=\emptyset$ in
      step 3 of Construction~\ref{acyclicdissection}.
  
\item $\frakC$ consists of all the core cycle actions of $H$, which
      means all the actions of $H$ since $H$ defines a Hamiltonian
      cycle.  \ So $\frakC = \{a_1, e_2, a_6, a_4, e_5, a_3\}$.

\item The construction of $\xi$ incorporates all of $\tau$, since $H$
      consists of a single unmarked node.  Thus
      $\taup=\frakC\inter\xi=\tau=\{e_2, e_5, a_3, a_6\}$ and
      $\taum=\sdiff{\frakC}{\xi}=\{a_1, a_4\}$.

\item The actions of $\taup$ may be released informatively in any
      order, for instance as the sequence $e_2, e_5, a_3, a_6$.

\item For each action in $\taum$, one finds an action in $\sigma$
      as per the proof of Theorem~\ref{longiarsnondet}:

      (Again, ``downstream'' refers to the partial order determined by
      $\taup$.)

\vspace*{-0.1in}

  \begin{enumerate}
  \addtolength{\itemsep}{-1pt}
    \item For action $a_1\in\taum$, action $a_2\in\sigma$ lies
         ``downstream'' from $a_1$, participates in the minimal
         nonface $\{a_1, a_2, a_3\}$ with $a_1$, and is not implied by
         $\{e_2, e_5, a_3, a_6\}$.

    \item For action $a_4\in\taum$, action $a_5\in\sigma$ lies
         ``downstream'' from $a_4$, participates in the minimal
         nonface $\{a_4, a_5, a_6\}$ with $a_4$, and is not implied by
         $\{e_2, e_5, a_3, a_6, a_2\}$.

  \end{enumerate}

\vspace*{-0.05in}

\item Consequently, all actions of $\sigma$ may be arranged into the
  informative action release sequence $e_2, e_5, a_3, a_6, a_2, a_5$.

\end{itemize}

\paragraph{Comment:}\  We have seen the following: (i) With $H$
  as in \fig{Ex241_H1}, Construction~\ref{acyclicdissection} produces
  an informative action release sequence for $G$ consisting of
  $\mskip1mu5$ actions in $\sigma$.  (ii) With $H$ as in
  \fig{Ex241_H2}, the construction produces an informative action
  release sequence consisting of all $\hspc{6}$ actions in $\sigma$.
  These sequence lengths match the assertions of
  Theorem~\ref{longiarsnondet} on page~\pageref{longiarsnondet}.

\subsubsection{A Directed Graph with a Disruptive but not Cycle-Breaking Strategy}

This example will illustrate an instance in which $\tau$ contains all
the cycle actions in a node during step 4 of
Construction~\ref{acyclicdissection}.  \fig{Ex249_GraphSigma}
depicts a graph $G$ and a maximal strategy $\sigma\in\DG$.
\fig{Ex249_relation} displays $G$'s action relation.
\fig{Ex249_H} shows a hierarchical cyclic subgraph $H$ of $G$, on
the same state space (but with fewer actions).  \ (Other such subgraphs
exist, of course.)

\begin{figure}[h]
\begin{center}
\hspace*{0.1in}
\begin{minipage}{2.2in}
\ifig{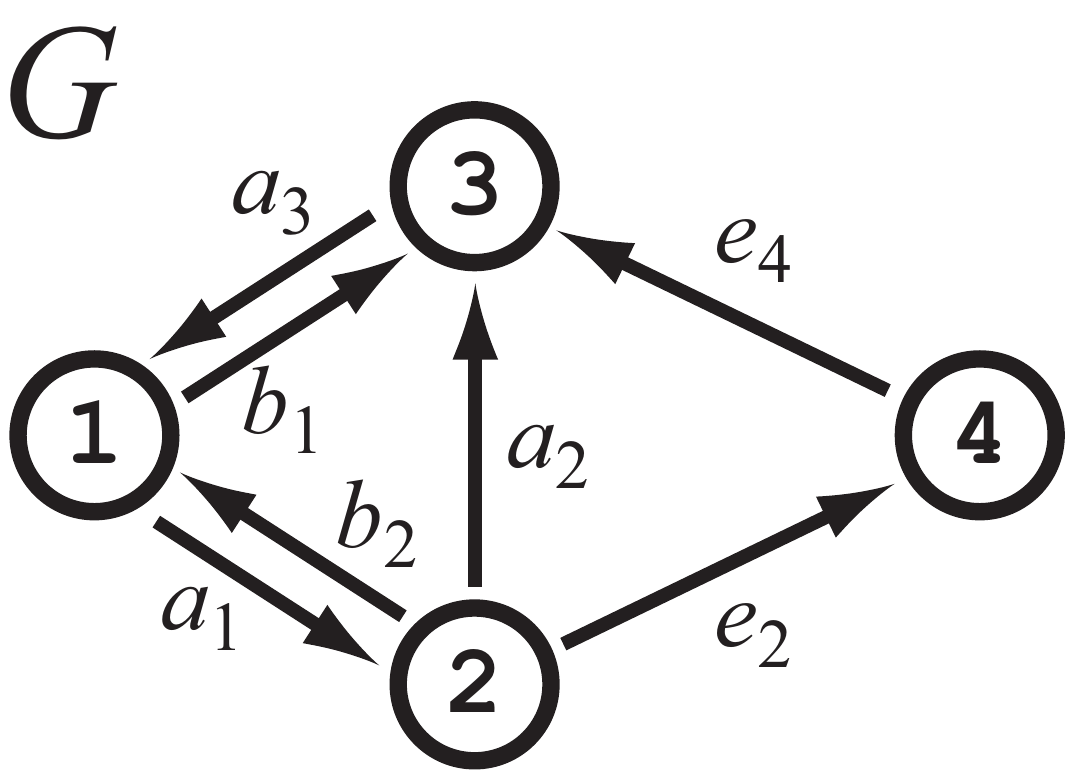}{height=1.5in}
\end{minipage}
\hspace*{0.4in}
\begin{minipage}{2.2in}
\ifig{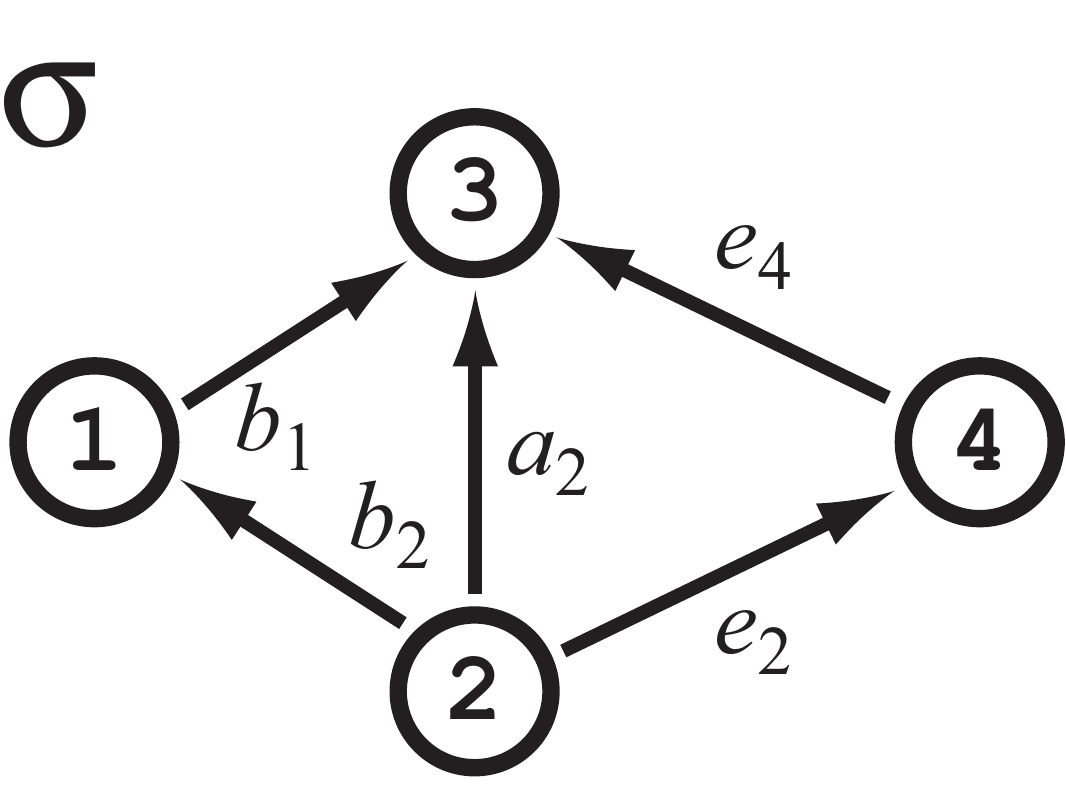}{height=1.5in}
\end{minipage}
\end{center}
\vspace*{-0.1in}
\caption[]{{\bf Left Panel:} A directed graph $G$, consisting of four
  states and seven directed edges.\quad
  {\bf Right Panel:} A maximal strategy $\sigma\in\DG$, depicted by
  its directed edges.}
\label{Ex249_GraphSigma}
\end{figure}

\begin{figure}[h]
\begin{center}
\vspace*{0.1in}
\begin{minipage}{2in}{$\begin{array}{c|ccccccc}
     A & a_1  & a_2  & a_3  & b_1  & b_2  & e_2  & e_4  \\[2pt]\hline
       &      & \one & \one &      & \one & \one & \one \\[2pt]
       & \one &      & \one &      &      & \one &      \\[2pt]
       & \one &      & \one &      &      &      & \one \\[2pt]
       & \one & \one &      & \one &      & \one & \one \\[2pt]
\sigma &      & \one &      & \one & \one & \one & \one \\[2pt]
\end{array}$}
\end{minipage}
\hspace*{0.2in}
\begin{minipage}{0.5in}{$\begin{array}{c}
\hbox{Goal} \\[2pt]\hline
1 \\[2pt]
4 \\[2pt]
2 \\[2pt]
3 \\[2pt]
3 \\[2pt]
\end{array}$}
\end{minipage}
\end{center}
\vspace*{-0.1in}
\caption[]{Action relation and goals for the graph of
   \fig{Ex249_GraphSigma}.  The row corresponding to maximal
   strategy $\sigma$ is labeled.  This strategy converges to state \#3.}
\label{Ex249_relation}
\end{figure}

\begin{figure}[h]
\vspace*{0.1in}
\begin{center}
\ifig{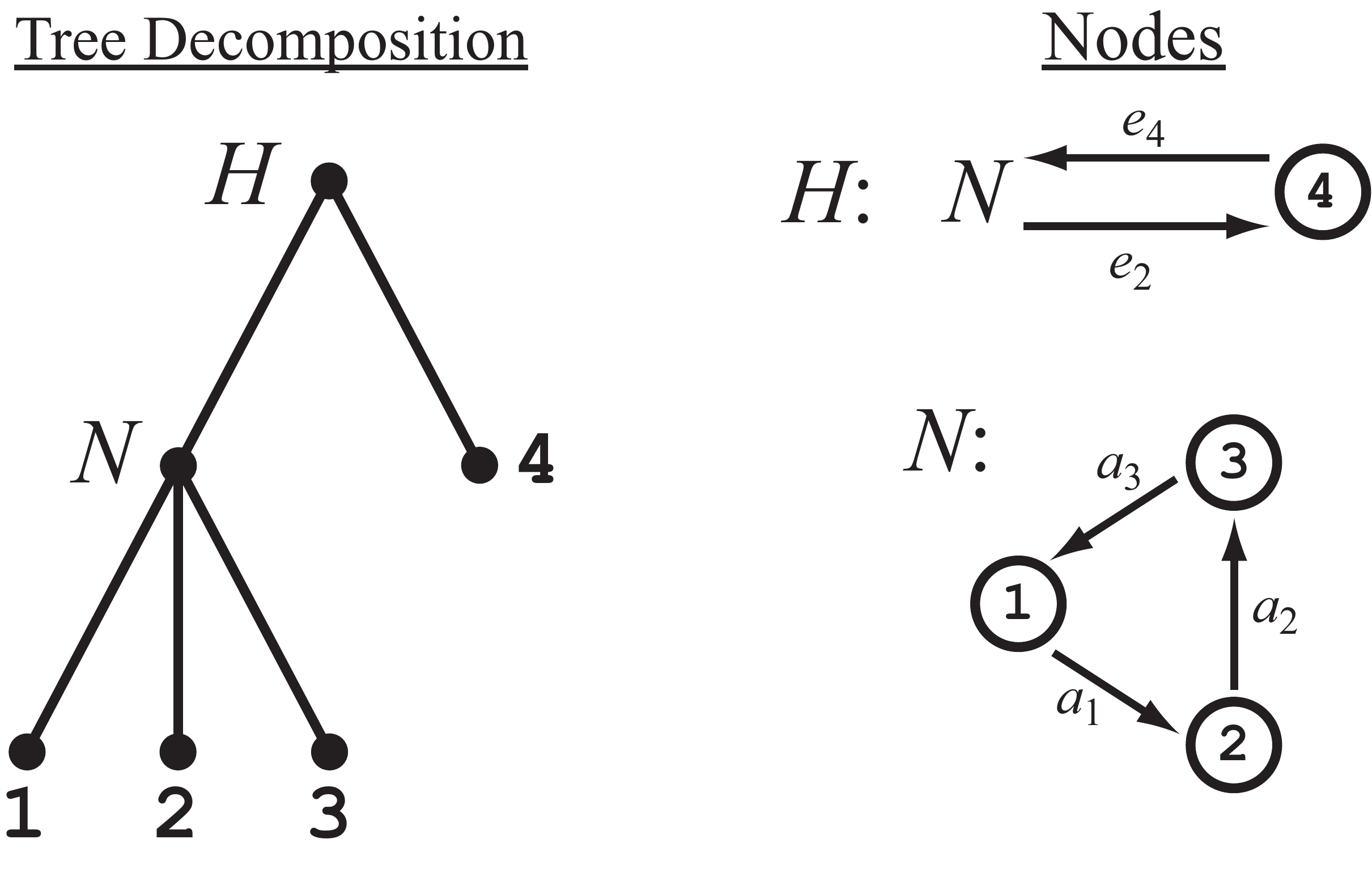}{height=3in}
\end{center}
\vspace*{-0.25in}
\caption[]{A hierarchical cyclic subgraph $H$ of graph $G$ from
  \fig{Ex249_GraphSigma}, on the same same state space (but with
  fewer actions).  The left panel shows the tree decomposition of $H$.
  The right panel shows each node's constituent parts.  The structure is
  very similar to that of \fig{Ex202b_HTree}.}
\label{Ex249_H}
\end{figure}

\paragraph{\underline{$\sigma=\{e_2, e_4, a_2, b_1, b_2\}$}}

\begin{itemize}
\addtolength{\itemsep}{-3pt}

\item $\tau$ is the intersection of $\sigma$ with $H$'s actions, so
      $\tau=\{e_2, e_4, a_2\}$.

\item $\tau$ is not cycle-breaking in $H$, since it contains both cycle
  actions of $H$'s root node.

\item $\tau$ is disruptive in $H$, since it contains only one of node
      $N$'s three cycle actions.

\item Since $\tau$ is disruptive, $H^*=H$ and $\tauc=\emptyset\,$ in
      step 3 of Construction~\ref{acyclicdissection}.
  
\item $\frakC$ consists of all the core cycle actions of $H$, so
      $\frakC = \{a_1, a_2, a_3, e_2\}$.

\item The construction of $\xi=\{a_2, e_2\}\,$ in step 4 occurs as
      follows, starting from $\xi=\emptyset$:

\vspace*{-0.1in}

  \begin{enumerate}
  \addtolength{\itemsep}{-1pt}
    \item For node $N$, $\CyN=\{a_1, a_2, a_3\}$ and
    $\CyN\inter\tau=\{a_2\}$, so one adds action $a_2$ to $\xi$.

    \item For node $H$, $\CyH=\{e_2, e_4\}$ and $\CyH\inter\tau=\CyH$,
    so one must discard some action of $\CyH$ that is not in $\frakC$.
    That action is $e_4$.  One adds action $e_2$ to $\xi$.
  \end{enumerate}

\item Consequently, $\taup=\frakC\inter\xi=\{a_2, e_2\}$ and
      $\taum=\sdiff{\frakC}{\xi}=\{a_1, a_3\}$.

\item The actions of $\taup$ may be released informatively in depth
      order, so as the sequence $\,e_2, a_2$.

\item For each action in $\taum$, one finds an action in $\sigma$
      as per the proof of Theorem~\ref{longiarsnondet}:

      (Once again, ``downstream'' refers to the partial order
      determined by $\taup$.)

\vspace*{-0.1in}

  \begin{enumerate}
  \addtolength{\itemsep}{-1pt}
    \item For action $a_1\in\taum$, action $b_2\in\sigma$ lies
       ``downstream'' from $a_1$, forms a minimal nonface with $a_1$,
       and is not implied by $\{e_2, a_2\}$.

    \item For action $a_3\in\taum$, action $b_1\in\sigma$ lies
      ``downstream'' from $a_3$, forms a minimal nonface with $a_3$,
      and is not implied by $\{e_2, a_2, b_2\}$.

  \end{enumerate}

\item Therefore $\,e_2, a_2, b_2, b_1\,$ is an informative action
      release sequence for $G$, contained in $\sigma$.

      (The order matters: Revealing either $b_1$ or $b_2$ at the
      beginning of the sequence would narrow the set of maximal
      strategies consistent with the revealed action to two instantly.
      One then could reveal only one more action informatively before
      identifying $\sigma$, that action being the other ``$b_i$''
      action not yet revealed.  Indeed, revealing actions
      $\vtsp{}b_1\mskip-1.75mu$ {\em and} $\,b_2\vtsp$ in either order
      identifies the maximal strategy to be $\sigma$.  Revealing
      action $a_2$ implies action $e_2$, while revealing $e_2$ at the
      beginning does not imply $a_2$.  Although not part of the
      construction, observe that revealing $b_1$ would in and of
      itself declare the goal to be state \#3.)

\end{itemize}

\clearpage
\section{The Stochastic Setting}
\markright{The Stochastic Setting}
\label{stochastic}

The aim of this section is to prove Theorem~\ref{longiars} from
page~\pageref{longiars} for the case in which the graph $G$ is pure
stochastic.  Throughout, this section assumes that all graphs are pure
stochastic, meaning each action is either deterministic or stochastic
(but not nondeterministic).

\paragraph{Caution:}\ Even though all actions in a pure stochastic
graph are deterministic or stochastic, there may still be a component
of nondeterminism in a strategy: When multiple actions have the same
source state, any one of those actions might execute from that state,
with the choice potentially made by an adversary.  (See again the
discussion of generalized control laws on
page~\pageref{actionsemantics}, as well as the definitions of ``moves
off'', ``contains a circuit'', and ``strategy complex''.)

\subsection{Expanding Fully Controllable Subgraphs via Minimal Nonfaces}
\markright{Expanding Fully Controllable Subgraphs via Minimal Nonfaces}

As mentioned on page~\pageref{minnonfacesemantics}, a minimal nonface
of $\DG$ in a pure stochastic graph $G$ defines an irreducible Markov
chain and thus a fully controllable subgraph of $G$.  If $G$ is itself
fully controllable, one may construct such a minimal nonface $\kappa$
for each maximal strategy $\sigma\in\DG$, for instance by considering
some action at a goal state of $\sigma$.  The actions of the minimal
nonface $\kappa$ that lie within $\sigma$ then form an informative
action release sequence $z$ for $G$, contained in $\sigma$.

One may expand the state space $\src(\kappa)$ covered by this minimal
nonface by considering some action outside $\sigma$, in a manner to be
discussed.  This process yields a new minimal nonface and thus
additional actions of $\sigma$ with which to enlarge the informative
action release sequence $z$.  Repeating this process one may
eventually encounter a situation in which there are no further useful
actions outside $\sigma$.  Instead, one forms a quotient graph by
identifying all the states covered thus far.  Recursively, one obtains
an informative action release sequence within this quotient graph.
Patching the two sequences together gives an overall informative
action release sequence contained in $\sigma$ of length one less than
the number of states in $G$'s state space.

\vspace*{0.1in}

The following construction and subsequent results describe this process formally:

\begin{construction}[Minimal Nonface Expansion]\label{expandminnonfaces}
Let $G=(V,\frakA)$ be a fully controllable pure stochastic graph with
$n = \abs{V} > 1$ and suppose $\mskip1.7mu\sigma$ is a maximal
strategy in $\DG$.

Construct a collection $\{b_1, \ldots, b_k\}$ of convergent actions in
$\frakA$, a collection $\{\kappa_1, \ldots, \kappa_k\}$ of minimal
nonfaces \!of $\DG$, a collection $\{\calA_1, \ldots, \calA_k\}$ of
subsets of $\vtsp\frakA$, and a collection $\{W_1, \ldots, W_k\}$ of
subsets of $\vlsp{}V$, with $k \geq 1$, as follows:
\end{construction}

\begin{enumerate}

\addtolength{\itemsep}{1pt}

\item Let $g\in\sdiff{V}{\src(\sigma)}$.  Choose $b_1\in\frakA$ so
      that $\{b_1\}\in\DG$ and $\src(b_1)=g$.  Such an action exists
      since $G$ is fully controllable and $\abs{V} > 1$.

\item Since $\sigma$ is maximal, $\sigma\union\{b_1\}\not\in\DG$, so
      there exists a minimal nonface $\kappa_1$ of $\DG$ such that
      $b_1\in\kappa_1\subseteq\sigma\union\{b_1\}$.
      \ \ (For later reference, observe also that $\abs{\kappa_1} > 1$.)

\item Let $\calA_1=\kappa_1$ and $W_1=\src(\kappa_1)$.

\item Set {\sc Done} to {\tt false}. \ While not {\sc Done}, run the
      following loop, starting from $i=1$:

  \vspace*{-0.1in}

  \begin{enumerate}

  \addtolength{\itemsep}{3pt}

  \item Consider the quotient graph $G/W_i$ and let prime notation
        refer to the correspondence between actions in $G$ and
        $G/W_i$, as per the discussion on page~\pageref{quotient}.

  \item \label{qloopstepb}Define
        $\xi_i=\setdef{a\in\sigma}{\src(a)\in\sdiff{V}{W_i}}$.  \ So
        $\,\xi_i\subseteq\sigma$, $\,\xi_i\in\DG$, and
        $\src(\xi_i)\inter{W_i}=\emptyset$.

  \item By Fact 1 on page~\pageref{quotientfacts},
        $\xi_i^\prime\in\Delta_{G/W_{\scriptstyle{i}}}$, so extend
        $\xi_i^\prime$ to a maximal simplex
        $\tau^\prime_i\in\Delta_{G/W_{\scriptstyle{i}}}$.

  \item If \label{qloopstepd}$\tau_i\subseteq\sigma$, then set $k$ to
        the current value of $\vtsp{}i$ and {\sc Done} to {\tt true}.
        \ The loop ends.

        \vspace*{0.05in}

        Otherwise: \hspace*{-0.15in}\begin{minipage}[t]{4.1in}
                   \begin{itemize}
                   \addtolength{\itemsep}{-1pt}

                   \item[--] Let $b_{i+1}\in\sdiff{\tau_i}{\sigma}$.

                   \item[--] As in step 2, there exists a minimal nonface
                         $\mskip1mu\kappa_{i+1}\mskip1mu$ of $\,\DG$ such that
                         $\mskip1mub_{i+1}\in\kappa_{i+1}\subseteq\sigma\union\{b_{i+1}\}$.
                         \ (Again, $\abs{\kappa_{i+1}} > 1$.)

                   \item[--] Let $\calA_{i+1}=\calA_i\union\kappa_{i+1}$
                         and $W_{i+1}=W_i\union\src(\kappa_{i+1})$.

                   \item[--] The loop continues, with $i+1$ in place of $i$.
                   \end{itemize}
                   \end{minipage}

  \end{enumerate}

\end{enumerate}

\vspace*{0.05in}

\begin{lemma}[Expansive Subspaces]\label{expansive}
Let the hypotheses and notation be as in Construction~\ref{expandminnonfaces}.

Then $\mskip1muW_i \subsetneq W_{i+1}$, for all $\,i$ such that
$\vlsp{}W_i$ and $\vtsp{}W_{i+1}\mskip-1mu$ are well-defined.

(Consequently, the loop of step 4 in the construction ends, that is,
$k$ is well-defined finite.)
\end{lemma}

\begin{proof}
It is enough to show that $\src(\kappa_{i+1}) \not\subseteq W_i$.
Suppose otherwise.  Since $b_{i+1}\in\kappa_{i+1}$, that would mean
$\src(b_{i+1})\in W_i$ and $\trg(b_{i+1})\subseteq W_i$. (The
inclusion holds because $\kappa_{i+1}$ is a minimal nonface in $\DG$,
so no action of $\kappa_{i+1}$ moves off $\src(\kappa_{i+1})$, and
because $G$ is pure stochastic.)  Thus $b^{{\kern .1em}\prime}_{i+1}$
would become self-looping in $G/W_i$, contradicting $b^{{\kern
.1em}\prime}_{i+1}\in\tau^\prime_i\in\Delta_{G/W_{\scriptstyle{i}}}$.
\end{proof}

\vspace*{0.01in}

\begin{lemma}[Fully Controllable Expansion]\label{expfullcontrol}
Let the hypotheses and notation be as in Construction~\ref{expandminnonfaces}.
\ Then $(W_i, \vtsp\calA_i)$ is a fully controllable pure stochastic
graph, for $i=1,\ldots,k$.
\end{lemma}

\begin{proof}
Observe that $k \geq 1$, since $\abs{V} > 1$. 

\vso

Base Case: \ $(W_1, \vtsp\calA_1) = (\src(\kappa_1), \,\kappa_1)$.  \
Since $\kappa_1$ is a minimal nonface in the strategy complex of a
pure stochastic graph, $(\src(\kappa_1), \,\kappa_1)$ is a fully
controllable pure stochastic graph.

\vso

Inductive Step: \ As in the base case, $(\src(\kappa_{i+1}),
\,\kappa_{i+1})$ is a fully controllable pure stochastic graph.
Inductively, $(W_i, \vtsp\calA_i)$ is a fully controllable pure
stochastic graph.  Showing that $W_i \inter \src(\kappa_{i+1}) \neq
\emptyset$ would therefore establish full controllability of the pure
stochastic graph $(W_{i+1}, \calA_{i+1})$.  Suppose this intersection
is empty.  Then $\kappa^\prime_{i+1}$ is a minimal nonface in
$\Delta_{G/W_{\scriptstyle{i}}}$.  On the other hand,
$\sdiff{\kappa_{i+1}\!}{\!\{b_{i+1}\}}\subseteq\vtsp\xi_i$, \vtsp{}so
$\kappa^\prime_{i+1} \subseteq \tau^\prime_i \in
\Delta_{G/W_{\scriptstyle{i}}}$, producing a contradiction.
\end{proof}

\vspace*{0.01in}

\begin{lemma}[Distinct Actions]\label{distinctbi}
Let the hypotheses and notation be as in Construction~\ref{expandminnonfaces}.

Then $\abs{\{b_1, \ldots, b_k\}} = k$, that is, the actions $b_1,
\ldots, b_k$ are distinct.
\end{lemma}

\begin{proof}
Suppose $1 \leq j \leq i < k$.  Then $\src(b_j)\in W_j$ and
$\trg(b_j)\subseteq W_j$.  Since $W_j \subseteq W_i$, action
$b^\prime_j$ is self-looping in $G/W_i$ and thus $b_j$ cannot be a
candidate for $b_{i+1}$.
\end{proof}

\vspace*{0.01in}

\begin{lemma}[Expansive Sets of Actions]\label{fullexp}
Let hypotheses and notation be as in Construction~\ref{expandminnonfaces}.
Suppose $1 < i \leq k$.
\ Let $\ell_i = \abs{\sdiff{W_i}{W_{i-1}}}$.
\quad (By Lemma~\ref{expansive}, $\mskip2mu\ell_i > 0$.)

\vst

Then there exist actions $\hspt\expansive_i \subseteq
\sdiff{\kappa_i\!}{\!\big(\calA_{i-1}\union\{b_i\}\big)}$ such that
$\abs{\expansive_i} = \ell_i$ and at most one action in
$\hspq\expansive_i$ has its source in \hspq$W_{i-1}$.
\quad (We refer to $\mskip3mu\expansive_i$ as an
\,\mydefem{expansive set} of actions.)

\vst

Moreover, suppose for all $\,\expansive \subseteq
\sdiff{\kappa_i\!}{\!\big(\calA_{i-1}\union\{b_i\}\big)}$ with
$\abs{\expansive} = \ell_i$, $\,\src(\expansive)\inter
W_{i-1}\neq\emptyset$.  Then $\src(b_i)\not\in W_{i-1}$ and one may
choose $\expansive_i$ to contain an action $\exa$ such that
$\src(\exa)\in W_{i-1}$ and such that the probability of reaching
$\vtsp\src(b_i)$ from $\vtsp\src(\exa)$ under actions of
$\vmsp\expansive_i$ is nonzero.
\end{lemma}

\noindent {\bf Comments:}\label{commentszero}
\ (a) Let $\calA_0 = \emptyset\,$, $W_0 = \{\src(b_1)\}$, and
    $\expansive_1 = \sdiff{\kappa_1\!}{\!\{b_1\}}$.  Then the lemma
    holds for $i=1$, with $\src(\expansive_1)\inter W_0=\emptyset$.
\ (b) For $i=1,\ldots,k$, $\vmsp\expansive_i \subseteq \calA_i\inter\sigma$,
    since $\kappa_i \subseteq \calA_i$ and $\sdiff{\kappa_i\!}{\!\{b_i\}}
    \subseteq \sigma$.

\begin{proof}
Assume $1 < i \leq k$.
\ Readily,
$\;\sdiff{W_i}{W_{i-1}} = \sdiff{\src(\kappa_i)}{W_{i-1}}\;$
and
$\;W_{i-1} = \src(\calA_{i-1}).$
\ Thus
$\;\sdiff{W_i}{W_{i-1}} \;=\; \sdiff{\src(\kappa_i)}{\src(\calA_{i-1})}
\;\subseteq\; \src\big(\sdiff{\kappa_i}{\calA_{i-1}}\big),$
meaning each state in $\sdiff{W_i}{W_{i-1}}$ is the source of some
action in $\kappa_i$ that is not also an action in $\calA_{i-1}$.  If
in fact every state in $\sdiff{W_i}{W_{i-1}}$ is the source of some
action in $\kappa_i$ that is neither an action in $\calA_{i-1}$ nor
the action $b_i$, then we may construct
$\mskip1.5mu\expansive_i\subseteq
\sdiff{\kappa_i\!}{\!\big(\calA_{i-1}\union\{b_i\}\big)}\mskip2mu$
such that $\mskip1mu\abs{\expansive_i} = \ell_i\mskip2mu$ and
$\mskip2mu\src(\expansive_i)\inter W_{i-1}=\emptyset$.

Otherwise, since all actions in a minimal nonface have distinct
sources, it is only possible to find $\ell_i-1$ actions in
$\sdiff{\kappa_i\!}{\!\big(\calA_{i-1}\union\{b_i\}\big)}$ whose
sources lie outside $W_{i-1}$.  Moreover,
$\src(b_i)\in\sdiff{W_i}{W_{i-1}}$.  By the proof of
Lemma~\ref{expfullcontrol}, $\src(\kappa_i)\inter W_{i-1}
\neq\emptyset$, meaning $\kappa_i$ contains at least one action with
source in $W_{i-1}$.  We now show by backchaining from $\src(b_i)$ how
to select one such action $e$ so that $\expansive_i$ may consist of
action $e$ and the $\mskip1.5mu\ell_i-1$ actions just mentioned.

\vspace*{0.07in}

To reduce index clutter, we fix $i$ and make the following
definitions for the rest of the proof:

\vspace*{-0.225in}

\begin{eqnarray*}
       \calA &=& \calA_{i-1} \quad \hbox{and} \quad W \;=\; \src(\calA),\\
       b &=& b_i,\\
  \kappa &=& \kappa_i,\\
   \outside &=&
   \setdef{a\in\sdiff{\kappa\!}{\!\big(\calA\union\{b\}\big)}}{\vtsp\src(a)\in\sdiff{\src(\kappa)\!}{\!W}}.\\
\end{eqnarray*}

\vspace*{-0.25in}

(By assumption for this case, $\vtsp\abs{\outside}=\ell_i-1\vtsp$ and
$\vtsp\src(b)\in\sdiff{\src(\kappa)\!}{\!W}$.)

\vspace*{0.1in}

We now define a backchaining algorithm, with a loop index $j$, for
constructing sets of actions $\emptyset \neq \tau^{(0)} \subsetneq
\cdots \subsetneq \tau^{(j)} \subsetneq \cdots$.  Inductively, each
iteration assumes that (i) $b\in\tau^{(j)} \subsetneq \kappa$, (ii)
$\tau^{(j)} \subseteq \outside\union\{b\}$, and (iii) for each
$s\in\src(\tau^{(j)})$, there exists a sequence of zero or more action
edges leading from $s$ to $\src(b)$, with the edges coming from
actions in $\sdiff{\tau^{(j)}}{\{b\}}$.

\vspace*{0.1in}

We initialize the loop with $\tau^{(0)}=\{b\}$.  The loop will end
by defining an action $\exa\vtsp$ such that we may let
$\expansive_i=\outside\union\{\exa\}$, establishing the lemma.
\ The loop starts from $j=0$:
      
\vspace*{0.1in}

\begin{minipage}{5.6in}

  \begin{enumerate}

  \item[(a)] Since $\kappa$ is a minimal nonface in $\DG$, with $G$
             pure stochastic, $(\src(\kappa), \vmsp\kappa)$ is a fully
             controllable graph in its own right and $\emptyset \neq
             \src(\tau^{(j)}) \subsetneq \src(\kappa)$, by
             Lemma~\ref{minnonfacestrat} on
             page~\pageref{minnonfacestrat}.  Thus some action
             $a^{(j)}\in\kappa$ moves off
             $\vmsp\sdiff{\src(\kappa)}{\src(\tau^{(j)})}$ in this
             graph.

  \item[(b)] If we can pick $a^{(j)}$ so that $\src(a^{(j)})\in W$,
             then we do so and in that case we let $\exa=a^{(j)}$.
             Either way, we define
             $\tau^{(j+1)}=\tau^{(j)}\union\{a^{(j)}\}$.  Condition
             (iii) above is satisfied by $\tau^{(j+1)}$ since it is
             satisfied by $\tau^{(j)}$ and
      $\mskip0.5mu\trg(a^{(j)})\inter\mskip1.1mu\src(\tau^{(j)})\neq\emptyset$.

  \item[(c)] If step (b) defined action $\exa$, then the loop ends.

             Otherwise, necessarily
             $a^{(j)}\!\in\sdiff{\outside}{\tau^{(j)}}$.  Thus, in
             this case, $\tau^{(j+1)}$ also satisfies conditions (ii)
             and (i) above, since in particular some action of
             $\kappa$ has source in $W\mskip-1mu$ but no action of
             $\tau^{(j+1)}$ does.  \ The loop continues, with $j+1$ in
             place of $j$.

  \end{enumerate}

\end{minipage}

\vspace*{0.15in}

By finiteness, the loop must eventually end, for some $j$.  The
probability of reaching $\src(b)$ from $\src(\exa)$ under actions of
$\sdiff{\tau^{(j+1)}\!}{\!\{b\}}$ is nonzero by condition (iii), so
the same will be true under actions of $\expansive_i= \outside \union \{\exa\}
\subseteq \kappa$.  Moreover,
$\exa\in\sdiff{\kappa\!}{\!\big(\calA\union\{b\}\big)}$ with
$\src(\exa)\in W$, since $\exa\in\kappa$ and
$\emptyset\neq\trg(\exa)\inter\mskip1.5mu\src(\tau^{(j)})\subseteq\sdiff{\src(\kappa)}{W}$,
whereas $\trg(a)\subseteq W$, for all $a\in\calA$, and $\src(b)\not\in W$.
\end{proof}

\paragraph{For the remainder of Section~\ref{stochastic}:} \ \label{Hidefs}Assume
the hypotheses and notation of Construction~\ref{expandminnonfaces}
starting on page~\pageref{expandminnonfaces}.  Let $\calA_0 =
\emptyset\,$ and $W_0 = \{\src(b_1)\}$.  Define $H_i = (W_i,
\vtsp\calA_i)$, for $i=0, 1, \ldots, k$, with $k \geq 1$.  Each $H_i$
is a fully controllable pure stochastic graph, by
Lemma~\ref{expfullcontrol}.
Also, for $i=1, \ldots, k$, $H_{i-1}$ is a subgraph of $H_i$, with
$\emptyset\neq{}W_{i-1}\subsetneq{}W_i$ and
$\calA_{i-1}\subsetneq\calA_i$, by Lemmas~\ref{expansive}
and \ref{distinctbi}, and since $\abs{\kappa_1} > 1$.
\ Let $\expansive_1 = \sdiff{\kappa_1\!}{\!\{b_1\}}$.  For $i=2,
\ldots, k$, define $\expansive_i$ via Lemma~\ref{fullexp}, choosing
$\expansive_i$ so that $\vtsp\src(\expansive_i)\inter{}W_{i-1}=\emptyset\vtsp$
whenever possible.

\vspace*{0.05in}

\begin{corollary}[Expansion Independence]\label{indepexp}
\quad
Let hypotheses and notation be as above.

Suppose $1 \leq i \leq k$.
\ Then $\mskip3mu\expansive_i \union \tau \in \DHi$, \,for every $\mskip2mu\tau\in\DHim$.
\end{corollary}

\begin{proof}
If $\src(\expansive_i)\subseteq\sdiff{W_i}{W_{i-1}}$, then the lemma's
assertion follows from Lemma~7.3(b)(i) in \cite{paths:plans}.

\vst

Otherwise, $i>1$.  Let $\exa\in\expansive_i$ be as per Lemma~\ref{fullexp}.
There exists a sequence of action edges

\vspace*{-0.1in}

$$\src(\exa)=v_1\;\xrightarrow{a_1=\mskip2mu\exa}\; v_2
       \;\xrightarrow{\phantom{1}a_2\phantom{1}}\;
       \cdots \;v_m
       \;\xrightarrow{\phantom{1}a_m\phantom{1}}\; v_{m+1}=\src(b_i),$$

\noindent for some $m \geq 1$, with $a_j\in\expansive_i$,
$v_j=\src(a_j)$, and $v_{j+1}\in\trg(a_j)$, for all $j=1, \ldots, m$.
Moreover, $v_1 \in W_{i-1}$, $v_j\not\in W_{i-1}$, for $j = 2, \ldots,
m$, \vtsp{}and $v_{m+1}=\src(b_i)\not\in W_{i-1}\mskip-1mu\union\mskip0.5mu\src(\expansive_i)$.

\vst

Suppose $\expansive_i\union\tau\not\in\DHi$, for some $\tau\in\DHim$.
Let $\minexp$ be a minimal nonface of $\DHi$, with
$\emptyset\neq\minexp\subseteq\expansive_i\union\tau$.  Some or all of
the actions $\{a_1, \ldots, a_m\}$ lie in $\minexp$.
\ Certainly $\exa\in\minexp$, again by Lemma~7.3(b)(i) in \cite{paths:plans}.
Since $\minexp$ is a minimal nonface, $\exa$ is the only action of
$\minexp$ with source $\src(\exa)$.  \ Since $\tau\in\DHim$,
$\vtsp\src(\sdiff{\minexp}{\expansive_i}) \subseteq W_{i-1}$.  The
actions in $\expansive_i$ all have distinct sources.  Thus no action
in $\minexp$ other than $a_j$ (if $a_j$ is even in $\minexp$) can have
source $v_j$, for $j = 2, \ldots, m$.

\vst

Consequently, there is a nonzero probability that the system will
transition to and stop at a state outside $\src(\minexp)$ when started
at $\src(\exa)$, while moving under actions of $\minexp$.  Some action of
$\minexp$ therefore moves off $\src(\minexp)$, which is a contradiction.
\end{proof}

\vspace*{0.05in}

\begin{corollary}[Cardinality of Expansive Actions]\label{expandingcard}
Let hypotheses and notation be as above.

\vst

Then $\,\bigabs{\mskip-2.25mu\bigunion_{i=1}^k\expansive_i\mskip1mu} = \abs{W_k} - 1$.
\end{corollary}

\begin{proof}
By Construction~\ref{expandminnonfaces}, Lemma~\ref{fullexp}, and
subsequent comments,

\vspace*{-0.2in}

$$\matchabs{\bigunion_{i=1}^k\expansive_i}
    \;=\;
           \sum_{i=1}^k\abs{\expansive_i}
    \;=\;
           \sum_{i=1}^k\ell_i
    \;=\;
           \sum_{i=1}^k\abs{\sdiff{W_i}{W_{i-1}}}
    \;=\;
           \sum_{i=1}^k\left(\abs{W_i} - \abs{W_{i-1}}\right)
    \;=\;
           \abs{W_k} - 1.
$$

\vspace*{-0.25in}

\end{proof}

\subsection{Informative Action Release Sequences from Expansive Sets of Actions}
\markright{Informative Action Release Sequences from Expansive Sets of Actions}

This subsection shows how the constructions of the previous
subsection produce informative action release sequences.
Some notational abbreviations will be useful:

\vspace*{-0.1in}

\paragraph{Notation and Terminology:}
\begin{enumerate}
\item Rather than merely write sequences of actions, $b_1, \ldots,
  b_m$, we may write sequences of sets of actions $\calB_1, \ldots,
  \calB_m$, assuming the sets $\calB_1, \ldots, \calB_m$ are nonempty
  and pairwise disjoint.

  The meaning of a set $\calB_i$ of actions is to indicate a multiplicity
  of sequences of actions, one for each possible permutation of the
  actions in the set $\calB_i$.  The sequence of sets $\calB_1, \ldots,
  \calB_m$ represents all possible orderings of the actions
  $\union_{i=1}^m\calB_i$ consistent with the top-level ordering
  $\calB_1, \ldots, \calB_m$.  \ Here, ``consistent'' means actions in
  $\calB_i$ must appear before actions in $\calB_j$ whenever $i < j$,
  but the ordering is otherwise unconstrained.

  For example, the sequence of sets $\,\{a, b\}, \{c\}, \{d, e, f\}\,$
  represents 12 sequences of actions:

  \vspace*{-0.2in}

   $$\begin{array}{cccccc}
   a,b,c,d,e,f \,&\, a,b,c,e,f,d \,&\, a,b,c,f,d,e \,&\, a,b,c,f,e,d \,&\, a,b,c,e,d,f \,&\, a,b,c,d,f,e \\[1pt]
     b,a,c,d,e,f \,&\, b,a,c,e,f,d \,&\, b,a,c,f,d,e \,&\, b,a,c,f,e,d \,&\, b,a,c,e,d,f \,&\, b,a,c,d,f,e \\[1pt]
     \end{array}$$

  \vspace*{-0.05in}

\item We say that a sequence $\calB_1, \ldots, \calB_m$ of sets of
  actions is \vtsp{\em informative for $\vtsp{}G\vlsp$} if each of the
  sequences of actions it represents is an informative action release
  sequence for graph $G$.

\item In place of a singleton set, we may also simply write the
  action it contains.  For instance, we could write the top-level
  sequence in the example above as $\,\{a, b\}, \,c, \{d, e, f\}\,$.
\end{enumerate}

\vspace*{0.1in}

\begin{lemma}[Expanding Informative Actions]\label{expandingiars}
\ Suppose $G=(V,\frakA)$ is a fully controllable pure stochastic graph
with $n = \abs{V} > 1$. 
\ Let $\sigma$ be a maximal strategy in $\DG$.

\vso

From $G$ and $\mskip1mu\sigma\nvtsp$ construct $\vtsp{}H_1, \ldots,
H_k$ and $\vtsp\expansive_1, \ldots, \expansive_k$, with $k \geq 1$,
as per Construction~\ref{expandminnonfaces} on
page~\pageref{expandminnonfaces}, Lemma~\ref{fullexp} on
page~\pageref{fullexp}, and the definitions and notation of
page~\pageref{Hidefs}.

\vso

Then, for each $i$, with $1 \leq i \leq k$, the sequence $\expansive_i,
\expansive_{i-1}, \ldots, \expansive_1$ is informative for $H_i$.
\end{lemma}

\begin{proof}
By induction on $i$.
\ Let $i$, with $1 \leq i \leq k$, be given.

\vst

The set $\expansive_i$ is a nonempty proper subset of a minimal
nonface of $\DG$ and thus of $\DHi$.   By Lemma~\ref{minnonfaceiars}
on page~\pageref{minnonfaceiars}, every ordering of actions in
$\expansive_i$ is an informative action release sequence for $H_i$.
\ Moreover, $\expansive_i\in\DHi$.

\vsr

If $i=1$, these observations establish the base case.

\vst

If $i > 1$, then inductively $\expansive_{i-1}, \ldots, \expansive_1$
is informative for $H_{i-1}$ and 
$\expansive_{i-1} \union \cdots \union \expansive_1 \in \DHim$.

By Corollary~\ref{indepexp} on page~\pageref{indepexp}, $\expansive_i
\union \tau \in \DHi$, for every $\tau\in\DHim$.  By construction, no
action in $\expansive_i$ is an action in the graph $H_{i-1}$.
Therefore, by Lemma~\ref{combinesubgraphiars} on
page~\pageref{combinesubgraphiars}, $\expansive_i, \expansive_{i-1},
\ldots, \expansive_1$ is informative for $H_i$.  Moreover,
$\expansive_i \union \cdots \union \expansive_1 \in \DHi$, since,
for instance,
$\expansive_i \union \cdots \union \expansive_1\subseteq\sigma$.
\end{proof}

\vspace*{0.1in}

\begin{corollary}[Expanding Informative Actions in $G$]\label{expandingiarsG}
Let the hypotheses and notation be as for Lemma~\ref{expandingiars}.
\ For each $i$ with $1 \leq i \leq k$, the sequence $\expansive_i,
\expansive_{i-1}, \ldots, \expansive_1$ is informative for $G$.
\end{corollary}

\begin{proof}
By the previous lemma, $\expansive_i, \expansive_{i-1}, \ldots,
\expansive_1$ is informative for $H_i$.  The comment after the
statement of Lemma~\ref{combinesubgraphiars} on
page~\pageref{combinesubgraphiars}~establishes the corollary.
\end{proof}

\vspace*{0.1in}

\subsection{An Informative Action Release Sequence from a Quotient}
\markright{An Informative Action Release Sequence from a Quotient}
\label{iarsfromquotient}

The loop in Construction~\ref{expandminnonfaces} may end in step 4(d)
(on page~\pageref{qloopstepd}) with $\tau_k=\emptyset$.  This will occur
if and only if $W_k=V$.  In that case, Corollary~\ref{expandingiarsG}
(above) and Corollary~\ref{expandingcard} (on
page~\pageref{expandingcard}) imply that the sequence $\expansive_k,
\ldots, \expansive_1$ provides an informative action release sequence
for $G$ of length $n-1$, with all actions of the sequence contained in
$\sigma$, and with $n=\abs{V}>1$.

Otherwise, the following lemma ensures that one may add a prefix of
informative actions to that sequence whenever one can find an
informative sequence in the quotient graph $G/W_k$.

\begin{lemma}[Informative Actions from Quotient]\label{quotiars}
\ Suppose $G=(V,\frakA)$ is a fully controllable pure stochastic graph
with $n = \abs{V} > 1$. 
\ Let $\sigma$ be a maximal strategy in $\DG$.

\vso

From \vtsp$G$ and $\sigma\nvtsp$ construct \vmsp$k$, $W_k$, $\tau_k$,
$H_k$, and $\vtsp\expansive_1, \ldots, \expansive_k$, as per
Construction~\ref{expandminnonfaces} on
page~\pageref{expandminnonfaces}, Lemma~\ref{fullexp} on
page~\pageref{fullexp}, and the definitions and notation of
page~\pageref{Hidefs}.  (Recall that $k \geq 1$.)

\vso

Suppose further that $a^\prime_1, \ldots, a^\prime_\ell$ is an
informative action release sequence for $G/W_k$, with $\ell \geq 1$
and $\{a^\prime_1, \ldots,
a^\prime_\ell\}\subseteq\tau^\prime_k\in\Delta_{G/W_{\scriptstyle{k}}}$.

\vsr

Then $\vtsp{}a_1, \ldots, a_\ell, \vmsp\expansive_k, \ldots,
\expansive_1$ is informative for $\vtsp{}G$, with all actions
contained in $\sigma$.
\end{lemma}

\begin{proof}
By construction, $\{a_1, \ldots,
a_\ell\}\subseteq\vtsp\tau_k\subseteq\sigma$ and
$\union_{i=1}^k\expansive_i \subseteq \sigma$.

\vso

By Lemma~\ref{expfullcontrol} on page~\pageref{expfullcontrol}, $H_k$
is a fully controllable subgraph of $G$.
\ Also, $\emptyset\neq W_k\subsetneq V$.

By Lemma~\ref{expandingiars} on page~\pageref{expandingiars},
$\vtsp\expansive_k, \ldots, \expansive_1$ is informative for $H_k$.
Any informative sequence of actions formed from $\vtsp\expansive_k,
\ldots, \expansive_1$ is a subset of $\sigma$, therefore convergent in
both $G$ and $H_k$.

Corollary~\ref{combinequotientgraphiars} on
page~\pageref{combinequotientgraphiars} therefore establishes the desired
result.
\end{proof}

\vspace*{0.1in}

The following theorem instantiates Theorem~\ref{longiars} of
page~\pageref{longiars} for pure stochastic graphs:

\begin{theorem}[Informative Action Release Sequences :
    Pure Stochastic Graphs]$\phantom{0}$\label{longiarsstoch}

Let $G=(V,\frakA)$ be a fully controllable pure stochastic graph
with $n = \abs{V} > 1$.

Suppose $\sigma$ is a maximal strategy in $\DG$.

Then $\vtsp\sigma$ contains an informative action release sequence for
$\vlsp{}G$ of length at least ${\kern .135em}n-1$.
\end{theorem}

\begin{proof}
By induction on $n$.

\vspace*{0.05in}

{Base Case:} \ {$n=2$.}

\vso 

In this case, $\DG$ consists of two (nonempty) maximal strategies, one
for each state in $V$.  (The strategy for state $v$ consists of all
actions with source $v$ that are not deterministic self-loops.  The
strategy converges to the other state.)  Any single action in one of
these strategies constitutes an informative action release sequence
for $G$ and is contained in the given strategy.

\vspace*{0.1in}

{Inductive Step:} \ {$n>2$.}

\vso

From $G$ and $\sigma$ construct \vmsp$k$, $W_k$, $\tau_k$, and
$\vtsp\expansive_1, \ldots, \expansive_k$, using
Construction~\ref{expandminnonfaces} on
page~\pageref{expandminnonfaces}, Lemma~\ref{fullexp} on
page~\pageref{fullexp}, and subsequent comments.
Recall that $k \geq 1$.

\vst

As discussed on page~\pageref{iarsfromquotient}, if
$\tau_k=\emptyset$, then the sequence $\expansive_k, \ldots,
\expansive_1$ provides an informative action release sequence for $G$
of length $n-1$, consisting of actions in $\sigma$.

\vst

Otherwise, let $\ell = \abs{\sdiff{V}{W_k}}$.  Then $\ell > 0$.  The
quotient graph $G/W_k$ is pure stochastic and fully controllable, by
Fact 3 on page~\pageref{quotientfacts}.  It has state space
$V^\prime=(\sdiff{V}{W_k})\union \{\wrep\}$, with $\wrep$ representing
all of $W_k$ identified to a single state.

\vst

Since the minimal nonface $\kappa_1$ in
Construction~\ref{expandminnonfaces} contains at least two actions,
$W_k$ contains at least two states.  Therefore $2 \leq \abs{V^\prime}
< n$.  Inductively, the theorem holds for graph $G/W_k$ and maximal
strategy $\tau^\prime_k$, producing an informative action release
sequence $a^\prime_1, \ldots, a^\prime_\ell$ for $G/W_k$ with
$\{a^\prime_1, \ldots,
a^\prime_\ell\}\subseteq\tau^\prime_k\in\Delta_{G/W_{\scriptstyle{k}}}$.
By Lemma~\ref{quotiars}, $a_1, \ldots, a_\ell, \vmsp\expansive_k,
\ldots, \expansive_1$ is informative for $\vtsp{}G$, with all actions
contained in $\sigma$.  Any consequent informative action release
sequence has length $\vtsp\ell + \abs{\union_{i=1}^k\expansive_i} =
\abs{\sdiff{V}{W_k}} + (\abs{W_k} - 1) = n - 1$, by
Corollary~\ref{expandingcard} on page~\pageref{expandingcard}.
\end{proof}

\clearpage
\subsection{Examples for Pure Stochastic Graphs}
\markright{Examples for Pure Stochastic Graphs}

This subsection shows how the proof of Theorem~\ref{longiarsstoch}
produces informative action release sequences for some pure stochastic
graphs and strategies.  \vtsp{}For clarity, figures discard self-loops.

\subsubsection{A Directed Graph with Several Cycles}

\begin{figure}[t]
\begin{center}
\hspace*{0.1in}
\begin{minipage}{2.7in}
\ifig{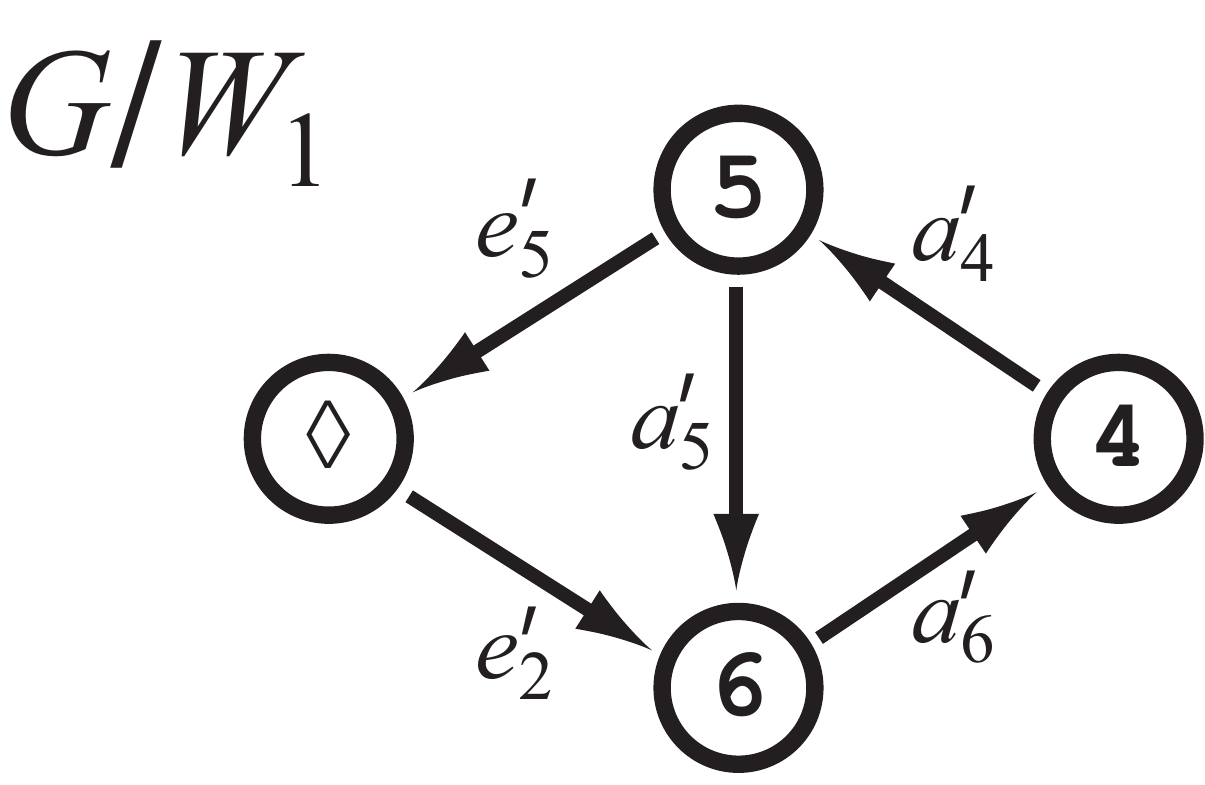}{scale=0.5}
\end{minipage}
\hspace*{0.4in}
\begin{minipage}{2.7in}
\ifig{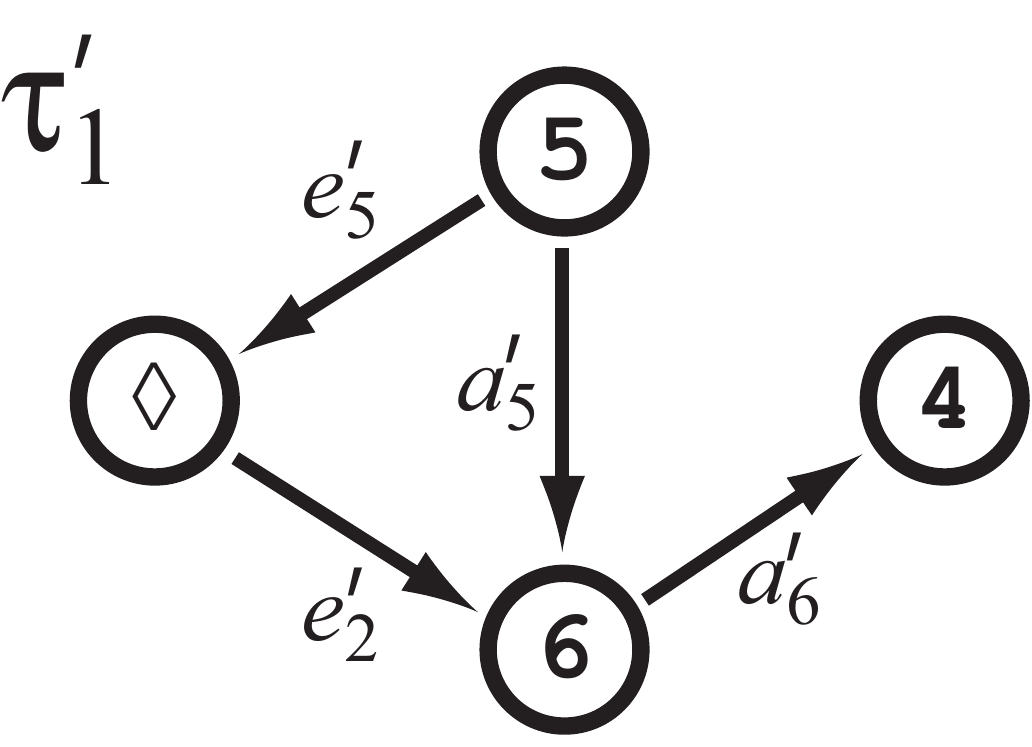}{scale=0.5}
\end{minipage}
\end{center}
\vspace*{-0.15in}
\caption[]{{\bf Left Panel:} The quotient graph $G/W_1$, with
  $G$ as in \fig{Ex241_GraphSigma} and $W_1=\{1, 2, 3\}$.\\[2pt]
  {\bf Right Panel:} The maximal strategy $\tau^\prime_1$, obtained in
  step 4(c) on page~\pageref{qloopstepb} during the first iteration of
  the loop of Construction~\ref{expandminnonfaces}, as applied to the
  graph $\mskip1muG$ and strategy $\mskip0.6mu\sigma\mskip0.8mu$ of
  \hsph\fig{Ex241_GraphSigma}.}
\label{Ex241_GWtau1}
\end{figure}

Consider again the strongly connected directed graph $G$ and maximal
strategy $\sigma$ of \fig{Ex241_GraphSigma} on
page~\pageref{Ex241_GraphSigma}.  Earlier, we viewed $G$ as a pure
nondeterministic graph with different hierarchical cyclic subgraphs.
Now, we view $G$ as a pure stochastic graph and apply
Construction~\ref{expandminnonfaces} to obtain an informative action
release sequence of length 5 for $G$, contained in $\sigma$.

\vspace*{0.1in}

For this example, it turns out that the loop of step 4 in the
construction runs once, ending with $k=1$, but without having covered
the entire state space of the graph.  Consequently, as indicated by
Theorem~\ref{longiarsstoch}'s inductive proof, one needs to invoke the
construction again, on a quotient graph.  Again, the loop runs only
once.  In total, there are three invocations of the construction.  The
synopses below show how local variables in the construction are instantiated.

\begin{enumerate}

\item In the first invocation of Construction~\ref{expandminnonfaces},
  the graph is $G$ as in \fig{Ex241_GraphSigma} and the maximal
  strategy is $\sigma=\{e_2, e_5, a_2, a_3, a_5, a_6\}$.  For $g$, one
  may use either state in $\sigma$'s goal set $\{1, 4\}$.  Using
  $g=1$, one finds $b_1=a_1$, yielding the minimal nonface
  $\kappa_1=\{a_1, a_2, a_3\}$.  Thus $W_1=\{1,2,3\}$.  The comments
  at the top of page~\pageref{commentszero} produce
  $\expansive_1=\{a_2, a_3\}$.

  Running the loop of step 4 in the construction, with $i=1$, one
  obtains $\xi_1=\{e_5, a_5, a_6\}$.  In $G/W_1$, $\vtsp\xi^\prime_1$
  has a single maximal extension, namely $\tau^\prime_1=\{e^\prime_2,
  e^\prime_5, a^\prime_5, a^\prime_6\}$.  \ \fig{Ex241_GWtau1}
  shows both $G/W_1$ and $\tau^\prime_1$, with $\wrepone$ representing
  all of $W_1$ identified to a singleton.

  Since $\tau_1\subseteq\sigma$, the loop ends with $k=1$.

\item In the second invocation of
  Construction~\ref{expandminnonfaces}, the graph is $G/W_1$ and the
  maximal strategy is $\tau^\prime_1=\{e^\prime_2, e^\prime_5,
  a^\prime_5, a^\prime_6\}$.  The strategy has goal state $4$, so let
  $g=4$.  Therefore, in this invocation of the construction,
  $b^{\mskip1mu\prime}_1=a^\prime_4$, yielding the minimal nonface
  $\kappa^\prime_1=\{a^\prime_4, a^\prime_5, a^\prime_6\}$.  (We use
  single prime notation to indicate actions in $G/W_1$, including
  references to local variables within this invocation of
  Construction~\ref{expandminnonfaces}.)

  We now write $U_1$ for $\src(\kappa^\prime_1)$, in order to avoid
  confusion with the earlier $W_1$.  Thus $U_1 = \{4, 5, 6\}$.  The
  comments at the top of page~\pageref{commentszero} produce
  $\expansive^\prime_1=\{a^\prime_5, a^\prime_6\}$.

  (Below, we will now also use double prime notation, specifically to
  indicate actions in $(G/W_1)/U_1$, including references to local
  variables within Construction~\ref{expandminnonfaces}.  We therefore
  write $\xi^\prime_1$ in place of $\xi_1$ in step 4(b) and
  $\xi^{\prime\prime}_1$ in place of $\xi^\prime_1$ in step 4(c).)

  Running the loop of step 4, with $i=1$, one obtains
  $\xi^\prime_1=\{e^{\prime}_2\}$.  In $(G/W_1)/U_1$,
  $\vtsp\xi^{\prime\prime}_1$ has a single maximal extension, namely
  itself.  We refer to that extension as
  $\rho^{\mskip0.5mu\prime\prime}_1$, in order to avoid confusion with
  the earlier $\tau_1$.  \fig{Ex241_GWUrho1} shows both $(G/W_1)/U_1$
  and $\rho^{\prime\prime}_1$, with $\wrepone$ as before and
  $\mskip1mu\wreptwo$ representing all of $U_1$ identified to a
  singleton.

  Since $\rho^\prime_1\subseteq\tau^\prime_1$, the loop ends with $k=1$.

\item Since the graph $(G/W_1)/U_1$ has only two states, one could now
  simply refer to the base case in the proof of
  Theorem~\ref{longiarsstoch}.  However, we will invoke
  Construction~\ref{expandminnonfaces} yet a third time, with graph
  $(G/W_1)/U_1$ and maximal strategy
  $\rho^{\prime\prime}_1=\{e^{\prime\prime}_2\}$.  This strategy has
  goal state $\wreptwo$, thus yielding minimal nonface
  $\kappa_1^{\prime\prime}=\{e^{\prime\prime}_2, e^{\prime\prime}_5\}$
  with source set $\{\wrepone, \wreptwo\}$.  The comments at the top
  of page~\pageref{commentszero} produce the expansive set of actions
  $\expansive^{\prime\prime}_1=\{e^{\prime\prime}_2\}$.  The loop ends
  because the source set is the entire state space, as discussed at
  the beginning of Section~\ref{iarsfromquotient} on
  page~\pageref{iarsfromquotient}.
\end{enumerate}

\vspace*{0.05in}

Finally, one assembles the various expansive sets in reverse order of
the recursive invocations of Construction~\ref{expandminnonfaces}.
This process produces the following sequence of sets of actions in
$G$:

\vspace*{-0.1in}

$$ \{e_2\}, \;\{a_5, a_6\}, \;\{a_2, a_3\}.$$

That sequence of sets represents four informative action release
sequences for $G$, each consisting of actions in $\sigma$:

\vspace*{-0.2in}

$$\begin{array}{cc}
      e_2, a_5, a_6, a_2, a_3 \;&\; e_2, a_5, a_6, a_3, a_2 \\[2pt]
      e_2, a_6, a_5, a_2, a_3 \;&\; e_2, a_6, a_5, a_3, a_2 \\
\end{array}$$

\vspace*{0.05in}

\noindent (We know from Section~\ref{multicyclegraph} \hspt that $\sigma$
also contains other informative action release sequences.)

\vspace*{0.2in}

\begin{figure}[h]
\vspace*{0.1in}
\begin{center}
\hspace*{-0.2in}
\ifig{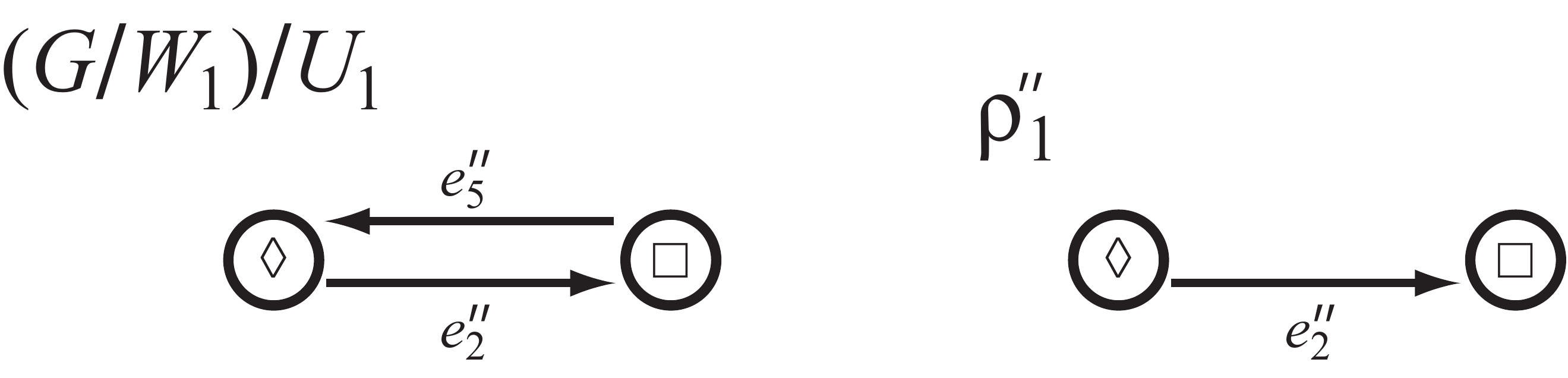}{scale=0.5}
\end{center}
\vspace*{-0.2in}
\caption[]{{\bf Left Panel:} The quotient graph $\left(G/W_1\right)\!/U_1$,
  with $G/W_1$ as in \fig{Ex241_GWtau1} and $U_1=\{4, 5, 6\}$.\quad
  {\bf Right Panel:} The maximal strategy $\rho^{\mskip0.5mu\prime\prime}_1$,
  obtained in step 4(c) during the first iteration of the loop of
  Construction~\ref{expandminnonfaces} as applied to the graph $G/W_1$
  and the strategy $\tau^\prime_1$ of \fig{Ex241_GWtau1}.\quad
  (In order to avoid overloaded letters, while retaining indices as in
  Construction~\ref{expandminnonfaces}, this figure refers to $U_1$,
  $\rho^{\prime\prime}_1$, and uses double prime notation to indicate
  actions in $\left(G/W_1\right)\!/U_1$.)}
\label{Ex241_GWUrho1}
\end{figure}

\clearpage

\begin{figure}[t]
\vspace*{-0.05in}
\begin{center}
\hspace*{0.1in}
\begin{minipage}{2.5in}
\ifig{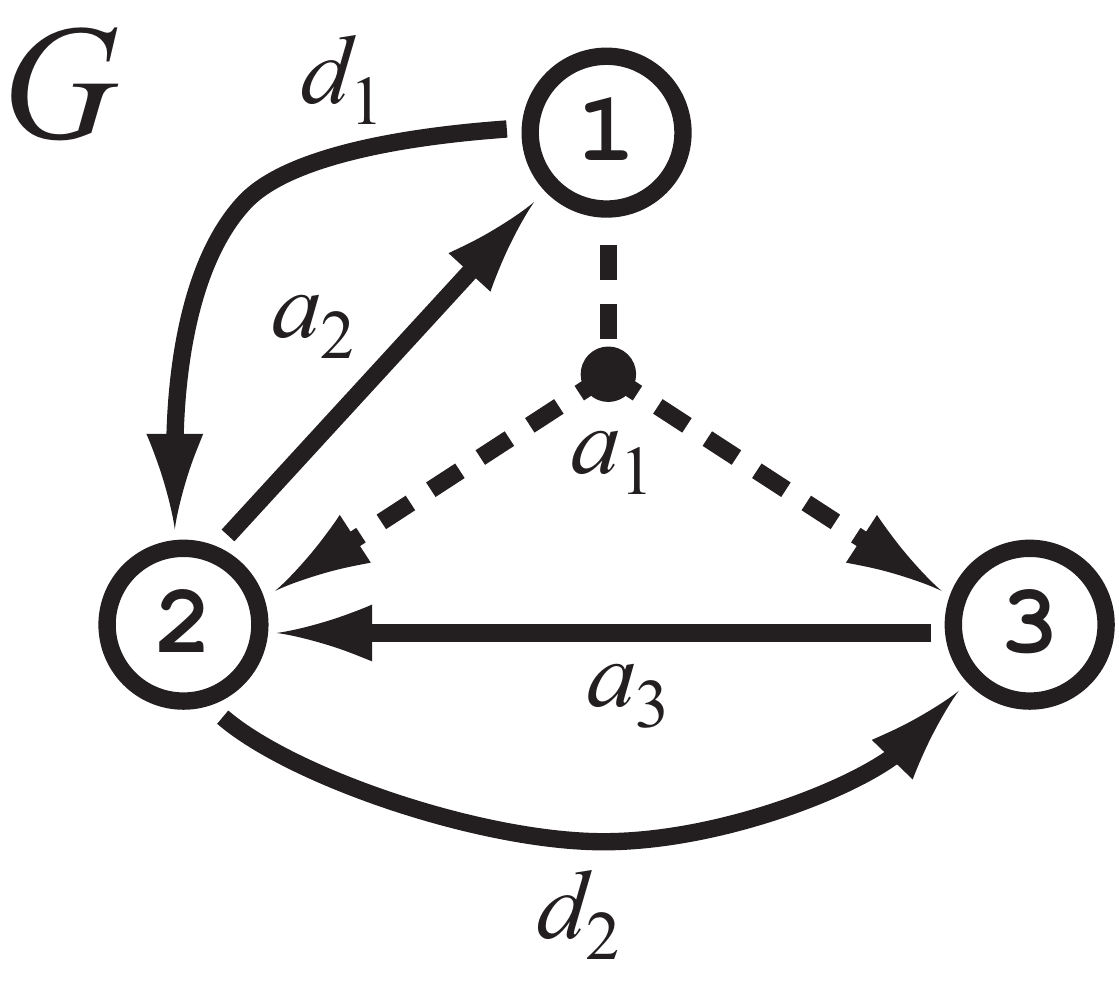}{scale=0.5}
\end{minipage}
\hspace*{0.4in}
\begin{minipage}{2.5in}
\ifig{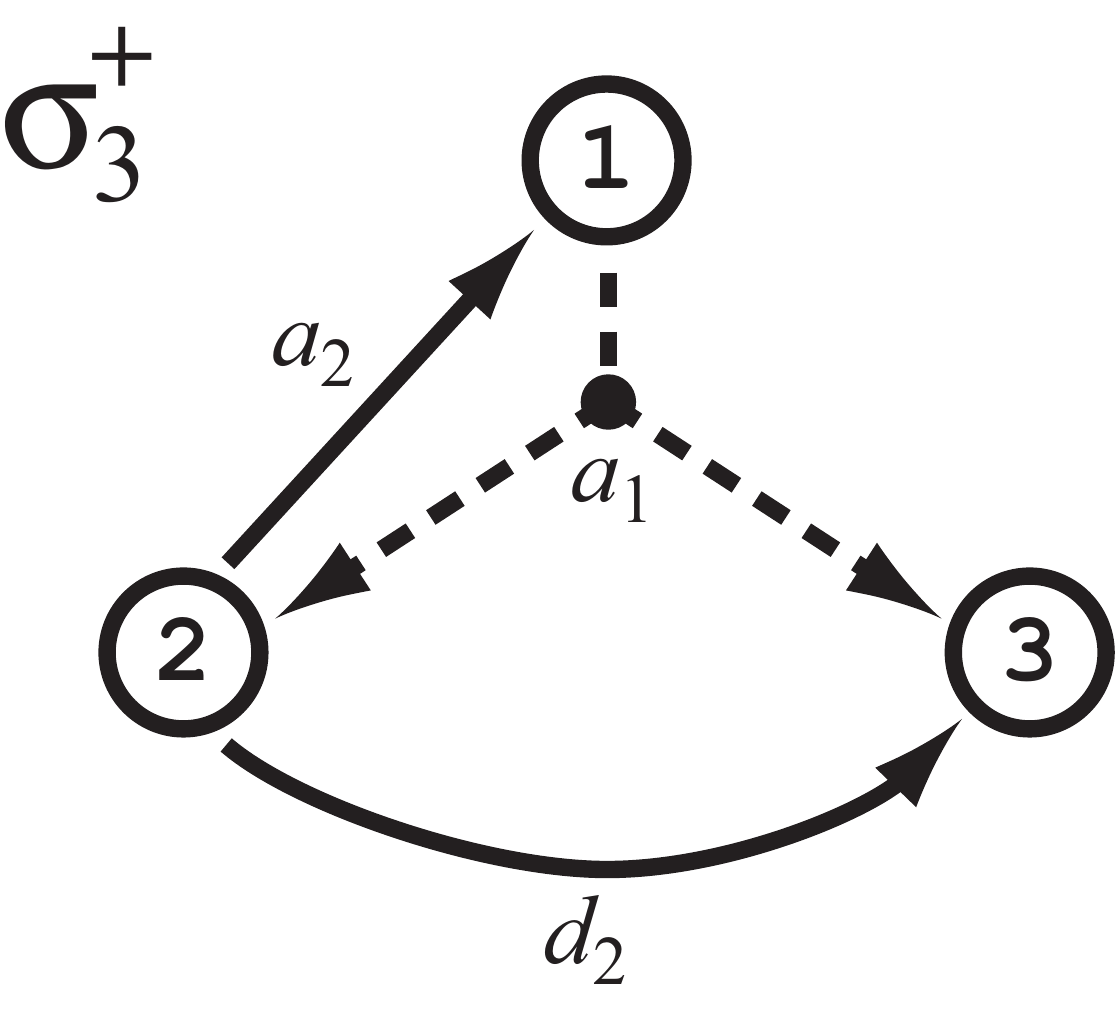}{scale=0.5}
\end{minipage}
\end{center}
\vspace*{-0.2in}
\caption[]{{\bf Left Panel:} A pure stochastic graph $G$, consisting
  of three states, four deterministic actions, and one stochastic
  action.  The stochastic action is $a_1 = 1 \rightarrow
  p\mskip1mu\{2,3\}$; its action edges appear as dashed lines.  The
  precise probability distribution $p$ is not significant here, except
  to indicate that each of the transitions $1 \rightarrow 2$ and $1
  \rightarrow 3$ has nonzero probability.\\[1pt]
  {\bf Right Panel:} \hspt{The} maximal strategy $\sigma^{+}_3\in\DG$,
  depicted by its actions.
  \quad See also \fig{Ex251_DGA}.}
\label{Ex251_GraphSigma}
\end{figure}

\begin{figure}[h]
\vspace*{-0.05in}
\begin{center}
\begin{minipage}{2.7in}
\ifig{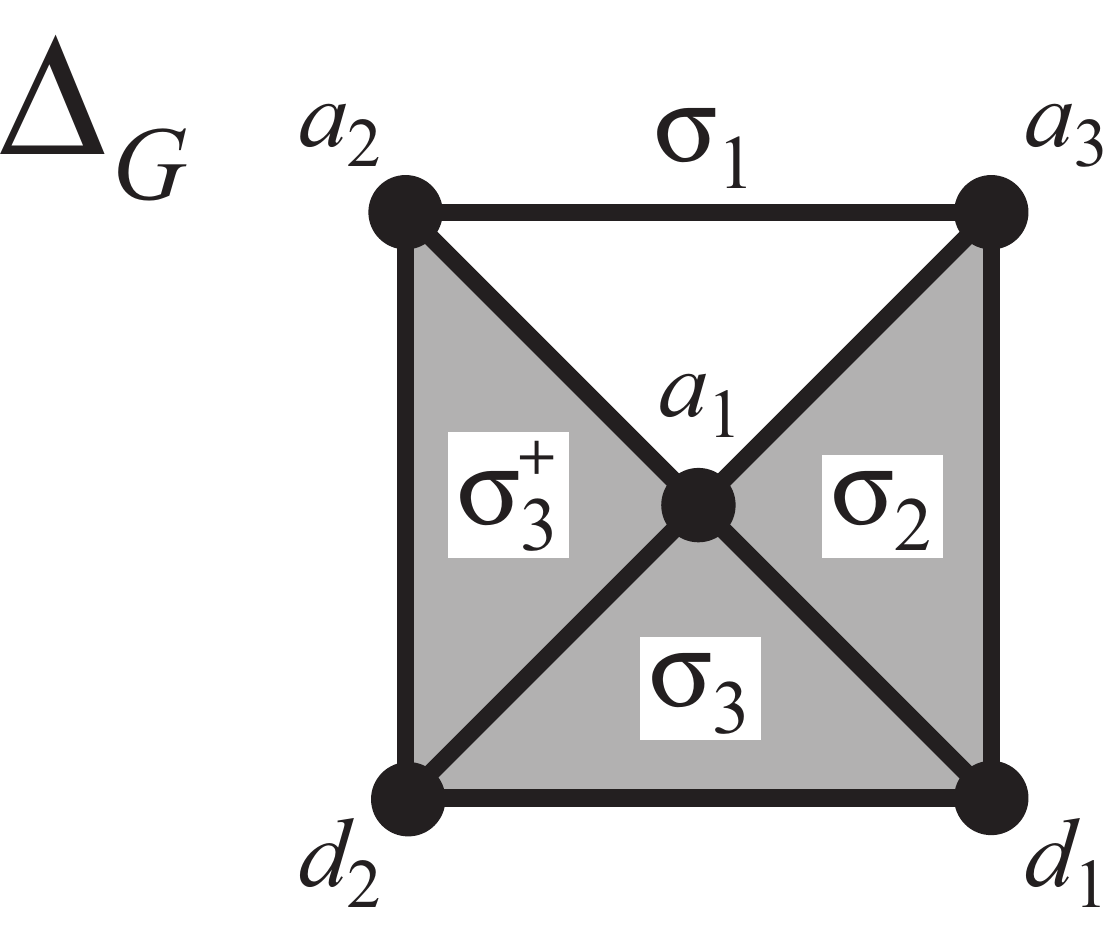}{height=1.975in}
\end{minipage}
\hspace*{0.6in}
\begin{minipage}{1.8in}{$\begin{array}{l|ccccc}
       A     & a_1  & a_2  & a_3  & d_1  & d_2 \\[2pt]\hline
\sigma_1     &      & \one & \one &      &      \\[2pt]
\sigma_2     & \one &      & \one & \one &      \\[2pt]
\sigma_3     & \one &      &      & \one & \one \\[2pt]
\sigma^{+}_3 & \one & \one &      &      & \one \\[2pt]
\end{array}$}
\end{minipage}
\begin{minipage}{0.5in}{$\begin{array}{c}
\hbox{Goal} \\[2pt]\hline
1 \\[2pt]
2 \\[2pt]
3 \\[2pt]
3 \\[2pt]
\end{array}$}
\end{minipage}
\end{center}
\vspace*{-0.2in}
\caption[]{The strategy complex $\DG$ of the graph $G$ from
  \fig{Ex251_GraphSigma} appears on the left.  \,Each maximal simplex
  is labeled with its strategy name, as specified by the relation on
  the right.}
\label{Ex251_DGA}
\end{figure}

\subsubsection{A Pure Stochastic Graph}

\fig{Ex251_GraphSigma} depicts a fully controllable pure
stochastic graph $G$, along with a maximal strategy that contains an
action with stochastic transitions.  The strategy complex $\DG$
appears in \fig{Ex251_DGA}, along with $G$'s action relation.
The maximal strategy under consideration is $\sigma^{+}_3$.

Unlike in a pure nondeterministic graph, cycling is permitted in a
pure stochastic graph, so long as the cycling is transient.  The
definition of ``moves off'' from page~\pageref{movesoff} captures this
distinction.  For instance, in the current example, $\{a_1, a_2\}$ is
a convergent (nonmaximal) strategy, with goal state \#3.  The set of
actions $\{a_1, a_2\}$ would {\em not{\kern .10em}} be convergent if
action $a_1$ were nondeterministic, since then an adversary could
force infinite cycling between states \#1 and \#2.  However, $a_1$ is
stochastic, so there is a nonzero probability that the system will
exit such a cycle, transitioning to state \#3 instead.  The precise
transition probabilities of action $a_1$ affect expected convergence
times, as discussed in \cite{paths:plans, paths:strategies}, but not
overall convergence.

\begin{figure}[t]
\begin{center}
\hspace*{0.1in}
\begin{minipage}{2in}
\ifig{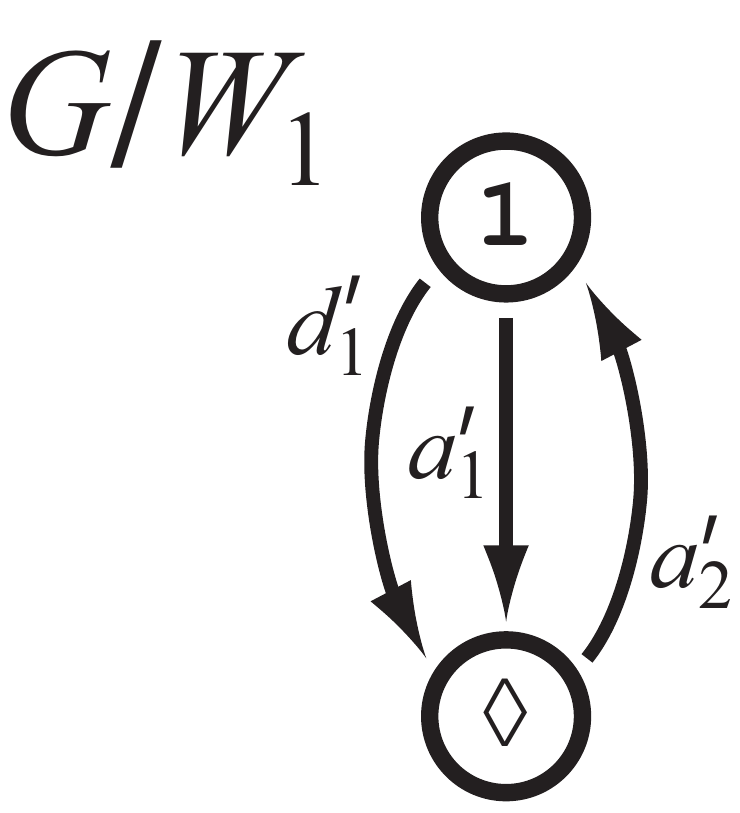}{scale=0.5}
\end{minipage}
\hspace*{0.4in}
\begin{minipage}{2in}
\ifig{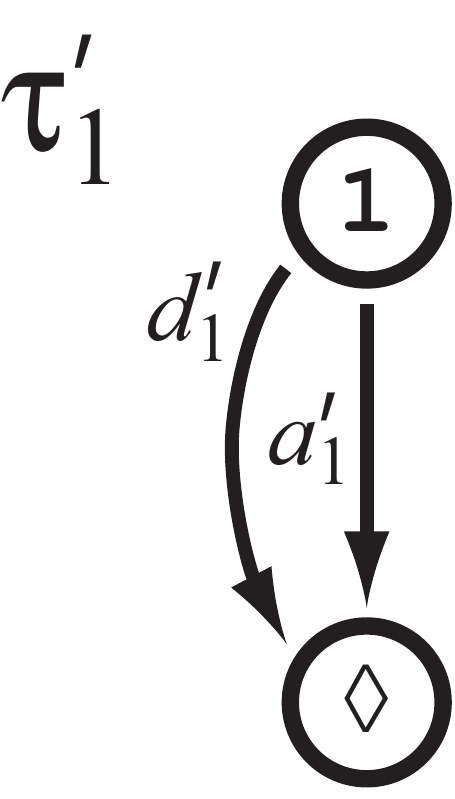}{scale=0.5}
\end{minipage}
\end{center}
\vspace*{-0.15in}
\caption[]{{\bf Left Panel:} The quotient graph $G/W_1$, with
  $G$ as in \fig{Ex251_GraphSigma} and $W_1=\{2, 3\}$.\\[2pt]
  {\bf Right Panel:} The maximal strategy $\tau^\prime_1$, obtained in
  step 4(c) during the first iteration of the loop of
  Construction~\ref{expandminnonfaces}, as applied to the graph $G$
  and strategy $\sigma^{+}_3$ of \fig{Ex251_GraphSigma}.}
\label{Ex251_GWtau1}
\end{figure}

\noindent Given $G$ and $\sigma=\sigma^{+}_3=\{a_1, a_2, d_2\}$ as in
\fig{Ex251_GraphSigma}, Construction~\ref{expandminnonfaces} computes
as follows:\label{examplestoch}

\vspace*{-0.025in}

\begin{enumerate}

\item There is one state outside $\src(\sigma)=\{1,2\}$, so $g=3$.  \
      Action $b_1$ in the construction should be an action with source
      $g$, meaning it is action $a_3$ of \fig{Ex251_GraphSigma}.

\item One may then use minimal nonface $\kappa_1=\{d_2, a_3\}$.
  (Another possibility is $\{a_1, a_2, a_3\}$.)

\item So $\calA_1 = \{d_2, a_3\}$ and $W_1 = \{2, 3\}$.

\item Now the loop of the construction runs:

\vspace*{0.05in}

\hspace*{0.05in}\begin{minipage}{5.64in}

\begin{itemize}

\item[$i=1$:]

\begin{enumerate}

\item \!\fig{Ex251_GWtau1} depicts graph $G/W_1$, omitting
      the self-looping actions at state $\wrepone$.

\item $\xi_1=\setdef{a\in\sigma}{\src(a)\not\in W_1}=\{a_1\}$.

\item $\xi^\prime_1$ is not maximal in
      $\Delta_{G/W_{\scriptstyle{1}}}$.  It has unique maximal extension
      $\tau^\prime_1 = \{a^\prime_1, d^{\mskip1mu\prime}_1\}$.

\item $\tau_1 \not\subseteq \sigma$.  Since
      $\sdiff{\tau_1}{\sigma}=\{d_1\}$, $\,b_2=d_1$.  Thus $\kappa_2 =
      \{d_1, a_2\}$, $\calA_2 = \calA_1 \union \kappa_2 = \{d_2, a_3,
      d_1, a_2\}$, and $W_2 = W_1 \union \src(\kappa_2) = \{1, 2,
      3\}$.

\end{enumerate}

\vspace*{0.05in}

\item[$i=2$:]

\begin{enumerate}

\item Graph $G/W_2 = (\{\Box\}, \emptyset)$, with $\Box$ representing
      all states of $G$ identified to a singleton.  \ So
      $\Delta_{G/W_{\scriptstyle{2}}}=\{\emptyset\}$, the empty
      simplicial complex.

\item $\xi_2=\emptyset$.

\item $\xi^\prime_2$ is maximal in $\Delta_{G/W_{\scriptstyle{2}}}$,
      so $\tau^\prime_2=\emptyset$.

\item $\tau_2\subseteq\sigma$, so the loop ends, with $k=2$.

\end{enumerate}

\end{itemize}

\end{minipage}

\end{enumerate}

\noindent Lemma~\ref{fullexp} and subsequent comments construct
expansive sets $\expansive_1$ and $\expansive_2$ as follows (variable
bindings for $g$, $W_1$, $W_2$, $b_1$, $b_2$, $\kappa_1$, and
$\kappa_2$ are as above, actions $d_2$ and $a_2$ are as in
\fig{Ex251_GraphSigma}):

\vspace*{0.035in}

\hspace*{-0.175in}\begin{minipage}{6.1in}

   \begin{enumerate}

   \setcounter{enumi}{-1}
   \addtolength{\itemsep}{-4pt}

   \item Let $W_0=\{g\}=\{3\}$.

   \item Since $\sdiff{W_1}{W_0}= \{2\} = \src(\sdiff{\kappa_1\!}{\!\{b_1\}})$,
         $\;\expansive_1 = \sdiff{\kappa_1\!}{\!\{b_1\}} = \{d_2\}$.
         \ (See also page~\pageref{commentszero}.)

   \item While $\sdiff{W_2}{W_1}=\{1\}\neq\{2\}=\src(\{a_2\})=
         \src(\sdiff{\kappa_2\!}{\!\{b_2\}})$,
         $\vmsp\src(b_2)\in\trg(a_2)$, so $\expansive_2 = \{a_2\}$.

   \end{enumerate}

\end{minipage}

\vspace*{0.1in}

By Corollary~\ref{expandingiarsG}, the sequence $\expansive_2,
\expansive_1$ is informative for $G$, with all actions contained in
$\sigma^{+}_3$.  We thus obtain the informative action release
sequence $\,a_2, d_2$.  \ Side note: This is not the only iars
contained in $\sigma^{+}_3$.  The longest such iars consists of all
actions in $\sigma^{+}_3$, in the order $\,a_1, d_2, a_2$.

\vspace*{-0.1in}

\paragraph{Interpretation:}\   \fig{Ex251_A2} depicts $\calA_2$
as a graph.  The graph consists of two independent two-cycles.
Relationally, we may think of these two two-cycles as two independent
bits of information, forming a basis for informative action release
sequences.  In more general examples (see
Section~\ref{expansivedependence}), there may be less independence.
Consequently, Lemma~\ref{fullexp} (page~\pageref{fullexp}) constructs
expansive sets of actions, which Lemma~\ref{expandingiars}
(page~\pageref{expandingiars}) then arranges informatively.

\begin{figure}[h]
\vspace*{0.1in}
\begin{center}
\ifig{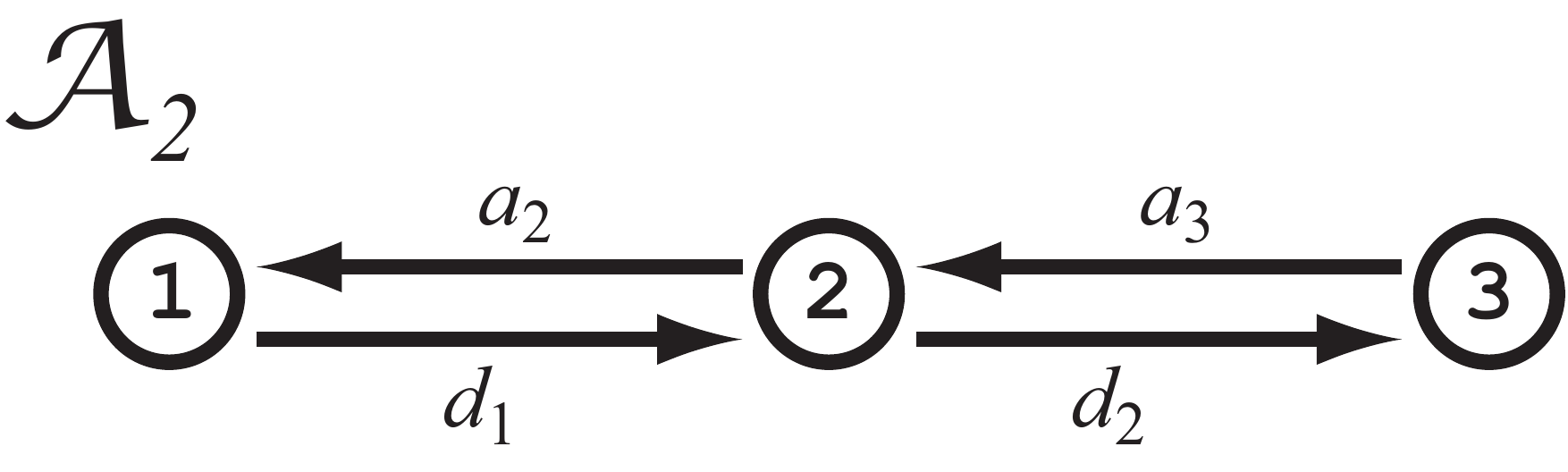}{scale=0.425}
\end{center}
\vspace*{-0.25in}
\caption[]{The set of actions $\calA_2$ viewed as a graph.
  Construction~\ref{expandminnonfaces} produces this set when applied
  to graph $G$ and maximal strategy $\sigma^{+}_3$ of
  \fig{Ex251_GraphSigma} in the manner discussed on
  page~\pageref{examplestoch}.  The actions of $\calA_2$ contained in
  $\sigma^{+}_3$ form an informative action release sequence for $G$.
  \ (In fact, any ordering of any convergent set of actions in
  $\calA_2$ is an iars contained in some strategy, by independence of
  the two two-cycles in $\calA_2$.)}
\label{Ex251_A2}
\end{figure}

\subsubsection{A Pure Stochastic Graph Highlighting Expansive Set Order}
\label{expansivedependence}

\begin{figure}[h]
\vspace*{-0.05in}
\begin{center}
\hspace*{0.05in}
\begin{minipage}{1.6in}
\ifig{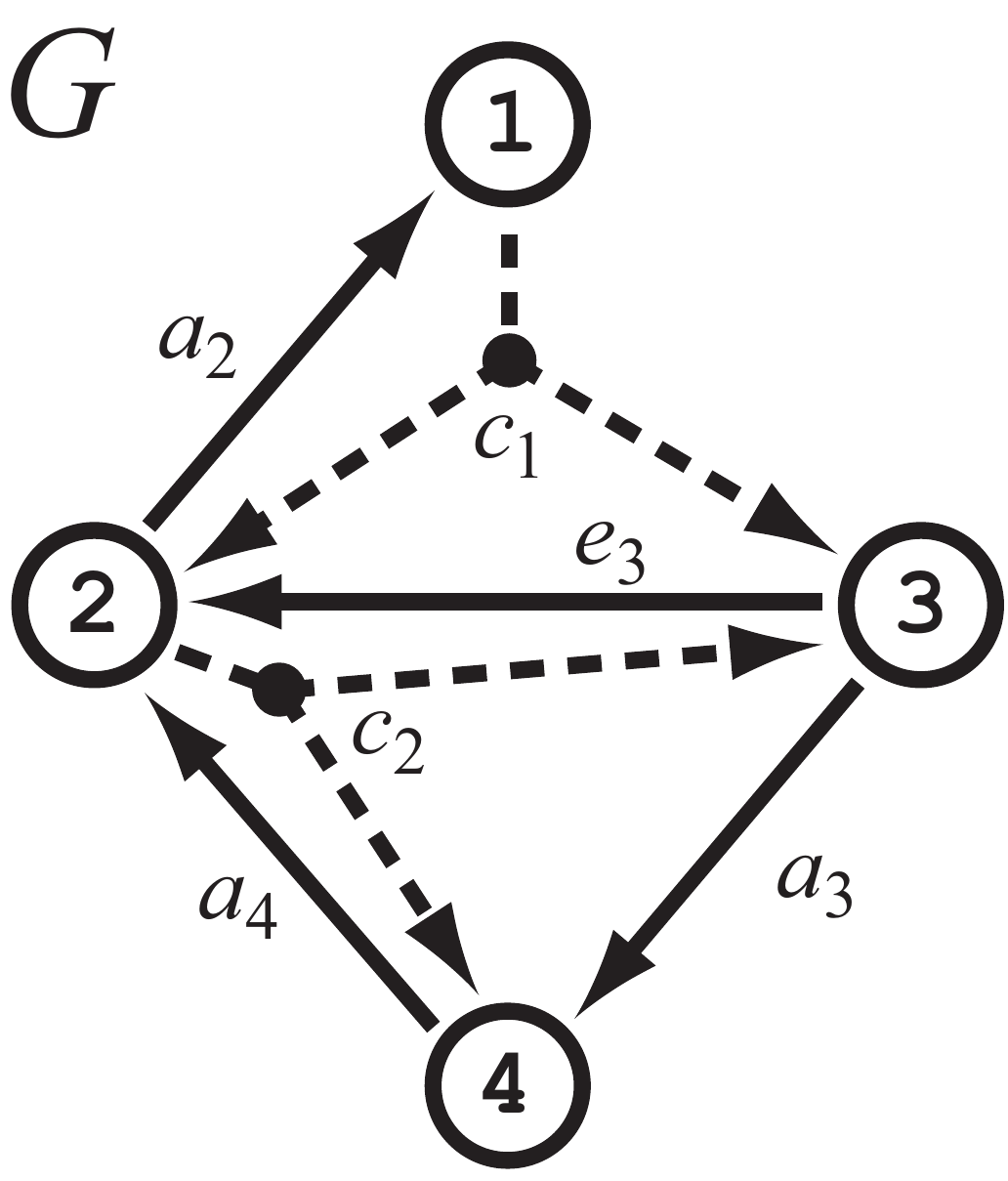}{scale=0.35}
\end{minipage}
\hspace*{0.05in}
\begin{minipage}{1.9in}{$\begin{array}{c|cccccc}
A      & a_2  & a_3  & a_4  & e_3  & c_1  & c_2  \\[2pt]\hline
       & \one & \one & \one & \one &      &      \\[2pt]
       &      & \one & \one & \one & \one &      \\[2pt]
       & \one &      & \one &      & \one & \one \\[2pt]
       &      & \one &      & \one & \one & \one \\[2pt]
       & \one & \one &      &      & \one & \one \\[2pt]
\sigma & \one & \one &      & \one &      & \one \\[2pt]
\end{array}$}
\end{minipage}
\begin{minipage}{0.45in}{$\begin{array}{c}
\hbox{Goal} \\[2pt]\hline
1 \\[2pt]
2 \\[2pt]
3 \\[2pt]
4 \\[2pt]
4 \\[2pt]
\{1,4\} \\[2pt]
\end{array}$}
\end{minipage}
\hspace*{0.375in}
\begin{minipage}{1.5in}
\ifig{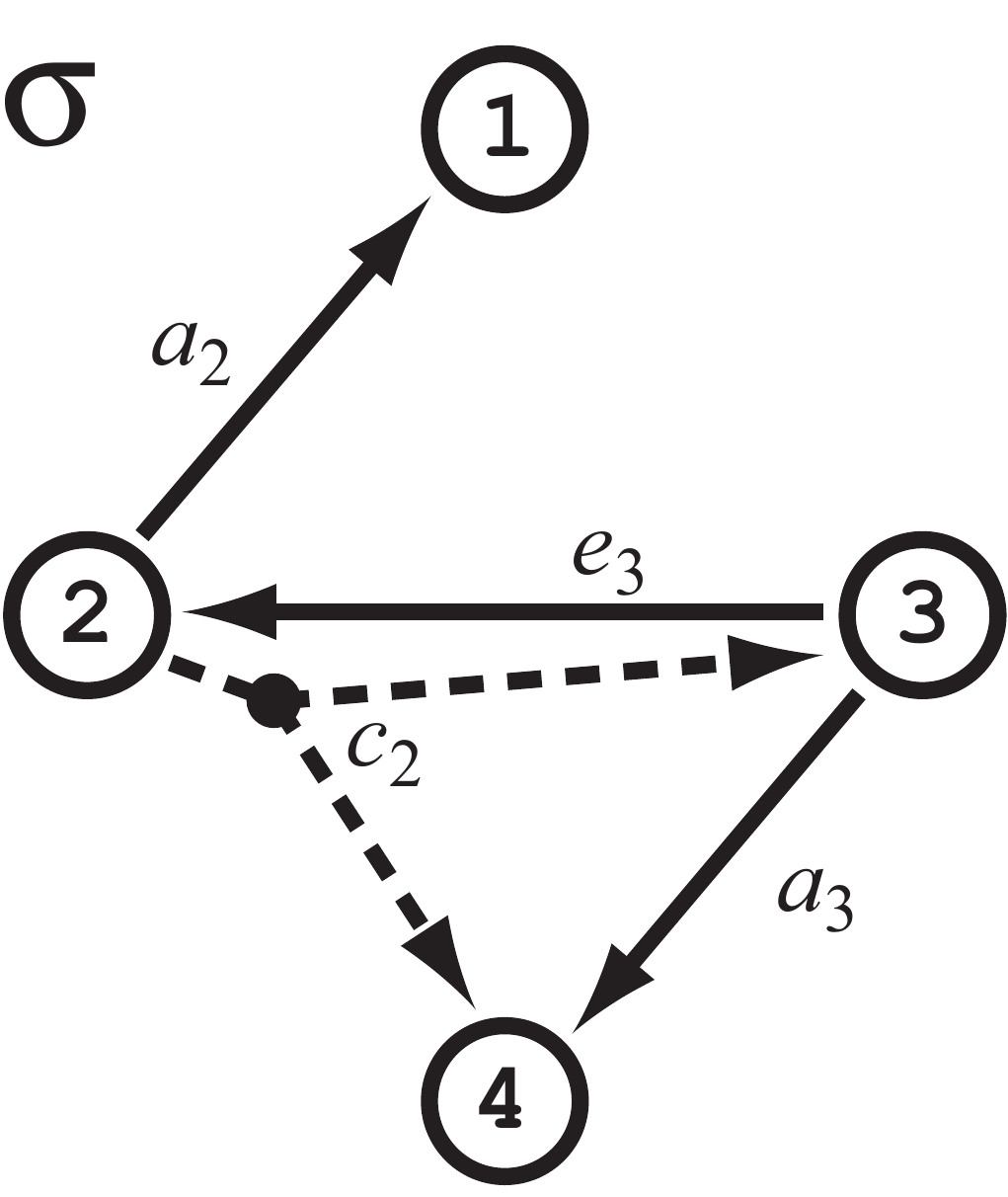}{scale=0.35}
\end{minipage}
\end{center}
\vspace*{-0.2in}
\caption[]{The left panel displays a pure stochastic graph $G$,
  consisting of four states, $\,1, 2, 3, 4$, \ four deterministic
  actions, $a_2$, $a_3$, $a_4$, $e_3$, \ and two stochastic actions,
  $c_1$, $c_2$.  \ The center panel shows $G$'s action relation.  \
  The right panel depicts maximal strategy $\sigma\in\DG$, via its
  actions.}
\label{ExH5}
\end{figure}

We now apply Construction~\ref{expandminnonfaces} to the pure
stochastic graph $G$ and maximal strategy $\sigma\in\DG$ of
\fig{ExH5}.  We may let $g=1$.  $\mskip1mu$Then $\mskip1mub_1=c_1$,
$\mskip3mu\kappa_1 = \{c_1, a_2, e_3\}$, and $W_1 = \{1,2,3\}$.  Thus
$\expansive_1=\{a_2,e_3\}$.  Now $\xi_1=\emptyset$, so there is a
choice in constructing $\tau^\prime_1$.  If we choose
$\tau^\prime_1=\{a^\prime_4\}$, then $\tau_1\not\subseteq \sigma$ and
so $b_2=a_4$.  There are two minimal nonfaces within $\sigma \union
\{a_4\}$.  Let us use $\kappa_2=\{a_4, c_2, a_3\}$.  Thus $W_2 =
\{1,2,3,4\}$ and the loop ends with k=2.  Constructing $\expansive_2$
involves a choice since actions $c_2$ and $a_3$ each have action
$a_4\mskip-1.5mu$'s source as a target.
Let us pick $\expansive_2=\{a_3\}$.

Corollary~\ref{expandingiarsG} on page~\pageref{expandingiarsG}
arranges the expansive sets of actions in the order $\expansive_2,
\expansive_1$.  Indeed, both $\,a_3, a_2, e_3\,$ and $\,a_3, e_3,
a_2\,$ are informative action release sequences for $G$.  Notice that
action $e_3$ implies action $a_3$ in relation $A$.  Consequently, the
order $\expansive_1, \expansive_2$ would \vtsp{\em not\,} be
acceptable.

\vst

Comment: \hsph{We} can lengthen $\,a_3, a_2, e_3\,$ to the iars
$\,c_2, a_3, a_2, e_3$.  \hsph{In} fact, 12 of the 24 possible
permutations of all the actions in $\mskip0.5mu\sigma$ constitute
informative action release sequences for $G$.

\clearpage
\section{Counterexamples for Mixed Graphs}
\markright{Counterexamples for Mixed Graphs}
\label{mixed}

The assertion of Theorem~\ref{longiars} on page~\pageref{longiars}
need not hold for graphs containing a mix of deterministic,
nondeterministic, and stochastic actions.  Of course, there are many
settings in which the assertion does hold.  For instance, if one can
find a hierarchical cyclic subgraph with the same state space as the
given graph, then one can again prove the theorem, even if the graph
contains stochastic actions.  All the proof needs is for the
hierarchical cyclic subgraph to be composed only of deterministic and
nondeterministic actions.  Absent such structure, it is very easy to
construct a counterexample involving a mix of deterministic,
nondeterministic, and stochastic actions.  With some added effort, one
may also construct counterexamples involving maximal strategies that
attain singleton goals.  This section presents such counterexamples.

\subsection{A Counterexample with a Large Goal Set}
\markright{A Counterexample with a Large Goal Set}

\begin{figure}[h]
\begin{center}
\begin{minipage}{2.7in}
\ifig{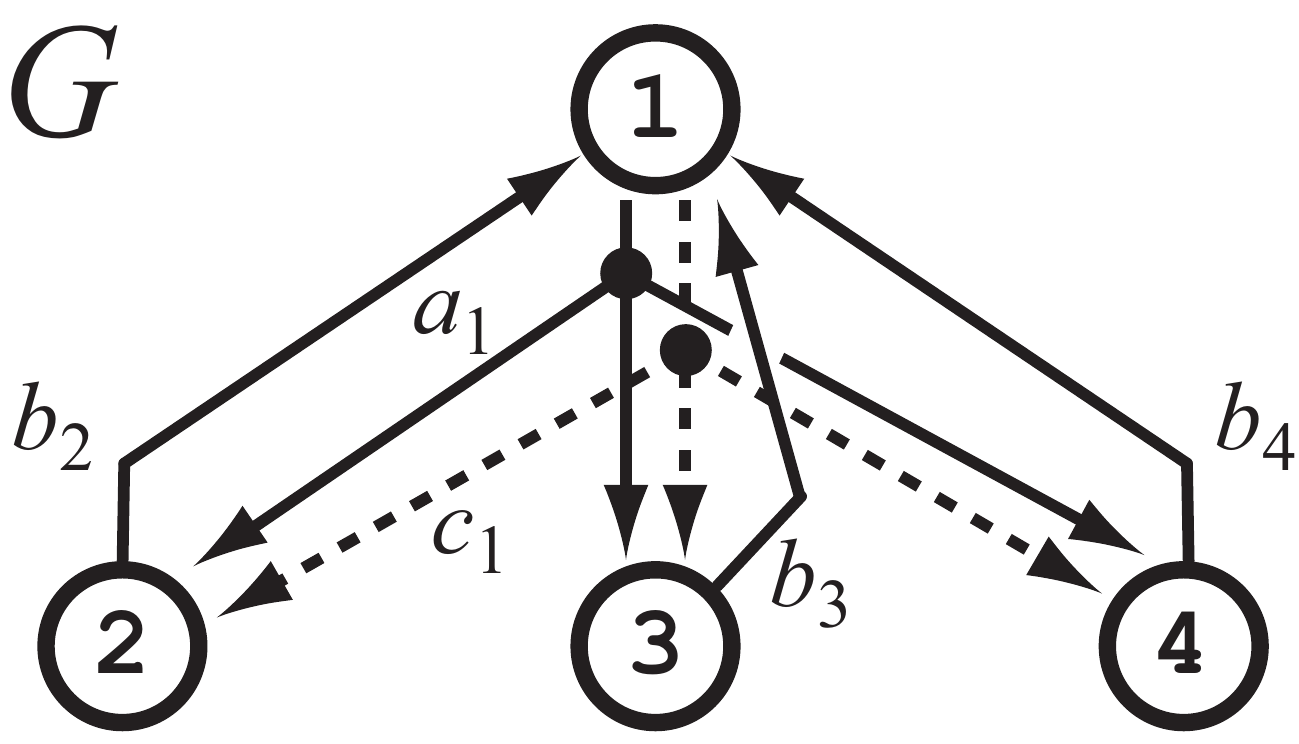}{scale=0.5}
\end{minipage}
\hspace*{0.4in}
\begin{minipage}{1.7in}{$\begin{array}{c|ccccc}
       A     & a_1  & c_1  & b_2  & b_3  & b_4  \\[2pt]\hline
\sigma_1     &      &      & \one & \one & \one \\[2pt]
\sigma_2     &      & \one &      & \one & \one \\[2pt]
\sigma_3     &      & \one & \one &      & \one \\[2pt]
\sigma_4     &      & \one & \one & \one &      \\[2pt]
\sigma_{234} & \one & \one &      &      &      \\[2pt]
\end{array}$}
\end{minipage}
\quad
\begin{minipage}{0.5in}{$\begin{array}{c}
\hbox{Goal} \\[2pt]\hline
1 \\[2pt]
2 \\[2pt]
3 \\[2pt]
4 \\[2pt]
\{2, 3, 4\} \\[2pt]
\end{array}$}
\end{minipage}
\end{center}
\vspace*{-0.1in}
\caption[]{\ {\bf Left Panel:} \ A graph $G$ with four states, $1, 2,
  3, 4$, \ three deterministic actions, $b_2$, $b_3$, $b_4$, \ one
  nondeterministic action, $a_1$, \ and one stochastic action,
  $c_1$.\\[2pt]
  {\bf Right Panel:} \,$G$'s action relation and goal sets.}
\label{Ex238_top}
\end{figure}

Consider the graph and action relation of \fig{Ex238_top}.  The
graph contains two actions that have identical action edges, but
differ in that one action is stochastic and the other action is
nondeterministic.  The stochastic action is $c_1 = 1 \rightarrow
p\mskip1.25mu\{2,3,4\}$, for some probability distribution $p$
ascribing nonzero probabilities to each target, and the
nondeterministic action is $a_1 = 1 \rightarrow \{2,3,4\}$.
Additionally, the graph contains three deterministic action, $b_i = i
\rightarrow 1$, for $i=2,3,4$.

\vst

The graph is fully controllable based just on the set of actions
$\{c_1, b_2, b_3, b_4\}$, as one can see from the action relation or
as follows: The system can attain state \#1 from any other state by
using strategy $\{b_2, b_3, b_4\}$.  The system can reach a desired
state in the set $\{2, 3, 4\}$ by repeatedly trying to do so using the
stochastic action $c_1$, cycling back to state \#1 if that action
transitions to the wrong state.  For instance, strategy $\{c_1, b_2,
b_3\}$ will converge to state \#4.

\vst

None of the actions $\{b_i\}$ can be in a strategy together with
action $a_1$, but action $c_1$ can be.  In fact, $\sigma_{234}=\{c_1,
a_1\}$ is a maximal strategy.  The longest informative action release
sequence contained in $\sigma_{234}$ is the strategy itself, revealed
in the order $\,c_1, a_1$.  That iars has length 2, which is less than
the number 3 demanded by Theorem~\ref{longiars}.

\vst

One may easily generalize this example to graphs with $n$ states, for
$n>4$, such that some maximal strategy consists of only two actions
and therefore has an iars of length at most 2.

Key to this example is the oddity of having two nearly identical
actions, with the only difference being that one action is stochastic
and the other is nondeterministic.  The graph would {\em not{\kern
.12em}} be fully controllable with just the nondeterministic action.
The stochastic action is needed to ``sample with replacement'', i.e.,
``try and try again, until success''.

From a worst-case perspective, the strategy consisting of the nearly
identical stochastic and nondeterministic actions amounts to no more
than the nondeterministic action itself.  So, why even include the
nondeterministic action in the graph?

The answer is that it is a choice a system may make.  Executing the
stochastic action may entail greater cost than executing the
nondeterministic action, because the nondeterministic action relieves
the system of guaranteeing stochastic behavior.  That may be desirable
in some settings.  At first it seems hardly so, because the only goal
set one can attain using any strategy containing the nondeterministic
action is a very large set (consisting of $n-1$ states in the
generalized version).  However, not caring about precise transitions
is sensible when the graph is part of a larger graph and it does not
matter what state the system passes through as a subgoal while
attaining some overall goal.  The next subsection explores such graphs
further.

\subsection{A Counterexample with a Small Goal Set}
\markright{A Counterexample with a Small Goal Set}

Previously, we saw the basis for a family of counterexamples in which
the graph has $n$ states but contains a maximal strategy consisting of
two actions with a goal set of size $n-1$.  One might therefore
hypothesize that Theorem~\ref{longiars} should merely assert the
existence of an informative action release sequence of length $n-k$,
with $k$ being the size of the goal set.  In fact, such a theorem
would also be false, since one can construct counterexamples to
Theorem~\ref{longiars} using strategies that have goal sets of size 1,
as this subsection demonstrates.

\begin{figure}[h]
\begin{center}
\begin{minipage}{2.6in}
\ifig{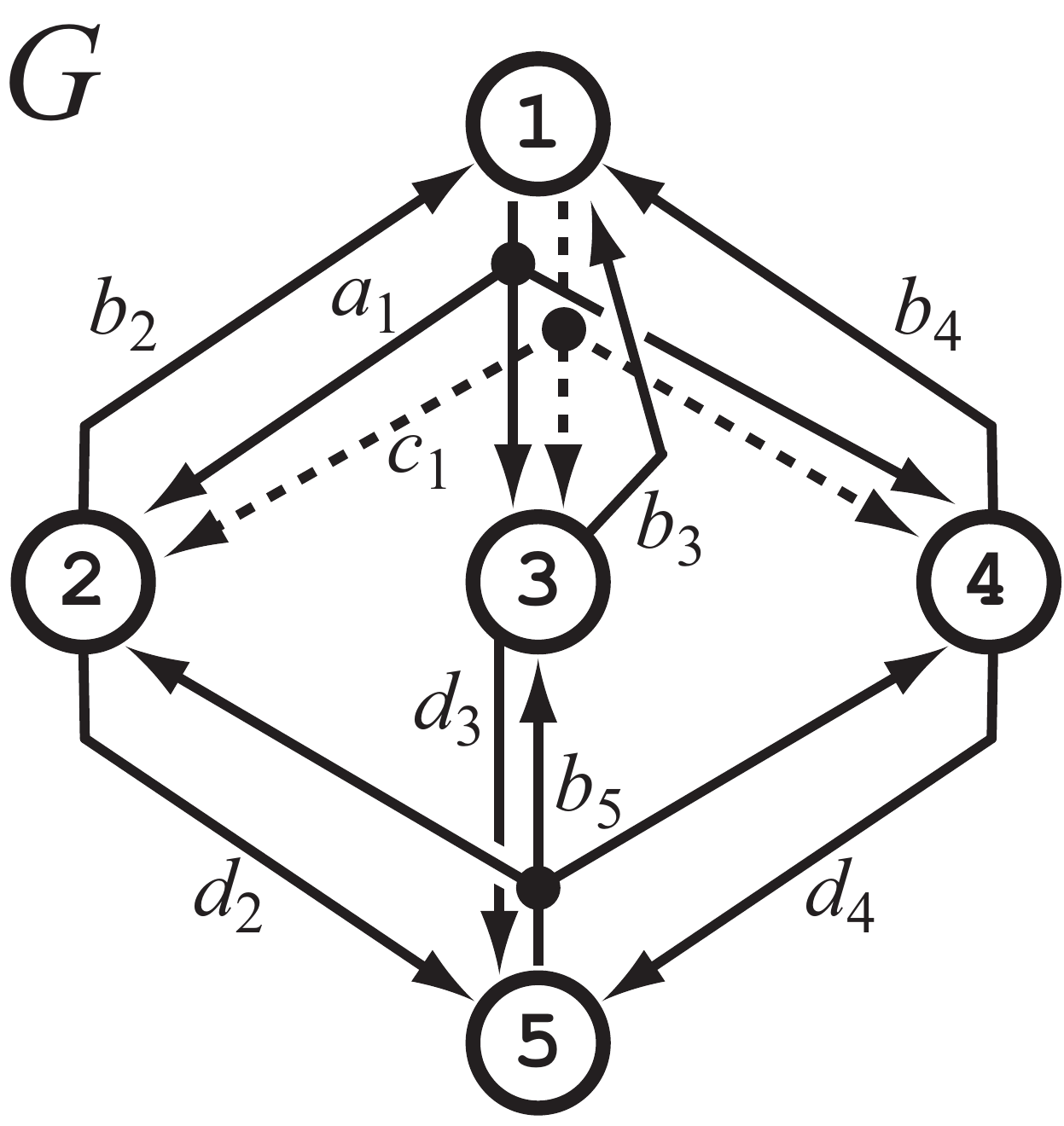}{scale=0.5}
\end{minipage}
\begin{minipage}{2.85in}{$\begin{array}{c|ccccccccc}
       A     & a_1  & d_2  & d_3  & d_4  & c_1  & b_2  & b_3  & b_4  & b_5  \\[2pt]\hline
\sigma_1     &      &      &      &      &      & \one & \one & \one & \one \\[2pt]
\sigma_2     &      &      &      &      & \one &      & \one & \one & \one \\[2pt]
\sigma_3     &      &      &      &      & \one & \one &      & \one & \one \\[2pt]
\sigma_4     &      &      &      &      & \one & \one & \one &      & \one \\[2pt]
\sigma_{234} & \one &      &      &      & \one &      &      &      & \one \\[2pt]
\sigma_5     & \one & \one & \one & \one & \one &      &      &      &      \\[2pt]
\sigma_{54}  &      & \one & \one & \one & \one & \one & \one &      &      \\[2pt]
\sigma_{53}  &      & \one & \one & \one & \one & \one &      & \one &      \\[2pt]
\sigma_{52}  &      & \one & \one & \one & \one &      & \one & \one &      \\[2pt]
\sigma_{15}  &      & \one & \one & \one &      & \one & \one & \one &      \\[2pt]
\end{array}$}
\end{minipage}
\begin{minipage}{0.6in}{$\begin{array}{c}
\hbox{Goal} \\[2pt]\hline
1 \\[2pt]
2 \\[2pt]
3 \\[2pt]
4 \\[2pt]
\{2, 3, 4\} \\[2pt]
5 \\[2pt]
5 \\[2pt]
5 \\[2pt]
5 \\[2pt]
\{1,5\} \\[2pt]
\end{array}$}
\end{minipage}
\end{center}
\vspace*{-0.2in}
\caption[]{{\bf Left Panel:} A graph with five states, $1,2,3,4,5$,
\,six deterministic actions, \hbox{$d_2$, $d_3$, $d_4$, $b_2$, $b_3$, $b_4$},
\ two nondeterministic actions, $a_1$, $b_5$, \ and one stochastic
action, $c_1$.\\[3pt]
  {\bf Right Panel:} \,The graph's action relation and goal sets.\\[2pt]
  (This figure is a copy, with minor notational changes, of Figures 67
  and 68 in \cite{paths:privacy}.)}
\label{Ex238_GraphArel}
\end{figure}

In constructing such counterexamples, one may take the fully
controllable graph on $n$ states of the previous subsection and glue
the set $\{2, \ldots, n\}$ to another graph, permitting direct motion
from any state in $\{2, \ldots, n\}$ to a new state, \#$(n+1)$.  In
order to retain full controllability, the system also needs to be able
to move back from state \#$(n+1)$.  In this subsection, we use a
single nondeterministic action.  In the next subsection, will will use
$n-1$ different nondeterministic actions.  Numerous variations exist.

\vst

\fig{Ex238_GraphArel} replicates a counterexample taken from
\cite{paths:privacy}, showing a graph and its action relation.
Maximal strategy $\sigma_5=\{a_1, d_2, d_3, d_4, c_1\}$ converges to
singleton goal state \#5.  The strategy contains 5 actions, but the
longest informative action release sequences contained in $\sigma_5$
have length 3, which is less than the 4 demanded by
Theorem~\ref{longiars}.  The reason no longer iars exists is because
any one of the ``downward'' actions in the set $\{d_2, d_3, d_4\}$
implies the other two.  That fact is clear from the action relation,
but can also be understood as follows: First, knowing that a strategy
contains one of the actions $d_2$, $d_3$, or $d_4$ means the strategy
cannot contain the nondeterministic action $b_5$.  Second, for a
maximal strategy, not containing $b_5$ means the strategy must contain
the entire set $\{d_2, d_3, d_4\}$.

\subsection{A Counterexample with a Small Goal Set and Nonequivalent Inferences}
\markright{A Counterexample with a Small Goal Set and Nonequivalent Inferences}

Finally, we construct a counterexample similar to that of
\fig{Ex238_GraphArel}, but without requiring equivalence between the
deterministic downward actions flowing into state \#$(n+1)$.  Instead,
any $\mskip-0.5mu${\em two\hspt} of these downward actions will imply
all the downward actions.  One may achieve this inference by replacing
the single nondeterministic action at state \#$(n+1)$ with $n-1$
different nondeterministic actions.  Each of these actions now has a
target set of size $n-2$ contained within the set $\{2, \ldots, n\}$.
With this counterexample in mind, one may see yet another infinite
family of counterexamples, parameterized now by the number of downward
actions $\{d_i\}_{i=2}^{n}$ that may be released before all are
implied (with $n$ sufficiently large).

\vspace*{0.1in}

\fig{Ex239_Graph} shows a graph $G$ in three panels.  \fig{Ex239_Arel}
shows $G$'s action relation.  There are six states.  Four of the
graph's deterministic actions, namely $d_2$, $d_3$, $d_4$, $d_5$,
transition {\em to\vmsp} state \#6, while four nondeterministic
actions, $e_2$, $e_3$, $e_4$, $e_5$, transition {\em away\vmsp} from
state \#6.  Each of those nondeterministic actions has a target set of
size three that is a subset of $\{2, 3, 4, 5\}$.  The following table
shows which pairings of $d_i$ and $e_j$ actions create minimal
nonfaces in $\DG$.  One sees that any single action drawn from $\{d_2,
d_3, d_4, d_5\}$ is potentially consistent with strategies not
involving any other $d_i$ action, but that any two of the $\{d_i\}$
actions imply them all.  (Any two of the $\{d_i\}$ eliminate all
$\{e_j\}$, so the given maximal strategy must contain all $\{d_i\}$.)

\vspace*{-0.05in}

$$\begin{array}{c|cccc}
    & e_5  & e_4  & e_3  & e_2  \\[2pt]\hline
d_2 & \one & \one & \one &      \\[2pt]
d_3 & \one & \one &      & \one \\[2pt]
d_4 & \one &      & \one & \one \\[2pt]
d_5 &      & \one & \one & \one \\[2pt]
\end{array}$$

\vspace*{0.05in}

Maximal strategy $\sigma\mskip-4mu=\mskip-4mu\{a_1, d_2, d_3, d_4,
d_5, c_1\}$ converges to goal state
\hspace*{-0.5pt}\#6\hspace*{-0.3pt} and contains
$\mskip-0.05mu6\mskip-0.75mu$ actions.  However, since at most two of
the actions $\{d_i\}$ are informative, the longest informative action
release sequences contained in $\sigma$ have length 4.  That is less
than the 5 demanded by Theorem~\ref{longiars}.

\begin{figure}[h]
\vspace*{-0.3in}
\setlength{\fboxrule}{0.75pt}
\begin{center}
\fbox{\begin{minipage}{3.25in}
\begin{center}
\ifig{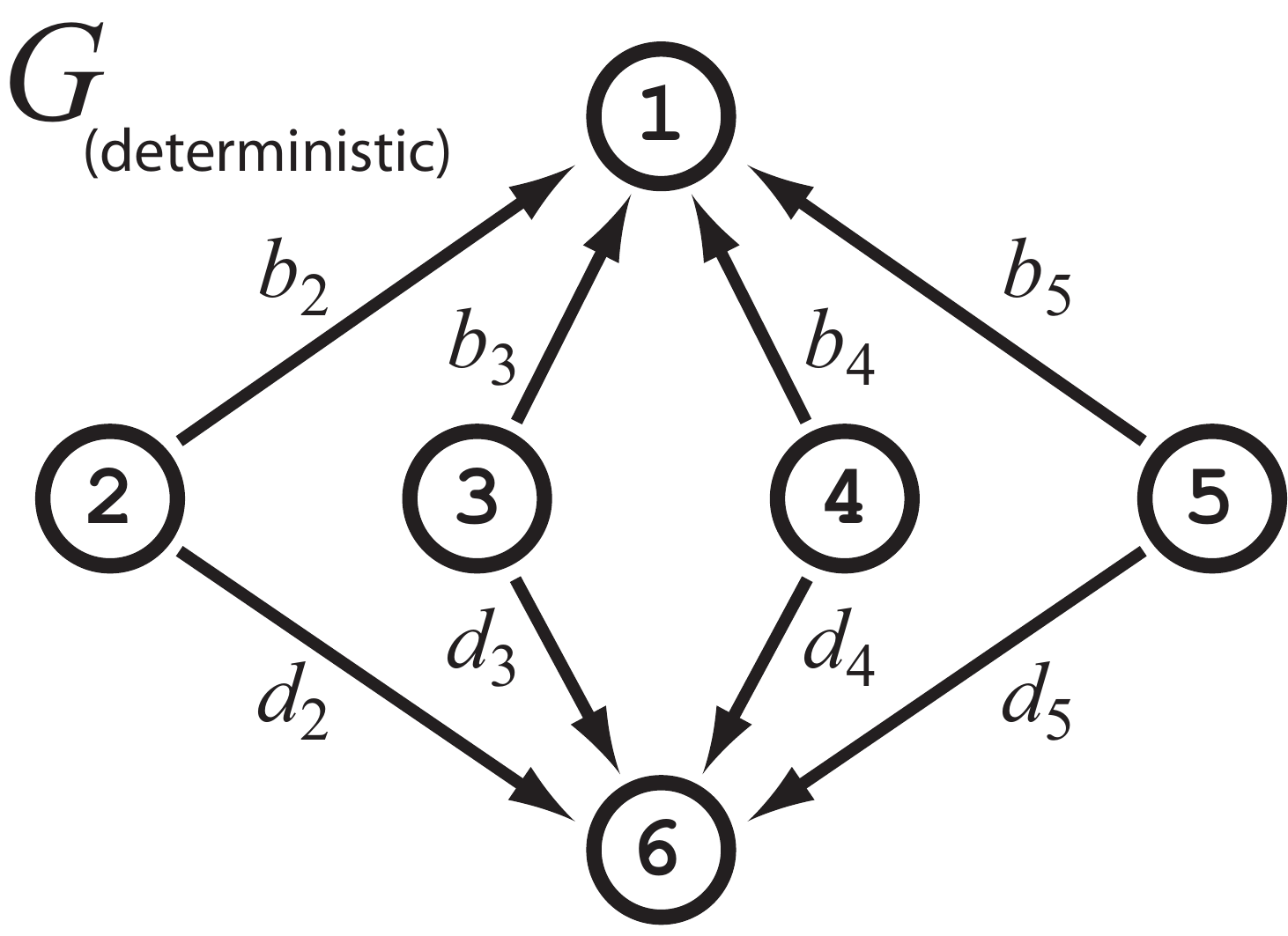}{scale=0.5}
\end{center}
\end{minipage}}\\
\vspace*{0.1in}
\fbox{\begin{minipage}{3.25in}
\begin{center}
\ifig{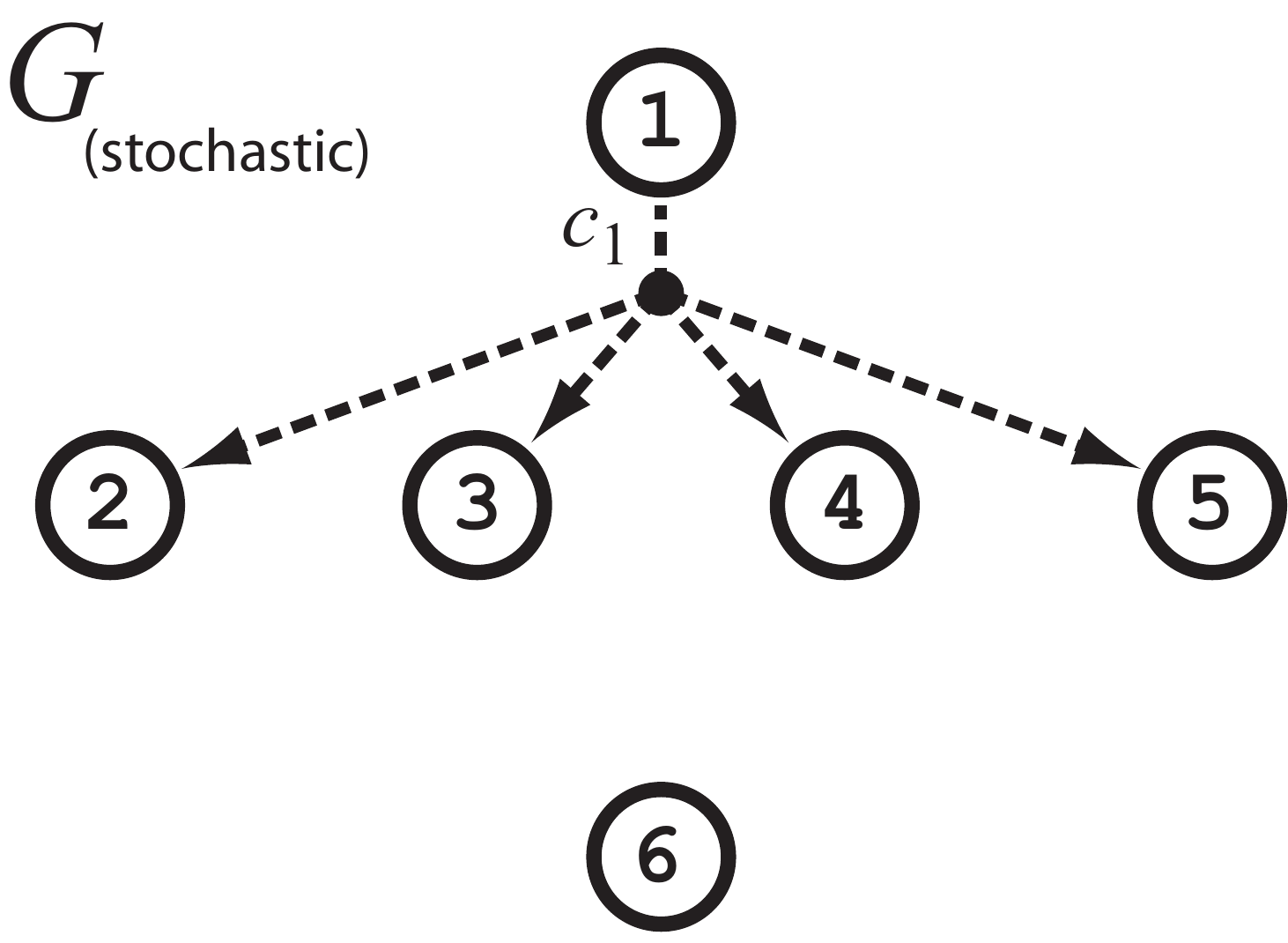}{scale=0.5}
\end{center}
\end{minipage}}\\
\vspace*{0.1in}
\fbox{\begin{minipage}{3.25in}
\begin{center}
\ifig{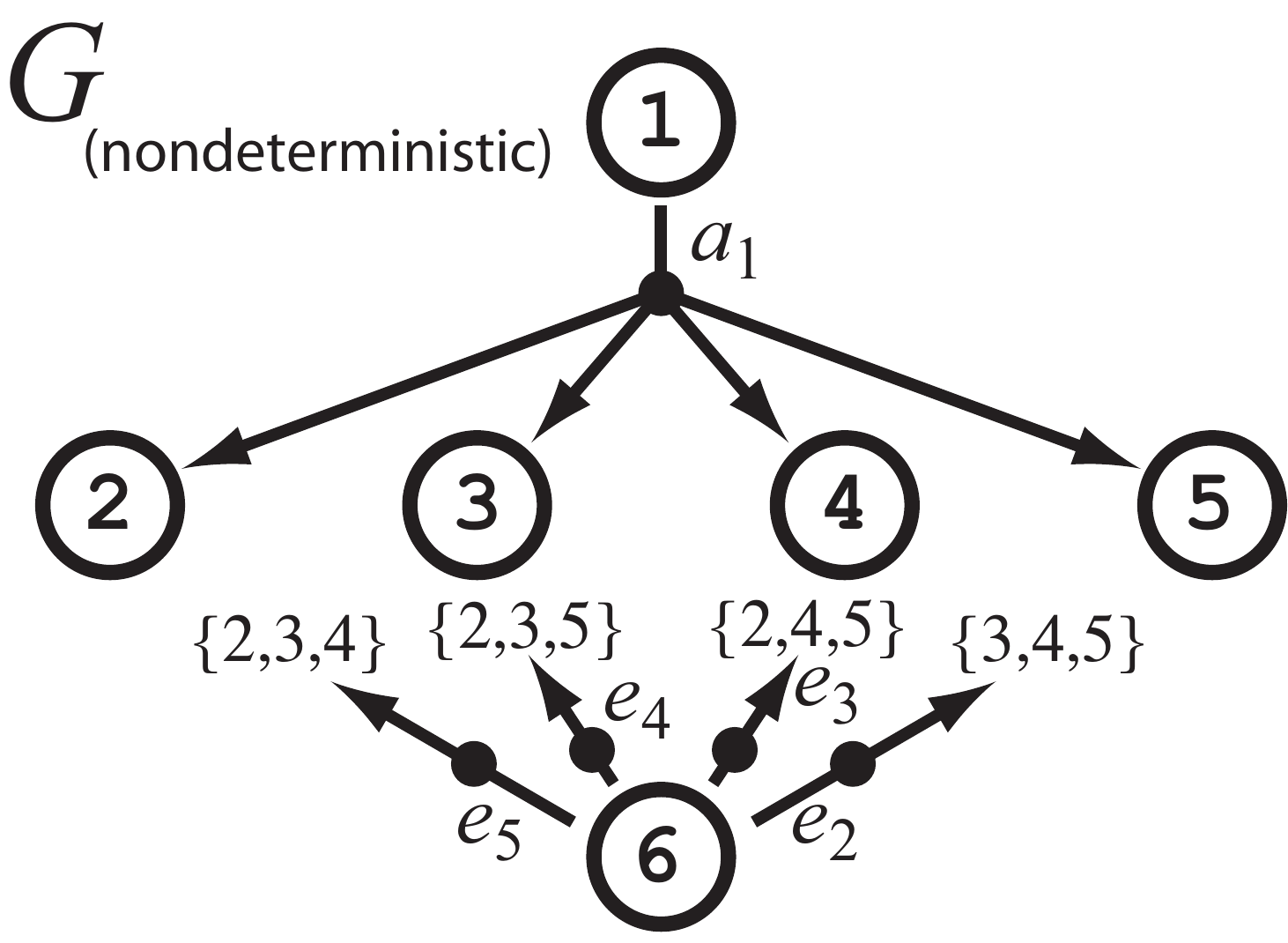}{scale=0.5}
\end{center}
\end{minipage}}
\end{center}
\vspace*{-0.2in}
\caption[]{A graph $G$ with eight deterministic actions, one
  stochastic action, and five nondeterministic actions, depicted in
  three panels.  \ \fig{Ex239_Arel} displays $G$'s action
  relation.  \\[3pt]
\small
  {\bf Top:}\ The top panel shows $G$'s deterministic actions.  Action
  $b_i$ moves ``\underline{b}ack'' from state \#$i$ to state \#1,
  while action $d_i$ moves ``\underline{d}own'' from state \#$i$ to
  state \#6, for $i=2,3,4,5$.  \\[3pt]
  {\bf Middle:}\ The middle panel shows $G$'s stochastic action $c_1$,
  with source state \#1 and target set $\{2,3,4,5\}$.\\[3pt]
  {\bf Bottom:}\ The bottom panel shows $G$'s nondeterministic actions.
  Action $a_1$ has the same source and targets as the stochastic
  action, but is nondeterministic.  The remaining four actions each
  have source state \#6, and some target set of size three in the set of
  states $\{2,3,4,5\}$.  Specifically, the target set of action $e_i$
  is $\{2,3,4,5\}$ but with state \#$i$ ``\underline{e}xcised'', for
  $i=2,3,4,5$.  The figure displays these four actions in abbreviated
  form, with written target sets rather than all the arrows drawn.}
\label{Ex239_Graph}
\end{figure}

\begin{figure}[h]
\begin{center}
\begin{minipage}{4.5in}{$\begin{array}{c|cccccccccccccc}
  A     & a_1  & d_2  & d_3  & d_4  & d_5  & c_1  & b_2  & b_3  & b_4  & b_5  & e_5  & e_4  & e_3  & e_2   \\[2pt]\hline
        &      &      &      &      &      &      & \one & \one & \one & \one & \one & \one & \one & \one \\[2pt]
        &      &      &      &      &      & \one &      & \one & \one & \one & \one & \one & \one & \one \\[2pt]
        &      &      &      &      &      & \one & \one &      & \one & \one & \one & \one & \one & \one \\[2pt]
        &      &      &      &      &      & \one & \one & \one &      & \one & \one & \one & \one & \one \\[2pt]
        &      &      &      &      &      & \one & \one & \one & \one &      & \one & \one & \one & \one \\[2pt]
        & \one &      &      &      &      & \one &      &      &      &      & \one & \one & \one & \one \\[2pt]
        &      &      &      &      & \one &      & \one & \one & \one & \one & \one &      &      &      \\[2pt]
        &      &      &      &      & \one & \one &      & \one & \one & \one & \one &      &      &      \\[2pt]
        &      &      &      &      & \one & \one & \one &      & \one & \one & \one &      &      &      \\[2pt]
        &      &      &      &      & \one & \one & \one & \one &      & \one & \one &      &      &      \\[2pt]
        & \one &      &      &      & \one & \one &      &      &      &      & \one &      &      &      \\[2pt]
        &      &      &      & \one &      &      & \one & \one & \one & \one &      & \one &      &      \\[2pt]
        &      &      &      & \one &      & \one &      & \one & \one & \one &      & \one &      &      \\[2pt]
        &      &      &      & \one &      & \one & \one &      & \one & \one &      & \one &      &      \\[2pt]
        &      &      &      & \one &      & \one & \one & \one & \one &      &      & \one &      &      \\[2pt]
        & \one &      &      & \one &      & \one &      &      &      &      &      & \one &      &      \\[2pt]
        &      &      & \one &      &      &      & \one & \one & \one & \one &      &      & \one &      \\[2pt]
        &      &      & \one &      &      & \one &      & \one & \one & \one &      &      & \one &      \\[2pt]
        &      &      & \one &      &      & \one & \one & \one &      & \one &      &      & \one &      \\[2pt]
        &      &      & \one &      &      & \one & \one & \one & \one &      &      &      & \one &      \\[2pt]
        & \one &      & \one &      &      & \one &      &      &      &      &      &      & \one &      \\[2pt]
        &      & \one &      &      &      &      & \one & \one & \one & \one &      &      &      & \one \\[2pt]
        &      & \one &      &      &      & \one & \one &      & \one & \one &      &      &      & \one \\[2pt]
        &      & \one &      &      &      & \one & \one & \one &      & \one &      &      &      & \one \\[2pt]
        &      & \one &      &      &      & \one & \one & \one & \one &      &      &      &      & \one \\[2pt]
        & \one & \one &      &      &      & \one &      &      &      &      &      &      &      & \one \\[2pt]
        &      & \one & \one & \one & \one &      & \one & \one & \one & \one &      &      &      &      \\[2pt]
        &      & \one & \one & \one & \one & \one &      & \one & \one & \one &      &      &      &      \\[2pt]
        &      & \one & \one & \one & \one & \one & \one &      & \one & \one &      &      &      &      \\[2pt]
        &      & \one & \one & \one & \one & \one & \one & \one &      & \one &      &      &      &      \\[2pt]
        &      & \one & \one & \one & \one & \one & \one & \one & \one &      &      &      &      &      \\[2pt]
\sigma  & \one & \one & \one & \one & \one & \one &      &      &      &      &      &      &      &      \\[2pt]
\end{array}$}
\end{minipage}
\begin{minipage}{0.5in}{$\begin{array}{c}
\hbox{Goal} \\[2pt]\hline
1 \\[2pt]
2 \\[2pt]
3 \\[2pt]
4 \\[2pt]
5 \\[2pt]
\{2,3,4,5\} \\[2pt]
1 \\[2pt]
2 \\[2pt]
3 \\[2pt]
4 \\[2pt]
\{2,3,4\} \\[2pt]
1 \\[2pt]
2 \\[2pt]
3 \\[2pt]
5 \\[2pt]
\{2,3,5\} \\[2pt]
1 \\[2pt]
2 \\[2pt]
4 \\[2pt]
5 \\[2pt]
\{2,4,5\} \\[2pt]
1 \\[2pt]
3 \\[2pt]
4 \\[2pt]
5 \\[2pt]
\{3,4,5\} \\[2pt]
\{1,6\} \\[2pt]
6 \\[2pt]
6 \\[2pt]
6 \\[2pt]
6 \\[2pt]
6 \\[2pt]
\end{array}$}
\end{minipage}
\end{center}
\caption[]{Action relation and goal sets for the graph $G$ of
           \fig{Ex239_Graph}.  The maximal strategy in the
           bottommost row is labeled $\sigma$ for reference in the
           text.}
\label{Ex239_Arel}
\end{figure}

\clearpage
\section*{Acknowledgments}
\addcontentsline{toc}{section}{Acknowledgments}

Many thanks to Rob Ghrist, Steve LaValle, Ben Mann, and Matt Mason,
for advice and support related to this research and its earlier
foundations within SToMP.

\vspace*{0.2in}

\markright{References}


\end{document}